\documentclass[11 pt, reqno]{amsart}
\usepackage[margin=1in]{geometry}
\usepackage{amssymb, amsmath, amsthm}
\usepackage{hyperref}
\usepackage[all]{xy}

\usepackage{tikz}
\usetikzlibrary{trees}
\usepackage{forest}

\usepackage{pdflscape}
\usepackage{rotating} 

\usepackage{mathpazo}
\usepackage{eulervm}

\usepackage{enumitem}
\setlist[itemize]{leftmargin=0.25in}

\graphicspath{{inkscape-pics/}}

\makeatletter
\newtheorem*{rep@theorem}{\rep@title}
\newcommand{\newreptheorem}[2]{%
\newenvironment{rep#1}[1]{%
 \def\rep@title{#2 \ref{##1}}%
 \begin{rep@theorem}}%
 {\end{rep@theorem}}}
\makeatother

\newtheorem{theorem}{Theorem}[section]
\newreptheorem{theorem}{Theorem}
\newtheorem{lemma}[theorem]{Lemma}
\newreptheorem{lemma}{Lemma}
\newtheorem{proposition}[theorem]{Proposition}
\newtheorem{corollary}[theorem]{Corollary}

\theoremstyle{definition}
\newtheorem{remark}[theorem]{Remark}

\newtheorem{definition}[theorem]{Definition}
\newtheorem{definitions}[theorem]{Definitions}
\newtheorem{example}[theorem]{Example}
\newtheorem{exercise}[theorem]{Exercise}
\newtheorem{notation}[theorem]{Notation}

\def\Q{\mathbb{Q}}
\def\R{\mathbb{R}}
\def\C{\mathbb{C}}
\def\Z{\mathbb{Z}}

\def\V{\mathcal{V}}
\def\incl{\hookrightarrow}
\def\to{\rightarrow}

\def\id{\mathrm{id}}
\def\x{\times}
\def\d{\partial}
\def\phi{\varphi}

\def\Aut{\mathrm{Aut}}
\def\Emb{\mathrm{Emb}}

\def\Diff{\mathrm{Diff}}
\def\Isom{\mathrm{Isom}}
\def\Conf{\mathrm{Conf}}

\def\B{\mathcal{B}}
\def\PB{\mathcal{PB}}

\def\K{\mathcal{K}}
\def\tK{\widetilde{\mathcal{K}}}
\def\L{\mathcal{L}}
\def\tL{\widetilde{\mathcal{L}}}
\def\T{\mathcal{T}}
\def\im{\mathrm{im}}

\def\Wh{\mathrm{Wh}}
\def\tgamma{\widetilde{\gamma}}

\makeatletter
\@namedef{subjclassname@2020}{%
  \textup{2020} Mathematics Subject Classification}
\makeatother

    \title[Spaces of knots in the solid torus, knots in the thickened torus, and links in the 3-sphere]{Spaces of knots in the solid torus, knots in the thickened torus, and links in the 3-sphere}
    \author{Andrew Havens}
    \author{Robin Koytcheff}
    \email{havens@math.umass.edu}
    \email{koytcheff@louisiana.edu}
    \address{Department of Mathematics \& Statistics, University of Massachusetts Amherst, Amherst, MA, USA 01003}
    \address{Department of Mathematics, University of Louisiana at Lafayette, Lafayette, LA, USA 70504}

\begin{document}

\begin{abstract}
%In this paper a problem of Arnol'd is solv'd.
We recursively determine the homotopy type of the space of any irreducible framed link in the 3-sphere, modulo rotations.  This leads us to the homotopy type of the space of any knot in the solid torus, thus answering a question posed by Arnold.  We similarly study spaces of unframed links in the 3-sphere, modulo rotations, and spaces of knots in the thickened torus.   The subgroup of meridional rotations splits as a direct factor of the fundamental group of the space of any framed link except the unknot.  Its generators can be viewed as generalizations of the Gramain loop in the space of long knots.  Taking the quotient by certain such rotations relates the spaces we study.  All of our results generalize previous work of Hatcher and Budney.  We provide many examples and explicitly describe generators of fundamental groups.
\end{abstract}

\maketitle

\vspace{-1pc}

\tableofcontents

\section{Introduction}
This paper concerns spaces of embeddings of 1-manifolds into certain 3-dimensional submanifolds of the 3-sphere.  This generalizes the study of knot types from path components to spaces.  Hatcher described the topology of spaces of knots using his resolution of the Smale conjecture.  He established results for torus and hyperbolic knots, and Budney later addressed all knot types by using the satellite decomposition of knots.  
Arnold \cite[Problem 1970-14]{ArnoldsProblems} posed the following problem: 
\begin{quote}
Evaluate the fundamental group of the space of embeddings of a circle into a solid torus (the answer is a knot invariant!).
\end{quote}
We solve this problem in detail, by adapting the methods of Hatcher and Budney to study spaces of links in $S^3$.  In turn, that allows us to describe spaces of knots in a thickened torus, a special case of knots in a thickened surface and thus virtual knot theory.  
We have tried to make our account as self-contained as possible.  We include many examples, especially of links corresponding to knots in the solid torus, with pictures and descriptions of generators of the fundamental groups of their embedding spaces.

\subsection{Context: previous results of Hatcher and Budney}
Before outlining our results, we review results on the space $\Emb(S^1, S^3)$ of (closed) knots in $S^3$ and the space $\Emb(\R, \R^3)$ of long knots in $\R^3$.
We use a subscript $f$ to denote the component of a knot $f:S^1 \incl S^3$ or a long knot $f:\R \to \R^3$ in these spaces.
Hatcher proved the following \cite{HatcherKnotSpacesArxiv, HatcherKnotSpacesWebsite}:

\begin{itemize}
\item If $f$ is the unknot, $\Emb_f(S^1, S^3) \simeq SO_4/SO_2$, while $\Emb_f(\R, \R^3) \simeq \ast$.
\item If $f$ is a torus knot, $\Emb_f(S^1, S^3) \simeq SO_4$, while $\Emb_f(\R, \R^3) \simeq S^1$.
\item If $f$ is a hyperbolic knot, $\Emb_f(S^1, S^3) \simeq SO_4 \x S^1$, while $\Emb_f(\R, \R^3) \simeq S^1 \x S^1$.
\end{itemize}

Consider $\pi_1(\Emb_f(S^1, S^3))$ for $f$ as above.
For any knot in $S^3$, there are two loops that are canonical up to homotopy: (1) a loop of rotations representing a generator in $\pi_1(SO_4)\cong \Z/2$ and (2) a loop of reparametrizations of the knot.  
For an unknot and a torus knot, the loops (1) and (2) are homotopic.  For an unknot, they are homotopic to a rotation of a tubular neighborhood that fixes the parametrized unknot itself, suggesting the quotient by $SO_2$.  
The above result says that there are no further generators or relations in $\pi_1$.  
For a hyperbolic knot, its proof shows that the loops (1) and (2) are independent.  In $\pi_1(SO_4 \x S^1) \cong \Z/2 \x \Z$, the $\Z/2$ factor is generated by the loop of rotations.  The loop  of reparametrizations is, up to a loop of rotations, a multiple of the generator of the $\Z$ factor.  In many cases, it is a proper multiple because of nontrivial hyperbolic isometries of the complement.  So a generator of the $\Z$ factor is given by a diffeotopy to an isometry followed by a fractional reparametrization.

A similar analysis applies to the spaces $\Emb_f(\R,\R^3)$ above, which are completely determined by their fundamental groups.  
(In general $\Emb_f(S^1, S^3) \simeq SO_4 \x_{SO_2} \Emb_f(\R, \R^3)$ \cite[Proposition 4.1]{Budney-Cohen}.)  
In this setting, the two canonical loops above correspond to (1) the Gramain loop of rotations about the long axis \cite{Gramain1977}, shown in Figure \ref{F:Gramain},
and 
\begin{figure}[h!]
\includegraphics[scale=0.3]{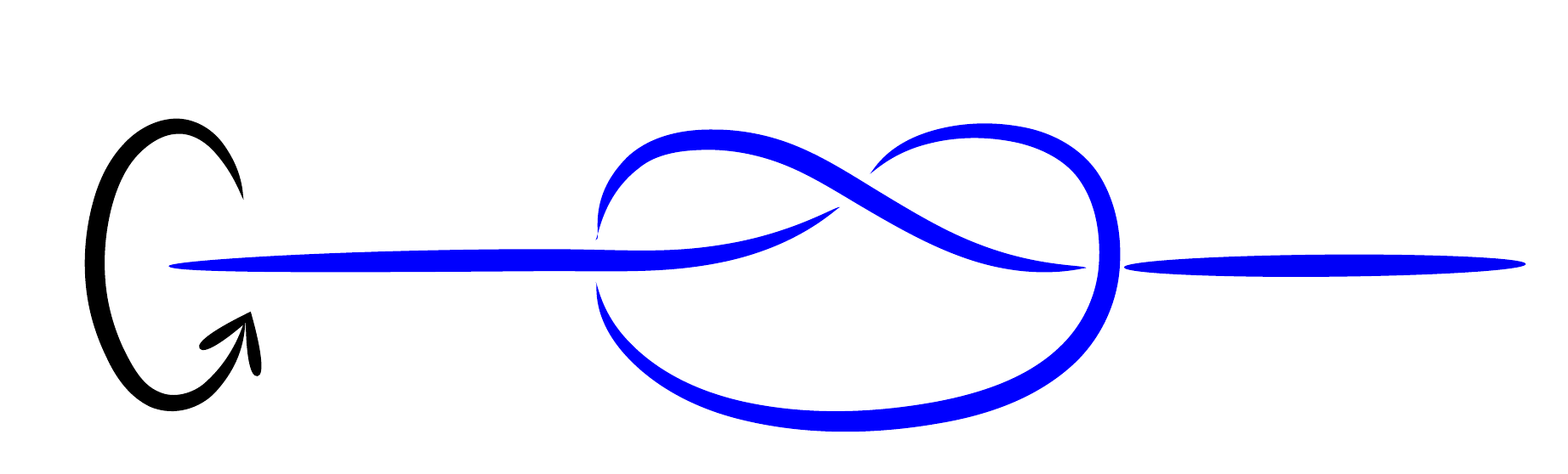}
\caption{The Gramain loop in $\Emb_f(\R, \R^3)$ where $f$ is a long right-handed trefoil.}
\label{F:Gramain}
\end{figure}
(2) the Fox--Hatcher loop, given by taking the one-point compactification $\overline{f}: S^1 \incl S^3$ of $f$, moving $\infty$ along the image of $\overline{f}$, and stereographically projecting to get a loop of long knots; see \cite[p.~3]{HatcherKnotSpacesWebsite} for a picture.
The latter loop will again generally be a multiple of a generator of $\pi_1$ for a hyperbolic long knot.  
These results can be summarized by saying that there are no exotic loops in these spaces, much like the Smale conjecture says that diffeomorphisms of $S^3$ are equivalent to rotations.  

Budney's results on satellite knots \cite{BudneyTop} further illustrate this phenomenon, though iterating satellite operations produces quite a rich topology.  His results roughly say that $\Emb_f(\R, \R^3)$ is a twisted product of (i.e.~fiber bundle involving) the spaces of the torus knots,  hyperbolic knots, and knot-generating links appearing in the satellite decomposition of $f$.  (Although satellite operations \cite{Schubert:KnotenVollringe} predate the aforementioned question of Arnold, the uniqueness of the satellite decomposition does not \cite{Jaco-Shalen-MemAMS, Johannson}.)  
The satellite operations are cables, connected sums, and hyperbolic splices.
Cabling a knot introduces a factor of $S^1$, while a connected sum of $n$ knots introduces the topology of 
configurations in the plane, and thus the $n$-strand braid group at the level of $\pi_1$.  A hyperbolic splice introduces two factors of $S^1$, but again, the product may be twisted.

\subsection{Main results}
We find that the above phenomenon extends from spaces of long knots and knots in $S^3$ to spaces $\tL_F$ of framed links $F$ in $S^3$ and ultimately spaces of knots in $S^1\x D^2$.
While there are many more isotopy classes of knots in $S^1 \x D^2$ than knots in the 3-sphere,\footnote{Up to 6 crossings, the number increases from 7 to 526 \cite{Gabrovsek-Mroczkowski}.  However, many of the examples we consider fall outside this range of crossing numbers, especially those involving satellite operations.} 
the homotopy types of their embedding spaces are just as tractable via various results in 3-manifold theory and the methods of Hatcher and Budney.  

The key is our computation of $\tL_F/SO_4$, using the fact that it is equivalent to the classifying space for the group of diffeomorphisms of the complement of $F$.
This result also generalizes early work of Goldsmith on spaces of torus links in $S^3$ \cite{Goldsmith:Math-Scand}, though she worked modulo reparametrization and relabeling of components rather than modulo rotations.  
While taking the quotient by rotations arguably removes some geometric meaning, $\tL_F$ is an $SO_4$-bundle over $\tL_F/SO_4$, and for certain links $F$ the latter spaces are homotopy equivalent to spaces of framed knots in $S^1 \x D^2$.  They thus lead to calculations of spaces $\mathcal{T}_f$ of (parametrized) knots $f$ in $S^1 \x D^2$, where we do not work modulo rotations.  
Since $\Diff^+(S^1) \simeq S^1$, one can visualize parametrized knots as knots with a bead representing a basepoint in $S^1$ and an arrow indicating the orientation.
Although we do not consider spaces of links that are unlabeled or unparametrized in this paper, this fact allows one to obtain some information on spaces of unparametrized knots in $S^1 \x D^2$ by using our results on $\mathcal{T}_f$.
We also obtain results on spaces $\L_f/SO_4$ of unframed (parametrized) links $f$ modulo rotations, which for certain links $f$ are homotopy equivalent to spaces $\V_f$ of (parametrized) knots in $S^1 \x S^1 \x I$.  
Throughout, we study these four types of embedding spaces:
\begin{equation}
\label{Eq:EmbSpacesWeConsider}
\tL_F/SO_4, \ \  \L_f/SO_4, \ \  \T_f, \ \text{ and } \  \V_f.
\end{equation}

To precisely state our results, we need some notation.
Let $f=(f_0,\dots, f_m)$ be an $(m+1)$-component link in $S^3$.  To $f$ we can associate a framed link $F=(F_0,\dots,F_m)$.
Let $\L$ be the space of links in $S^3$, let $\tL$ be the space of framed links in $S^3$, and let $\L_f$ and $\tL_F$ be the components of  $f$ and $F$ respectively.  

We write $T_{p,q}$ for the $(p,q)$-torus link; $S_{p,q}$ for the $(p,q)$-Seifert link;  $R_{p,q}$ for the union of $S_{p,q}$ with a core circle $C_1$ of the torus on which it lies; and $KC_{n+r}$ for the $(n+r+1)$-component keychain link.

For irreducible $F$, the JSJ tree of the complement $C_F$ has vertices labeled by submanifolds of $C_F$.  
Each submanifold labeling a vertex is the complement of a Seifert-fibered or hyperbolic link, and relabeling the vertices by these links gives the companionship tree $\mathbb{G}_F$ of $F$.  
The link $F$ is then built out of these links, via the splicing operation, denoted $\bowtie$.  
Below we use this symbol only when it corresponds to a sub-decomposition of $C_F$ (see Remark \ref{SplicingRemark}).
In our conventions for $\mathbb{G}_F$, each component of such a link corresponds to a half-edge on its vertex, and each component of $F$ corresponds to a half-edge in $\mathbb{G}_F$ that is not part of a whole edge joining a pair of vertices.  
Below, each of $J, J_1, \dots, J_r$ is an arbitrary irreducible link with a distinguished component.
The notation $(\varnothing, \dots, \varnothing, J_1,\dots, J_r) \bowtie L$ with $L=(L_0,\dots,L_{n+r})$ means graphically that one attaches each tree $\mathbb{G}_{J_i}$ to the vertex $\mathbb{G}_L$, joining the half-edge of the distinguished component of $J_i$ to the half-edge of $\mathbb{G}_L$ labeled $n+i$; 
see Figure \ref{F:TreeF=JbowtieL} in Section \ref{S:Splicing}.

We write $F \rtimes B$ to denote a fiber bundle with base $B$ and fiber $F$.  

For more details, see Section \ref{S:BasicDefs} for basic conventions, 
Section \ref{S:SeifertPrelims} for Seifert-fibered links, and Section \ref{S:JSJ} for JSJ trees and the splicing operation $\bowtie$.

\subsubsection*{\textbf{Spaces of framed links in the 3-sphere (modulo rotations)}}
The spaces  $\tL_F/SO_4$ lead to results on the other three types of spaces in \eqref{Eq:EmbSpacesWeConsider}, so we consider this our Main Theorem:

\begin{theorem}
\label{MainT:FramedLinks}
(A)
If $F$ is a nontrivial irreducible framed link, then $\tL_F/SO_4$ is aspherical, and generators of $\pi_1(\tL_F/SO_4)$ roughly correspond to loops of rotations of incompressible tori in $C_F$, as well as possibly braids coming from motions of incompressible tori in Seifert-fibered submanifolds of $C_F$.  

More precisely, we can recursively compute $\pi_1(\tL_F/SO_4)$ from the companionship tree $\mathbb{G}_F$ of $F$, as follows.  Designate some vertex corresponding to a link component as the root, and apply the isomorphisms below to reduce the calculation to smaller subtrees of $\mathbb{G}_F$, where the first two cases are precisely those where $\mathbb{G}_F$ has only one vertex:
\begin{itemize}
[leftmargin=0.4in]
\item[(i)]
If $F$ is an $(n+1)$-component Seifert-fibered link whose complement has $r$ singular fibers ($0 \leq r \leq 2$), then $\pi_1(\tL_F/SO_4) \cong \Z^{2n} \x \PB_{n+r}$, where $\PB_{n+r}$ is the $(n+r)$-strand pure braid group.
\item[(ii)]
If $F$ is an $(n+1)$-component hyperbolic link, then $\pi_1(\tL_F/SO_4) \cong \Z^{2(n+1)}$.
\item[(iii)]
If $T_{p,q}$ is an $(n+r+1)$-component torus link 
(i.e.~$\gcd(p,q)=n+r+1$) and \\
$F = (\varnothing, \dots, \varnothing, J_1,\dots, J_r) \bowtie T_{p,q}$,
then $\pi_1(\tL_F/SO_4) \cong \Z^{2n} \x \PB_{n+r+2} \x \prod_{i=1}^r  \pi_1(\tL_{J_i}/SO_4)$.
\item[(iv)]
If $S_{p,q}$ is an $(n+r+1)$-component Seifert link (i.e.~$\gcd(p,q)=n+r$) and \\
$F = (\varnothing, \dots, \varnothing, J_1,\dots, J_r) \bowtie S_{p,q}$,
then $\pi_1(\tL_F/SO_4) \cong \Z^{2n} \x \PB_{n+r+1} \x \prod_{i=1}^r  \pi_1(\tL_{J_i}/SO_4)$.
\item[(v)]
If $R_{p,q}:=S_{p,q} \cup C_1$ 
has $(n+r+1)$ components (i.e.~$\gcd(p,q)=n+r-1$) and \\
$F = (\varnothing, \dots, \varnothing, J_1,\dots, J_r) \bowtie R_{p,q}$,
then $\pi_1(\tL_F/SO_4) \cong \Z^{2n} \x \PB_{n+r} \x \prod_{i=1}^r  \pi_1(\tL_{J_i}/SO_4)$.
\item[(vi)]
%If $KC_{n+r}$ is the $(n+r+1)$-component keychain link and
If $F = (\varnothing, \dots, \varnothing, J_1,\dots, J_r) \bowtie KC_{n+r}$, then $\pi_1(\tL_F/SO_4) \cong \Z^{2n} \x \left(\B_{F} \ltimes \prod_{i=1}^r  \pi_1(\tL_{J_i}/SO_4) \right)$, where  $\B_F$ is a subgroup of the 
%$(n+r)$-strand 
braid group $\B_{n+r}$ determined by the symmetries of $\{J_1, \dots, J_r\}$.
\item[(vii)]
If $F = (\varnothing, \dots, \varnothing, J_1,\dots, J_r) \bowtie L$ with $L=(L_0, \dots, L_{n+r})$ hyperbolic, then 
$\pi_1(\tL_F/SO_4) \cong$ \\ $\Z^{2n+1} \x \left(\Z \ltimes \prod_{i=1}^r  \pi_1(\tL_{J_i}/SO_4) \right)$ 
where the semi-direct product is determined by the action of hyperbolic isometries of $L$ on the $J_i$.
\end{itemize}

(B)
If $F$ is an arbitrary framed link consisting of irreducible sublinks $F_1, \dots, F_k$, 
then roughly the homotopy type of $\tL_F/SO_4$ is at least as big as that of $\tL_{F_1}/SO_4\x \dots \x \tL_{F_k}/SO_4$.

More precisely, suppose $F_1, \dots, F_k$ are contained in disjoint 3-balls $B_1, \dots, B_k$, and let $Q_k:=S^3 - \coprod_{i=1}^k B_i$.
Let $\Emb_0(Q_k, C_F)$ be the subspace of embeddings of $Q_k$ into $C_F$ which extend to diffeomorphisms in $\Diff(C_F; \d C_F)$.  
Let $\Emb_{F_i}\left(\coprod S^1 \x  D^2, D^3\right)$ be the component of a framed link in $D^3$ corresponding to $F_i$ in the space of framed links in $D^3$.
Then up to homotopy, there is a fibration 
\begin{equation}
\label{Eq:SplitLinksFibn1}
\Emb_0 \left(Q_k, C_F\right) \to \prod_{i=1}^k \Emb_{F_i}\left(\coprod S^1 \x  D^2, D^3\right)  \to \tL_F/SO_4.
\end{equation}
For $i=1,\dots,k$, we have $\Emb_{F_i}\left(\coprod S^1 \x  D^2, D^3\right)\simeq SO_3 \x (C_{F_i} \rtimes (\tL_{F_i}/SO_4))$.
Restricting to the sublinks $F_1, \dots, F_n$ induces a surjection in $\pi_1$ from $\tL_F/SO_4$ to (the aspherical, connected space) $\prod_{i=1}^k\tL_{F_i}/SO_4$.
\end{theorem}

Theorem \ref{MainT:FramedLinks} comprises a handful of statements in the main body.  The asphericity in part (A) is Proposition \ref{FramedLinksKpi1}, part (b).  Statements (A) (i) - (vii) are Corollaries \ref{SeifertFramedLinks} and \ref{DiffHyp} and Propositions  
\ref{P:Cable}, 
\ref{P:SpliceIntoRpq}, 
%(a), (b), 
\ref{P:ConnectSum}, and 
\ref{P:HypSplice}.  
Part (B) on split (i.e.~reducible) links is Proposition \ref{P:SplitLinks}.  
Statements (A) (i) - (vii) are given at the space level, i.e., after applying the classifying space functor to the discrete groups above.  This converts iterated semi-direct products into iterated fiber bundles, where the spaces involved are products of circles and configuration spaces of points in $\R^2$.  
Throughout, we alternately work in terms of fundamental groups (semi-direct products) and homotopy types of spaces (fiber bundles)
% a.k.a.~twisted products).
As a byproduct, $\pi_1(\tL_F/SO_4)$ is torsion-free for irreducible $F$, as the fundamental group of a $K(\pi,1)$ space that is homotopy equivalent to a finite-dimensional manifold.

In Theorem \ref{MainT:FramedLinks}, part (B), we study neither the space $\Emb(Q_k, C_F)$ nor the long exact sequence in homotopy  from the fibration \eqref{Eq:SplitLinksFibn1} in detail.  
If $F$ has just two irreducible summands $F_1$ and $F_2$, Proposition \ref{P:SplitLinksTwoSummands} gives a sharper statement:
% than Theorem \ref{MainT:FramedLinks}, part (B) (i.e. Proposition \ref{P:SplitLinks}):
\[
\tL_F/SO_4 \simeq  SO_3 \x ((C_{F_1} \x C_{F_2}) \rtimes (\tL_{F_1}/SO_4 \x \tL_{F_2}/SO_4)).
\]
Below, we similarly give sharper statements on spaces of knots in $S^1 \x D^2$ and $S^1 \x I \x I$ which correspond to split links than on spaces of split links in general.  
A more detailed study of spaces of arbitrary split links could be an interesting direction for future work.

The next main result generalizes the fact that the Gramain subgroup splits as a direct factor of the fundamental group of the space of a long knot:

\begin{reptheorem}{T:MeridiansFactor}
If $F$ is an $(m+1)$-component framed link that is not a framed unknot, then the loops of meridional rotations $\mu_0, \dots, \mu_m$ generate a subgroup $\Z^{m+1}$ that is a direct factor of $\pi_1(\tL_F / SO_4)$.
\end{reptheorem}

\subsubsection*{\textbf{Spaces of unframed links in the 3-sphere (modulo rotations)}}

Proposition \ref{LfModSO4isKpi1} says that 
\[
\pi_1(\L_f/SO_4) \cong \pi_1(\tL_F/SO_4) / \langle \mu_0, \dots, \mu_m\rangle
\]
and all the higher homotopy groups of $\L_f/SO_4$ and $\tL_F/SO_4$ agree.  In particular, if $f$ is irreducible, then 
$\L_f/SO_4$ is also a $K(\pi,1)$ space.
Its fundamental group is explicitly described for various irreducible links in Sections \ref{S:Seifert}, \ref{S:Hyperbolic}, and \ref{S:Splicing} as well as in the proof of Theorem \ref{T:MeridiansFactor}, which draws upon those previous sections.  
Corollary \ref{SeifertSpacesOfLinks} specifies $\pi_1(\L_f/SO_4)$ for all Seifert-fibered $f$.  
If $f$ is an $(m+1)$-component hyperbolic link, then $\pi_1(\L_f/SO_4) \cong \Z^{m+1}$ (Corollary \ref{nCompHypLink}).
In Section \ref{S:Splicing}, we describe $\pi_1(\L_f/SO_4)$ for examples of nontrivial splices $f$.

\subsubsection*{\textbf{Spaces of knots in a solid torus}}

Let $\T_f$ be the component of a knot in a solid torus.  We view $f$ as a 2-component link $f=(f_0,f_1)$ where $f_1$ is an unknot, by Proposition \ref{TfandLinksWithUnknots}.  By Proposition \ref{TfisKpi1}, 
if $F$ is the associated framed link, $\pi_i(\T_f) \cong \pi_i(\tL_F/SO_4)$ for all $i \geq 2$, and
\[
\pi_1(\T_f) \cong \pi_1(\tL_F/SO_4) / \langle \mu_0 \rangle.
\]  
If $f$ is irreducible, $\T_f$ is a $K(\pi,1)$ space, and we can use Theorem \ref{MainT:FramedLinks} together with the proof of Theorem \ref{T:MeridiansFactor} to identify the subgroup $\langle \mu_0 \rangle$ and completely determine $\T_f$.  

If $f$ is split, we get a more precise description of $\T_f$ by applying Corollary \ref{TfSplit}  than by applying Theorem \ref{MainT:FramedLinks}, part (B).  Namely, $\T_f \simeq S^1 \x \Emb_f(S^1, D^3)$.  In turn, $\Emb_f(S^1, D^3)$ can be determined from $\Emb_{\underline{f}}(\R, \R^3)$ where $\underline{f}$ is a long knot whose closure is $f$.  Finally, $\Emb_{\underline{f}}(\R, \R^3)$ can be determined by \cite[Theorem 1.1]{BudneyTop}, which (by Proposition \ref{FramedKnotsInS3VsLongKnots}) our Theorem \ref{MainT:FramedLinks} generalizes.

From Theorem \ref{MainT:FramedLinks}, part (A), Corollary \ref{TfSplit} and an analysis of the cases in Theorem  \ref{T:MeridiansFactor}, we obtain the following direct factors in $\pi_1(\T_f)$.  As in Theorem \ref{T:MeridiansFactor}, these can be regarded as analogues of the Gramain subgroup.
We denote the meridional and longitudinal rotations of the solid torus  $\lambda_1$ and $\mu_1$ respectively because they correspond to rotations of a complementary solid torus where the roles of meridian and longitude are reversed. 

\begin{reptheorem}{T:FactorsInTf}
Let $f$ be an irreducible 2-component link corresponding to a knot in the solid torus.  Then $\mu_1$
%, the longitudinal rotation of the solid torus, 
generates a copy of $\Z$ that is a direct factor of $\pi_1(\T_f)$.
If $f$ is irreducible and not the Hopf link, then $\lambda_1$
%, the meridional rotation of the solid torus, 
generates a further direct factor of $\Z$ in $\pi_1(\T_f)$.  
If $f$ is a split link, then $\lambda_1$ generates a central subgroup isomorphic to $\Z/2$.
\end{reptheorem}
\noindent
The meridional rotation of the solid torus $\lambda_1$ coincides with the meridional rotation of $F_0$ (essentially the Gramain loop of $F_0$) in some but not all cases, including cases where $f$ is irreducible.
Theorem \ref{T:FactorsInTf} implies that $\T_f \simeq S^1 \x \L_f/SO_4$ (Corollary \ref{TfisS1xLf}).
Thus for a 2-component link $f=(f_0,f_1)$ with $f_1$ an unknot, and for $F$ a framed link which maps to $f$, we have the series of quotients
\[
\tL_F/SO_4 \twoheadrightarrow \T_f \twoheadrightarrow \L_f/SO_4.
\]
The second map is the quotient by a direct factor of $S^1$, and if $f$ is irreducible, so is the first map.

One can pass to spaces $\overline{\T}_f^+$ of oriented, unparametrized knots in $S^1 \x D^2$ by using the fact that $\Diff^+(S^1)$ acts freely on $\mathcal{T}_f$.  For irreducible $f$, the relevant bundle involves only $K(\pi,1)$ spaces, and $\pi_1(\overline{\T}_f^+)$ is the quotient of $\pi_1(\T_f)$ by the loop $\lambda_0$ of reparametrizations, which can be identified by Theorem \ref{MainT:FramedLinks}, part (A).  Then $\pi_1$ of the space $\overline{\T}_f$ of unoriented, unparametrized knots in $S^1 \x D^2$ is an extension of $\Z/2$ by $\pi_1(\overline{\T}_f^+)$.  
It seems that for irreducible $f$, both $\pi_1(\overline{\T}_f^+)$ and $\pi_1(\overline{\T}_f)$ could always be torsion-free, just as $\pi_1(\T_f)$ is.

Below are some special cases of irreducible links $f$ corresponding to knots in $S^1 \x D^2$, some of which are shown in Figure \ref{F:ExKnotsInSolidTorus}:
\begin{itemize}[leftmargin=0.25in]
\item[(a)] If $f$ is the Hopf link, $\L_f/SO_4 \simeq \ast$, and $\pi_1(\T_f) \cong \Z$, generated by a longitudinal rotation of $S^1 \x D^2$ or a reparametrization.  
So the space of unparametrized such knots is contractible.
(The fact that $\T_f$ has the simplest homotopy type for this $f$ mirrors the fact that the longitudinal vector field has minimal energy among vector fields on $S^1 \x D^2$ induced by boundary-fixing Hamiltonian isotopies of $D^2$ \cite[Section 5]{Gambaudo-Lagrange}; the periodic orbits of such vector fields include many knots in $S^1 \x D^2$.)
\item[(b)] If $f$ is a torus knot $T_{p,q}$ in its standard position, then $\pi_1(\T_f) \cong \Z^2$, generated by the two rotations $\lambda_1$ and $\mu_1$ of $S^1 \x D^2$ (Proposition \ref{TorusKnotInTorus}).  A generator of the quotient by reparametrization, $\Z$, can be identified with a motion of a barber pole that can be viewed as either a rotation or a translation, since modulo a loop of reparametrizations of a torus knot, $\lambda_1 \sim \mu_1$.\footnote{The second author credits Rob Kusner for essentially pointing out this connection to him in 2012.}
\item[(c)] If $f$ is hyperbolic, $\pi_1(\T_f) \cong \Z^3$ (Corollary \ref{HypKnotInTorus}).  In addition to the two rotations, there is a generator coming from a hyperbolic isometry and reparametrization, like for knots in $S^3$.  
\item[(d)] If $f$ is the connected sum of $r$ copies of the same knot $j$, embedded in $S^1 \x D^2$ to have winding number 1, then $\pi_1(\T_f) \cong \Z \x \B_{1,r} \ltimes (\pi_1(\tL_J/SO_4))^{r}$, where $\B_{1,r}$ is the annular braid group on $r$ strands (see Example \ref{Ex:MoreGeneralConnectSum}).  
\item[(e)] Exercise \ref{Ex:WhiteheadDoubleInSolidTorus}, parts (a) and (b), 
provides two links corresponding to embeddings of a Whitehead double into $S^1 \x D^2$ with different results for $\pi_1(\T_f)$.
\end{itemize}

\begin{figure}[h!]
(a)\includegraphics[scale=0.23]{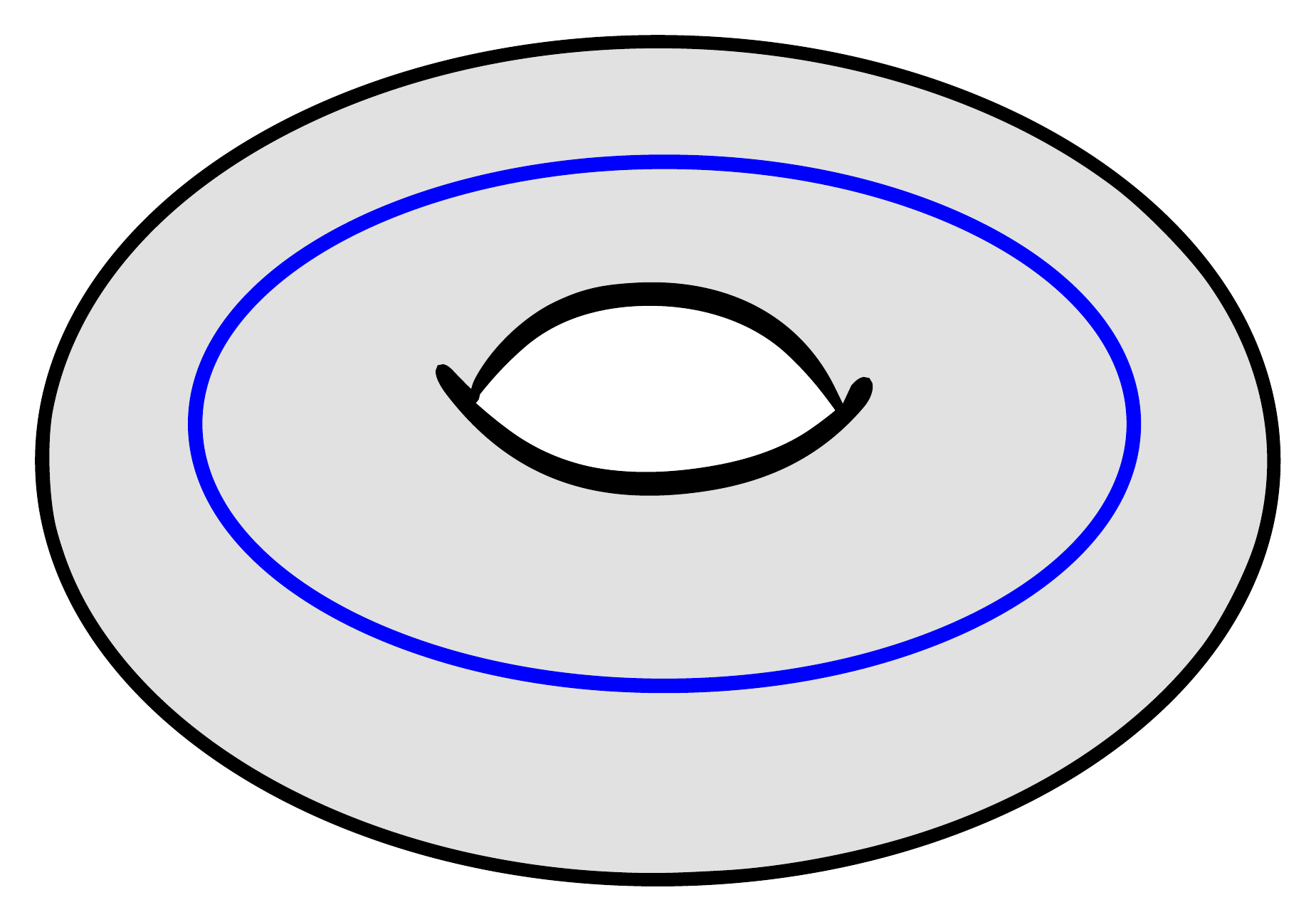} \quad
(b)\raisebox{0pc}{\includegraphics[scale=0.3]{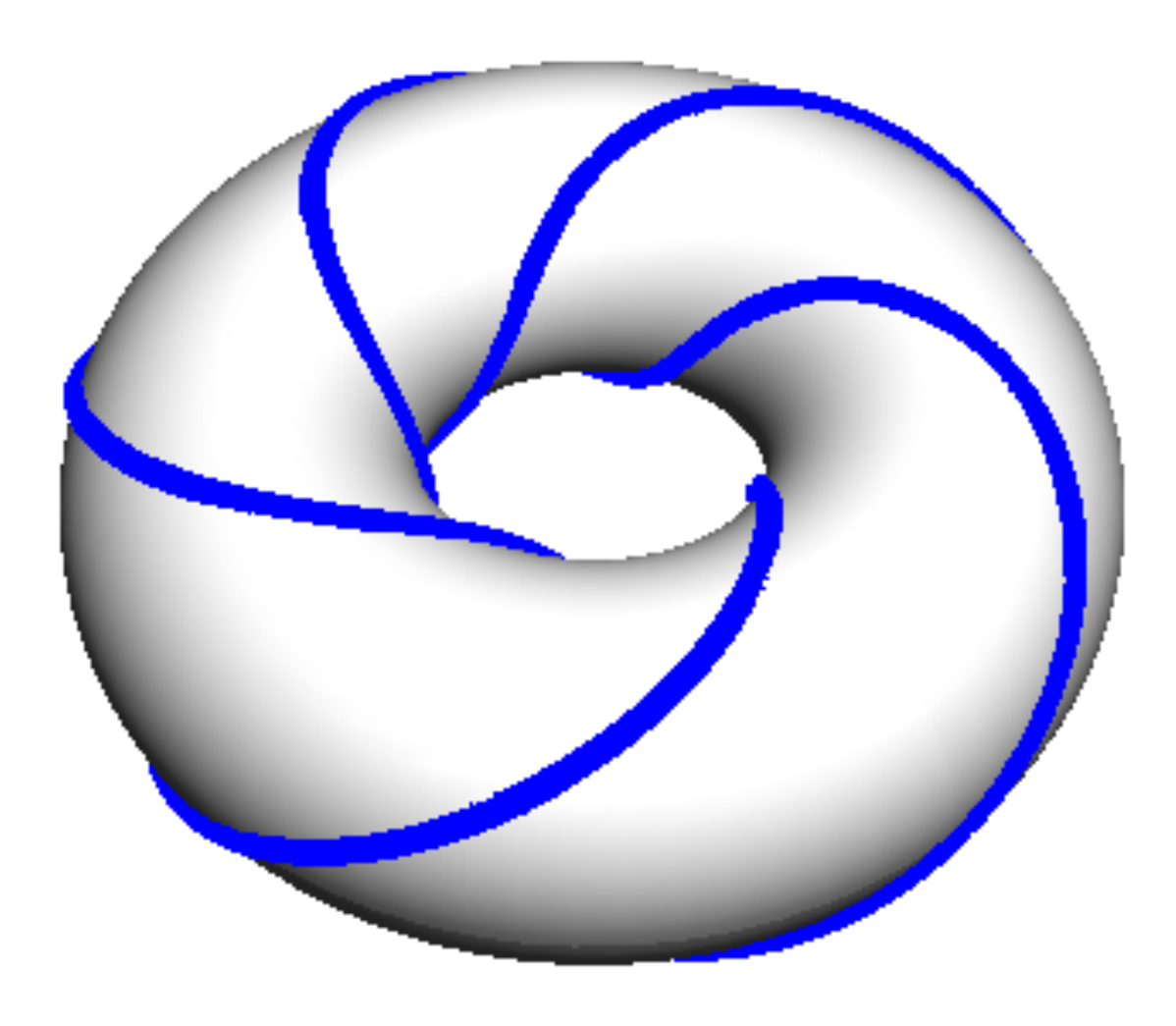}} \quad
(c)\includegraphics[scale=0.23]{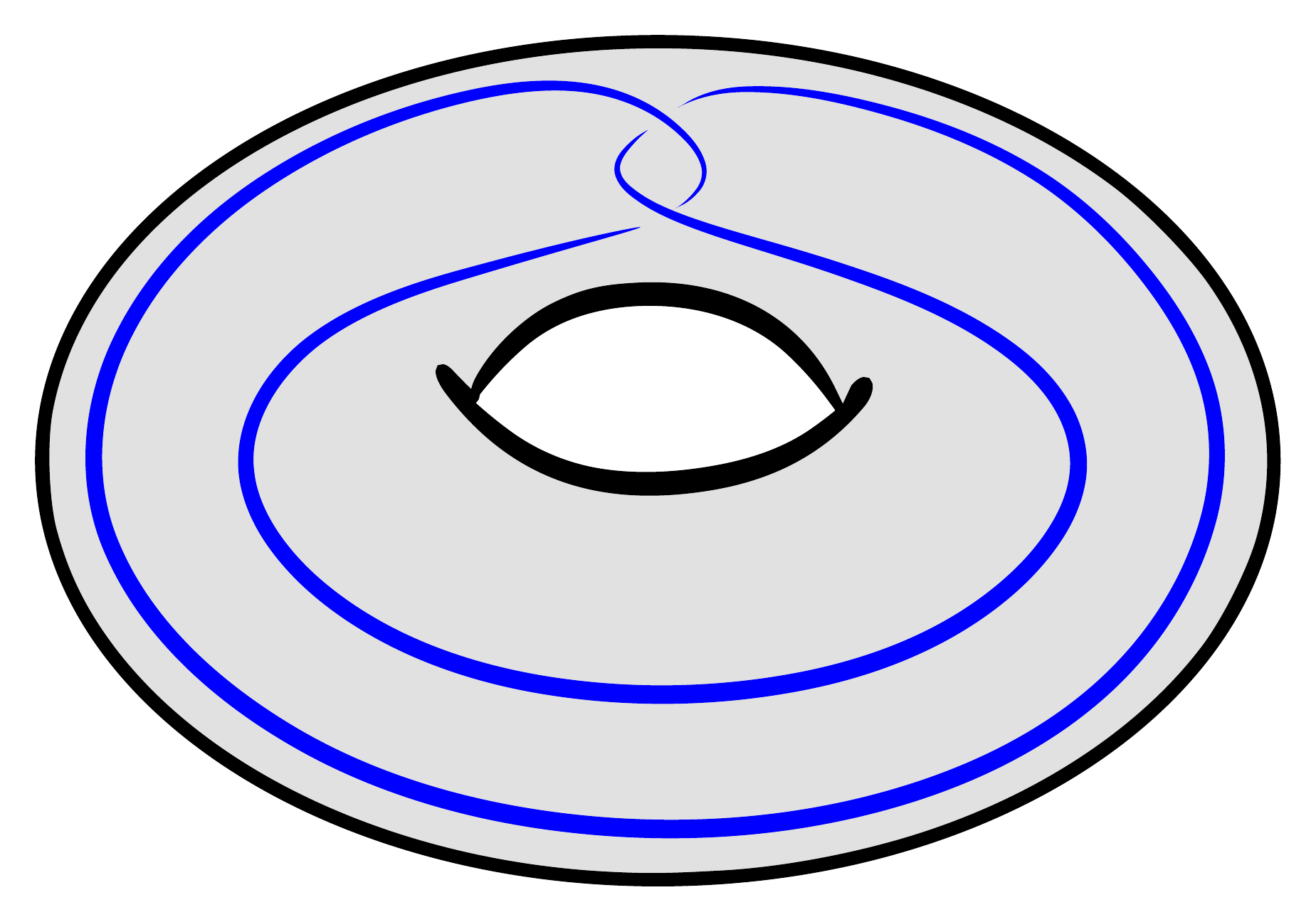}
\caption{Depictions of knots in the solid torus associated to (a) the Hopf link, (b) the $(3,5)$-Seifert link, and (c) the Whitehead link, which is hyperbolic.}
\label{F:ExKnotsInSolidTorus}
\end{figure}

Section \ref{S:Splicing} has calculations of $\pi_1(\L_f/SO_4)$ in many examples of nontrivial splices $f$, with explicit generators.  These results immediately transfer to $\T_f$ via the equivalence  $\T_f \simeq S^1 \x \L_f/SO_4$.  Most of them are summarized in Table \ref{T:ExamplesSolidTorus} at the end of the paper.

Again, there are many more knots in the solid torus than in the 3-sphere, and accordingly, remembering only the knot $f_0: S^1 \incl S^3$ associated to $f: S^1 \incl S^1 \x D^2$ (by embedding $S^1 \x D^2 \incl S^3$) may omit much information.
For example, there are many non-isotopic knots $f$ for which $f_0$ is the unknot, including multiple knots with winding number 0 (e.g.~a component of the Whitehead link or of a Bing double, as in Example \ref{Ex:BorrPartialSplice}), knots with Seifert-fibered complements (a component of the Hopf link or of $S_{1,n}$), knots with hyperbolic complements (e.g.~again a component of the Whitehead link), and knots whose complements have nontrivial JSJ trees (e.g.~again a Bing double).  

Similarly, knottedness of $f_0$ imposes few restrictions on $f$.  
There are knots $f$ with $f_0$ nontrivial where the complement of $f$ is hyperbolic (e.g.~Examples \ref{Ex:HypLinkFromTrefoil} and \ref{Ex:HypLinkFrom8-18}), is Seifert-fibered (the knotted component of $S_{p,q}$), or has a nontrivial JSJ tree (e.g.~Whitehead doubles, as in Exercise \ref{Ex:WhiteheadDoubleInSolidTorus}).
The latter examples are moreover different knots $f$ with the same $f_0$.
In addition, there are cases where the complement of $f_0$ has a trivial JSJ tree, while the complement of $f$ does not.  
For example, if $f_0$ is a Seifert-fibered or hyperbolic link and $f_1$ is a meridian of $f_0$, then the complement of $f$ has a nontrivial JSJ tree.

\subsubsection*{\textbf{Spaces of knots in a thickened torus}}
Let $\V_f$ be the component of a knot in the thickened torus $S^1 \x S^1 \x I$, where the notation is chosen to suggest a connection to virtual knot theory.  
We view $f$ as a 3-component link $(f_0, f_1,f_2)$ where $(f_1,f_2)$ is the Hopf link, by Proposition \ref{VfandLinksWithHopf}.  Then
\[
\V_f \simeq \L_f/SO_4
\]  
by Proposition \ref{VfIsLfModSO4}.
As with $\T_f$, the homotopy type of $\V_f$ can be determined by using Theorem \ref{MainT:FramedLinks}, part (A) and Theorem \ref{T:MeridiansFactor} (and its proof) if $f$ is irreducible.  
For Seifert-fibered $f$, $\V_f \simeq S^1 \x S^1$ (Proposition \ref{P:SeifertVf}), while for hyperbolic $f$, $\V_f \simeq (S^1)^3$ (Corollary \ref{HypKnotVf}).
If $f$ is split, Corollary \ref{VfSplit} gives the most precise description, namely $\V_f \simeq S^1 \x S^1 \x \Emb_f(S^1, D^3)$, and the rightmost factor is in turn calculated as described above for $\T_f$.
For any knot $f$ in a thickened torus, the two rotations of $S^1 \x S^1 \x I$ generate a factor $\Z^2$ in $\pi_1(\V_f)$ (Corollary \ref{FactorsInVf}), thus providing an analogue of a Gramain subgroup of $\pi_1(\V_f)$, just as Theorems \ref{T:MeridiansFactor} and \ref{T:FactorsInTf} do for $\pi_1(\L_f/SO_4)$ and $\pi_1(\T_f)$.  
Examples of knots in $S^1 \x S^1 \x I$ given throughout the paper are summarized in Table \ref{T:ExamplesThickenedTorus}.

As in the case of knots in $S^1 \x D^2$, one can understand spaces of oriented (or unoriented) unparametrized knots in $S^1 \x S^1 \x I$ by taking the appropriate quotient of $\pi_1(\V_f)$ (and an extension of $\Z/2$ by the resulting group).

\subsubsection*{\textbf{Miscellaneous remarks}}
Apropos of virtual knot theory, one might consider spaces of knots in more general thickened surfaces.  Unlike knots in a thickened torus, these do not correspond to links in $S^3$.  So we do not consider them here, though they may well be tractable via their JSJ decompositions and the results of \cite{HatcherMcCullough}.  We expect smaller homotopy types for higher genus, since $\Diff(S_{g})$ has contractible components for $g \geq 2$.  More specifically, we expect that the fundamental group has no Gramain-type subgroup of rotations for $g \geq 2$.  
Similar methods could be fruitful for studying spaces of knots in more general 3-manifolds.

Knots in the solid torus are of interest in other contexts in addition to their relationship to links in the 3-sphere.
These include settings mentioned in the book of Arnold's problems, such as contact geometry \cite[Problem 1994-21]{ArnoldsProblems} \cite{Chmutov-Goryunov-Murakami} and dynamical systems \cite[Problem 1973-23]{ArnoldsProblems} \cite{Gambaudo-Ghys:1997, Gambaudo-Lagrange}, as well as related applications such as tokamaks in nuclear fusion \cite{Moffatt:PNAS}.

\subsection{Organization of the paper}  
In Section \ref{S:Prelims}, we set notation and terminology and review results on spaces of diffeomorphisms and diffeomorphisms of 3-manifolds.
We relate knots in $S^1 \x D^2$ and $S^1 \x S^1 \x I$ to links in $S^3$, thus allowing us to write both $\L_f$ and either $\T_f$ or $\V_f$ for certain links $f$.
We then review Seifert-fibered submanifolds of $S^3$ and the JSJ decomposition.

In Section \ref{S:Asphericity}, we establish relationships between spaces of long knots and spaces $\tL_F/SO_4$. 
We then show that $\tL_F/SO_4$, $\L_f/SO_4$, $\T_f$, and $\V_f$ are all equivalent to certain classifying spaces, which are $K(\pi,1)$ spaces if $F$ (or $f$) is irreducible.

In Sections \ref{S:Seifert} and \ref{S:Hyperbolic}, we completely determine the homotopy types of these spaces in the Seifert-fibered and hyperbolic cases respectively, generalizing work of Hatcher.  

In Section \ref{S:Splicing}, we describe spaces of links obtained by satellite operations, namely generalizations of cables, connected sums, and hyperbolic splices, generalizing work of Budney.  These steps form the bulk of the proof of Theorem \ref{MainT:FramedLinks}.  We provide many examples.

In Section \ref{S:SplitLinksTorus}, we prove statements about spaces of split framed links, including Theorem \ref{MainT:FramedLinks}, part (B).  We also describe $\T_f$ and $\V_f$ for knots $f$ associated to split links.  These results do not rely on any material after Section \ref{S:Asphericity}.  

In Section \ref{S:Splittings} we give splittings of certain subgroups of $\pi_1(\tL_F/SO_4)$, $\pi_1(\T_f)$, and $\pi_1(\V_f)$.  

Section \ref{S:Tables} contains two tables which summarize the homotopy types and fundamental groups of $\T_f$ and $\V_f$ in our examples.   Many of these are spaces $\T_f$ where $f$ is a nontrivial splice, from Section \ref{S:Splicing}.

\subsection{Acknowledgments}  
The first author thanks R.~Inan\c{c} Baykur for support and encouragement in participating in this project.
The second author thanks Ryan Budney for conservations about approaches to this question, Sam Nariman for a conversation about reducible 3-manifolds, and Rafa\l\ Komendarczyk for a brief discussion of connections to other problems.  
Both authors thank Nikolay Buskin and Richard Buckman for bringing to our attention the problem of Arnold, for useful early conversations, and for inspiration to explore various examples.  
We thank the referee for useful comments which lead to some strengthening of our results.  
We acknowledge the use of KLO \cite{KLO} and especially SnapPy \cite{SnapPy} for verifying the symmetries of various links, and the use of Inkscape and Grapher for producing graphics.
The second author was supported by the Louisiana Board of Regents Support Fund, contract number LEQSF(2019-22)-RD-A-22.

\section{Preliminaries}
\label{S:Prelims}
We begin in Section \ref{S:BasicDefs} by setting notation and definitions.  In Section \ref{S:DiffResults}, we review some long-known results on diffeomorphisms of surfaces and 3-manifolds, providing proofs of two relatively short results.  In Section \ref{S:KnotsVia2CompLinks}, we relate knots in a solid torus and a thickened torus to links in $S^3$.  
We then review results specialized to links in $S^3$: 
Section \ref{S:SeifertPrelims} covers Seifert-fibered submanifolds of $S^3$ or equivalently Seifert-fibered links, 
while Section \ref{S:JSJ} covers JSJ decompositions of submanifolds of $S^3$, companionship trees, and satellite operations (i.e.~splicing).

\subsection{Notation and basic definitions}
\label{S:BasicDefs}
\ \\
\noindent
Manifolds and diffeomorphisms:
\begin{itemize}[leftmargin=0.25in]
\item 
We write $D^n$ for the closed $n$-dimensional disk and $I$ for $[0,1]$.  We often view $D^2 \subset \C$.  
\item 
We say "manifold" to mean a manifold with possibly nonempty boundary. 
\item 
We write $\Diff(M)$ for the space of diffeomorphisms of a manifold $M$.
\item 
Given a submanifold $S \subset M$, let $\Diff(M; S)$ be the subspace of diffeomorphisms of $M$ which restrict to the inclusion on $S$. We sometimes abbreviate $\Diff(M; \d M)$ as $\Diff(M; \d)$.   
\item 
For submanifolds $S_1$ and $S_2$ of $M$, let $\Diff(M, S_1; \, S_2)$ be the subspace of diffeomorphisms of $M$ which preserve $S_1$ setwise and fix $S_2$ pointwise.
\item 
If $M$ is orientable, we write $\Diff^+$ to indicate diffeomorphisms of $M$ which are orientation-preserving, both in the absolute and relative cases.  
The same convention applies to the group of isometries $\Isom(M)$ of a hyperbolic 3-manifold $M$.
A diffeomorphism which pointwise fixes a boundary component of a connected manifold $M$ must be orientation-preserving.
\item 
Let $\Diff_0(M)$ denote the component of the identity in $\Diff(M)$.  Define $\Diff_0(M; S)$ similarly.
\item
An element of $\pi_0$ of a space of diffeomorphisms is called a \emph{mapping class}.
If $\Sigma_g^b$ is a surface of genus $g$ with $b$ boundary components, then $\mathrm{PMod}(\Sigma_{g,n}^b):=\pi_0\Diff(\Sigma_g^b; \ \d \Sigma_g^b \cup \{x_1,\dots, x_n\})$, the \emph{pure mapping class group} of $\Sigma_g^b$ with $n$ marked points or punctures.
If $b=0$ or $n=0$ we omit the superscript or subscript entirely.
Special cases include the $n$-strand pure braid group $\PB_n \cong \mathrm{PMod}(\Sigma_{0,n}^1)$, the group of framed $n$-strand pure braids $\mathrm{PMod}(\Sigma_{0}^{n+1})$, and the spherical $n$-strand pure braid group $\mathrm{PMod}(\Sigma_{0,n})$.
\item 
Define $\Conf(n,M):=\{(x_1, \dots, x_n) \in M^n : x_i \neq x_j  \ \forall \, i \neq j\}$, the space of ordered configurations in a manifold $M$.  We consider mainly $\Conf(n,\R^2)$, which is a $K(G,1)$ space for $G=\PB_n$.
\end{itemize}
Surfaces and 3-manifolds:
\begin{itemize}[leftmargin=0.25in]
\item 
A 3-manifold $M$ is \emph{irreducible} if every embedded 2-sphere in $M$ bounds a 3-ball.  By the Sphere Theorem, $M$ is irreducible if and only if $\pi_2(M)=0$.
\item 
If $S$ is a compact, connected, orientable surface in a 3-manifold with boundary $M$ and $S$ is properly embedded (i.e.~$S \cap \d M = \d S$), we say $S$ is \emph{incompressible} if the inclusion $S\incl M$ is injective on $\pi_1$.
\item 
Suppose that $T=\d \nu$, where $\nu$ is a solid torus $S^1 \x D^2$.
A \emph{meridian} $m$ is a simple closed curve on $T$ which represents a generator of $\pi_1(T)$ but bounds a disk in $\nu$.  A \emph{longitude} $\ell$ is a simple closed curve on $T$ which has zero linking number with $S^1 \x \{0\} \subset \nu$ and intersection number $\pm 1$ with $\ell$.  
A knot $f:S^1 \incl S^3$ gives rise to such a solid torus $\nu$ via a tubular neighborhood, and an orientation of $S^1$ orients $\ell$, from which we can canonically orient $m$ by demanding that $m \cdot \ell=+1$, using an orientation of $S^3$.  Equivalently (orienting $T$ outward normal last), $m$ has linking number $+1$ with $f$.
\end{itemize}
Links, knots, and long knots:
\begin{itemize}[leftmargin=0.25in]
\item 
Let $\Emb(M,N)$ be the space of smooth embeddings of a manifold $M$ into a manifold $N$, with the $\mathcal{C}^\infty$ Whitney topology.  
Given a fixed embedding $f: S \incl N$ of a submanifold $S \subset M$, we write $\Emb(M,N; S)$ or $\Emb(M,N;f)$ for the space of smooth embeddings of $M$ into $N$ which restrict to $f$ on $S$. 
A special case is where $M$ is a submanifold of $N$ and $f$ is the inclusion.
\item 
A \emph{link} is a smooth embedding $f = (f_1,\dots, f_m): \coprod^m S^1 \incl S^3$.  When there is no risk of misinterpretation, we may blur the distinction between $f$ and its image.  Sometimes we index the components of $f$ starting at 0 rather than 1.  If the name of a link already has a subscript, e.g.~$j_i$, we include a second subscript, e.g.~$j_{i,k}$, to specify a component.  
Links which have been tabulated but which lack colloquial names will be identified by their names in the tables of Thistlethwaite \cite{ThistlethwaiteTable} and Rolfsen \cite{Rolfsen}.
We sometimes write for example $\coprod S^1$ for a disjoint union of an unspecified finite number of copies of $S^1$, so as to suppress the number of link components from the notation.
These remarks also apply to framed links, defined below.
\item
A \emph{knot} is a link $S^1 \incl S^3$ with one component.  We use the term \emph{closed knot} synonymously, to distinguish from long knots, defined below.
\item 
An $n$-component link $L$ is \emph{irreducible} if there exists an embedded $S^2$ that separates components of $L$.  We say that $L$ is \emph{split} if it is \emph{not} irreducible.  A knot is thus always irreducible.
\item 
Fix orientations on $S^1 \x D^2$ and $S^3$.
A \emph{framed link} is a smooth, orientation-preserving embedding $F=(F_1,\dots, F_m): \coprod^m S^1 \x D^2 \incl S^3$.
The \emph{$i$-th framing number} of $F$ is the linking number $\mathrm{lk}(F|_{(S^1)_i \x \{0\}}, \ F|_{(S^1)_i \x \{1\}})$, where $(S^1)_i$ is the $i$-th summand of $S^1$.  That is, it is the linking number of the $i$-th ``core'' with the $i$-th longitude.
The restriction of a framed link $F$ to the cores $\coprod^m S^1 \x \{0\}$ gives a link $f$.  
A \emph{framing} of a link $f$ is a class of framed link $F$ up to isotopy fixing $f$, and it is given precisely by the $m$ framing numbers.
Thus restriction to the core induces a bijection to isotopy classes of links from isotopy classes of links with \emph{$\vec{0}$-framing} 
a.k.a.~the \emph{homological framing}. 
A planar projection of a link gives rise to its \emph{blackboard framing} given by the writhe (i.e.~the signed sum of crossings) of each component.
We will use the uppercase and lowercase of the same letter to denote a framed link and the corresponding link, e.g.~$F \leftrightarrow f$.
\item 
For a link $f$, we use $\nu(f)$ to denote a tubular neighborhood of $f$.  The \emph{exterior of $f$} is defined $S^3 - \nu(f)$, which is diffeomorphic to the \emph{complement} $S^3 - \im(F)$ of the framed link $F$.  We thus denote it by either $C_f$ or $C_F$.
\item 
A \emph{long knot} is an embedding $f:I \incl I \x D^2$ whose values and derivatives on $\d I$ agree with the embedding induced by the diffeomorphism $I \overset{\cong}{\to} I \x \{0\}$.  A \emph{framed long knot} is an embedding $F:I \x D^2 \incl I \x D^2$ such that $F$ and its derivative are the identity on $\d I \x D^2$. 
\end{itemize}
Spaces of links and long knots and maps between them:
\begin{itemize}[leftmargin=0.25in]
\item For each integer $m\geq 1$, let $\L(m) := \Emb(\coprod^m S^1, S^3)$ be the space of $m$-component links.  Let $\L := \coprod_{m \geq 1} \L(m)$ be the space of all links.  For a link $f$, let $\L_f$ be the component of $f$ in $\L$.
\item Let $\tL$ be the space of all framed links, and let $\tL_F$ be the component in $\tL$ of the framed link $F$.  
{The space $\tL$ is homotopy equivalent to the space of links with a normal vector at each point.}
\item
Let $\K$ be the space of long knots.  
{It is homeomorphic to the space $\Emb_c(\R, \R \x D^2)$  of embeddings $\R \incl \R \x D^2$ given by the inclusion of the $x$-axis outside a compact set, say $[-1,1]$.}
\item
Let $\tK$ be the space $\Emb_c(\R \x D^2, \R \x D^2)$ of framed long knots.\footnote{Beware that the space of closed knots (respectively framed closed knots) is a subspace of $\L$ (respectively $\tL$), not $\K$ (respectively $\tK$).
Our notation is not overloaded because we make no use of long links.} 
\item
Equipping a link or long knot with the $\vec{0}$-framing and restricting to the core yield well defined maps $\pi_0 \L \to \pi_0 \tL \to \pi_0 \L$ and $\pi_0 \K \to \pi_0 \tK \to \pi_0 \K$ where each composition is the identity.
\item
There is a closure operation $\K \to \L(1)$ (or $\tK \to \tL(1)$) that sends a (framed) long knot to a (framed) closed knot, as described in Section \ref{S:KnotsVia2CompLinks} below.  
We will denote a (framed) long knot and its closure by either $(F, \overline{F})$ or $(\underline{F}, F)$.
The exterior of a long knot $f$ in $I \x D^2$ is homeomorphic to the exterior of the closure $\overline{f}$ in $S^3$.
\end{itemize}
Loops of links:
\begin{itemize}[leftmargin=0.25in]
\item 
For a framed link $F=(F_0, \dots, F_m)$, and  $i=0,1,\dots, m$, we write 
 $\lambda_i$ to denote a loop of longitudinal rotations (or reparametrizations) of the $i$-th component and
$\mu_i$ to denote a loop of meridional rotations of the $i$-th component.  
We view these as their classes in $\pi_1(\tL_F)$ or more often $\pi_1(\tL_F/SO_4)$.  
If the 0-th component of $F=(F_0,F_1)$ corresponds to a knot in $S^1 \x D^2$ (so $F_1$ is the unknot), then $\lambda_1$ and $\mu_1$ correspond modulo $SO_4$ to loops which fix $F_1$ and rotate $F_0$ in some way.
They correspond respectively to meridional and longitudinal rotations $\mu$ and $\lambda$ of the solid torus, not vice-versa.  
In this case, we also write $\rho$ for the loop $\lambda_0$ that reparametrizes the knotted component.  We also use $\lambda, \mu,$ and $\rho$ to denote similar loops of a knot in $S^1 \x S^1 \x I$.
\end{itemize}
Groups:
\begin{itemize}[leftmargin=0.25in]
\item 
We write ``+'' and ``0'' for the group operation and identity element in abelian groups, but juxtaposition and ``1'' in groups that are not (necessarily) abelian.  We use $e$ for the identity element in contexts possibly involving both types of groups, e.g., certain short exact sequences.  
\item 
We use exponentiation to denote conjugation, i.e., $g^h := h^{-1} g h$.
\item 
We write $\Z\langle a_1, \dots, a_n \rangle$ to denote the free abelian group on generators $a_1, \dots, a_n$.
\item 
We write $G\cong K \rtimes H$ or $G \cong H \ltimes K$ to indicate that $G$ is a semi-direct product of $H$ and $K$, where $K$ is the normal subgroup, hence the kernel of a map $G \to H$.   
Indeed, equivalently, there is a short exact sequence
\[\{e\} \to K \to G \to H \to \{e\}\]
 with a splitting $H \to G$.  A sequence as above with a splitting $G \to K$ is equivalent to the stronger condition that $G\cong H\x K$.  In the latter case, we call $H$ and $K$ \emph{direct factors}.  
\item 
We write $\mathfrak{S}_n$ for the symmetric group on $n$ letters and $\mathfrak{S}_n \wr K$ for the group $K^n \rtimes \mathfrak{S}_n$ where $\mathfrak{S}_n$ acts by permuting the $n$ factors of $K$.
 \item 
 Write $\mathfrak{S}_r^\pm$ for the signed symmetric group $\mathfrak{S}_r \wr \Z/2 := \mathfrak{S}_r \ltimes (\Z/2)^r$.  An element $\sigma$ in $\mathfrak{S}_r^\pm$ can be viewed as a permutation of $\{-r, \dots, r\}$ such that $\sigma(-i)= -\sigma(i)$.  It can also be viewed as a permutation of $\{1,\dots,r\}$ together with the additional data of a sign $+$ or $-$ for each value $\sigma(i)$.  
\item 
Write $\B_n$ and $\PB_n := \ker(\B_n \to \mathfrak{S}_n)$ for the $n$-strand braid (respectively pure braid) group.
\end{itemize}

\subsection{Background on homotopy types of spaces of diffeomorphisms}
\label{S:DiffResults}
We now review theorems related to homotopy types of spaces of diffeomorphisms and spaces of embeddings.
%, many of which are well known.

By work of Palais \cite{Palais-LocalTriv} or Lima \cite{Lima}, restricting diffeomorphisms of a manifold $M$ to a submanifold $S$ gives a locally trivial fiber bundle
\begin{equation}
\label{DiffToEmbFibn}
\Diff(M; S) \to \Diff(M) \to \Emb(S, M).
\end{equation}
This result generalizes the isotopy extension theorem \cite[Theorem 8.1.3]{Hirsch}, which by itself will also useful here.
One can similarly restrict embeddings of $M$ into another manifold $N$ to obtain a fibration 
\[
\Emb(M, N; S) \to \Emb(M,N) \to \Emb(S, N).
\]
An important special case of a slight variant of \eqref{DiffToEmbFibn} is when $M=D^{n+1}$, $S=\{(1,0,\dots, 0)\}$, and we take 
linear isometric maps instead of smooth maps to get the fiber bundle
\[
SO_n \to SO_{n+1} \to S^n.
\]
Recall that $SO_2 \cong S^1$, while $SO_3 \cong \R P^3 \cong S^3/\{\pm 1\}$, and $SO_4 \cong S^3 \x SO_3$ (as spaces, but not as groups).  For $n=2$, the bundle above is doubly covered (in the fibers) by the Hopf fibration.  For $n=3$, it is a trivial bundle.

We will extensively use the long exact sequence in homotopy groups for a fibration $F \to E \to B$, which ends at $\dots \to \pi_0(F) \to \pi_0(E) \to \pi_0(B)$.  Generally, the $\pi_0$ terms are just sets, but we will often consider fibrations of topological groups, in which case the maps on $\pi_0$ are homomorphisms and any path-component is homeomorphic to the component of the identity.  
We will also consider the situation where there is a free action of a group $G$ on the total space $E$ of a fiber bundle which restricts to a free $G$-action on the base $B$.  Taking the quotient by the action gives the fiber bundle $F \to E/G \to B/G$.

By a theorem of Whitehead, a map of CW-complexes inducing isomorphisms on all homotopy groups is a homotopy equivalence.  
We apply this fact to spaces of $C^\infty$-diffeomorphisms and $C^\infty$-embeddings of compact manifolds.  
Such a space can be given the structure of an infinite-dimensional manifold, modeled on a locally convex topological vector space, in fact a Fr\'echet space \cite[Section 3.1]{Brylinski}.
These spaces are thus dominated by CW-complexes \cite[Theorem 13]{Palais-HtpyThy}, which is sufficient for Whitehead's theorem to apply; see \cite[Lemma 6.6]{Palais-HtpyThy} or \cite[Proposition A.11]{HatcherAT}.

For a group $G$, we will use the classifying space $BG$, which can be obtained as the quotient of a contractible space $EG$ by a free action of $G$.  
%The space $BG$ classifies principal $G$-bundles in that homotopy classes of maps from $X$ to $BG$ are in bijection with principal $G$-bundles over $X$.
One can regard $B(-)$ as a functor from groups to based spaces, via a certain construction of $EG$ and $BG$.  
The based loop space functor $\Omega(-)$ is a one-sided inverse to $B(-)$ up to homotopy in the sense that $\Omega B G \simeq G$.  
So $\pi_i(BG) \cong \pi_{i-1}(G)$ for all $i\geq 1$, as one already sees from the long exact sequence in homotopy of the fibration $G \to EG \to BG$.
We will often consider classifying spaces $BG$ for discrete $G$, in which case $EG \to BG$ is a covering space with deck transformation group $G$, and  $BG$ is aspherical, i.e.~an Eilenberg--MacLane space $K(G,1)$.  
The functor $B(-)$ takes a short exact sequence of discrete groups to a fibration of $K(G,1)$ spaces.
%So if $G$ is discrete, $\Omega B G \cong G$.  
Recall that if a $K(G,1)$ space is a finite-dimensional CW complex, then $G$ must be torsion-free.

We list further relevant results, with some more useful terminology and notation:
\begin{itemize}
[leftmargin=0.25in]
\item 
(Alexander \cite[Theorem 1.1]{Hatcher3Mfds}): Every smoothly embedded $S^2$ in $S^3$ bounds a ball $D^3$.
Every smoothly embedded torus in $S^3$ bounds a solid torus on at least one side.  (Note also that a torus in $S^3$ is unknotted if and only if it bounds a solid torus on both sides.)
\item 
(Smale \cite{Smale:DiffD2}, 1959): $\Diff^+(S^2) \simeq SO_3$ or equivalently $\Diff(D^2; \d) \simeq *$.
\item 
(Earle--Eells \cite{Earle-Eells}, 1969; Earle--Schatz \cite{Earle-Schatz}, 1970; Gramain \cite{GramainMCG}, 1973):  
\begin{itemize}[leftmargin=*]
\item $\Diff(S^1 \x I; \d) \simeq \Omega \Diff^+(S^1) \simeq \Omega S^1 \simeq \Z$.  (A diffeomorphism representing a generator of $\pi_0 \Diff(S^1 \x I; \d)$ is called a \emph{Dehn twist}.  For a simple closed curve $C$ on any surface $S$, it gives rise to a diffeomorphism of $S$ supported in a neighborhood of $C$, called a \emph{Dehn twist along $C$}.)
\item $\Diff_0(S^1 \x S^1) \simeq S^1 \x S^1$. 
\item For a compact orientable surface $S$ of genus $\geq 2$, possibly with boundary, $\Diff_0(S; \d S) \simeq \ast$.
\end{itemize}
\item 
Let $P_n := D^2 - (\coprod_{i=1}^n \mathrm{int} D_i)$, where each $D_i$ is a disk in the interior of $D^2$.  Then $P_n \cong \Sigma_0^{n+1}$.  Call $\d D^2$ the \emph{outer boundary component} and the $\d D_i$ the \emph{inner boundary components}.
We have $\mathrm{PMod}(\Sigma_{0,r}^{n+1}) \cong \Z^n \x \PB_{n+r}$.   Indeed, the map to $\PB_{n+r}\cong \mathrm{PMod}(\Sigma_{0,n+r})$ is given by  \emph{capping} the $n$ inner boundary components, i.e.~gluing punctured disks to them.  The map to $\Z^n$ is obtained by $n$ maps $\phi_i$ that ``fill in'' all the punctures $x_1, \dots, x_r$ and all the $n$ removed disks except the $i$-th one (and using that $\Diff(S^1 \x I; \d) \simeq \Z$).
This group mutually generalizes the pure braid group $\PB_r$ ($n=0$) and the group of framed $n$-strand pure braids ($r=0$).
\item 
The center $Z(\PB_n)$ is isomorphic to $\Z$, generated by a Dehn twist along the boundary of $D^2$.  For $n\geq 3$, $\PB_n/Z(\PB_n) \cong \mathrm{PMod}(\Sigma_{0,n+1})$
The center splits off as a direct factor, i.e.~for $n\geq3$, $\PB_n \cong \Z \x  \mathrm{PMod}(\Sigma_{0,n+1})$.  See \cite[Section 9.3]{Farb-Margalit} for details.
\item 
(Hatcher \cite{HatcherIncomprSurfs}, 1976; Ivanov \cite{Ivanov, Ivanov2}, 1976--79):
\begin{itemize}[leftmargin=*]
\item
For an orientable, connected, compact, irreducible 3-manifold $M$, and an incompressible surface $S$ in $M$,  $\Emb(S, M; \d S)$ has contractible components, provided $S$ is neither a torus nor a fiber of a bundle structure on $M$.  If $S$ is a torus that is not the fiber of a bundle structure on $M$, then each component of $\Emb(S, M; \d S)$ is equivalent to $\Diff_0(S^1 \x S^1) \simeq S^1 \x S^1$.
\item
For a \emph{Haken} 3-manifold $M$, $\Diff(M; \d M)$ has contractible components.  
\end{itemize}

It suffices to recall that for orientable 3-manifolds, $M$ is Haken if and only if it is compact and irreducible and contains an incompressible surface.  The exterior of a nontrivial, irreducible link is a Haken manifold, since the boundary of a tubular neighborhood of any component is an incompressible surface.  

The stated results about $\Emb(S, M; \d S)$ also hold when $S=\coprod_{i=1}^n S_i$ is a disjoint union of incompressible surfaces $S_i$.  Indeed, this follows by induction on $n$ from the fibration 
\[
\Emb \left(S_n, M - \coprod_{i=1}^{n-1} S_i \right) \to \Emb \left(\coprod_{i=1}^n S_i, M \right) \to 
\Emb \left(\coprod_{i=1}^{n-1} S_i, M \right)
\]
and its long exact sequence in homotopy.
\item (Hatcher \cite{HatcherSmaleConj}, 1983):  $\Diff(D^3; \d) \simeq *$ or equivalently $\Diff^+(S^3) \simeq SO_4$.
\end{itemize}

With the above results in hand, the proofs of the last two results in this subsection (Lemma \ref{DiffSolidTorus} and Lemma \ref{DiffExteriorOfHopfLink}) are not long, so we include them.  We need Lemma \ref{DiffSolidTorus} only for $g=1$, but the proof for arbitrary $g$ is just as easy.

\begin{lemma}
\label{DiffSolidTorus}
For a handlebody $M$ of any genus $g\geq 0$, $\Diff(M; \d) \simeq *$.  In particular, 
$\Diff(S^1 \x D^2; \d) \simeq *$.
\end{lemma}
\begin{proof}
We proceed by induction on the genus $g$, where the basis case $g=0$ holds by Hatcher's work on the Smale Conjecture.
Let $M$ be a handlebody of genus $g$, and let $S$ be a properly embedded non-separating $D^2$ in $M$.  
Restricting a diffeomorphism to $S$ gives a fibration
\begin{equation}
\label{SurfaceInM}
\Diff(M'; \d M') \to \Diff(M; \d M) \to \Emb(S,M; \d S)
\end{equation}
where $M'$ is the genus-$(g-1)$ handlebody obtained by cutting $M$ along $S$.  By the induction hypothesis, the fiber is contractible, and by the theorem of Hatcher and Ivanov, the base has contractible components.  To see that the base is connected, consider two embeddings $e, e'$of $S$ in $M$ which agree with the inclusion on $\d S$.  By considering innermost circles in $e(S) \cap e'(S)$, we can find an isotopy of $e'$ to reduce the number of such circles.  We ultimately conclude that $e(S) \cup e'(S)$ is a 2-sphere in  $M$.  By embedding $M$ in $\R^3$, we view this 2-sphere as a subspace of $\R^3$.  By the theorem of Alexander, it bounds a 3-ball, and $e$ and $e'$ are isotopic.  Hence the base space is connected, and the total space is contractible.
\end{proof}

Lemma \ref{DiffExteriorOfHopfLink} is a special case of Proposition \ref{DiffSeifertFramed}, but it will help to first understand this case directly.
Understanding diffeomorphisms of $S^1 \x S^1 \x I$ will ultimately give us the homotopy type of the space of the Hopf link in $S^3$.
Considering only those that fix the boundary pointwise below rules out the rotations of each $S^1$ factor.  This is reflected by the fact that $\Diff(S^1 \x S^1; \d)$ is homotopy-discrete, whereas each component of $\Diff(S^1 \x S^1)$ is equivalent to $S^1 \x S^1$.

\begin{lemma}
\label{DiffExteriorOfHopfLink}
$
\Diff(S^1 \x S^1 \x I; \d)  \simeq \Omega(\Diff_0(S^1 \x S^1)) \simeq \Omega (S^1 \x S^1) \simeq  \Z \x \Z 
$.
\end{lemma}

\begin{proof}
By the theorem of Hatcher and Ivanov, it suffices to consider $\pi_0$ of this space, since $S^1 \x S^1 \x I$ is Haken.  
Define a map $\Phi: \pi_0(\Diff(S^1 \x S^1 \x I)) \to \Z \x \Z$ as follows.
Write $S^1=[0,2\pi]/\sim$ with $0$ as its basepoint, and for $f \in \Diff(S^1 \x S^1 \x I)$, consider the union of segments $\{0\} \x \{0\} \x I \cup -f(\{0\} \x \{0\} \x I)$ as a class in $H_1(S^1 \x S^1 \x I) \cong \Z \x \Z$.  
This map descends to a map $\Phi$ on $\pi_0$.  
The surjectivity of $\Phi$ is established by considering the two lifts of the Dehn twist that generate $\pi_0\Diff(S^1 \x I; \d)$ using the two possible projections  $S^1 \x S^1 \x I \to S^1 \x I$.  Explicitly, these generating diffeomorphisms are given by $(\theta, \phi, t) \mapsto (\theta + 2\pi t, \phi, t)$ and $(\theta, \phi, t) \mapsto (\theta, \phi+ 2\pi t, t)$.

For injectivity, suppose $\Phi(f)=(0,0)$.  
Let $A$ be the annulus $S^1 \x \{0\} \x I$.
Consider the lift of $f$ to a diffeomorphism $\widetilde{f}$ of the universal cover $\R \x \R \x I$, and let $\widetilde{A}$ be the subspace corresponding to the universal cover of $A$.  
Then since $\Phi(f)=(0,0)$, the restriction of $\widetilde{f}$ to $\d \widetilde{A}$ is the identity.  By a straight-line homotopy, $\widetilde{f}|_{\widetilde{A}}$ is homotopic to the inclusion, and this homotopy descends to a homotopy from $f|_A$ to the inclusion of $A$.
A theorem of Waldhausen \cite[Corollary 5.5]{Waldhausen:Annals1968} says that a homotopy of an incompressible surface constant on the boundary gives rise to an isotopy with the same endpoints, so there is an isotopy from $f|_A$ to the inclusion of $A$.  By isotopy extension, we get a diffeotopy of the whole space $S^1 \x S^1 \x I$ which restricts to the identity on both its boundary and the annulus $A$.  
By restricting diffeomorphisms to $A$, we essentially cut along $A$ to get a solid torus:
\[\pi_0\Diff(S^1 \x I \x I; \d) \to \ker \Phi \to \pi_0\Diff(A; \d)\]
By Lemma \ref{DiffSolidTorus}, $\pi_0\Diff(S^1 \x I \x I; \d)$ is trivial\footnote{Note that we are only using the connectedness of this space of diffeomorphisms, which can be deduced from Cerf's result $\pi_0\Diff^+(S^3)=\{e\}$, obtained well before the proof of the Smale conjecture.}, so the second map above is injective.  But any element of $\ker \Phi$ maps to $0 \in \Diff(A; \d)$ because the first component of $\Phi(f)$ is zero.  So $\ker \Phi$ is trivial.
\end{proof}

\begin{definition}
In the context of other 3-manifolds, we will use the two generators of this group $\Z^2$, i.e.~the lifts of the annular Dehn twists to $S^1 \x S^1 \x I$ in the proof of surjectivity above.  Specifically, if $T$ is a torus boundary component of a submanifold $M \subset S^3$, they extend to diffeomorphisms of $M$ supported in a collar neighborhood of $T$.  
By Alexander's theorem, $T$ bounds a solid torus on at least one side.  Given such a solid torus, we can identify one factor of $S^1$ in $S^1 \x S^1 \x I$ as a meridian and the other as a longitude, and we will refer to these diffeomorphisms of $M$ as a \emph{meridional Dehn twist} and a \emph{longitudinal Dehn twist} according to the factor which appears in the support.  Beware that these do \emph{not} correspond to the identically named Dehn twists on the torus: in each case, the diffeomorphism is the identity on one of the $S^1$ factors, not on the $I$ factor.
\end{definition}

\subsection{Spaces of knots in certain 3-manifolds via spaces of links}
\label{S:KnotsVia2CompLinks}
To relate knots in $S^1 \x D^2$ and $S^1 \x S^1\x I$ to links in $S^3$, first write the 3-sphere as the unit sphere in $\C^2$:
\[
S^3 := \{ (z_1,z_2) \in \C^2 : |z_1|^2 + |z_2|^2 = 1\}.
\]
There is a solid torus $U \subset S^3$ given by
\[
U:=\left\{(z_1,z_2) \in \C^2 : |z_1|^2 + |z_2|^2 = 1,\ |z_2|^2 \leq \frac{1}{{2}} \right\}
\]
with boundary the torus 
\[
T:=\left\{(z_1,z_2) \in \C^2 : |z_1|^2= \frac{1}{2}=|z_2|^2 \right\}.
\]
The closure of the complement of $U$ in $S^3$ is the solid torus
\[
U':= \left\{(z_1,z_2) \in \C^2 : |z_1|^2 + |z_2|^2 = 1,\ |z_2|^2 \geq \frac{1}{{2}} \right\}.
\]
Then $S^3 = U \cup U'$ is a Heegaard decomposition of $S^3$ into two solid tori.  Writing $S^3 = \R^3 \cup \{\infty\}$, we view the solid torus $U$ as a subset of $\R^3$, so $\infty \in U'$.  The longitudinal and meridional rotations $\lambda$ and $\mu$ of $U$ extend to rotations of $S^3$ which respectively fix and do not fix $\infty$.  The solid tori $U'$ and $U$ are respectively neighborhoods of the components $C_1$ and $C_2$ of the Hopf link, shown in Figure \ref{F:SomeSeifertFiberedLinks} after a rotation that puts it in $\R^3$.  So $\lambda$ is a rotation along $C_1$ and $\mu$ is a rotation along $C_2$.

This decomposition helps describe the closure of a long knot.  Given a framed long knot $F \in \tK$, we obtain a framed closed knot $\overline{F} \in \tL$ by gluing together the ends of the solid cylinder $I \x D^2$ in the domain and codomain of $F$.  
We view the result in the domain as $S^1 \x D^2$, identifying the basepoint of $I / \d I$ with $1 \in S^1$ and the forward orientation of $I$ with the counter-clockwise orientation of $S^1$.   
We view the codomain as $U'$, identifying the basepoint of $I/ \d I$ with $\infty \in S^3$.  
Thus $\overline{F}$ is a closed framed knot in $S^3$ sending $(1,0)$ to $\infty$ and with prescribed values and derivatives on all of $\{1\} \x D^2$.
A similar construction with $I$ instead of $I \x D^2$ yields a closed knot $\overline{f} \in \L$ from a long knot $f\in \K$.

Given a two-component link $f=(f_0,f_1)$ in $S^3$, we may view $f$ as the knot $f_0$ in the exterior of $f_1$.  See Figure \ref{F:ExamplesOfKnotsInSolidTorus}.

\begin{figure}[h!]
\includegraphics[scale=0.23]{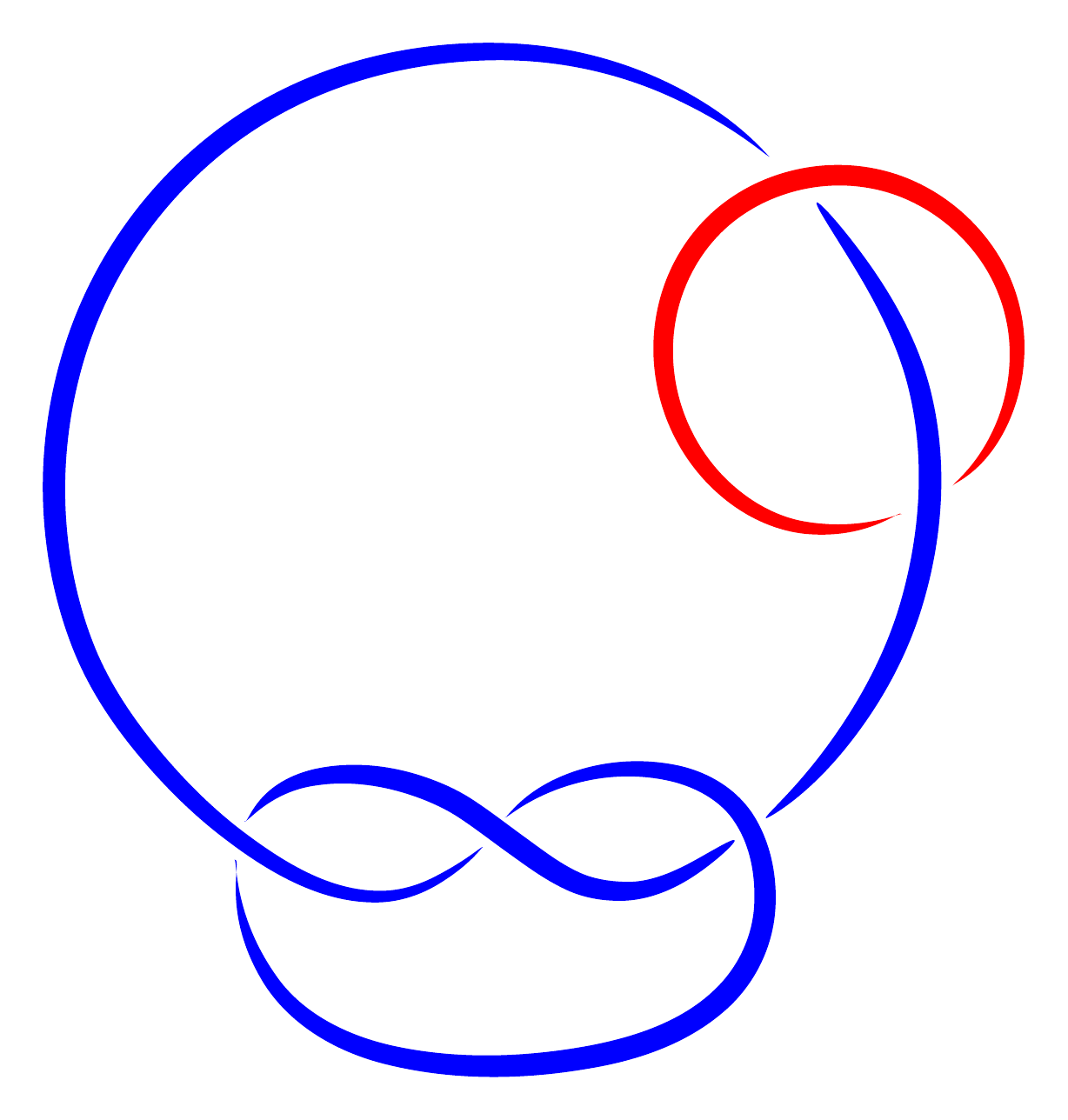}
\qquad \qquad
\includegraphics[scale=0.23]{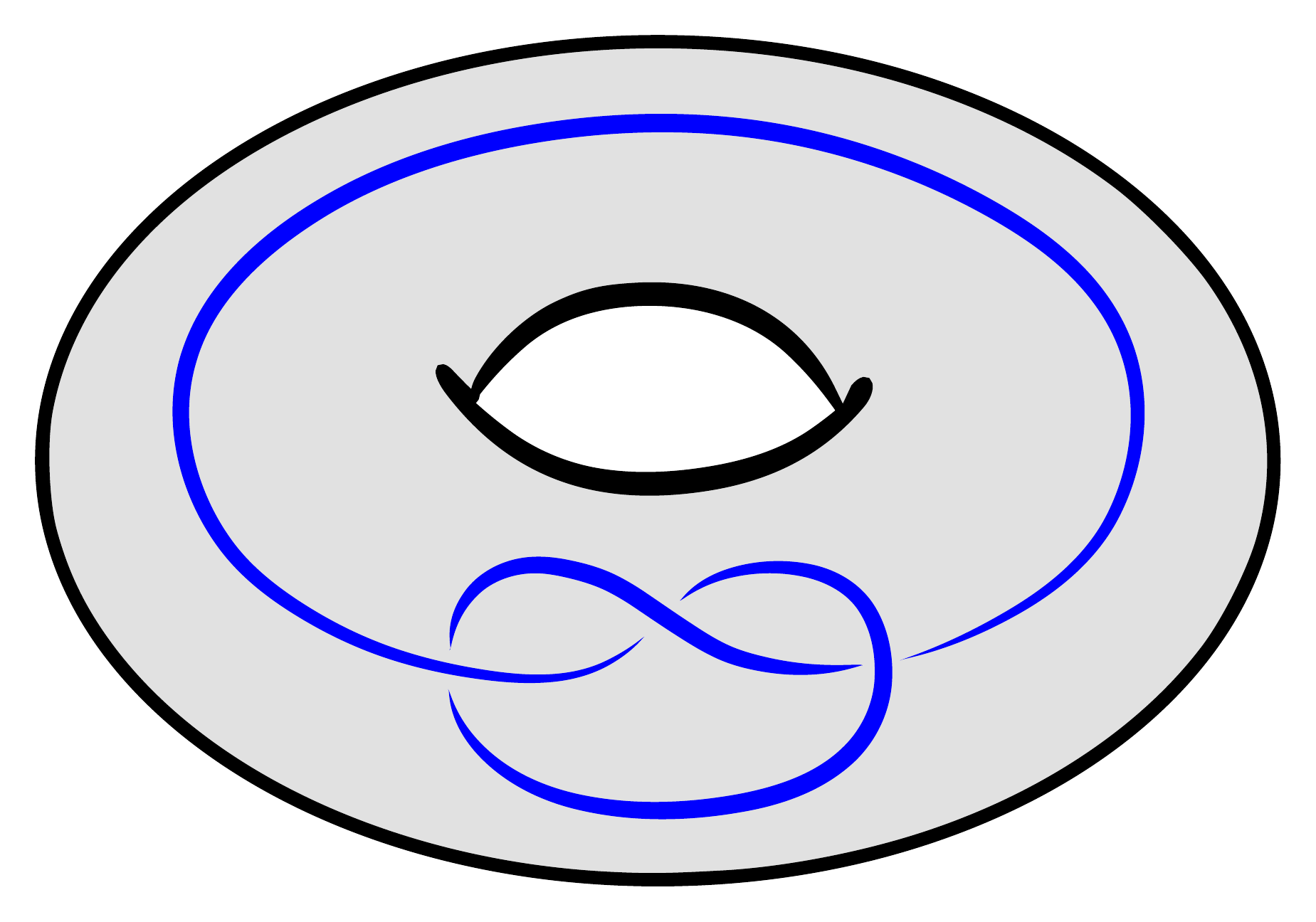}\\
\vspace{2pc}
\includegraphics[scale=0.18]{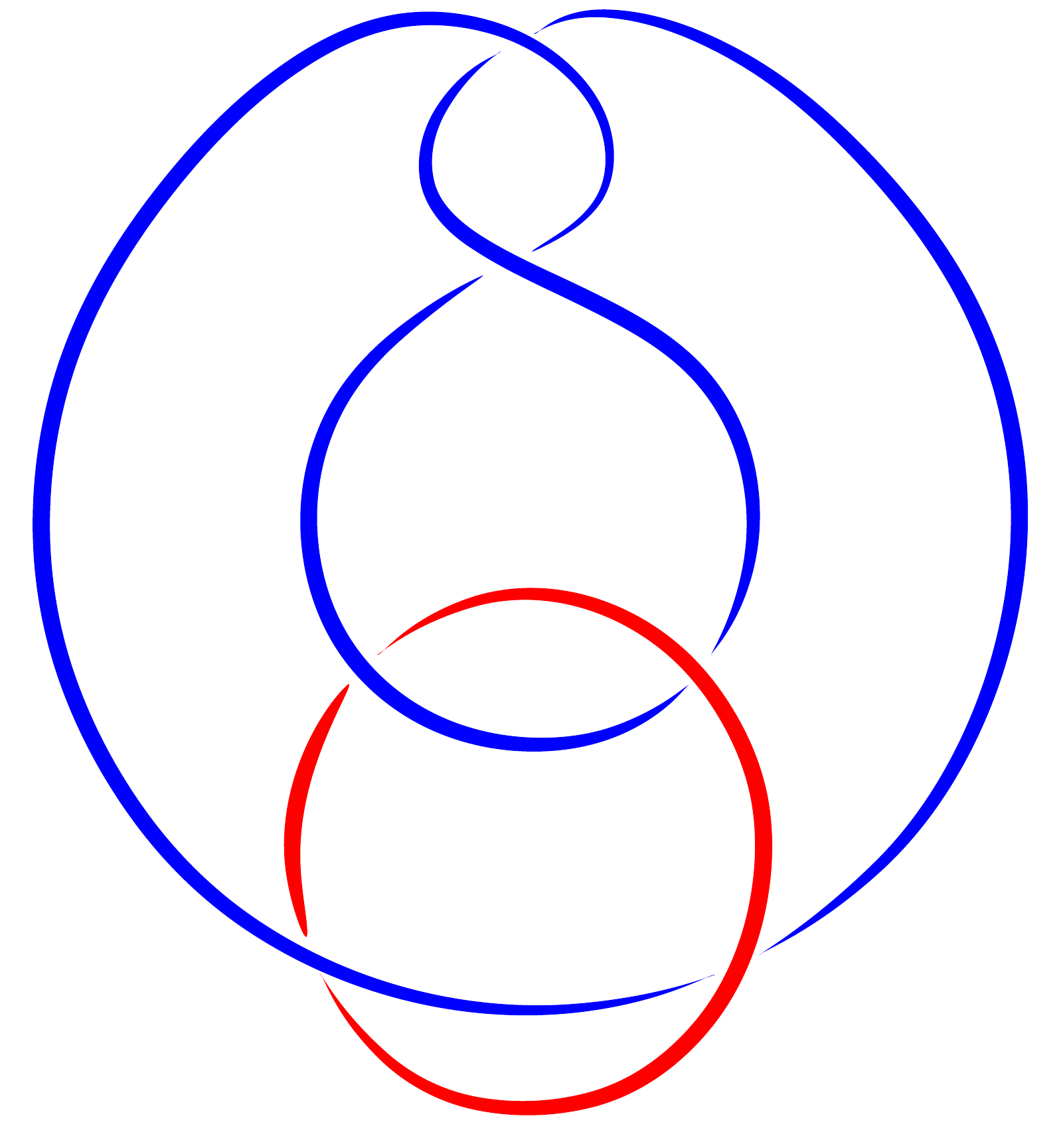}
\qquad \qquad
\includegraphics[scale=0.23]{whitehead-in-solid-torus.pdf}
\caption{Two examples of 2-component links $f=(f_0,f_1)$ with $f_1$ unknotted (left), together with their respective associated knots in the solid torus (right).}
\label{F:ExamplesOfKnotsInSolidTorus}
\end{figure}

\begin{proposition}
\label{TfandLinksWithUnknots}
Let $f=(f_0,f_1)$ be a 2-component link in $S^3$.
Then the pair $(S^3 - \nu(f_1), f_0)$ is diffeomorphic to the pair $(U, g)$ for some knot $g$ in the solid torus $U$ if and only if $f_1$ is unknotted.
%corresponds to a knot in $U=S^1 \x D^2$ if and only if (at least) one component is unknotted.  
 In this case, an isotopy of $f$ gives rise to an isotopy of $g$ and vice-versa.  Thus isotopy classes of knots in $S^1 \x D^2$ are in bijection with isotopy classes of 2-component links in $S^3$ where the second component is unknotted.
\end{proposition}
\begin{proof}
For the first statement, $S^3 - \nu(f_1)$ is diffeomorphic to a solid torus if and only if $f_1$ is the unknot, since the complement of any nontrivial knot has a $\Z^2$ subgroup in its fundamental group \cite[p.~104]{Rolfsen}.  
In this case, if $h:S^3 - \nu(f_1) \to U$ is such a diffeomorphism, then $g=h\circ f_0$.  
For the second statement, suppose that $f_1$ is an unknot and that $h : S^3 - \nu(f_1) \to U$ is a diffeomorphism.  Let $f_t$ be an isotopy of $f$ in $S^3$ and $(f_i)_t$ the restriction to the $i$-th component.  Let $F_t$ be the extension to a diffeotopy of $S^3$, guaranteed by the isotopy extension theorem \cite[Theorem 8.1.3]{Hirsch}.  
Then $t \mapsto h \circ (F_t)^{-1} \circ (f_0)_t$ is an isotopy of the knot $g$ in $U$.  
For the reverse direction, an isotopy $g_t$ of the knot $g$ in $U$ extends to an diffeotopy $G_t$ of $U$ which the identity on $\d U$.  In turn, $h^{-1} \circ G_t$ extends by the identity to a diffeotopy of $S^3$ (which fixes $f_0$).  
The third statement follows from the first two.
\end{proof}

Proposition \ref{TfandLinksWithUnknots} concerns only isotopy classes, i.e.~path components in embedding spaces.  
(The topology of those components is addressed in Proposition \ref{TfisKpi1}.) 
Nonetheless, it suffices for the purposes of associating (framed) 2-component links to knots $f$ in the solid torus and vice-versa.

Now define
\[
V:= \left\{(z_1,z_2) \in \C^2 : |z_1|^2 + |z_2|^2 = 1,\ \frac{1}{3} \leq |z_1|^2, |z_2|^2 \leq \frac{2}{3} \right\}.
\]
Thus $V$ is the closure of a neighborhood of the torus $T$, so $V \cong S^1 \x S^1 \x I$.  The meridional and longitudinal rotations of $U$ mentioned above extend to rotations of $V$.
A similar argument as in the proof of Proposition \ref{TfandLinksWithUnknots} establishes an analogous result for knots in $V$:

\begin{proposition}
\label{VfandLinksWithHopf}
Let  $f=(f_0,f_1,f_2)$ be a 3-component link in $S^3$.  
The pair $(S^3 - \nu(f_1 \cup f_2), f_0)$ is diffeomorphic to $(V, g)$ for some knot $g:S^1 \incl U$ if and only if $(f_1, f_2)$ is a Hopf link.  An isotopy of $f$ gives rise to an isotopy of $g$ and vice-versa.  Thus isotopy classes of knots in $S^1 \x S^1 \x I$ are in bijection with isotopy classes of 3-component links in $S^3$ where the last two components are a Hopf link.
\qed
\end{proposition}

\begin{notation}
Let $\T$ be the space of knots in a solid torus, that is, smooth embeddings $S^1 \incl S^1 \x D^2$.  
By Proposition \ref{TfandLinksWithUnknots}, a 2-component link $f$ determines a path-component $\T_f$ in $\pi_0(\T)$.  Similarly, a knot  $f$ in the solid torus determines a path-component $\L_f$ in $\pi_0(\L)$.

Let $\V$ be the space of knots in thickened torus, that is, smooth embeddings $S^1 \incl S^1 \x S^1 \x I$.  
By Proposition \ref{VfandLinksWithHopf}, a 3-component link $f$ determines a path-component  $\V_f$ in $\pi_0(\V)$.  Similarly, a knot  $f$ in the thickened torus determines a path-component $\L_f$ in $\pi_0(\L)$.

In either of these two cases, we will often use the same letter $f$ to refer to both the knot and the corresponding link.
\end{notation}

\subsection{Preliminaries on Seifert-fibered links}
\label{S:SeifertPrelims}
Seifert-fibered manifolds form one of the two classes of submanifolds in the JSJ decomposition.  Following \cite{Hatcher3Mfds} and \cite{BudneyJSJ}, we review their definition and the classification of Seifert-fibered link complements, which are precisely all Seifert-fibered submanifolds of $S^3$.

A Seifert fibering of a 3-manifold $M$ consists of a map $M\to B$ to a surface $B$ with genus $g$ and $n$ boundary components.  On the complement of finitely many points $x_1, \dots, x_r$, the map is an $S^1$-bundle.  
We call the preimages of the $x_i$ \emph{singular fibers}.
For a Seifert-fibered submanifold of $S^3$, $B$ must be orientable.  In this case, one can construct $M$ by starting with $B\x S^1$ and removing a solid torus neighborhood of each $\{x_i\} \x S^1$.  For each $i$, one then glues back a solid torus with its meridian attached along a curve of slope $a_i/b_i \in \Q$, where $a_i$ and $b_i$ correspond to the fiber and base direction respectively.
For $B$ orientable, this data determines $M$, and we write $M=M(g, n; a_1, b_1, \dots, a_r, b_r)$.  Assuming $n>0$, another such Seifert-fibered manifold $M(g, n; a'_1, b'_1, \dots, a'_r, b'_r)$ is diffeomorphic to this one if and only if $a_i/b_i \equiv a'_i/b'_i \mod 1$ for all $i$.  Thus integer slopes $a_i/b_i$ may be removed from the list.

Define the following subsets of $S^3$:
\begin{itemize}
\item
$T_{p,q}:=\{ (z_1,z_2) \in S^3  : z_1^q = z_2^p \} = \{(e^{2\pi p it}/\sqrt{2},\  e^{2 \pi q it}/\sqrt{2}): t \in \R\}$
where $p,q \in \Z-\{0\}$.
\item
$C_1:=\{(z_1,0) \in S^3 \}=\{(e^{2\pi it},0): t \in \R\}$
\item
$C_2:=\{(0,z_2) \in S^3 \} = \{(0,e^{2\pi it}): t \in \R\}$
\item 
$S_{p,q}:=T_{p,q} \cup C_2$
\item
$R_{p,q}:=T_{p,q} \cup C_1 \cup C_2$
\item
$KC_n := C_1 \cup \bigcup_{k=1}^n \{(e^{2\pi i k/n}/\sqrt{2}, \ e^{2\pi i t}/\sqrt{2}) : t \in \R \}$
\end{itemize}

\begin{figure}[h!]
\includegraphics[scale=0.21]{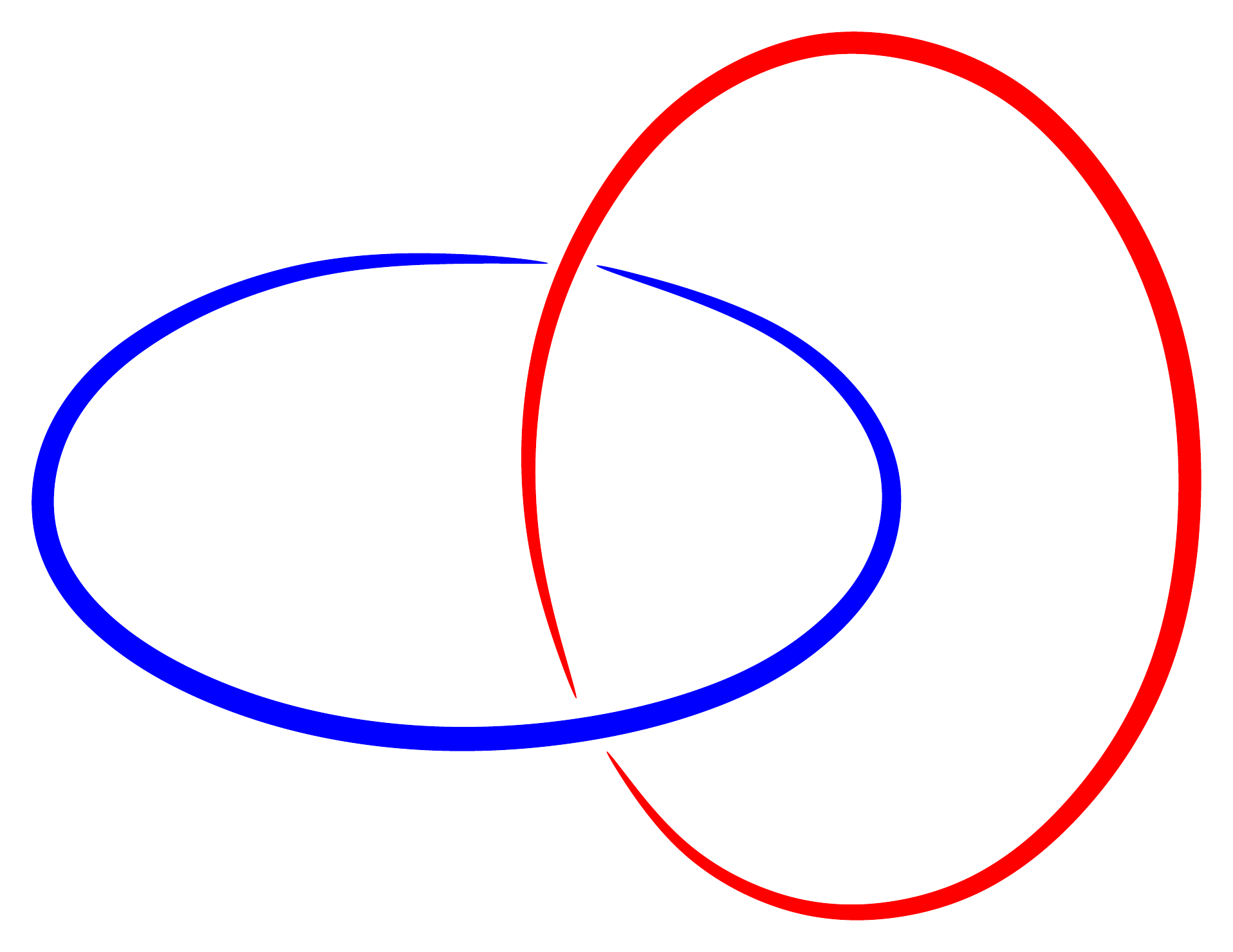} \qquad \qquad
\includegraphics[scale=0.25]{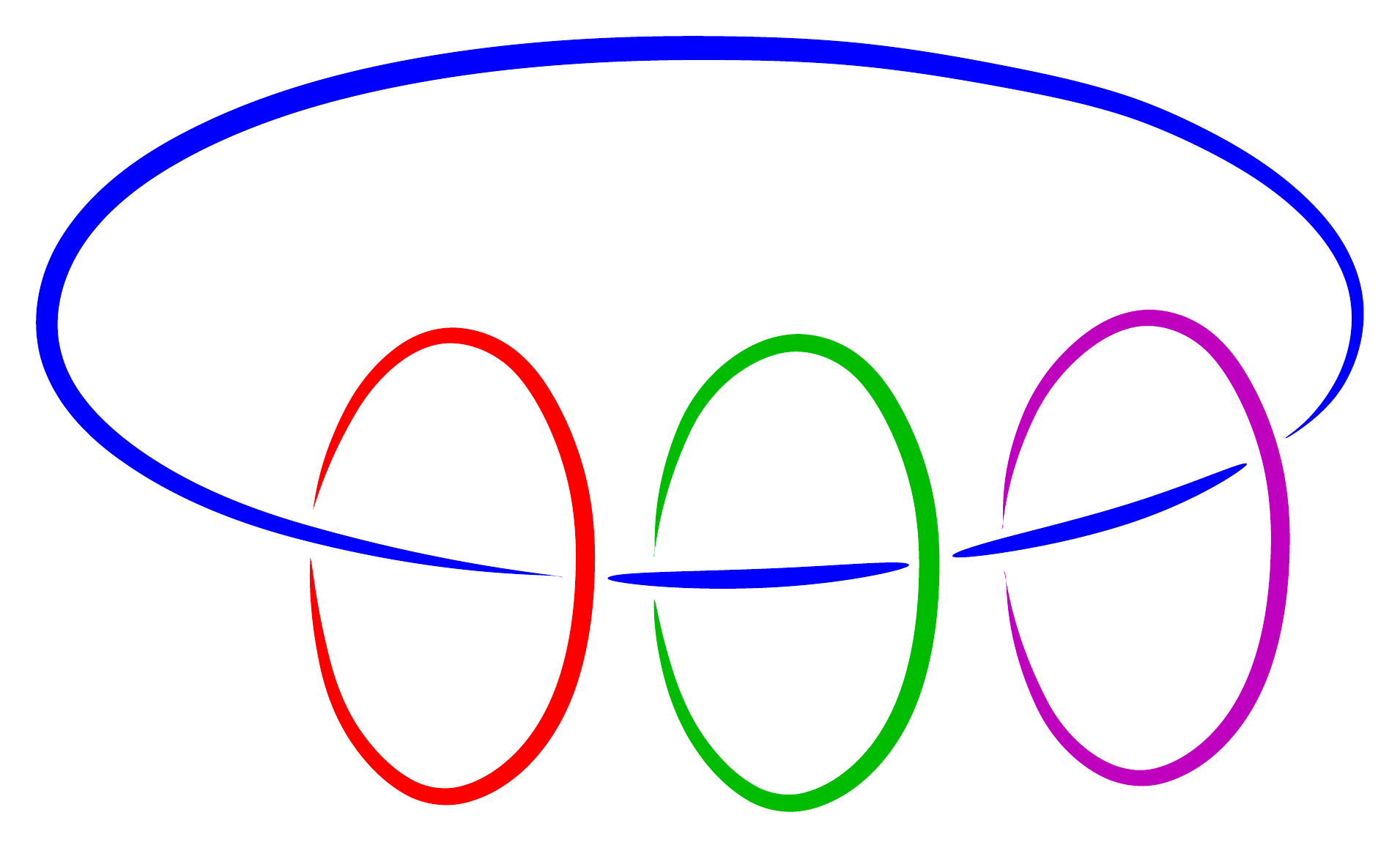}
\caption{The Hopf link $(\sim C_1 \cup C_2 \sim KC_1)$ and the keychain link $KC_3$.}
\label{F:SomeSeifertFiberedLinks}
\end{figure}

We visualize $C_1$ in the $xy$-plane and $C_2$ tangent to the $z$-axis at $0\in \R^3$.
Beware that for some authors, the roles of $p$ and $q$ in $T_{p,q}$ may be reversed, 
 the definitions or depictions of $C_1$ and $C_2$ may be reversed, and
 $S_{p,q}$ may include $C_1$ instead of $C_2$.

If $\gcd(p,q)=1$ and $p \neq 1 \neq q$, then $T_{p,q}$ has one component, is not the unknot, and is called the $(p,q)$-\emph{torus knot}.  
It winds $p$ times around $C_2$ and $q$ times around $C_1$.
The second condition excludes $T_{1,q}$ and $T_{p,1}$, which are precisely those $T_{p,q}$ which are unknots.

In general, the link $T_{p,q}$ is called the $(p,q)$-\emph{torus link}.  
The link $T_{p,q}$ is isotopic to $T_{q,p}$.  Negating either $p$ or $q$ reverses the orientation on all the components, but this link is isotopic to the original one, since torus knots are invertible.
In general, $T_{p,q}$ is a link with $\gcd(p,q)$ components.  Each component is a copy of the torus knot $T_{p',q'}$ where $p'=p/\gcd(p,q)$ and $q'=q/\gcd(p,q)$.  
This can be seen by viewing $T_{p,q}$ as a union of parallel lines of slope $q'/p'$ on $\R^2 / \Z^2$.
The linking number between a pair of components is $p'q'$.\footnote{Indeed, the linking number $Lk$ of $T_{p',q'}$ with a normal perturbation of $T_{p',q'}$ on the torus satisfies $Lk=Tw+Wr$, where $Tw$ is twist and $Wr$ is writhe \cite{Calugareanu:Gauss, Pohl:SelfLinking}.  By writing $T_{p',q'}$ as the closure of the appropriate $p'$-strand braid, $Wr=q'(p'-1)$, the number of (positive) crossings.   The twist $Tw$ is the linking number with $C_1$, which is $q'$.}  For example $T_{2,4}$ is a 2-component link with unknotted components and linking number 2, $T_{4,6}$ consists of 2 trefoils with linking number 6, and $T_{n,n}$ is a union of fibers in the Hopf fibration, any pair of which has linking number 1.

\begin{figure}[h!]
(a)  \includegraphics[scale=0.24]{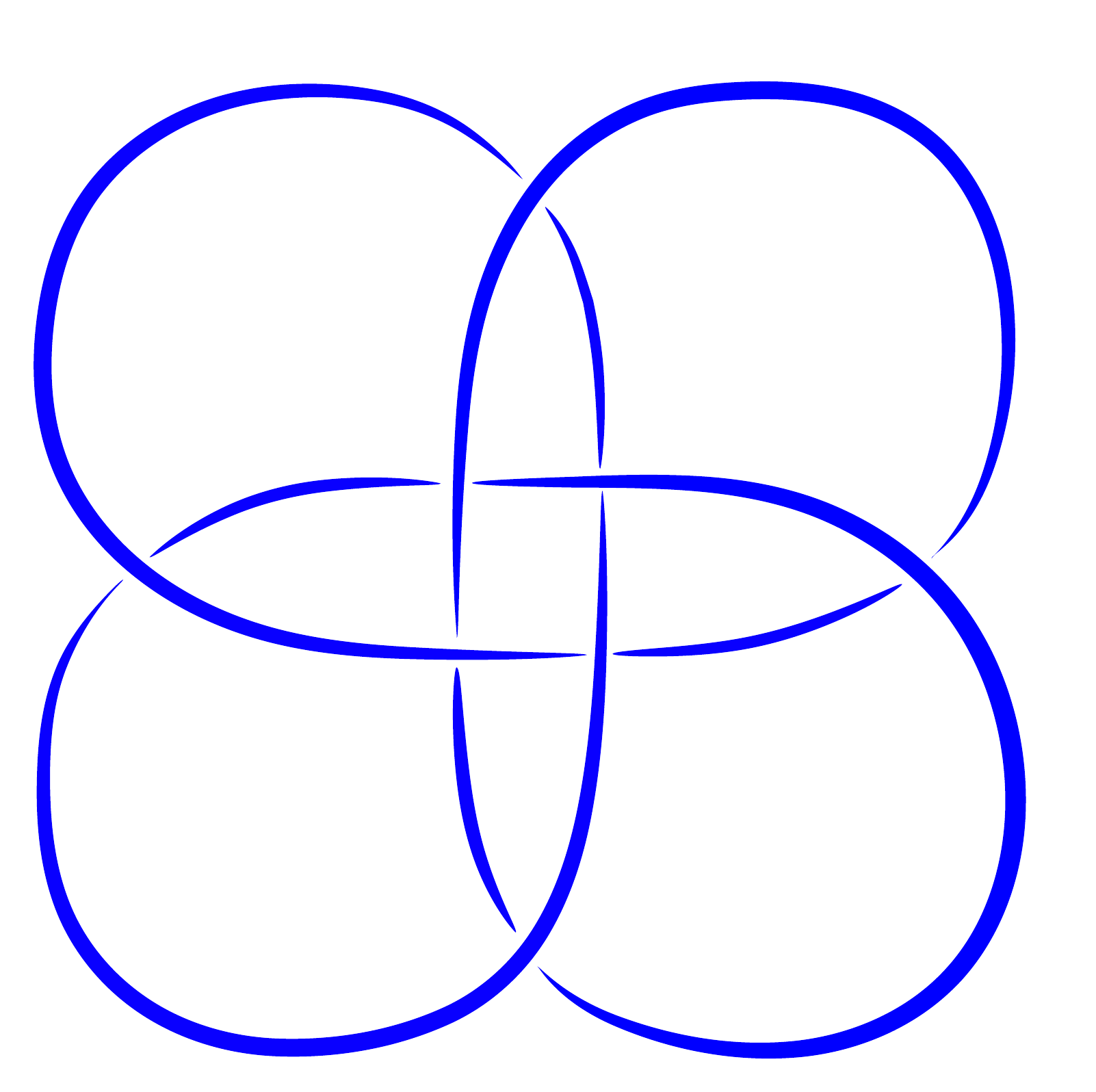} \qquad
(b) \includegraphics[scale=0.24]{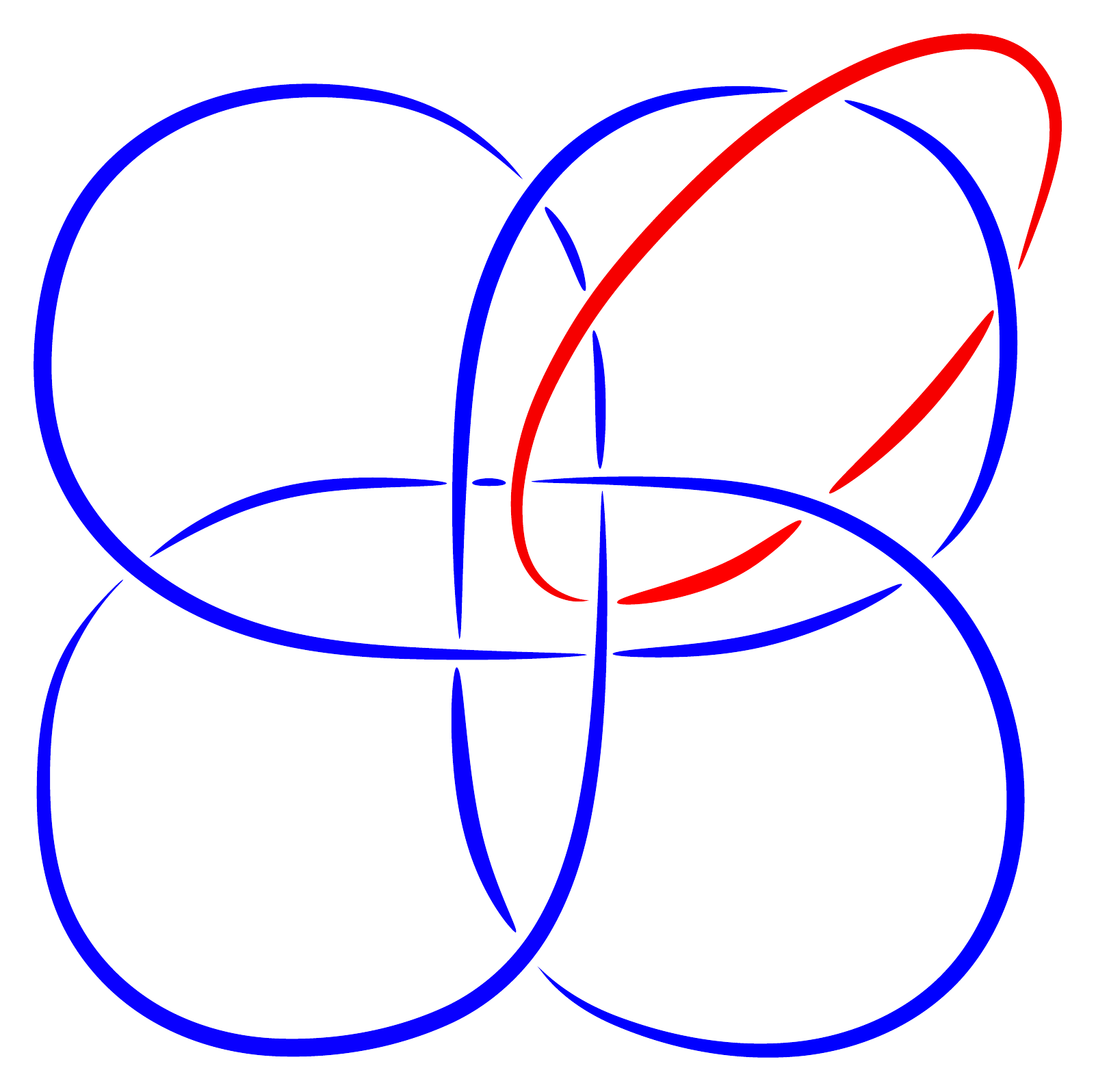} \qquad
(c)\includegraphics[scale=0.24]{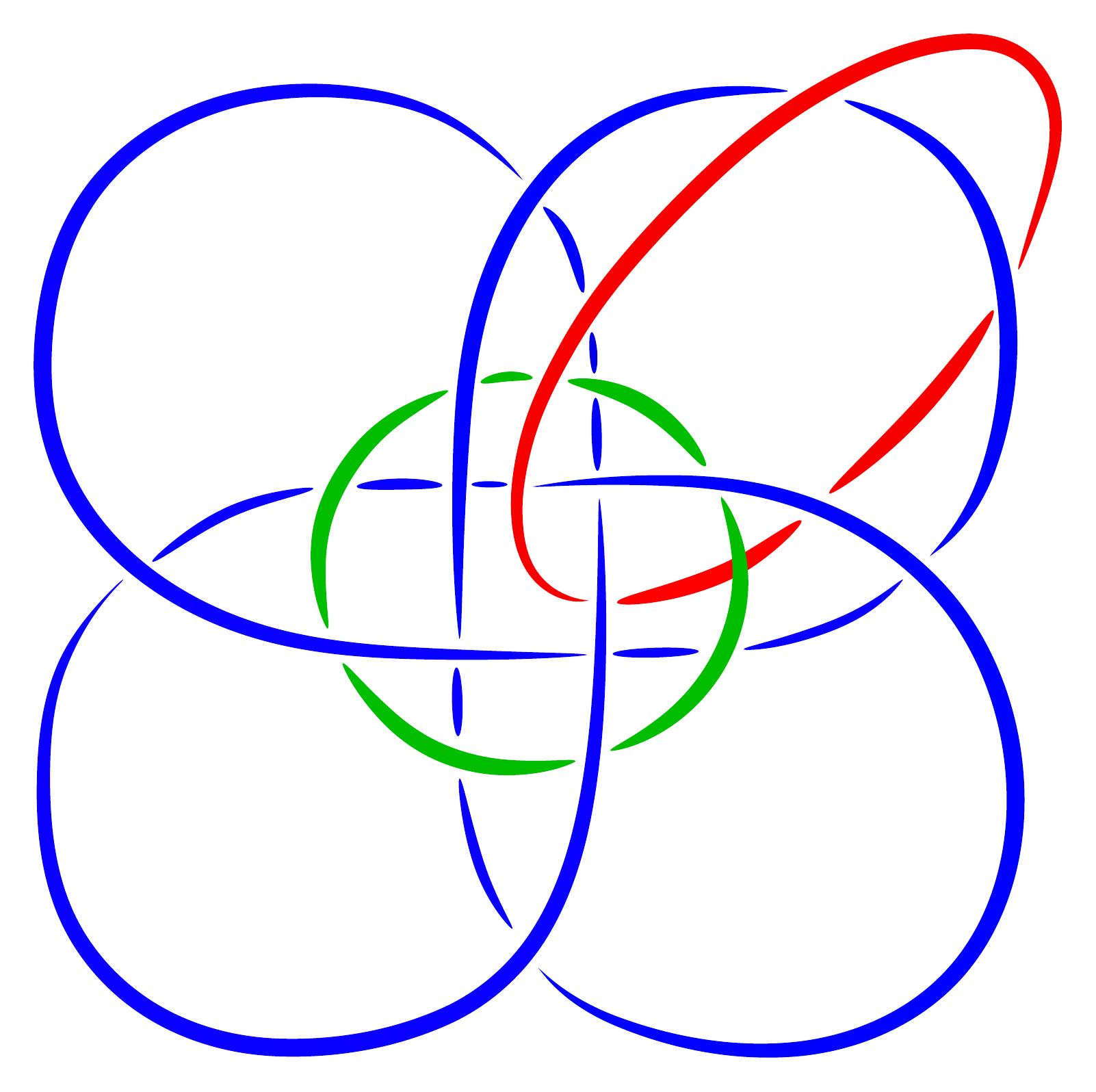} \\
(d) \includegraphics[scale=0.30]{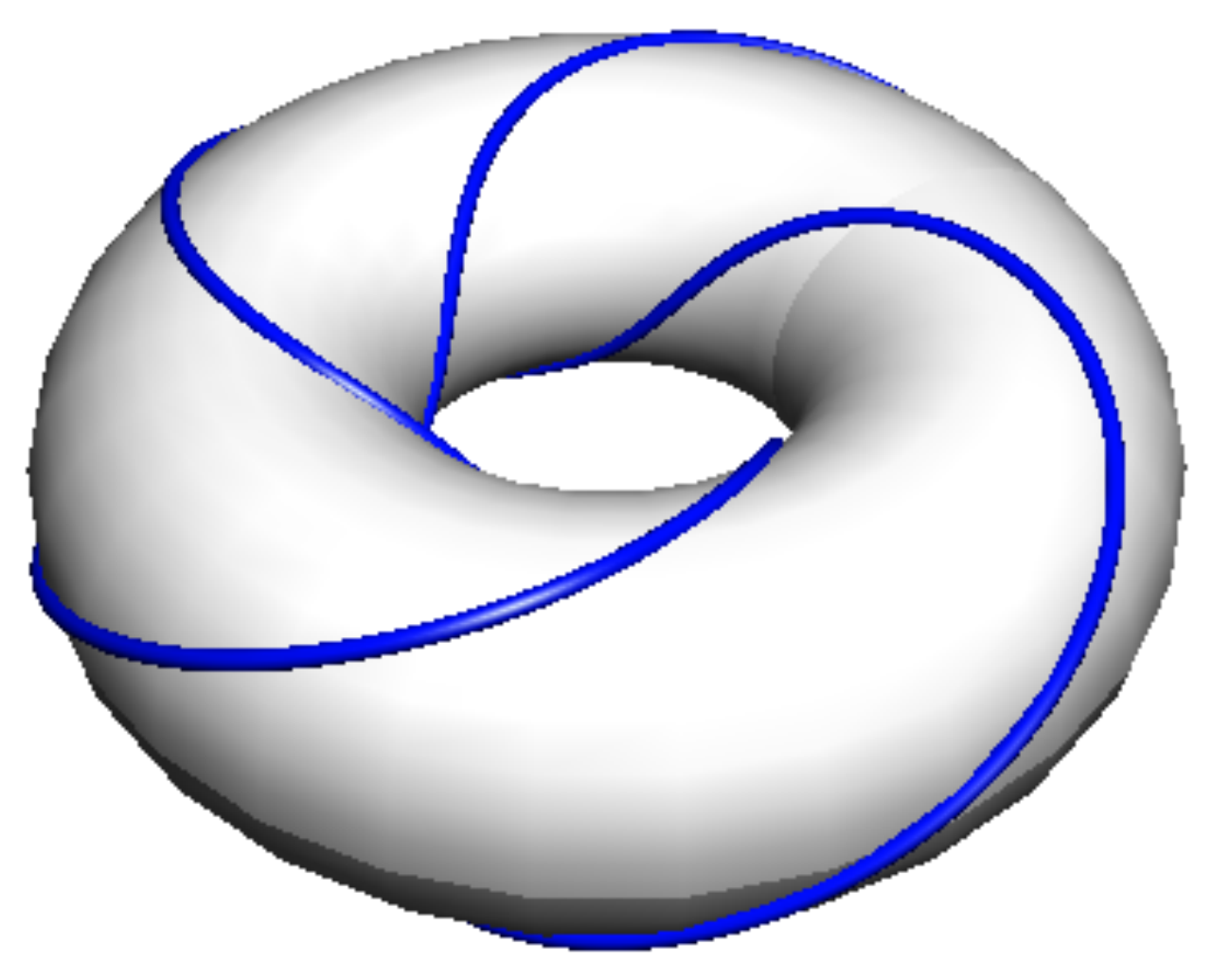} \
(e) \includegraphics[scale=0.30]{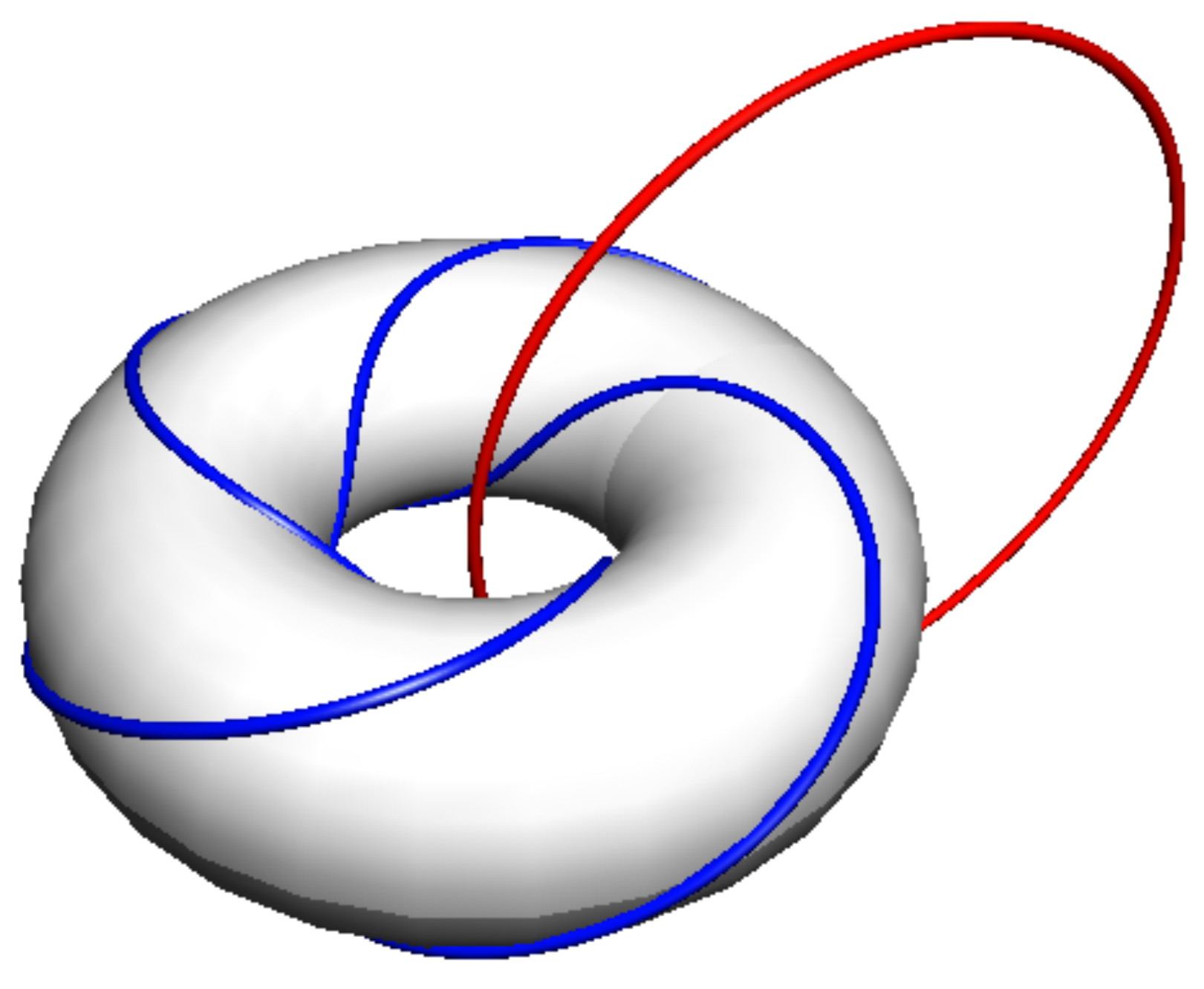} \
(f) \includegraphics[scale=0.30]{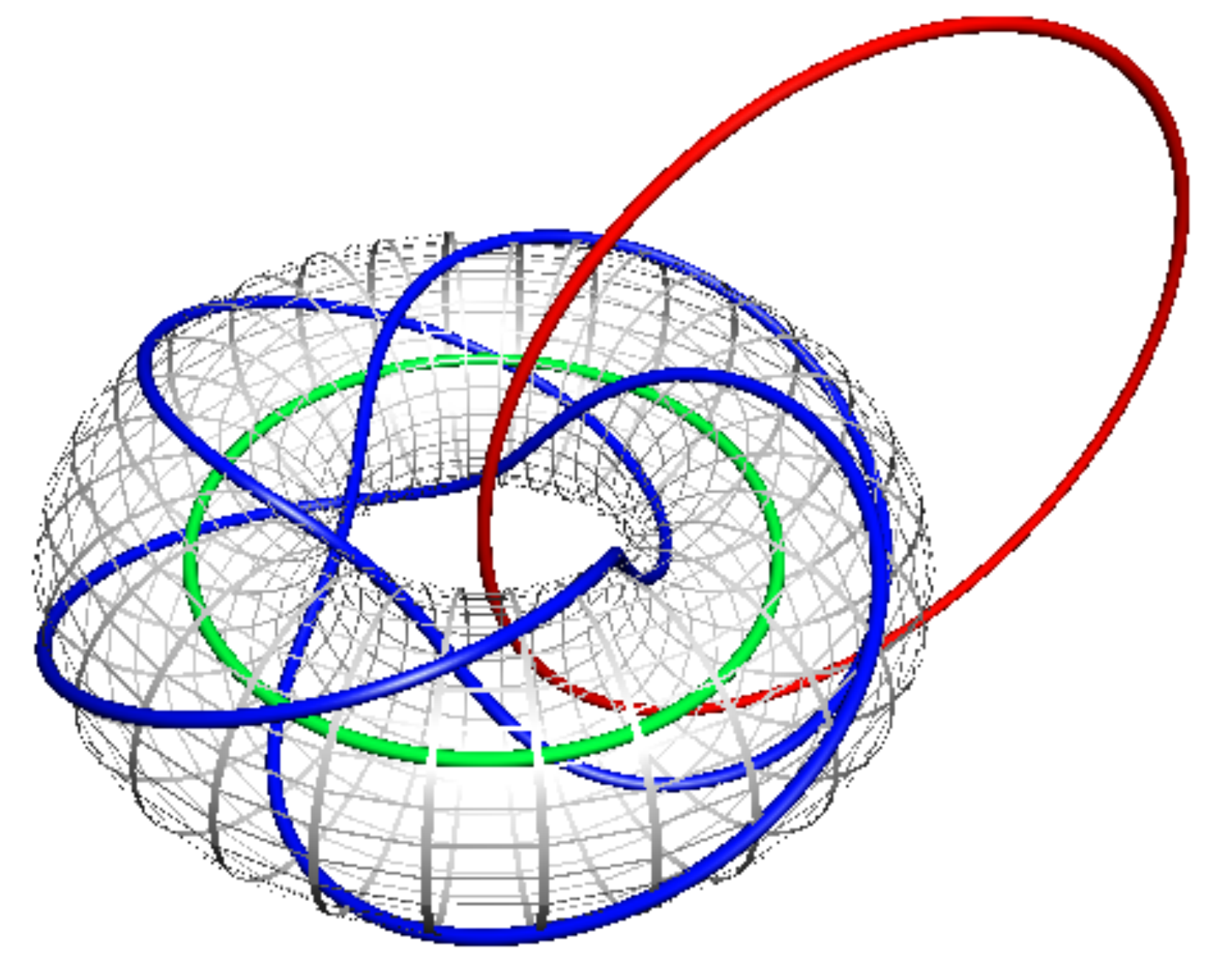} 
\caption{The (3,4)-torus knot $T_{3,4}$ in (a) and (d), the (3,4)-Seifert link $S_{3,4} = T_{3,4} \cup C_2$ in (b) and (e), and the link $R_{3,4} = T_{3,4} \cup C_1 \cup C_2$ in (c) and (f).
We can associate a knot in $S^1 \x D^2$ to $S_{3,4}$ and a knot in $S^1 \x S^1 \x I$ to $R_{3,4}$.}  
\label{F:3-4-SeifertLink}
\end{figure}

If $\gcd(p,q)=1$ and $p \neq \pm 1$, then $S_{p,q}$ is called the $(p,q)$-\emph{Seifert link}.  
The link $S_{p,q}=T_{p,q} \cup C_2$ is isotopic to $T_{q,p} \cup C_1$.  
As far as we know, the link $R_{p,q}$ does not have a name, and our notation for it is non-standard.
See Figure \ref{F:3-4-SeifertLink} for pictures of $S_{p,q}$ and $R_{p,q}$.  

The link $KC_n$ is called the $(n+1)$-component \emph{keychain link}.  There are symmetries which arbitrarily permute all the components except $C_1$.  We thus call $C_1$ the \emph{special component}.  It links nontrivially with all the other components.  
  
The \emph{Hopf link} (with $+1$ linking number) is $C_1 \cup C_2$.  It can also be described as $KC_1$, \ $T_{2,2}$, \ $S_{1,1} =T_{1,1} \cup C_2$, \ $T_{1,1}\cup C_1$, \  $S_{1,q}=T_{1, q} \cup C_2$, or $T_{p, 1} \cup C_1$.

Of course, each of these named links refers to any link in the same isotopy class.  When a choice of orientation is important, we take the one determined by the parametrizations by $t\in \R$ above.

A \emph{Seifert-fibered link} is any link whose exterior is diffeomorphic to a Seifert-fibered manifold.
Beware that $(p,q)$-Seifert links are a proper subset of the similarly named Seifert-fibered links.  

\begin{proposition}[\cite{BurdeMurasugi, BudneyJSJ}]
\label{SeifertFiberedLinks}
Let $p$ and $q$ be nonzero integers.  
Define $p' := p/\gcd(p,q)$ and $q':= q/\gcd(p,q)$, and let $r,s \in \Z$ such that $rp' - sq' =1$.
A Seifert-fibered link has (at least) one of the five forms below:
\begin{itemize}
\item $T_{p,q}$, with exterior $M(0, \gcd(p,q); r/q' , s/p' )$,
\item $T_{p,q} \cup C_2 \  (=S_{p,q})$, with exterior $M(0, 1+\gcd(p,q); s/p' )$,
\item  $T_{p,q} \cup C_1 \ (\sim S_{q,p})$, with exterior $M(0, 1+\gcd(p,q); r/q' )$,
\item  $T_{p,q} \cup C_1 \cup C_2 \ (=R_{p,q})$, with exterior $M(0, 2+\gcd(p,q);\ )$, or
\item $KC_n$, the $(n+1)$-component keychain link, with exterior $M(0,n+1; \ )$.
\qed
\end{itemize}
\end{proposition}

In particular, a Seifert-fibered link exterior is always fibered over a genus-0 surface $B$ with at most two singular fibers.  
Note that the effect of adding a component $C_i$ to the link is to drill out a neighborhood of a singular fiber, leaving one more boundary component in $B$.

It will be useful to understand the relationship between a Seifert-fibered link and the fibering on its exterior.  We describe this in the case of a nontrivial torus knot $T_{p,q}$ and leave the general case to the reader.  The space $S^3 - (C_1 \cup C_2)$ is homeomorphic to $S^1 \x S^1 \x (0,\infty)$.  Foliating the torus by copies of $T_{p,q}$ gives a trivial $S^1$-bundle structure on this space.  We thus see $S^3$ as a Seifert fibered manifold over $S^2$ with $C_1$ and $C_2$ as singular fibers.  
With $r$ and $s$ as in Proposition \ref{SeifertFiberedLinks} with $p'=p$ and $q'=q$, one can see that at least up to the choice of orientation, the fiber slopes are $r/q$ and $s/p$, since
$ \left(\begin{array}{cc} p & s \\ q & r \end{array} \right)^{-1} = 
\left(\begin{array}{cc} r & -s \\ -q & p \end{array} \right)$.

\subsection{JSJ decompositions of 3-manifolds and splicing of knots and links}
\label{S:JSJ}
We now review the JSJ decomposition of an irreducible 3-manifold $M$, due to Jaco, Shalen \cite{Jaco-Shalen-Symp, Jaco-Shalen-MemAMS}, and Johannson \cite{Johannson}, mainly when $M$ is the exterior of a link in $S^3$.  The relevant version is the decomposition of $M$ along a collection $T$ of incompressible, non-boundary-parallel tori such that each piece is Seifert-fibered or atoroidal \cite[Theorem 1.9]{Hatcher3Mfds}.  
By Thurston's work, each non-Seifert-fibered piece can be given the structure of a hyperbolic manifold of finite volume \cite{Thurston:BullAMS1982}.  
In general, the JSJ decomposition is represented by the \emph{JSJ graph}, where a vertex is labeled by the closure of a component of $M-T$ and an edge represents a torus in $T$.  
The Jordan curve theorem implies that if $M$ is a submanifold of $S^3$, the JSJ graph is a tree.  

In more detail, suppose $M=C_F$ is the complement of a $\vec{0}$-framed link $F$, or equivalently the exterior of the link $f$ associated to $F$.  
Each vertex $v$ in the JSJ graph $G_F$ of $C_F$ represents a submanifold $M_v \subset S^3$ with $\d M_v$ a union of tori.  We also write $G_F(v):=M_v$.  The vertex $v$ such that $\d C_F \subset M_v$ is designated as the root of $G_F$.  
Each $M_v$ is diffeomorphic to the complement of a nontrivial, irreducible ($\vec{0}$-framed) $k$-component link $L_v$ in $S^3$ \cite[Proposition 2.4]{BudneyJSJ}.  
We call $L_v$  a \emph{Seifert-fibered link} or \emph{hyperbolic link} according as $M_v$ is Seifert-fibered or hyperbolic.  
Call either such type of link \emph{simple}.
Replacing the label $M_v$ of each vertex $v$ in the JSJ tree of $M$ by the label $L_v$ yields the \emph{companionship tree} of the link $f$, or equivalently of the framed link $F$.  We denote this tree by $\mathbb{G}_f$ or $\mathbb{G}_F$, and we also write $\mathbb{G}_F(v):=L_v$.  
The unframed isotopy classes of the links $L_v$ are uniquely determined. (If $C_L$ and $C_{L'}$ are glued along boundary tori corresponding to link components $L_i$ and $L'_j$, the framings of $L_i$ and $L_j'$ may be simultaneously altered to yield the same result.)
This  decomposition of $F$ into the $L_v$, called the \emph{satellite decomposition}, goes back to Schubert \cite{Schubert:KnotenVollringe}, though its uniqueness is due to Jaco, Shalen, and Johannson.
See Figure \ref{F:Borr-Borr-Borr}, as well as Section \ref{S:Splicing}, for $\mathbb{G}_f$  in various examples of links $f$.

\begin{figure}[h!]
\begin{tabular}{l}
(a) \includegraphics[scale=0.14]{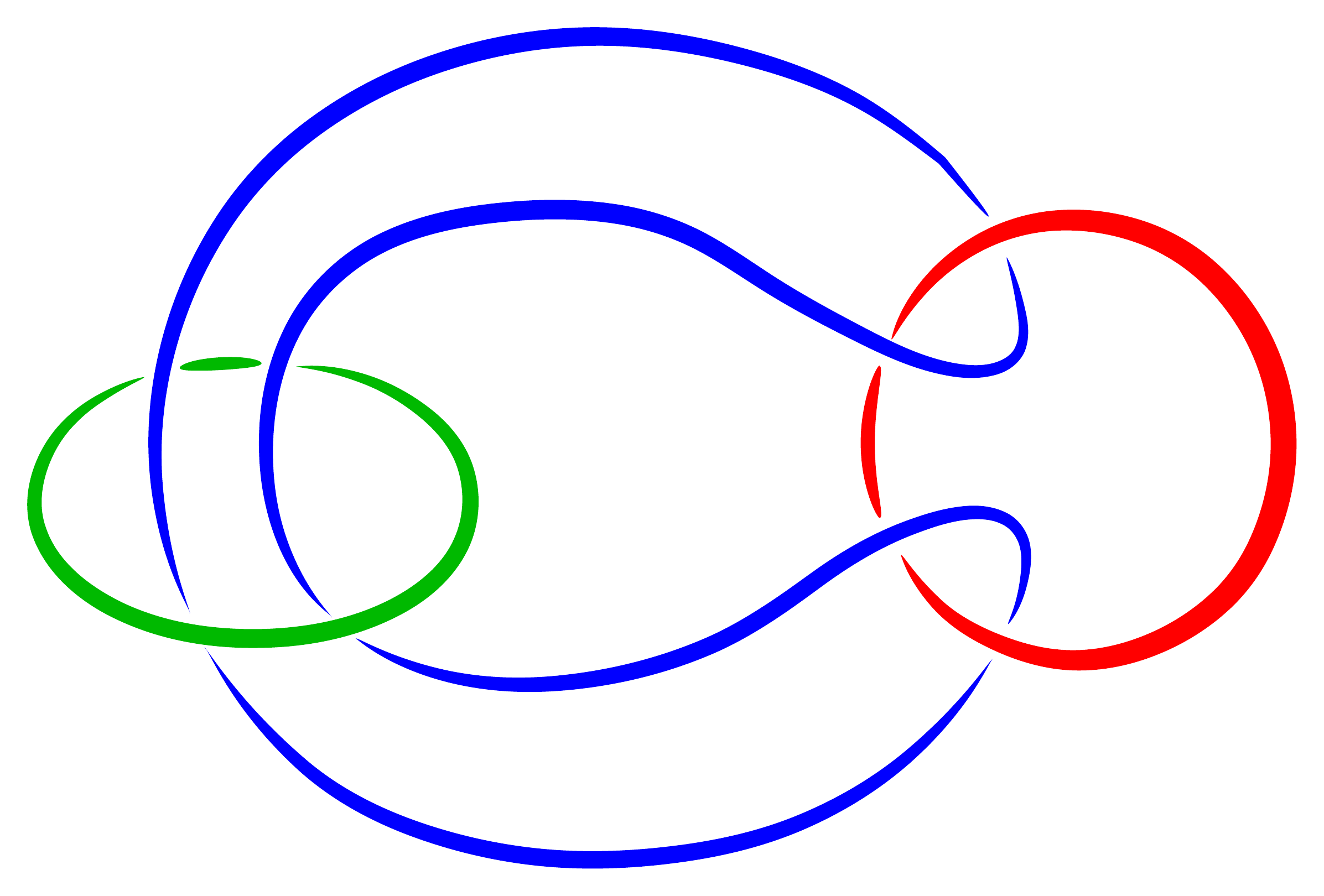} \\
(b)  
\raisebox{-1pc}{
\includegraphics[scale=0.14]{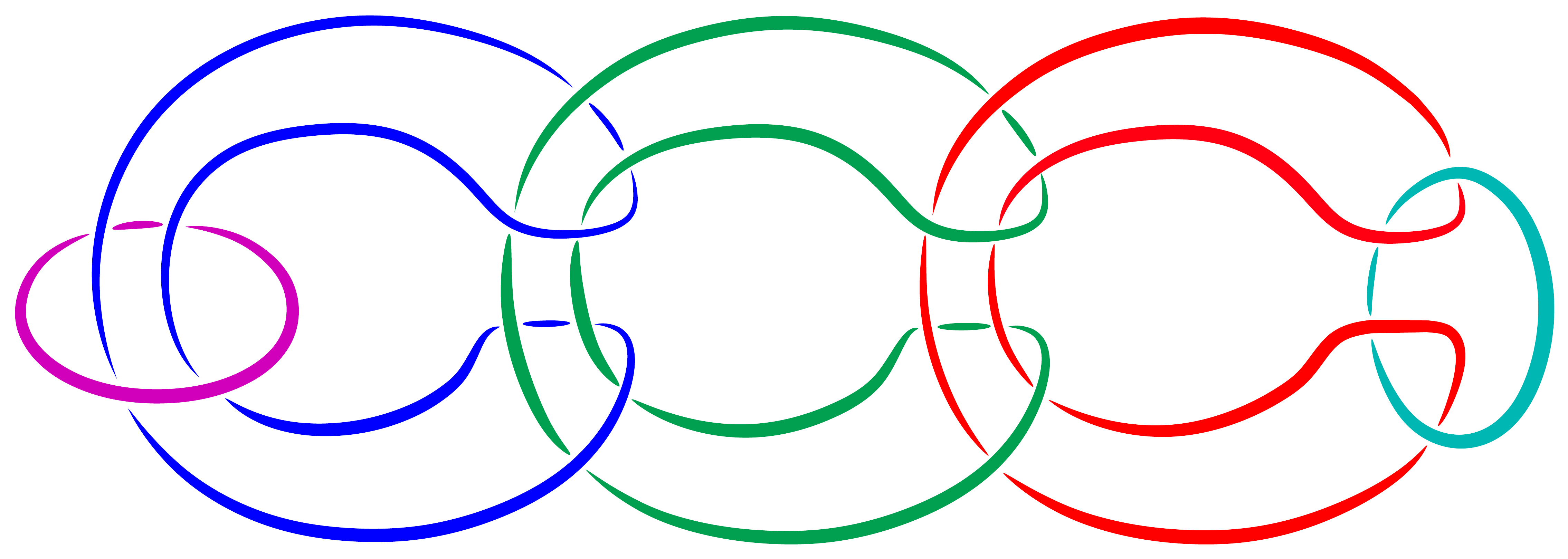}}
\end{tabular}
 \qquad
(c) \raisebox{-6pc}{
\begin{tikzpicture}
[grow=north, level distance=35pt]
\node[] {}[sibling angle=20]
child {node[draw,circle] {Borr} 
child {node[draw,circle] {Borr} 
child{node[draw,circle] {Borr} child{} child{}} child{}} child{}};
\end{tikzpicture}
}
\caption{(a) The Borromean rings $\mathrm{Borr}$, a hyperbolic KGL.  (b) Milnor's $n$-component Brunnian link $f$ for $n=5$.  (c)  The companionship tree $\mathbb{G}_f$ of $f$, which is a linear tree on $3\, (=n-2)$ vertices.}
\label{F:Borr-Borr-Borr}
\end{figure}
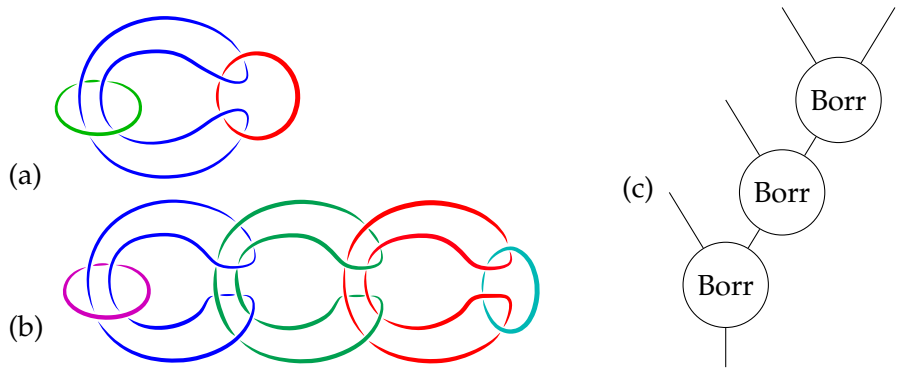

To a certain extent, the gluing together of link complements $G_F(v)$  can be reinterpreted in terms of composition of embeddings related to $\mathbb{G}_F(v)$.
Suppose that $F=(F_0,\dots, F_n)$ and
 $M=C_F$ has $\d M$ contained in one piece $M_v$ of its JSJ decomposition.
Let $L=(L_0, \dots, L_{n+r})$ be the link whose complement $C_L$ is $M_v$.  
Then $M$ is obtained from $C_L$ by gluing knot complements $M_1, \dots, M_r$ to some of the boundary tori  $T_1, \dots, T_{n+r}$ of $C_L$.
After relabeling, we may assume that $M_i$ is glued to $C_L$ along $T_{n+i}$.  
Each $M_i = C_{K_i}$ for some nontrivial framed knot $K_i$, and (essentially by Alexander's theorem on tori in $S^3$) a longitude of $K_i$ is glued to a meridian of $T_{n+i}$, while a meridian of $K_i$ is glued to a longitude of $T_{n+i}$.  
Each $M_i$ may itself have a nontrivial JSJ decomposition.  

Then for $i=1,\dots, n$, we can identify $F_i$ with $L_i $, and the component $F_0$ is given by 
\begin{equation}
\label{L0FormulaSplicing}
F_0 = 
(H_{r} \circ \underline{K}_r \circ H_{r}^{-1}) 
\circ \dots \circ
(H_{1} \circ \underline{K}_1 \circ H_{1}^{-1}) 
\circ L_0
\end{equation}
where $\underline{K}_i \in \tK$ is a framed long knot whose closure is $K_i \in \tL$, and where $H_{i}$ is an embedding of $I \x D^2$ such that $H_i|_{\{0\} \x \d D^2} = L_{n+i}$.
By disjointness of the supports of the $H_i$, the order of composition of the parenthesized terms above is inconsequential.  We say that $F$ is the result of \emph{splicing} the framed knots $(K_1, \dots, K_r)$ into the link $L$ along the components $L_{n+1}, \dots, L_{n+r}$, or the result of a \emph{satellite operation}.
We write $F =  (\varnothing, \dots, \varnothing, K_1, \dots, K_r) \bowtie L $.
This is a variation of Budney's notation \cite[Definition 4.8]{BudneyJSJ} which for $n=0$ agrees with his.  
(Beware that some authors may use the term \emph{satellite operation} for just the special case where $n=0$ and $r=1$.)
One may think of the knots $K_1, \dots, K_r$ as inputs for the link $L$ (so the notation is read in the reverse order from the commonly used function notation $f(x)$).  Graphically, it says that the result of removing the vertex labeled $L$ from $\mathbb{G}_F$ is the disjoint union of the trees $\mathbb{G}_{K_i}$. 

For links $J_1, \dots, J_r$ each with a distinguished component $J_{i,0}$, we write $F =  (\varnothing, \dots, \varnothing, J_1, \dots, J_r) \bowtie L $ in the analogous situation where $C_{K_i}$ replaced by $C_{J_i}$ and with the gluing done along the boundary torus corresponding to $J_{i,0}$.  The embedding $F_0$ is given by \eqref{L0FormulaSplicing}, with $\underline{K}_i$ replaced by $\underline{J}_{i,0}$.
Again, the component of $\d C_L$ corresponding to $L_0$ is the component of $\d C_F$ corresponding to $F_0$.  

In establishing the induction step in our Main Theorem, we will use not just $\mathbb{G}_F$ but an arbitrary choice of a distinguished component $F_0$, which makes $\mathbb{G}_F$ rooted.  
As in Figure \ref{F:Borr-Borr-Borr}, we have modified the conventions for ${G}_F$ (and thus $\mathbb{G}_F$) in \cite{BudneyJSJ} by adding half-edges to each vertex $v$, joined only to $v$, to represent the tori in $\d M_v$ that also lie in $\d C_F$.  Call such half-edges \emph{leaf half-edges}, and call the one corresponding to the distinguished component the \emph{root half-edge}.
So if $F$ is a knot, the only leaf half-edge is the root half-edge.  We also omit the edge-orientations used in \cite{BudneyJSJ}  (which indicate which side of a torus, if any, is a knot complement rather than a solid torus).

\begin{remark}[Assumptions on $F=(\varnothing, \dots, \varnothing, J_1, \dots, J_r) \bowtie L$]
\label{SplicingRemark}  
The symbol $F =  (\varnothing, \dots, \varnothing, J_1, \dots, J_r) \bowtie L$ could be interpreted for arbitrary links $L,J_1, \dots, J_r$. In some cases it will produce split links or links in 3-manifolds other than $S^3$, and in some cases it will not correspond to a sub-decomposition of $C_F$.  (For example, for a knot $J$, $F=(J, (J,J) \bowtie \mathrm{KC}_2) \bowtie KC_2$ could be interpreted to mean $F=(J,J,J) \bowtie KC_3$, but only the latter expression could correspond to part of the JSJ decomposition of $C_F$.)
To guarantee the validity of our theorem statements, we will write this expression to mean that $F$ is an irreducible link in $S^3$ and $L$ labels a vertex of $\mathbb{G}_F$.  
So for example, no $J_i$ may be an unknot or Hopf link or split link, and the operation $F=(J, (J,J) \bowtie \mathrm{KC}_2) \bowtie KC_2$ will not be considered.

A complete characterization of the possible link-labelings of companionship trees for links in $S^3$ is given in \cite[Proposition 4.20, Proposition 4.29]{BudneyJSJ}.  
For us, it suffices to note that if $n=0$ and $( J_1, \dots, J_r) \bowtie L$ is a knot in $S^3$, then $L$ must be a (simple) \emph{knot generating link} or \emph{KGL}, meaning an $(n+1)$-component link $(L_0, L_1, \dots, L_r)$ such that the $n$-component sublink $(L_1, \dots, L_r)$ is the unlink; see \cite[Proposition 2.2, Definition 4.4]{BudneyJSJ}.  
One allows $r=0$: a simple (nontrivial) KGL is either a torus knot or a hyperbolic knot.  For a simple KGL $(L_0, L_1, \dots, L_r)$ with $r>0$, $L_0$ may or may not be the unknot.
More generally, if $(\varnothing, \dots, \varnothing, J_1, \dots, J_r) \bowtie L$ is a link in $S^3$ and the $J_i$ are all knots, then the sublink $(L_0, L_{n+1}, \dots, L_{n+r})$ must be a KGL.
\end{remark}

\section{Relationships among various embedding and diffeomorphism spaces}
\label{S:Asphericity}
In this Section, we provide general descriptions of and relationships between various spaces of embeddings.  
Recall that 
\begin{align*}
\L &:= \coprod_{m\geq 1} \L(m):=\coprod_{m\geq 1} \Emb \left(\coprod^m S^1, S^3\right), & 
\tL &:=\coprod_{m\geq 1} \tL(m):= \coprod_{m\geq 1} \Emb \left(\coprod^m S^1 \x D^2, S^3\right), \\
\K &:= \Emb_c(\R, \R \x D^2),  &
 \tK &:= \Emb_c(\R \x D^2, \R \x D^2), \\
\T &:= \Emb(S^1, S^1 \x D^2), \qquad \text{ and} &
\V &:= \Emb(S^1, S^1 \x S^1 \x I).
\end{align*}
A subscript $f$ indicates the component of such a space in which $f$ lies.  In particular, if $F$ is a framed knot $S^1 \x D^2 \incl S^3$, then $\tL_F \subset \Emb(S^1 \x D^2, S^3)$.

In Section \ref{S:Gramain}, we relate $\tL(1)$ to $\tK$ and $\K$ (Proposition \ref{FramedKnotsInS3VsLongKnots}).  As a result, the meridional rotations appearing in the remaining results of this section can be viewed as generalizations of Gramain loops of long knots.
In Section \ref{S:LFmodSO4=BDiff(CF)}, we first describe $\tL_F/SO_4$ as the classifying space of the group of diffeomorphisms of the complement $C_F$ (Proposition \ref{FramedLinksKpi1}).
We then relate $\tL_F/SO_4$ to $\L_f/SO_4$, $\T_f$, and $\V_f$, in Proposition \ref{LfModSO4isKpi1}, Proposition \ref{TfisKpi1}, and Corollary \ref{VfIsLfModSO4} respectively.
% In doing so, we also describe the relationships between $\tL_F/SO_4$ and the other three spaces.  At the level of $\pi_1$, each of these three is the quotient by loops of meridional rotations along certain components.  

\subsection{Closed knots in the 3-sphere, long knots, and the Gramain loop}
\label{S:Gramain}
To relate the space $\tL$ of framed links in $S^3$, the space $\tK$ of framed long knots, and the space $\K$ of (unframed) long knots,  
%Recall that we view framed embeddings as thickened embeddings.
we will use the closure map $\tK \to \tL,\  F \mapsto \overline{F}$, as well as the fibration $\tK \to \K, \ F \mapsto F|_{I \x \{0\}}$ given by restricting to the core.
In Proposition \ref{FramedKnotsInS3VsLongKnots} below, part (a) is related to \cite[p.~8]{HatcherKnotSpacesWebsite} and \cite[Proposition 4.1]{Budney-Cohen}.  Part (b) is essentially how Gramain established the nontriviality of his eponymous loop \cite{Gramain1977}.  
Part (c) is implicit when comparing Hatcher's results on closed knots in $\L(1)$ and Budney's results on long knots in $\K$.

\begin{proposition}
\label{FramedKnotsInS3VsLongKnots}
Let $F$ be a framed long knot (with any framing number).
\begin{itemize}
\item[(a)] 
The maps 
\begin{equation}
\label{2MapsFromFramedLongKnots}
\tL_{\overline{F}}/SO_4  \leftarrow \tK_F \to \K_f
\end{equation}
given by closure and restriction to the core are both homotopy equivalences.
\item[(b)] 
The loop $\mu_0$ of meridional rotations in $\tL_{\overline{F}}$ corresponds to the Gramain loop $g_f$ given by spinning a long knot around the long axis, via the isomorphism $\pi_1(\tL_{\overline{F}}/SO_4) \cong \pi_1(\K_f)$ induced by \eqref{2MapsFromFramedLongKnots}.
\item[(c)]
The loop $\lambda_0$ of reparametrizations in $\tL_{\overline{F}}$ corresponds to the Fox--Hatcher loop in $\K_f$. 
\end{itemize}
\end{proposition}

\begin{proof}
We first prove (a).  Given $G \in \tL_{\overline{F}}$, we can apply an element in $SO_4$ to give $G$ a prescribed value and tangent vector to the core at $(1,0) \in S^1 \x D^2$.  By Gram--Schmidt, we can perform a homotopy of $G$ so that $dG_{(1,0)} = \mathrm{id} \in SO_3$.  Finally, we can perform a further homotopy of $G$ to have prescribed values and derivatives on all of $\{1\} \x D^2$.  This allows us to construct a homotopy inverse to $\tK_F \to \tL_{\overline{F}}/SO_4$, so the left-hand map is an equivalence.  
Next, we have $\tK \simeq \Omega SO_2 \simeq \Z\x \K$, where $\Z$ gives the framing number and where the right-hand map in \eqref{2MapsFromFramedLongKnots} is the projection to $\K$.  Thus $\tK_F \simeq \{n\} \x \K_f$ for some $n\in \Z$, so $\tK_F \simeq\K_f$.

We now prove (b).  Since $\tK_F \to \K_f$ is a fibration inducing an isomorphism on $\pi_1$, the loop $g_f$ on $\K_f$ given by spinning around the long axis lifts to a loop $g_F$ on $\tK_F$.  We cannot realize this lift by merely ``spinning'' the thickened knot (for framed long knots are the identity on $\d I \x D^2$), but we can combine this motion with a simultaneous meridional rotation in the opposite direction to realize $g_F$.\footnote{The reader may find it amusing to realize $g_F$ using a knotted belt with anchored ends.}
The image of  $g_F$ in $\tL_F$ is then given by the product, in any order, of a loop of rotations in $SO_4$ (spinning the thickened knot) and a loop of meridional rotations.   See Figure \ref{F:GramainVsMeridRot}.

\begin{figure}[h!]
\includegraphics[scale=0.24]{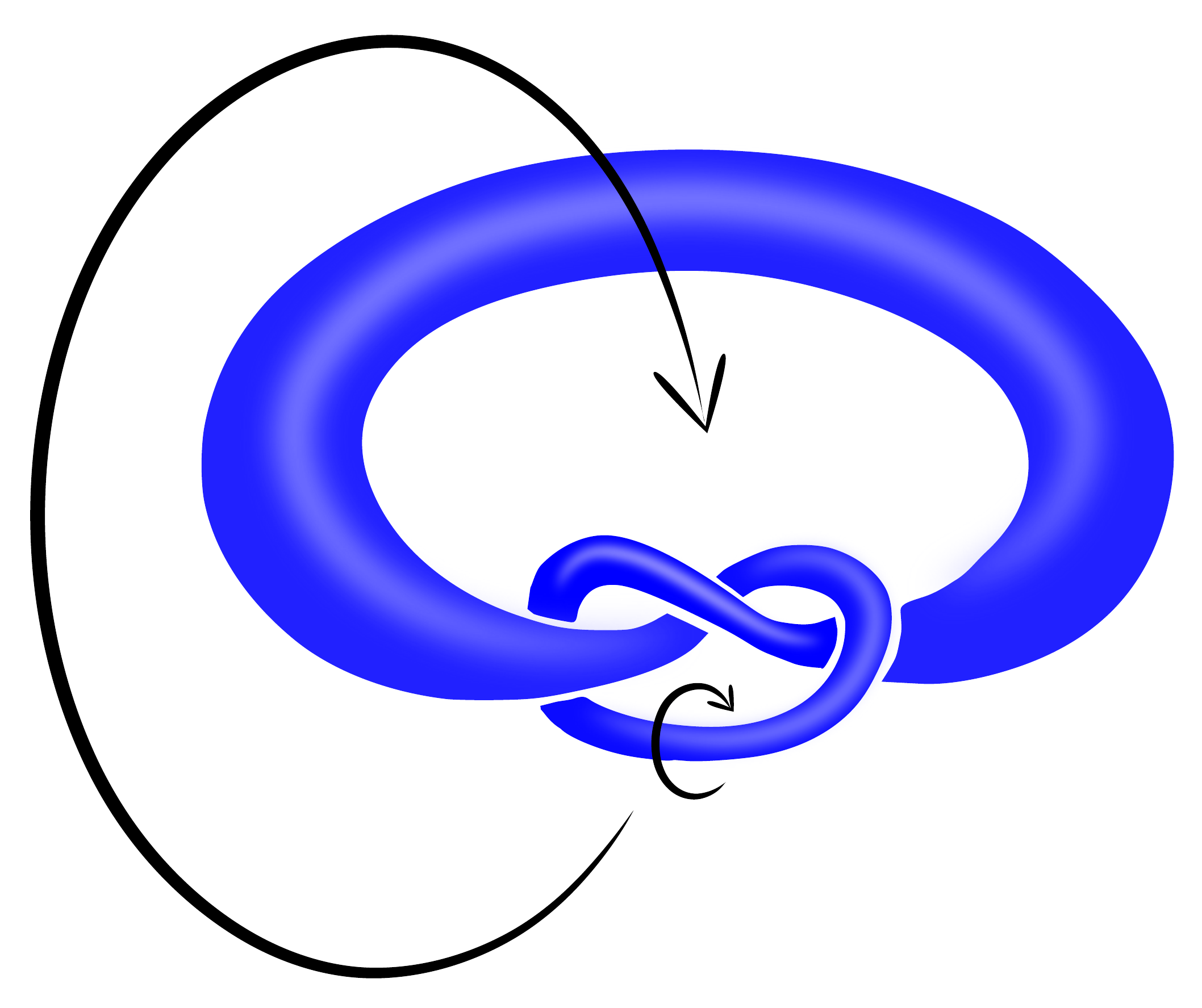}
\caption{A framed knot $\overline{F}$ obtained as the closure of a framed long knot $F$. The loop of rotations in $SO_4$ along the circle $C_2 \subset S^3$ (the large arrow) is the sum of a meridional rotation of $\overline{F}$ (the small arrow) and the image $g_F$ of the Gramain loop in $\tL_{\overline{F}}$.  The support of $g_F$ lies in a copy of $I \x D^2$ inside  $S^1 \x D^2 = \mathrm{domain}(\overline{F})$.}
\label{F:GramainVsMeridRot}
\end{figure}

For part (c), let $G \in \tL_{\overline{F}}$.  First by a homotopy we may assume that $dG_{(t,z)} \in SO_3$ for all $(t,z)$ in the core $S^1 \x \{0\}$, where we now parametrize $S^1$ by $t\in [0,1]$.  
Then multiplying the loop $\lambda_0$ pointwise by a loop $\alpha_t \in SO_4$, $t\in [0,1]$, such that $\alpha_t$ sends $(G(t,0), dG_{t,0})$ to $(\infty, \mathrm{id})$ gives a loop of embeddings with fixed values and derivatives at the basepoint (though their images move in $S^3$).
Since the derivatives at the basepoint are fixed, we may apply a further homotopy so that all of $\{1\} \x D^2$ is fixed.  
This allows us to continuously realize the homotopy inverse $\tL_{\overline{F}}/SO_4 \to \tK_F$ over all times in $\lambda_0$ and see that upon the projection to $\K_f$ we obtain the Fox--Hatcher loop.
\end{proof}

We will say \emph{Gramain loop} to refer to not only the loop of long knots, but also the associated loops of framed long knots and framed closed knots.  We will write e.g.~$g_f$, $g_F$, and $g_{\overline{F}}$ for these loops.

\begin{remark}
\label{R:GramainSplitting}
In \cite[Section 6]{BudneyTop}, Budney essentially showed that $\Z\langle g_f \rangle$ splits as a direct factor of $\pi_1(\K_f)$.  
(He constructed a canonical map $\K_f \to \Z$ such that $\Z\langle g_f \rangle \to \K_f \to \Z$ is $\pm n$ for some $n>0$.  
The construction inducts on the height of $\mathbb{G}_f$, and the step producing $n>1$ occurs when $f$ is an $n$-fold connected sum.  Then $\Z \langle g_f \rangle$ is the diagonal in a factor $\Z^n$.  Since a generator of the diagonal is primitive in $\Z^n$, $\Z \langle g_f \rangle$ splits as a direct factor, even though the splitting is not canonical.)
We will generalize this splitting to multiple components in Theorem \ref{T:MeridiansFactor}.
\end{remark}

\subsection{Spaces of links, classifying spaces, and meridional rotations}
\label{S:LFmodSO4=BDiff(CF)}
We now continue with results on the homotopy types of 
$\tL_F/SO_4$, $\L_f/SO_4$, $\T_f$, and $\V_f$.
We will see that each of these spaces is equivalent to the classifying space of a certain space of diffeomorphisms.
If $F$ is an irreducible link, then this classifying space is a $K(\pi,1)$ in all four cases.
For any link $F$, each of the last three of these four spaces differs from $\tL_F/SO_4$ only at $\pi_1$ by a quotient by certain meridional rotations.

The next proposition is a generalization of a result of Hatcher \cite{HatcherKnotSpacesArxiv} from knots to links.

\begin{proposition}
\label{FramedLinksKpi1}
Let $F$ be an $n$-component framed link in $S^3$, and let $C_F:= S^3 - \im F$.
\begin{itemize}
\item[(a)]
There is a homotopy equivalence $\tL_F/SO_4 \simeq B \Diff(C_F; \d C_F)$.
\item[(b)]
If $F$ is irreducible, then $\tL_F/SO_4$ is a $K(\pi,1)$ space for the group  $\pi_0\Diff(C_F; \d C_F)$.
\end{itemize}
\end{proposition}

\begin{proof}
Restricting diffeomorphisms of $S^3$ to the image of the framed link $F$, we have the fibration 
\begin{equation}
\label{DiffToEmb}
\Diff(C_F; \d F) \to \Diff^+(S^3) \to \tL_F
\end{equation}
where a diffeomorphism of $S^3$ fixing a codimension-0 submanifold must be orientation-preserving.
Now $SO_4$ acts on $\Diff^+(S^3)$ by composition.  This action and the resulting action on $\tL_F$ are free (even if $F$ is the unknot, because $F$ is framed).
Taking the quotient by this action yields the bundle 
\begin{equation}
\label{DiffToEmbModSO4}
\Diff(C_F; \d C_F) \to \Diff^+(S^3)/SO_4 \to \tL_F/SO_4.
\end{equation}
Part (a) follows from the fact that the total space is contractible \cite{HatcherSmaleConj}.
For part (b), if $F$ is irreducible, then the fiber has contractible components \cite{HatcherIncomprSurfs, Ivanov, Ivanov2}.  Thus the only nontrivial homotopy group of the base is $\pi_1 (\tL_F/ SO_4)\cong \pi_0(\Diff(C_F; \d C_F))$.
\end{proof}

The next result about unframed closed links in $S^3$ requires the classification of Seifert-fibered link exteriors in Proposition \ref{SeifertFiberedLinks}.   
We closely follow arguments from Hatcher's work on the case of knots \cite{HatcherKnotSpacesArxiv, HatcherKnotSpacesWebsite}.
Because of Proposition \ref{FramedKnotsInS3VsLongKnots} (b), we may view the kernel in $\pi_1(\tL_F/SO_4)$ below as an analogue of the Gramain subgroup in fundamental groups of spaces of long knots.

\begin{proposition}
\label{LfModSO4isKpi1} 
Let $f$ be a nontrivial link with $n$ components.
Let $F$ be any framed link corresponding to the unframed link $f$.
Then $\pi_i(\L_f / SO_4) \cong \pi_i(\tL_{F} / SO_4)$ for all $i\geq 2$, and
there is a short exact sequence 
\begin{equation}
\label{Pi1ExactSequenceLF}
\{e\} \to \Z^n \to \pi_1(\tL_{F} / SO_4) \to \pi_1(\L_f / SO_4) \to \{e\}
\end{equation}
where the first map sends the $i$-th standard generator to a loop of rotations along a meridian of the $i$-th component of $F$.
Thus $\L_f/SO_4 \simeq B\Diff^+(S^3; f)$, and if $f$ is irreducible, then $\L_f / SO_4$ is a $K(\pi,1)$ space for the group $\pi=\pi_0(\Diff^+(S^3; f))$.
\end{proposition}

\begin{proof}
We first show that for any framed link $F$, there is a fibration
\begin{equation}
\label{FramedLinkFibn}
(S^1)^n \to \tL_F/SO_4 \to \L_{f}/SO_4.
\end{equation}
Indeed, restricting a framed link from $\coprod_n S^1\x D^2$ to $\coprod_n S^1 \x 0$ is a fibration $\tL \to \L$ which passes to the quotients by $SO_4$:
\[
\Emb(\nu(f), S^3;\, {f}) \to \tL / SO_4 \to \L / SO_4
\]
The fiber $\Emb(\nu(f), S^3; \, {f})$ over ${f}$  is the space of tubular neighborhoods of $f$, which is homotopy equivalent to the space of linear automorphisms of its (trivial) normal bundle \cite[Theorem 4.5.3]{Hirsch}.  It is thus equivalent to $(\Omega SO_2 \x S^1)^n \simeq \Z^n \x (S^1)^n$: each factor $\Omega SO_2 \simeq \Z$ measures the framing on a link component, while each factor $S^1$ measures rotations by the same angle in each fiber of the normal bundle over a given component.  
Over each component $\L_f$ lie components $\tL_{F(\vec{k})}$ indexed by $\vec{k}\in \Z^n$ where $\vec{k}$ is the list of framing numbers.  The space $\tL_F$ is just one such component, so restricting to it yields the fibration \eqref{FramedLinkFibn}.
Its long exact sequence in homotopy establishes the short exact sequence \eqref{Pi1ExactSequenceLF} except for the injectivity of the map from $\Z^n$.
The claimed description of that map is however clear.  

It remains to show the injectivity of the map $\pi_1((S^1)^n) \to \pi_1(\tL_{F} / SO_4)$; this will establish that $\L_f/SO_4$ is a $K(\pi,1)$ space for irreducible $f$ and the claimed isomorphisms of higher homotopy groups in general, since $(S^1)^n$ is a $K(\pi,1)$ space.
It will be useful to view the long exact sequence in homotopy of \eqref{FramedLinkFibn} as coming from a shifted version of this fibration, which essentially comes from applying $\Omega(-)$ to the total space and base.
Let $\nu(f)$ be the tubular neighborhood $\im\, F$ of the unframed link $f$, and let $M:= S^3 - \nu(f)\, (\,= C_F)$.
The shifted version of the fibration is
\begin{equation}
\label{LinkCompVsExt}
\Diff(M; \d M) \to  \Diff^+(S^3; f) \to   \Emb(\nu(f), S^3;\, f).
\end{equation}  
Again,  $\Emb(\nu(f), S^3;\, f) \simeq \Z^n \x (S^1)^n$, but an element of $\Diff^+(S^3; f)$ takes the longitude of a component of $f$ to a null-homologous curve in $S^3 - \im(f)$, i.e.~a curve that links trivially with $f$, i.e.~a  curve representing $\vec{0} \in \Z^n$.  
Hence we may rewrite the fibration \eqref{LinkCompVsExt} as 
\begin{equation}
\label{LinkCompVsExtS1}
\Diff(M; \d M) \to \Diff^+(S^3; f) \to (S^1)^n.
\end{equation}  
Thus we want to show that the map
\begin{equation}
\label{DehnTwistsShouldInject}
\pi_1((S^1)^n) \to \pi_0(\Diff(M; \d M)) \  \left(\cong \pi_1(\tL_F/SO_4)\right)
\end{equation}
 is injective.
It is given by sending the $i$-th standard generator in $\pi_1((S^1)^n)$ to a diffeomorphism $\tau_i$ of $M$ given by an annular Dehn twist along the meridian $m_i$ and supported in a collar neighborhood of the $i$-th boundary torus. 

First suppose that $f$ is irreducible, so $M$ is a Haken 3-manifold.  
Consider the effect of $\tau_i$ on $\pi_1(M, x_j)$ where $x_j$ lies in the $j$-th boundary component.  Then $\tau_{i*} =\id$  if $i\neq j$ and $\tau_{i*}$ is conjugation by $m_i$ if $i=j$.  Considering all $j$ together, we get a map 
\begin{equation}
\label{MapToAut}
\Z^n \to \prod_{j=1}^n \Aut(\pi_1(M, x_j))
\end{equation}
which sends the $i$-th standard generator to $(\id, \dots, \id, c_i,\id, \dots,\id)$ where $c_i$ is conjugation by $m_i$.
It suffices to check that  \eqref{MapToAut} is injective, since it factors through the map $\Z^n \to \pi_0 \Diff(M; \d M)$.

Suppose $(k_1, \dots, k_n) \in \Z^n$ lies in the kernel of \eqref{MapToAut}.
Then $(c_1^{k_1}, \dots, c_n^{k_n})=(\id, \dots, \id)$.  
Thus for each $i$,  $m_i^{k_i}$ is in the center of $\pi_1(M, x_i)$.
A theorem of Waldhausen \cite{Waldhausen:Zentrum} says that if $M$ is Haken and not Seifert-fibered, then the center of $\pi_1(M,x_i)$ is trivial.  
Thus if $f$ is not Seifert-fibered, $m_i^{k_i}=e$.  
Since $M$ is a $K(\pi,1)$ space \cite[Corollary 3.9]{Hatcher3Mfds} and a finite-dimensional manifold, $\pi_1(M,x_i)$ is torsion-free.  So $k_i=0$ for all $i$, and we are done.

So suppose $M$ is Seifert-fibered.
Then the base surface of $M$ is a disk with $n$ boundary components, $r \geq 0$ singular fibers, and fiber data $a_1/b_1, \dots, a_r/b_r$.
Thus 
\[
\pi_1(M, x_i) \cong 
\left\langle \{m_j\}_{j=1}^{n-1}, \, \{q_\ell\}_{\ell=1}^r, \, h \, \right| \left. \, 
[m_j, h]=1, \, [q_\ell, h] = 1, \,
q_\ell^{b_\ell}h^{a_\ell}=1
\right\rangle
\]
where $h$ corresponds to the fiber.
Any central element in $\pi_1(M, x_i)$ must map into the center of the quotient
\[
\pi_1(M, x_i) / \langle h \rangle \cong  \left\langle \{m_j\}_{j=1}^{n-1}, \, \{q_\ell\}_{\ell=1}^r \right| \left. \, q_\ell^{b_\ell}=1
\right\rangle.
\]
which is a free product of copies of $\Z$ and $\Z/b_\ell$, generated by the $m_j$ and $q_\ell$.
Thus $m_i^{k_i}$ is trivial, so again $k_i=0$ for all $i$.  

Finally, if $f$ is a split link so that $M=M_1 \# \dots \# M_p$, then $\pi_1(M)\cong \pi_1(M_1) \ast \dots \ast \pi_1(M_p)$.  
Because each diffeomorphism $\tau_i$ is the identity on all of $\d M$, it preserves each summand $M_k$.  
Considering the inclusions $\pi_1(M_k) \incl \pi_1(M)$ and projections $\pi_1(M) \twoheadrightarrow \pi_1(M_k)$, we see that $\tau_i$ induces an automorphism of $\pi_1(M_k, x)$ for each $k$ and any $x \in \d M_k$.
We therefore can use a similar argument as in the irreducible case, just replacing $M$ in \eqref{MapToAut} by the summand $M_k$ that contains the $j$-th boundary component of $M$.
\end{proof}

We now describe the homotopy types of both $\T$ and $\V$ in terms of $\tL$ and $\L$.

\begin{proposition}
\label{TfisKpi1}
Let $f:  S^1 \incl S^1 \x D^2$ be a knot in the solid torus.
If $F$ is the corresponding 2-component framed link, then 
$\pi_i(\T_f) \cong \pi_i(\tL_F/SO_4)$ for $i\geq 2$, and there is short exact sequence of groups
\[
\{e\} \to \Z \to \pi_1(\tL_F/SO_4) \to \pi_1(\T_f) \to \{e\}
\] 
where a generator of $\Z$ is sent to the meridional rotation $\mu_0$ of the knotted component.   
If $G:=\Diff(S^1 \x D^2;\, f \cup \d)$, then $\T_f \simeq BG$ and if $F$ is irreducible, then $\T_f$ is a $K(\pi_0(G),1)$ space. 
\end{proposition}

\begin{proof}
Let $\widetilde{\T}$ be the space of framed knots in a solid torus, i.e.~embeddings $S^1 \x D^2 \incl S^1 \x D^2$.  
The framed link $F \in \tL$ corresponds to a framed knot in $\widetilde{\T}$, which by abuse of notation we also call $F$.
Let $\widetilde{\T}_F$ be its path component.
By Lemma \ref{DiffSolidTorus}, $\Diff(S^1 \x D^2; \d)\simeq \ast$.
If $\nu(F)$ is a neighborhood of $F$ in the solid torus, then arguing as in Proposition \ref{FramedLinksKpi1}, we see that $\widetilde{\T}_F \simeq B\Diff(S^1 \x D^2 - \nu(F); \d) = B\Diff(C_F; \d C_F) \simeq \tL_F/SO_4$.   
Arguing as in Proposition \ref{LfModSO4isKpi1}, we get a fibration $S^1 \to \widetilde{\T}_F \to \T_f$, where the fiber corresponds to meridional rotations of $F$.  
It only remains to verify the injectivity of the map $\Z \to \pi_1(\tL_F/SO_4)$.
Proposition \ref{LfModSO4isKpi1} established the injectivity of a map $\Z^2 \to \pi_1(\tL_F/SO_4)$, and the map in question is its restriction to a factor of $\Z$ and hence injective.  
\end{proof}

Alternatively, $\T_f$ is the classifying space of the group $\Diff(S^1 \x D^2;\, f \cup \d)$; we leave the details to the reader.

\begin{corollary}
\label{TfVsLf}
If $f$ is any knot in the solid torus, then $\pi_i(\T_f)\cong \pi_i(\L_f/SO_4)$ for all $i\geq 2$, and there is a short exact sequence
\[
\{e\} \to \Z \to \pi_1(\T_f) \to \pi_1(\L_f/SO_4) \to \{e\}.
\]
\end{corollary}
\begin{proof}
This follows immediately from Proposition \ref{LfModSO4isKpi1} and Proposition \ref{TfisKpi1}.
\end{proof}

Proposition \ref{TfisKpi1} and Corollary \ref{TfVsLf} roughly say that $\T_f$ lies between $\tL_F/SO_4$ and $\L_f/SO_4$.

\begin{proposition}
\label{VfIsLfModSO4}
Let $f$ be any knot in $S^1 \x S^1 \x I$.  Then $\V_f \simeq \L_f/SO_4$.
\end{proposition}
\begin{proof}
By Proposition \ref{VfandLinksWithHopf}, $f$ corresponds to an isotopy class of link where the last two components are the Hopf link.  Let $H$ (and $h$) denote the framed (and unframed) Hopf link, and consider the fibration 
\[
\V_f \to \L_f \to \L_h
\]
given by restricting an embedding to the last two components.  Since $f$ and $h$ are not the unknot, $SO_4$ acts freely both $\L_f$ and $\L_h$.  
By Lemma \ref{DiffExteriorOfHopfLink} and Proposition \ref{FramedLinksKpi1} part (a), the space $\L_H$ of the framed Hopf link satisfies $\L_H/SO_4 \simeq S^1 \x S^1$, where the factors of $S^1$ correspond to meridional rotations of the two components.  By Proposition \ref{LfModSO4isKpi1}, $\L_h/SO_4 \simeq \ast$.  Thus taking the quotient by $SO_4$ of the above fibration yields the fibration
\[
\V_f \to \L_f/SO_4 \to \ast
\]
which completes the proof.
\end{proof}

\section{Spaces of Seifert-fibered links}
\label{S:Seifert}

We now determine the homotopy types of spaces of Seifert-fibered links.  
The list of Seifert-fibered links and their fiberings in Proposition \ref{SeifertFiberedLinks} will be crucial here.  
Proposition \ref{DiffSeifertFramed} gives a uniform result for $\tL_F/SO_4$.  To ultimately understand the spaces $\L_f/SO_4$, we  calculate quotients by certain elements in Proposition \ref{DiffSeifertUnframed}.
We then convert this into results on $\L_f/SO_4$ for all the possible types of Seifert-fibered links $f$ in Corollary \ref{SeifertSpacesOfLinks}.
The results for $\T_f$ and $\V_f$ similarly follow by taking the appropriate quotient of the group in Proposition \ref{DiffSeifertFramed}.  In these settings, the list of possible Seifert-fibered links is much shorter.  We describe generators of $\pi_1$ in many cases.

\begin{proposition}
\label{DiffSeifertFramed}
Let $M$ be a Seifert-fibered submanifold of $S^3$ over a genus-0 surface with $n+1$ boundary components and $r$ singular fibers.  
(So $n \geq 0$ and $0 \leq r \leq 2$.)  Then
\[
\pi_0\Diff(M; \d M) \cong 
\Z^n \times \mathrm{PMod}(\Sigma_{0,r}^{n+1}) \cong
\Z^{2n} \x \PB_{n+r}.
\]
\end{proposition}

\begin{proof}
Recall  that $P_n$ is the genus-0 surface with $n+1$ boundary components.
A result of Johannson \cite[Proposition 25.3]{Johannson} (building on work of Waldhausen \cite[pp.~85-86]{Waldhausen:Annals1968}) says that for $M$ as above, there is an exact sequence of groups
\begin{equation}
\label{JohannsonSES}
\{e\} \to H_1(P_n, \d P_n) \to \pi_0\Diff(M; \d M) \to \pi_0 \Diff(P_n; \d P_n \cup \{x_1, \dots, x_r\}) \to \{e\}
\end{equation}
which splits as a semi-direct product (because $\d M \neq \varnothing$).  The right-hand map comes from the fact that a diffeomorphism preserves the Seifert fibering, so it induces a diffeomorphism of the base.
Identifying the outer two groups gives 
\begin{align}
\pi_0\Diff(M; \d M) 
\cong \Z^n \rtimes \mathrm{PMod}(\Sigma_{0,r}^{n+1})
\cong \Z^n \rtimes (\PB_{n+r} \x \Z^n).
\label{Semidirect}
\end{align}
Let $T_i := D_i \x S^1$ be a boundary torus of $M$ corresponding to an inner boundary component of $P_n$.
Johannson's work shows that in the factor $\Z^n \cong H_1(P_n, \d P_n)$, a generator is an annular Dehn twist supported in a collar neighborhood of $T_i$, where the twist is along the fiber $S^1$.

To verify that \eqref{Semidirect} is a direct product, first let $r=0$.  
Then $M=P_n \x S^1$, and we can define a right-inverse $\pi_0\Diff(M; \d M) \to \Z^n$ to the first map in \eqref{JohannsonSES} as follows.  
This is the step where we will use that the base surface of $M$ has genus 0.  
Suppose the removed disks in $P_n$ have centers on the horizontal axis.  
Label the boundary components $c_0, \dots, c_n$, where $c_0$ is the unit circle.
For $j=1, \dots, n$, let $\alpha_j$ be the straight line segment from $x := (i, 1) \in D^2 \x S^1$ to $c_j \x \{1\}$.  
For $f \in \Diff(M; \d M)$, we can view $[\alpha_j \cup -f(\alpha_j)]$ as an element of $\pi_1(M,x)$.
Let $w_j(f)$ be the winding number of the image of $\alpha_j \cup -f(\alpha_j)$ under the projection $P_n \x S^1 \to S^1$.  
That is, $w_j(f)$ is the image of $[\alpha_j \cup -f(\alpha_j)] \in \pi_1(M, x)$ under the projection $F_n \x \Z \to \Z$.  
Define $\Phi: \pi_0(\Diff(M; \d M)) \to \Z^n$ by $\Phi(f)=(w_1(f), \dots, w_n(f))$; this is the claimed right-inverse.  Thus if $r=0$, $\pi_0(\Diff(M; \d M)) \cong \Z^{2n} \x \PB_n$.

To establish the direct product in the general case, suppose 
$M$ is Seifert-fibered over $P_n$ with $r$ singular fibers.
Since $M\subset S^3$, we have $0 \leq r \leq 2$, though we need not use this fact.  
Let $M'$ be the result of removing tubular neighborhoods $\nu_1, \dots, \nu_r$ of the singular fibers from $M$.  Then $M'$ is the Seifert-fibered manifold $M(0,n+r; )$.
There is a homomorphism $\pi_0\Diff(M'; \d) \to \pi_0\Diff(M; \d)$ given by extending by the identity on the $\nu_i$.
From the case $r=0$, this is a map $\Z^{n+r} \x \mathrm{PMod}(\Sigma_0^{n+r}) \to \Z^{n} \rtimes \mathrm{PMod}(\Sigma_{0,r}^{n})$ which we view as a map of short exact sequences of the form \eqref{JohannsonSES}.  From the description of generators above, it is surjective on both the left- and right-hand terms of this sequence (with kernel $\Z^r$ in both cases), hence surjective on all three terms, by the Four Lemma.  It is then straightforward to check that this surjectivity and the splitting of $\pi_0\Diff(M'; \d)$ as a direct product implies the splitting of $\pi_0\Diff(M; \d)$ as a direct product.
Hence $\pi_0\Diff(M; \d) \cong \Z^{2n} \x \PB_{n+r}$.
\end{proof}

\begin{corollary}
\label{SeifertFramedLinks}
Let $F$ be an $(n+1)$-component framed link whose complement $C_F$ is Seifert-fibered over a genus-0 surface with  $r$ singular fibers (and $n+1$ boundary components), where $n\geq 0$ and $0 \leq r \leq 2$.  Then $\pi_1(\tL_F/SO_4) \cong \Z^{2n} \x \PB_{n+r}$.
\end{corollary}
\begin{proof}
By Proposition \ref{FramedLinksKpi1}, $\tL_F/SO_4$ is a $K(\pi,1)$ space for the group $\Diff(C_F; \d)$, which by Proposition \ref{DiffSeifertFramed} is $\Z^{2n} \x \PB_{n+r}$.
\end{proof}

\begin{proposition}
\label{DiffSeifertUnframed}
Let $M$ be the exterior of a nontrivial $(n+1)$-component Seifert-fibered link $f$ that is Seifert-fibered over a genus-0 surface with $r$ singular fibers, where $n \geq 0$ and $0 \leq r \leq 2$.  
Then the quotient of $\pi_0\Diff(M; \d M)$ by the annular Dehn twists along the meridians of $L$ is isomorphic to $\Z^{n-1} \x \PB_n$.
\end{proposition}

\begin{proof}
Recall the structure of the Seifert fibering outlined at the end of Section \ref{S:SeifertPrelims}.  In particular, the $n+1$ meridians of $L$ are the boundary components of $P_n$, unless $L$ is a keychain link, in which case, they are the $n$ inner boundary components of $P_n$ and the fiber.  

If $M$ is the complement of a keychain link, then $\pi_0\Diff(M; \d M) \cong \Z^n \x \Z^n \x \PB_n$, where the meridional Dehn twists generate a subgroup $\Z^{n+1}$, with $n$ copies from the left-hand factor of $\Z^n$ and one copy from the diagonal in the right-hand factor of $\Z^n$. 

Now suppose $M$ is the complement of any other nontrivial Seifert-fibered link.
If $n+r<2$, then since $M$ is by assumption not the complement of the unknot, $n=1$ and $r=0$. 
By Proposition \ref{SeifertFiberedLinks}, $M$ is then the complement of the Hopf link $KC_1$.
So we may suppose $n+r \geq 2$.  Then $n$ of the meridional Dehn twists generate copies of $\Z$ in the left-hand factor of $\pi_0\Diff(M; \d M) \cong \Z^{2n} \times \PB_{n+r}$, while the last one generates $Z(\PB_{n+r})$.  The group $\PB_{n+r} / Z(\PB_{n+r})$ is either $\mathrm{PMod}(\Sigma_{0,n+r+1})$ or trivial, according as $n+r \geq 3$ or $n+r=2$.  
The quotient of $\pi_0\Diff(M; \d M)$ is then respectively
 $\Z^n \x \mathrm{PMod}(\Sigma_{0,n+r+1}) \cong \Z^{n-1} \x \PB_{n+r}$ or $\Z^n \cong \Z^{n-1} \x \PB_2$, as desired.
\end{proof}

\begin{corollary}
\label{SeifertSpacesOfLinks}
Let $f$ be a Seifert-fibered link that is not the unknot.
\begin{enumerate}
\item[(a)]
If $f$ is a torus knot $T_{p,q}$ (so $p,q>1$ and $\gcd(p,q)=1$), then $\pi_1(\L_f / SO_4) \cong \{e\}$.
\item[(b)]
If $f$ is a torus link $T_{p,q}$ with $n+1\geq 2$ components, none of which is an unknot (i.e.~$p/\gcd(p,q)>1$ and $q/\gcd(p,q)>1$), then $\pi_1(\L_f/SO_4) \cong \Z^{n-1} \x \PB_{n+2}$.
\item[(c)]
If $f$ is a Seifert link $S_{p,q}$ with $n+1$ components and $p/\gcd(p,q)>1$, then $\pi_1(\L_f/SO_4) \cong \Z^{n-1} \x \PB_{n+1}$.
\item[(d)]
If $f$ is a link $R_{p,q}$ with $n+1$ components, then $\pi_1(\L_f/SO_4)$ $\cong $
$\Z^{n-1} \x \PB_{n}$.
\item[(e)]
If $f$ is the $(n+1)$-component keychain link, then $\pi_1(\L_f/SO_4) \cong \Z^{n-1} \x \PB_n$.  
\end{enumerate}
\end{corollary}

\begin{proof}
By Proposition \ref{LfModSO4isKpi1}, the groups in question are the quotients by meridional rotations given in Proposition \ref{DiffSeifertUnframed}.  
Proposition \ref{SeifertFiberedLinks} tells us the number $r$ of singular fibers for each link.
In case (a), $n=0$ and $r=2$; 
in case (b), $r=2$; 
in case (c), $r=1$;
and in cases (d) and (e), $r=0$.
(The extra condition in cases (b) and (c) are needed to guarantee these values of $r$.)
The five cases above are all the possibilities (and they are disjoint), by Proposition \ref{SeifertFiberedLinks} and the following facts: 
\begin{itemize}
\item for $n\geq 3$, $T_{n,n} \sim S_{n-1,n-1} \sim R_{n-2,n-2}$; 
\item for $n\geq 2$ and $p'>1$, $T_{np',n} \sim S_{(n-1)p',n-1}$ (and $T_{n,nq'} \sim T_{nq',n}$); and
\item $S_{1,q}$ is the Hopf link, while for $n \geq 2$, $S_{n,nq'} \sim R_{n-1, (n-1)q'}$.
\end{itemize}
The Hopf link is covered by case (e)  with $n=1$, which yields the trivial group.
\end{proof}
  
We now describe explicit generators of $\pi_1(\L_f/SO_4)$ in some special cases.  See also Remark \ref{TorusKnotRmk} below.

In case (c) with $\gcd(p,q)=1$ (i.e.~$n+1=2$), we have a $(p,q)$-Seifert link, and we can take as a generator of $\PB_2 \cong \Z$ a loop that reparametrizes $C_2$ while leaving the knotted component fixed. 

In case (d) with $\gcd(p,q)=1$, we have a 3-component link $R_{p,q}$, and we can take independent reparametrizations of the unknotted components $C_1$ and $C_2$ as generators for $\Z \x \PB_2 \cong \Z \x \Z$.  

In case (e) of $KC_n=(L_0, L_1, \dots, L_n)$, we can take independent reparametrizations of all but one of $L_1, \dots, L_n$ as generators of the factor $\Z^{n-1}$.  Let  $p_{ij}$, $1 \leq i < j \leq n$ be the standard generators of $\PB_n$.  We can represent $p_{ij}$ by a loop where $L_i$ passes through $L_{i+1}, \dots, L_j$ counter-clockwise, then $L_j$ passes through $L_i$ counter-clockwise, and finally $L_i$ passes clockwise back through $L_{j-1},\dots, L_{i+1}$, returning to its original position.

\begin{figure}[h!]
\includegraphics[scale=0.2]{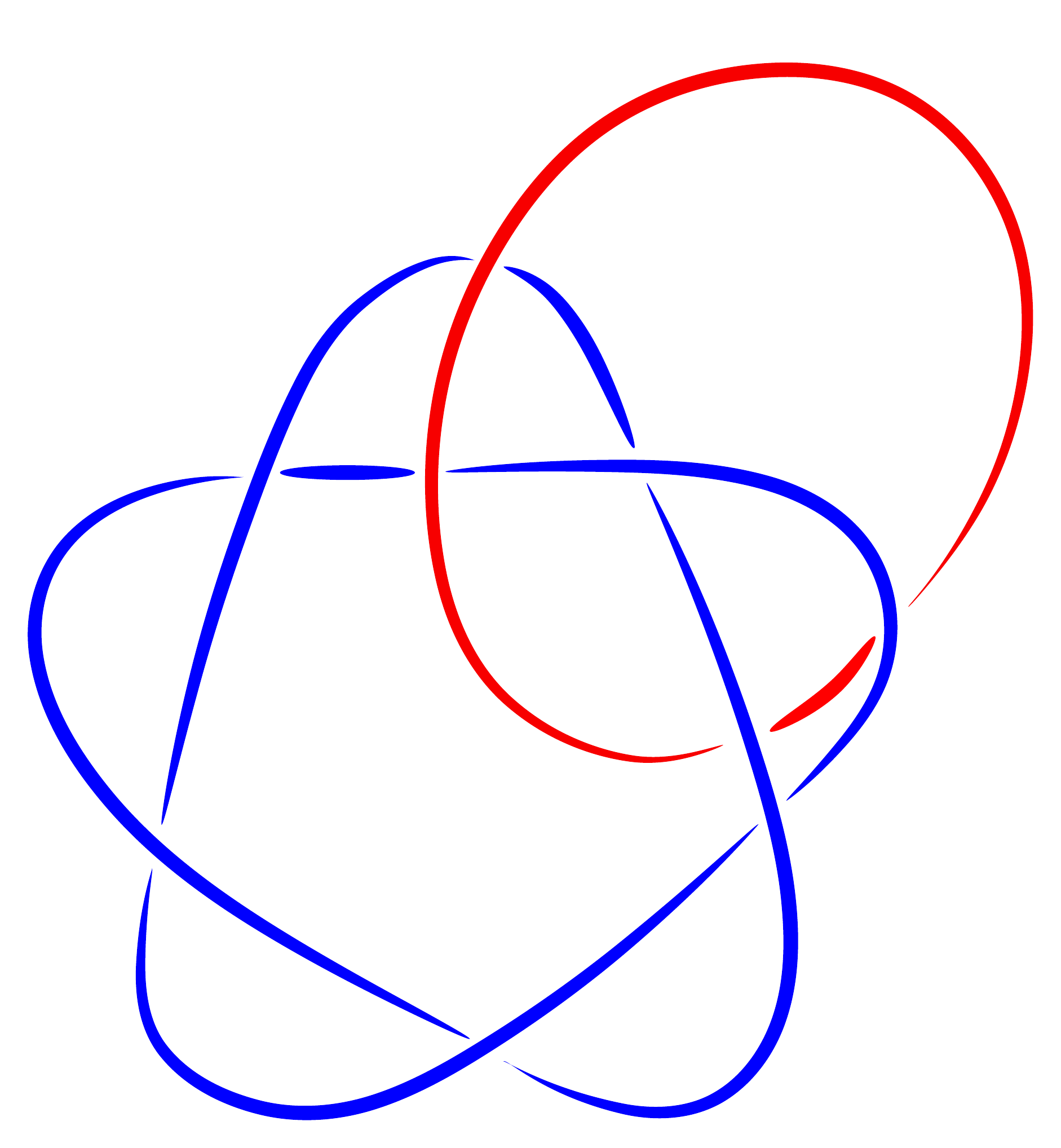} \qquad \qquad
\includegraphics[scale=0.35]{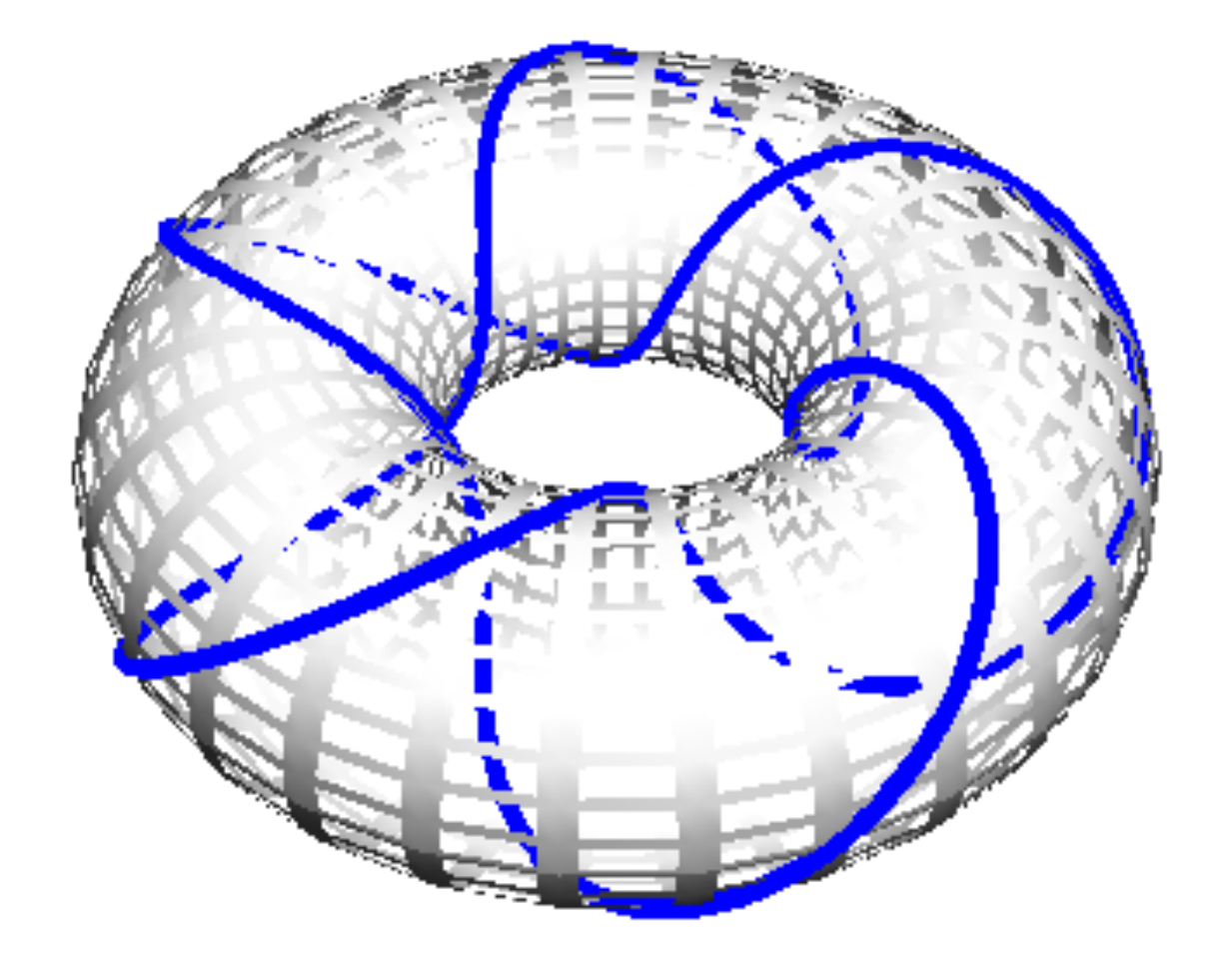}
\caption{The $(2,5)$-Seifert link, which may also be regarded as a knot $f$ in $S^1 \x D^2$.  Then $\pi_1(\T_f) \cong \Z\langle \lambda, \mu\rangle.$}
\label{F:2-5-SeifertLink}
\end{figure}

\begin{remark}
The above calculation of spaces $\L_f/SO_4$ of keychain links $f$ is implicit in Budney's decompositions of spaces of knots \cite{BudneyCubes, BudneyTop, BudneySplicing}.  
Presumably, an argument as in the work of Brendle and Hatcher \cite{Brendle-Hatcher} would show that these are equivalent to  the spaces of $(n+1)$-component \emph{necklaces} studied by Bellingeri and Bodin \cite[Theorem 2.3]{Bellingeri-Bodin}, which are keychain links consisting of Euclidean circles.  
Our quotient by $SO_4$ corresponds to fixing $L_0$ as the circle $C_1$ and fixing an additional component $L_{n+1}$ as $C_2$.  The bound on the radii of the components in necklaces can be viewed as the presence of such an additional component.  Because $L_0$ is fixed, our space of keychain links is most closely related to their space of normalized necklaces.  To get the latter space, one only has to take the quotient of $\L_f/SO_4$ by the group $\mathfrak{S}_n \wr S^1$ of relabelings and orientation-preserving reparametrizations of $L_1, \dots, L_n$.  The fundamental group of the resulting space has no factors of $\Z$ and features braids rather than pure braids; because $L_{n+1}$ is fixed, it is the braid subgroup $\B_{1,n}<\B_{n+1}$ with induced permutations fixing one point.  That is, the space of normalized necklaces is a $K(\pi,1)$ for the annular braid group.  
We will see this same group in Example \ref{Ex:MoreGeneralConnectSum}.
\end{remark}

\begin{proposition}
\label{TorusKnotInTorus}
Let $f$ be a knot in a solid torus whose complement is Seifert-fibered.  
If the associated 2-component link is the Hopf link, then $\T_f \simeq S^1$.
Otherwise $\T_f \simeq S^1 \x S^1$.
\end{proposition}
\begin{proof}
Let $M$ be the exterior of the knot $f$ in $S^1 \x D^2$.
By Proposition \ref{TfisKpi1}, $\T_f$ is a $K(\pi,1)$ with $\pi_1(\T_f)$ given by the quotient of $\pi_0 \Diff(M; \d M)$ by a Dehn twist along the meridian of $f$.

First consider the case that $f$ corresponds to the Hopf link.  
By Lemma \ref{DiffExteriorOfHopfLink},  $\pi_0 \Diff(M; \d M) \cong \Z^2$.  Its quotient by the Dehn twist along one of the components is $\Z$, completing the proof in this case.

Now suppose $f$ does not correspond to the Hopf link.
By Proposition \ref{TfandLinksWithUnknots} and Proposition \ref{SeifertFiberedLinks} (and the discussion before Proposition \ref{SeifertFiberedLinks}), $f$ corresponds to the component $T_{p,q}$ in the Seifert link $S_{p,q}$ where $\gcd(p,q)=1$ and $p>1$.  (This includes the possibility $S_{p,1} \sim T_{2p,2} \sim T_{2,2p}$.)
The link exterior $M$ is Seifert-fibered over a genus-0 surface $B$ with 2 boundary components and one singular fiber.  
By Proposition \ref{DiffSeifertFramed}, $\pi_0\Diff(M; \d M) \cong \Z^3$, generated by a Dehn twist along a boundary component of $B$, a Dehn twist along the fiber, and a pure braid on 2 strands.  
The Dehn twist along the meridian of $f$ 
generates one factor of $\Z$, so the quotient by it is $\Z^2$.  
\end{proof}

\begin{remark}
\label{TorusKnotRmk}
Let $f$ be the link $T_{p,q}$ in its standard position on a torus (slightly thinner than $U$ so that $f$ lies in the interior of $U$), e.g.~as shown in Figure \ref{F:2-5-SeifertLink}.
There are loops $\mu$,  $\lambda$, and $\rho$ given respectively by reparametrization (of all the components, ``diagonally''), meridional rotation, and longitudinal rotation.  These give rise to loops in $\L_f$.
Let $f'=S_{p,q}$.  If $\gcd(p,q)=1$, then we can make sense of $\T_{f'}$ and loops $\mu$,  $\lambda$, and $\rho$ in $\T_{f'}$.  
In either case, they come from loops in $SO_4$ given by 
$\left( \begin{array}{cc} e^{2\pi i t} & 0 \\ 0 & 1 \end{array} \right), 
\left( \begin{array}{cc} 1 & 0 \\ 0 & e^{2\pi i t} \end{array} \right),$ and 
$\left( \begin{array}{cc} e^{2\pi i pt} & 0 \\ 0 & e^{2\pi i q t} \end{array} \right)
 \in \C^{2 \x 2}$, $t \in [0,1]$.
Both $\mu$ and $\lambda$ represent the generator in $\pi_1 \L_f \cong \pi_1 SO_4 \cong \Z/2$.
Since $\rho \simeq p \mu + q \lambda$, so does $\rho$ if $p+q$ is odd, while $\rho$ is nullhomotopic if $p+q$ is even.  
This agrees with a result of Goldsmith \cite[Theorem 3.7]{Goldsmith:Math-Scand}.    
In the case of $\T_{f'}$, we cannot use $SO_4$ to find a homotopy between any pair of the three, but still $\rho \simeq p \mu + q \lambda$, so $\mu$ and $\lambda$ generate $\pi_1(\T_{f'}) \cong \Z^2$.
\end{remark}

\begin{proposition}
\label{P:SeifertVf}
Let $f$ be a Seifert-fibered 3-component link which corresponds to a knot in the thickened torus $S^1 \x S^1 \x I$.  Then $\V_f \simeq S^1 \x S^1$.
\end{proposition}

\begin{proof}
The link $f$ is precisely a 3-component link that contains the Hopf link as a sublink.  Thus possible cases include $f=KC_2$ and $f=R_{p,q}=T_{p,q} \cup C_1 \cup C_2$ for any integers $p$ and $q$.  By considering the linking numbers of components of $T_{p,q}$ with each other and with $C_1$ and $C_2$, we see that these are the only possibilities.  (Note that $S_{2,2q}\sim R_{1,q}$, and $T_{3,3} \sim S_{2,2} \sim R_{1,1}$.)  By Proposition \ref{VfIsLfModSO4} and Corollary \ref{SeifertSpacesOfLinks}, parts (d) and (e), $\V_f$ is a $K(\pi,1)$ with fundamental group $\Z\x\Z$.
\end{proof}

As generators of $\pi_1(\V_f)$ in the Seifert-fibered case, we can take longitudinal and meridional rotations of the torus.

\begin{example}
\label{Ex:T_3,3}
The Seifert-fibered link $T_{3,3}$ has the Hopf link as a sublink so gives rise to a knot in $S^1 \x S^1 \x I$.  See Figure \ref{F:T_3-3}.  
As for any such Seifert-fibered link, $\pi_1(\V_f) \cong \Z \x \Z$ by Proposition \ref{P:SeifertVf}, with generators given by longitudinal and meridional rotations of the torus.
\end{example}

\begin{figure}[h!]
\includegraphics[scale=0.25]{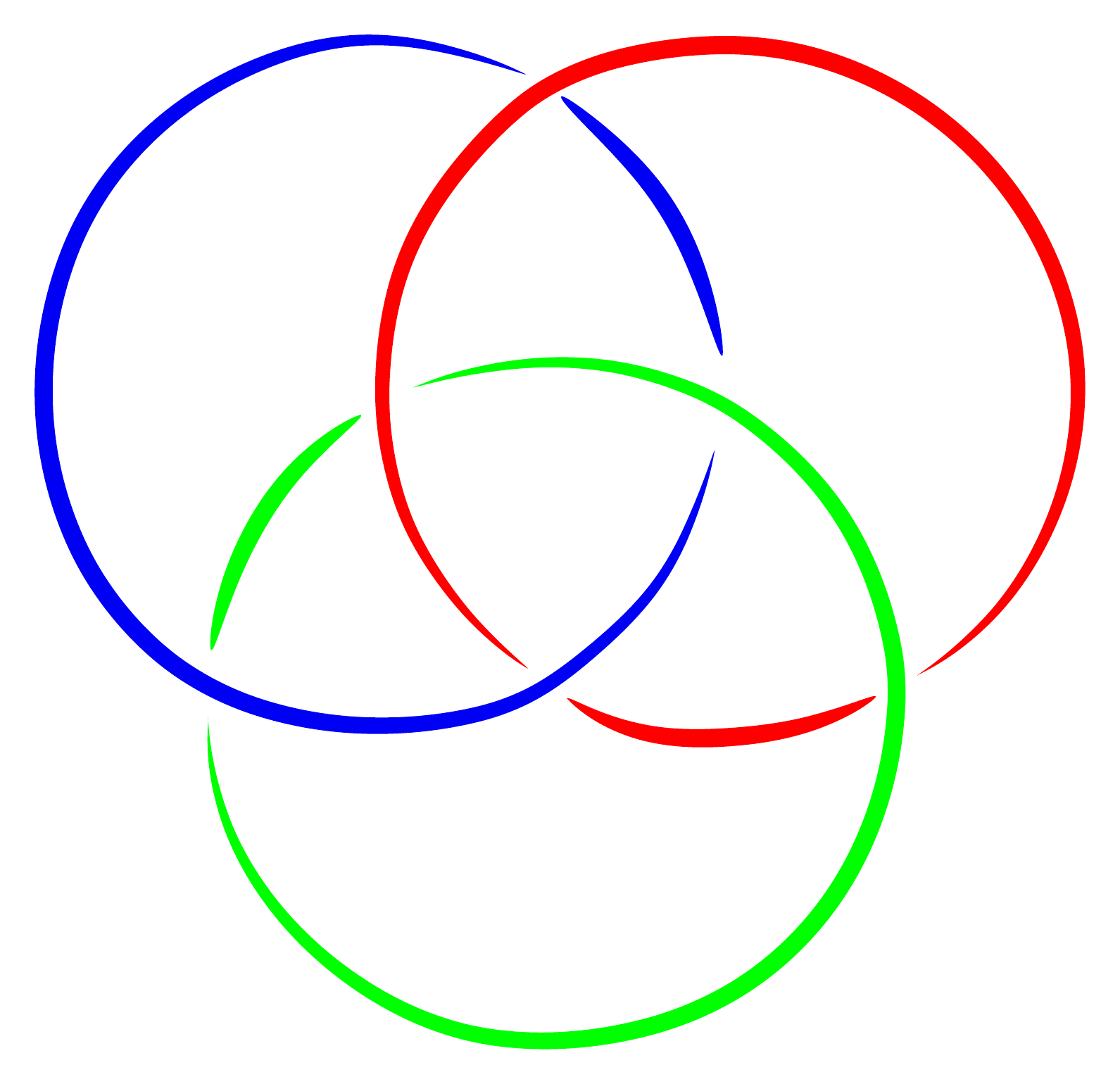}
\caption{The Seifert-fibered link $T_{3,3} \sim T_{2,2} \cup C_2 \sim T_{1,1} \cup C_1 \cup C_2$, which gives rise to a knot $f$ in $S^1 \x S^1 \x I$, with $\pi_1(\V_f) \cong \Z \langle \lambda, \mu \rangle.$}
\label{F:T_3-3}
\end{figure}

\begin{example}
\label{Ex:KC2}
The 3-component keychain link $KC_2$ also gives rise to a knot in $S^1 \x S^1 \x I$.  
Again, $\pi_1(\V_f) \cong \Z \x \Z$.  For generators in this case, we can take a longitudinal rotation of the torus and either a meridional rotation or reparametrization.
\end{example}

\section{Spaces of hyperbolic links}
\label{S:Hyperbolic}
We now determine the homotopy types of spaces of hyperbolic links.  Our starting point (Proposition \ref{HMProp}) is a special case of work of Hatcher and McCullough \cite{HatcherMcCullough} where the boundary consists of only tori.  
It quickly yields the homotopy type of spaces $\tL_F/SO_4$ of framed links (Corollary \ref{DiffHyp}).
We sketch its proof because we need it for a necessary refinement (Proposition \ref{DiffS3LHypLink}) which generalizes work of Hatcher   \cite{HatcherKnotSpacesArxiv, HatcherKnotSpacesWebsite} from knots to links.  That refinement gives us 
the homotopy types of $\L_f/SO_4$ for hyperbolic links $f$ (Corollary \ref{nCompHypLink}), as well as of the spaces $\T_f$ and $\V_f$ where applicable (Corollaries \ref{HypKnotInTorus} and \ref{HypKnotVf}).  Recall that hyperbolic links are precisely those irreducible links whose complements have trivial JSJ decompositions and are not Seifert-fibered, so they are the second basis subcase in the recursion in Theorem \ref{MainT:FramedLinks}, part (A).
We conclude this Section with a handful of examples.

\begin{proposition}[Proposition 3.2 in \cite{HatcherMcCullough}]
\label{HMProp}
Let $M$ be a hyperbolic 3-manifold whose boundary is a nonempty disjoint union of tori.  Let $R$ be a union of $n>0$ components of $\d M$.  
Let $\Diff_R(M)$ be the group of diffeomorphisms of $M$ whose restrictions to $R$ are isotopic to the identity.
\begin{itemize}
\item[(a)]
There is a short exact sequence 
\begin{equation}
\label{HypLinkSES}
\xymatrix@R-0.5pc{
\{e\}  \ar[r] & \pi_1(\Diff_0(R))  \ar[r]  & \pi_0(\Diff(M; R))  \ar[r] & \pi_0(\Diff_R(M)) \ar@{^(->}[d] \ar[r] & \{e\} \\
 & \Z^{2n} \ar@{<->}[u]_-\cong
 %\ar@2{~}[u] 
 & & \mathrm{Isom}(M) & 
}
\end{equation}
where the right-hand group is identified with a subgroup of the finite group $\Isom(M)$ of hyperbolic isometries of $M$.  
\item[(b)]
Moreover $\pi_0 (\Diff(M; R)) \cong H_1(R; \, \Z) \cong \Z^{2n}$.
\end{itemize}
\end{proposition}

\begin{proof}[Sketch of proof]
Establishing \eqref{HypLinkSES} is similar to the proof of Proposition \ref{LfModSO4isKpi1}.  The key is to obtain the injectivity of the first map, much like the injectivity of \eqref{DehnTwistsShouldInject}.  In this hyperbolic case, both meridional and longitudinal Dehn twists are nontrivial (whereas in the general case, only the meridional ones were guaranteed to be nontrivial, since a longitude may be central in $\pi_1(C_L) = \pi_1(M)$).  The identification of $\pi_0(\Diff_R(M))$ comes from a variant of Mostow rigidity, and $\Isom(M)$ is finite because $M$ has finite volume in the case where $\d M$ consists of only tori.

The key to prove part (b) is to exhibit $\pi_0(\Diff(M; R))$ as a subgroup of $\R^{2n}$, which together with part (a) suffices.  
We use this description to prove Proposition \ref{DiffS3LHypLink} below.
The space $\Diff(M; R)$ is the fiber of the map $\Diff(M) \to \Diff(R)$, but it is also equivalent to its homotopy fiber, i.e.,
\[ 
\Diff(M;R) \simeq \{ (\phi, \gamma) :  \phi \in \Diff(M), \, \gamma : I \to \Diff(R) \mbox{ with } \gamma(0) = \phi|_{R} \mbox{ and } \gamma(1) = \id_{R}\}.
\]  
We restrict our attention to $\pi_0$.
By Mostow rigidity, we can take $\phi$ to be a hyperbolic isometry of $M$ and $\gamma$ to be a path of rotations of the $n$ boundary tori.  
Since a Riemannian isometry is locally determined, and since the geometry of the boundary tori determine the geometry of their cusp neighborhoods, $\phi$ is determined by $\phi|_{R} \in (S^1)^{2n}$. That is, the restriction $\phi \mapsto \phi|_{R}$ is injective on isometries.  
Since $Isom(M)$ is finite, its image under this map is a discrete set in $(S^1)^{2n}$.  
We then replace $(\phi, \gamma)$ by $(\phi|_{R}, \gamma)$, and since the isometries have a discrete image, 
we may identify $\pi_0(\Diff(M; R))$ with equivalence classes of $(\phi|_{R}, \gamma)$ up to homotopies of $\gamma$ fixing \emph{both} endpoints. 
Finally, such equivalence classes correspond precisely to paths $\tgamma$ in the universal cover $\R^{2n}$ ending at $0$.  So we have constructed an inclusion $\pi_0(\Diff(M; R)) \incl \R^{2n}$, as desired.  
\end{proof}

Specializing Proposition \ref{HMProp}(b) to the case $R=\d M$ and applying Proposition \ref{FramedLinksKpi1} immediately yields the desired calculation for framed links:

\begin{corollary}
\label{DiffHyp}
Let $F$ be an $n$-component framed link in $S^3$ with hyperbolic complement $M$.  
Then 
\begin{flalign*}
&& \pi_1(\tL_F/SO_4) \cong \Z^{2n}. && \qed
\end{flalign*}
\end{corollary}

Next, we consider unframed hyperbolic links, which involves a straightforward generalization of work of Hatcher \cite{HatcherKnotSpacesArxiv, HatcherKnotSpacesWebsite} from hyperbolic knots to hyperbolic links.

\begin{proposition}
\label{DiffS3LHypLink}
Let $f$ be an $n$-component hyperbolic link in $S^3$ with exterior $M$. 
Then \[\pi_0(\Diff^+(S^3; f)) \cong \Z^{n}.\]
\end{proposition}

\begin{proof}
The fibration \eqref{LinkCompVsExtS1} yields the exact sequence 
\begin{equation}
\label{HypLinkSESMeridians}
\xymatrix{
0 \ar[r] & \pi_1((S^1)^n) \ar^-{\iota}[r] & \pi_0(\Diff(M; \d M)) \ar[r] & \pi_0(\Diff^+(S^3; f)) \ar[r] & 0
}
\end{equation}
where the first map sends the $i$-th standard generator to a Dehn twist along the $i$-th {meridian}.
All the groups are indeed abelian since by Proposition \ref{HMProp} the middle term is $\Z^{2n}$.  

It suffices to show that the quotient is torsion-free.  Suppose not.  Then some $\mu = (a_1, \dots, a_n) \in \Z^n \cong \pi_1((S^1)^n)$ satisfies $\iota(\mu)=k[\psi]$ for some $\psi \in \Diff(M; \d M)$ and some $k>1$, with $[\psi]$ nonzero in $\pi_0(\Diff(M; \d M)) \cong \Z^{2n}$.  We may suppose $\gcd(a_1, \dots, a_n)=1$, i.e.~$\mu$ is not a proper multiple.  
As in the proof of Proposition \ref{HMProp}, we view $\pi_0(\Diff(M; \d M)) \cong \Z^{2n}$ as a subgroup of $\R^{2n}$.
Under this identification, $\iota(\mu)$ corresponds to $(a_1, \dots, a_n, 0, \dots, 0)$ 
and hence $[\psi]$ corresponds to  $(a_1/k, \dots, a_n/k, 0, \dots, 0)$ (where we map $\R\to S^1$ by $t \mapsto e^{2\pi i t}$).

Recall that this identification of $[\psi]$ is obtained via the isometry $\phi \in \mathrm{Isom}(M)$ that is homotopic to $\psi$ (by a homotopy not fixing $\d M$).  
One chooses a path $\gamma$ from $\psi$ to $\phi$, 
and $\gamma|_{\d M}$ is a path in $(S^1)^{2n}$ starting at $(1,\dots,1)$, which one
lifts to a path in $\R^{2n}$ starting at $\vec{0}$.  Then $(a_1/k, \dots, a_n/k, 0, \dots, 0)$ is the endpoint.
The  isometry $\phi$ is nontrivial since $k>1$ and $\gcd(a_1, \dots, a_n)=1$.
Since $\phi|_{\d M}$ consists of only meridional rotations, the extension of $\phi$ to $S^3$ acts by the identity on $L$.  But the group of hyperbolic isometries is finite, so this extension of $\phi$ is an orientation-preserving periodic diffeomorphism of $S^3$ whose fixed point set contains $L$.  
This contradicts the Smith conjecture that such a fixed-point set must be an unknot \cite{MorganBass} (though for $n>1$, one may simply apply Smith's theorem that the fixed-point set of such an action must be a knot \cite{Smith:Annals1939}).
\end{proof}

\begin{corollary}
\label{nCompHypLink}
For an $n$-component hyperbolic link $f$, $\L_f/SO_4 \simeq (S^1)^n$.  
\end{corollary}
\begin{proof}
A hyperbolic link is necessarily irreducible, so by Proposition \ref{LfModSO4isKpi1}, $\L_f / SO_4$ is a $K(\pi,1)$ for the group $\Diff^+(S^3; f)$.  By Proposition \ref{DiffS3LHypLink} this group is $\Z^n$. 
\end{proof}

\begin{corollary}
\label{HypKnotInTorus}
If $f$ is a 2-component hyperbolic link whose first component is unknotted, then $\T_f \simeq (S^1)^3$.
\end{corollary}
\begin{proof}
Let $M$ be the exterior of the link $f$ in $S^3$.   By Proposition \ref{DiffHyp}, $\pi_0 \Diff(M; \d M) \cong \Z^4$.
By Proposition \ref{TfisKpi1}, $\T_f$ is a $K(\pi, 1)$, and $\pi_1(\T_f)$ is the quotient of $\pi_0 \Diff(M; \d M)$ by a Dehn twist along the meridian of the first component.  One can now argue as in the proof of Proposition \ref{DiffS3LHypLink}, but with the subgroup of Dehn twists along just the first meridian, rather than both meridians.  The result is that the quotient is $\Z^3$.
\end{proof}

\begin{corollary}
\label{HypKnotVf}
If $f=(f_0,f_1,f_2)$ is a hyperbolic link where $(f_1,f_2)$ is the Hopf link, then $\V_f \simeq (S^1)^3$.
\end{corollary}
\begin{proof}
Apply Proposition \ref{VfIsLfModSO4} and Corollary \ref{nCompHypLink}.
\end{proof}

In each example below, we consider a hyperbolic link $f$ corresponding to knots $f \in \T$ or $f \in \V$.  In some cases, the hyperbolicity follows from a theorem of Menasco that a non-split, prime, non-torus alternating link  is hyperbolic \cite{Menasco1984}.
In all cases, one can use software such as SnapPy \cite{SnapPy} (based on SnapPea) or KLO \cite{KLO} to verify that they are hyperbolic.

The key to understanding generators of $\pi_1(\L_f/SO_4)$ for each link $f$ is the exact sequence \eqref{HypLinkSES}.  The right-hand side $\pi_0(\Diff_R(M))$ is the finite subgroup of isometries in the link's symmetry group $\mathrm{Isom}(M)$ which extend to diffeomorphisms of $S^3$ preserving all the components of $f$ and their orientations.    
SnapPy reveals this group by computing $\mathrm{Isom}(M)$, as well as whether elements extend to the link and their effects on boundary tori.
By the validity of the Smith conjecture \cite{MorganBass}, $\pi_0(\Diff_R(M))$ embeds in the space $\Diff^+(S^1)$ of symmetries of some component and hence must be cyclic.  
By the validity of Thurston's geometrization conjecture, the action of these symmetries on $S^3$ is equivalent to an action of a subgroup of $SO_4$ (though this was known to be true in many cases before geometrization was fully proven).
This together with basic facts about lattices implies that in general, one generator of $\pi_1(\L_f/SO_4)$ is some fraction of an integer combination of the loops of reparametrizations $\lambda_i$.  A generating set consists of this element together with all but one of the $\lambda_i$.

To describe generators for $\pi_1(\T_f/SO_4) \cong \Z^3$ (when applicable), we start with the two generators for $\pi_1(\L_f/SO_4)$, where we view $\lambda_1$ as a meridional rotation of the solid torus.  The third generator is simply given by the longitudinal rotation of the solid torus.

\begin{example}[The Whitehead link]
\label{Ex:WhiteheadLink}
The full symmetry group $\Isom(M)$ of the Whitehead link exterior $M$  is the dihedral group $D_4$ of order 8.  All of these extend to the link, and the effect on the $i$-th oriented component corresponds to the effect of $D_4$ on the $i$-th standard basis vector.  (One can realize a $\Z/2 \oplus \Z/2$ subgroup in the leftmost picture in Figure \ref{F:WhiteheadLink} via the $180^\circ$ rotations around the three coordinate axes as generators.)  
Thus the subgroup preserving the two components and their orientations is trivial.  So for generators of $\pi_1(\L_f/SO_4)$, we can take reparametrizations of the two components.
\end{example}

\begin{figure}[h!]
\raisebox{1pc}{\includegraphics[scale=0.25]{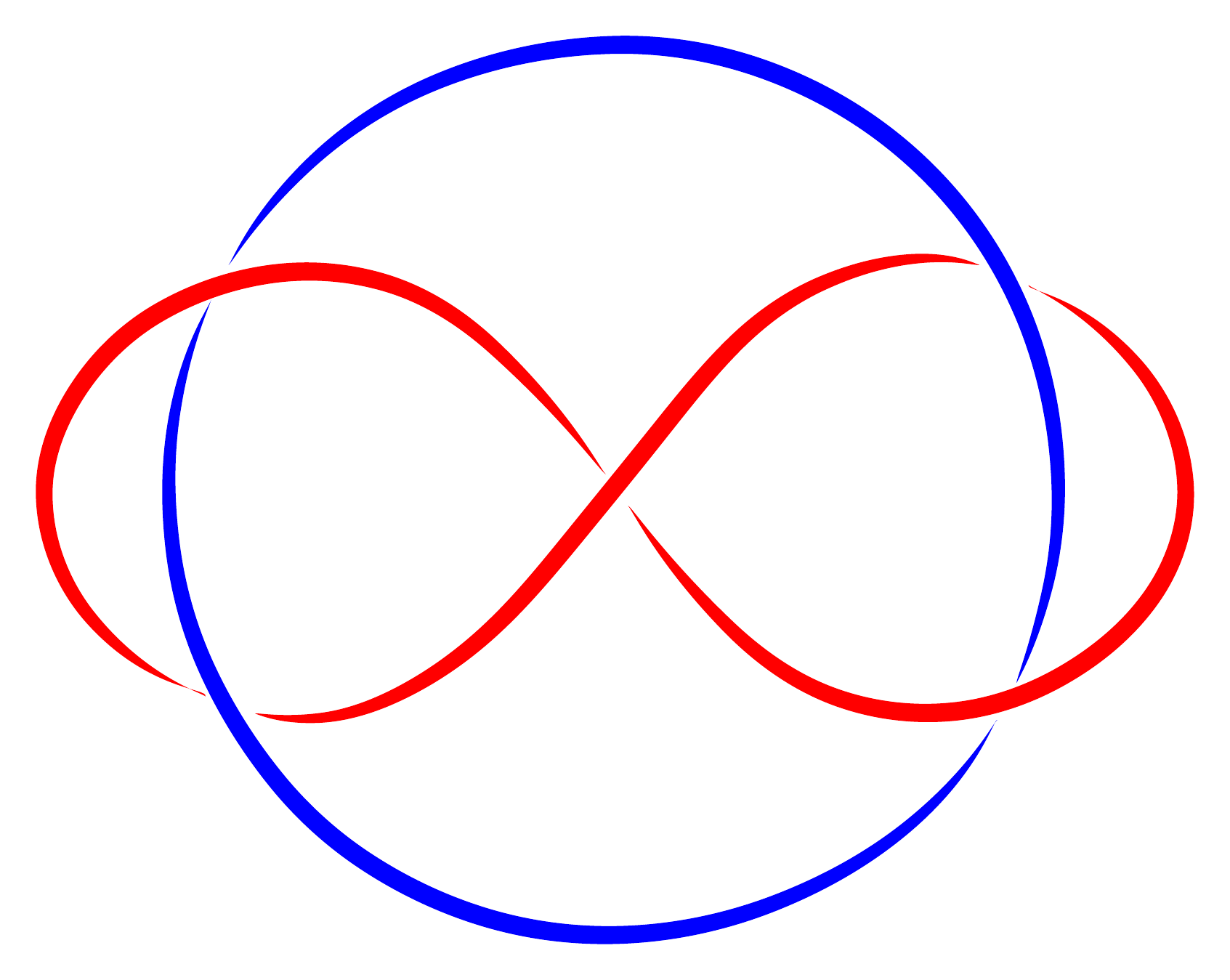}}
\qquad \qquad
\includegraphics[scale=0.23]{whitehead-link-2.pdf}
\caption{Two planar projections of the Whitehead link $f$, where $\pi_1(\L_f/SO_4)\cong \Z\langle \lambda_0, \lambda_1\rangle$.}
\label{F:WhiteheadLink}
\end{figure}

\begin{example}[A hyperbolic link from the figure-eight knot]
\label{Ex:HypLinkFromFig8}
The full symmetry group $\Isom(M)$ of the figure-eight knot exterior $M$ is $D_4$.  All of these isometries extend to the knot, by the theorem of Gordon and Luecke that homeomorphisms of knot complements send a meridian to a meridian \cite{Gordon-Luecke}.  Those preserving orientation form a subgroup $D_2 \cong \Z/2 \x \Z/2$.\footnote{This $D_2<D_4$ is of course the subgroup containing the $180^\circ$ rotation and reflections across the coordinate axes.  The isometry of $M$ corresponding to a reflection of $\R^3$ (and an isotopy between mirror images of the knot) can be identified with a reflection across a diagonal in $D_4$.}
If $f$ is the link L8n1 a.k.a.~$8^2_{16}$ shown in Figure \ref{F:HypLinkFromFig8}, then $\Isom(C_f) \cong \Isom^+(C_f) \cong (\Z/2)^3$, but the
isometries which extend to $f$ form the Klein four group, given by $180^\circ$ rotations in the coordinate axes, much like the orientation-preserving isometries of the figure-eight knot complement.  
They all preserve the two components, but the subgroup preserving orientations of the components is $\Z/2$, given by the rotation in the plane.  Thus 
$\pi_1(\L_f/SO_4) \cong \Z \langle \frac{1}{2}(\lambda_0+ \lambda_1), \ \lambda_0\rangle
\cong \Z \langle \frac{1}{2}(\lambda_0+ \lambda_1), \ \lambda_1\rangle$, 
where the loop $\frac{1}{2}(\lambda_0+ \lambda_1)$ can be visualized by a path of rotations in the plane (which is trivial modulo $SO_4$) followed by a path of reparametrizations of each component by $\pi$.
\end{example}

\begin{figure}[h!]
\includegraphics[scale=0.25]{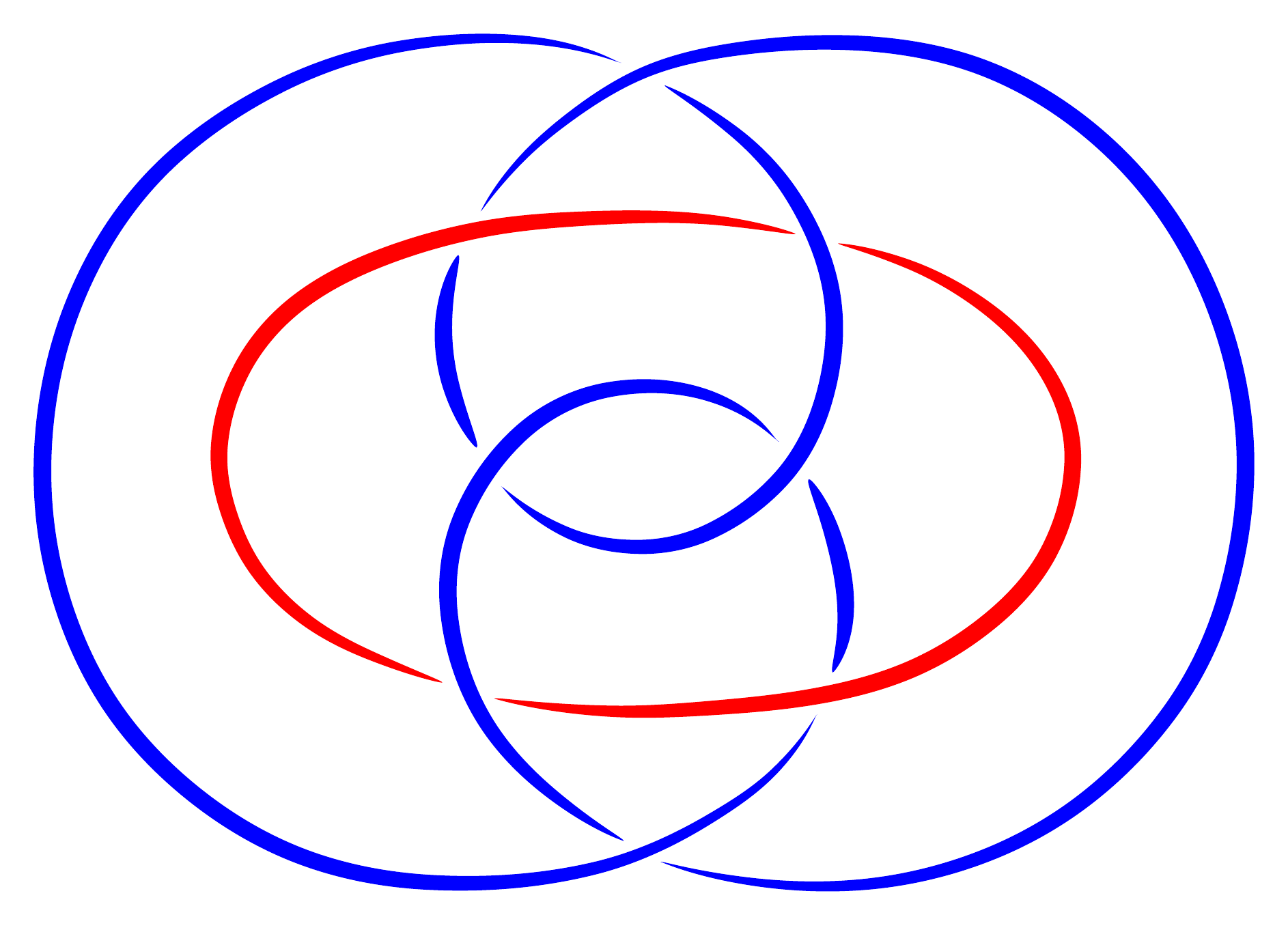}
\caption{The hyperbolic link L8n1 a.k.a.~$8^2_{16}$, which has one unknot component and one figure-eight knot component.  It has a symmetry preserving both components and their orientations by rotating by $\pi$ in the plane.  For this link $f$, 
$\pi_1(\L_f/SO_4) \cong \Z \langle \frac{1}{2}(\lambda_0+ \lambda_1), \ \lambda_0\rangle
\cong \Z \langle \frac{1}{2}(\lambda_0+ \lambda_1), \ \lambda_1\rangle$.
}
\label{F:HypLinkFromFig8}
\end{figure}

\begin{example}[A hyperbolic link from the $(2,k)$-torus knot]
\label{Ex:HypLinkFromTrefoil}
For any odd $k\geq 3$, there is a hyperbolic link $f$ with an unknot component and a $(2,k)$-torus knot component, as shown in Figure \ref{F:HypLinkFromTrefoil}.
We have $\Isom(M) \cong \Isom^+(M)\cong D_k$, with all isometries extending to the link and preserving the two components.  These symmetries are apparent from the picture.  Those that preserve orientations form a subgroup $\Z/k$.  
Thus $\pi_1(\L_f/SO_4) \cong \Z \langle \frac{1}{k}(\lambda_0+ \lambda_1), \ \lambda_0\rangle
\cong \Z \langle \frac{1}{k}(\lambda_0+ \lambda_1), \ \lambda_1\rangle$, where the first generator can be visualized by a path of rotations in the $xy$-plane by $2\pi/k$, followed by a path of reparametrizations of both components by $2\pi/k$.  
\end{example}

\begin{figure}[h!]
\includegraphics[scale=0.22]{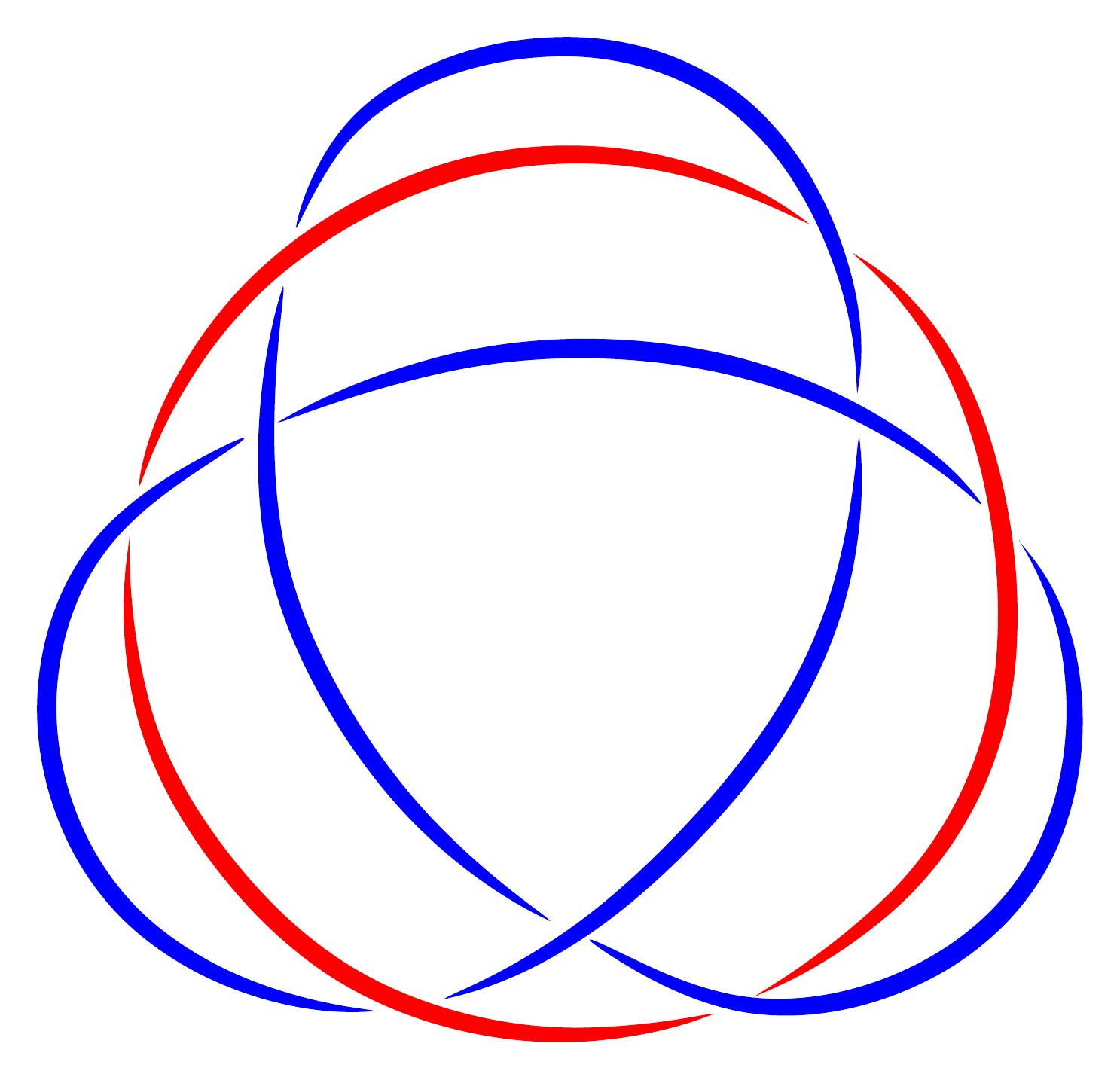}
\caption{The hyperbolic link L9a32 a.k.a.~ $9^2_{40}$, which has one unknot component and one trefoil component.  It has $\Z/3$ symmetry preserving both components and their orientations, given by rotation by $2\pi/3$ in the plane.  
For this link $f$, 
$\pi_1(\L_f/SO_4) \cong \Z \langle \frac{1}{3}(\lambda_0+ \lambda_1), \ \lambda_0\rangle
\cong \Z \langle \frac{1}{3}(\lambda_0+ \lambda_1), \ \lambda_1\rangle$.
}
\label{F:HypLinkFromTrefoil}
\end{figure}

\begin{example}[A hyperbolic link from the knot $8_{18}$]
\label{Ex:HypLinkFrom8-18}
For the link $f$ shown in Figure \ref{F:HypLinkFrom8-18}, a similar analysis as above applies, but with $D_4$ and $\Z/4$ rather than $D_k$ and $\Z/k$.  
\end{example}

\begin{figure}[h!]
\includegraphics[scale=0.22]{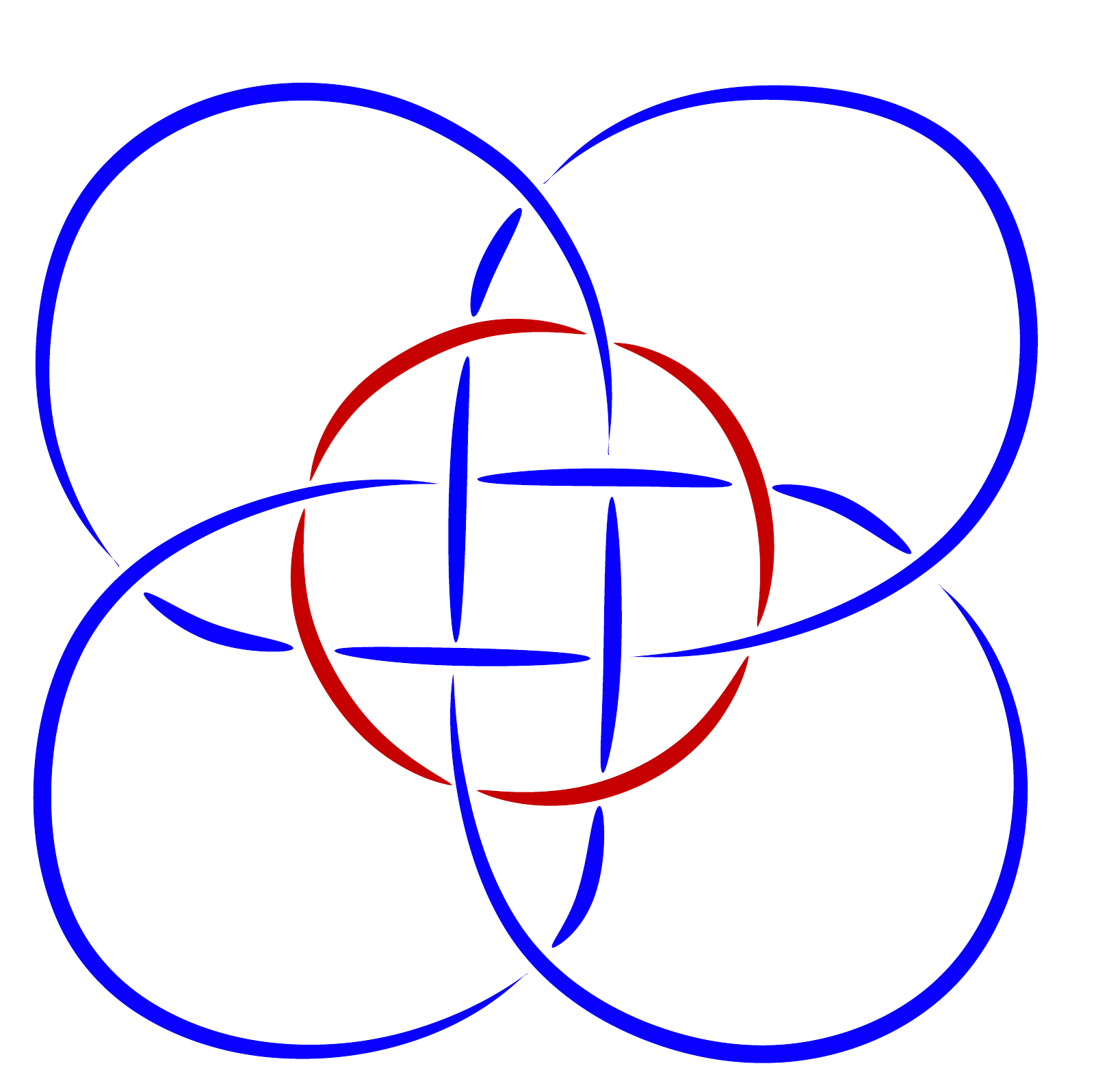}
\caption{A hyperbolic link $f$ with one unknot component and one $8_{18}$ component, and with $\Z/4$ symmetry preserving both components and their orientations.  Thus
$\pi_1(\L_f/SO_4) \cong \Z \langle \frac{1}{4}(\lambda_0+ \lambda_1), \ \lambda_0\rangle
\cong \Z \langle \frac{1}{4}(\lambda_0+ \lambda_1), \ \lambda_1\rangle$.
}
\label{F:HypLinkFrom8-18}
\end{figure}

\begin{example}
\label{Ex:L6a5}
The hyperbolic link L6a5 a.k.a.~$6^3_1$ contains the Hopf link as a sublink.  See Figure \ref{F:HypLinkVf}.  We have $\Isom(M) \cong\Isom^+(M) \cong D_6$.  A $D_3$ subgroup is clear from the picture, but $\Isom(M)$ is an extension by $D_3$ of $\Z/2$ given by a $180^\circ$ rotation along the meridian of a solid torus containing the link.
However, there are no nontrivial symmetries preserving all three components and their orientations.  Thus $\pi_1(\L_f/SO_4) \cong \Z\langle \lambda_0, \lambda_1, \lambda_2 \rangle$.  Equivalently, $\pi_1(\V_f)$ is generated by the two rotations of the torus and reparametrization.
\end{example}

\begin{figure}[h!]
\includegraphics[scale=0.18]{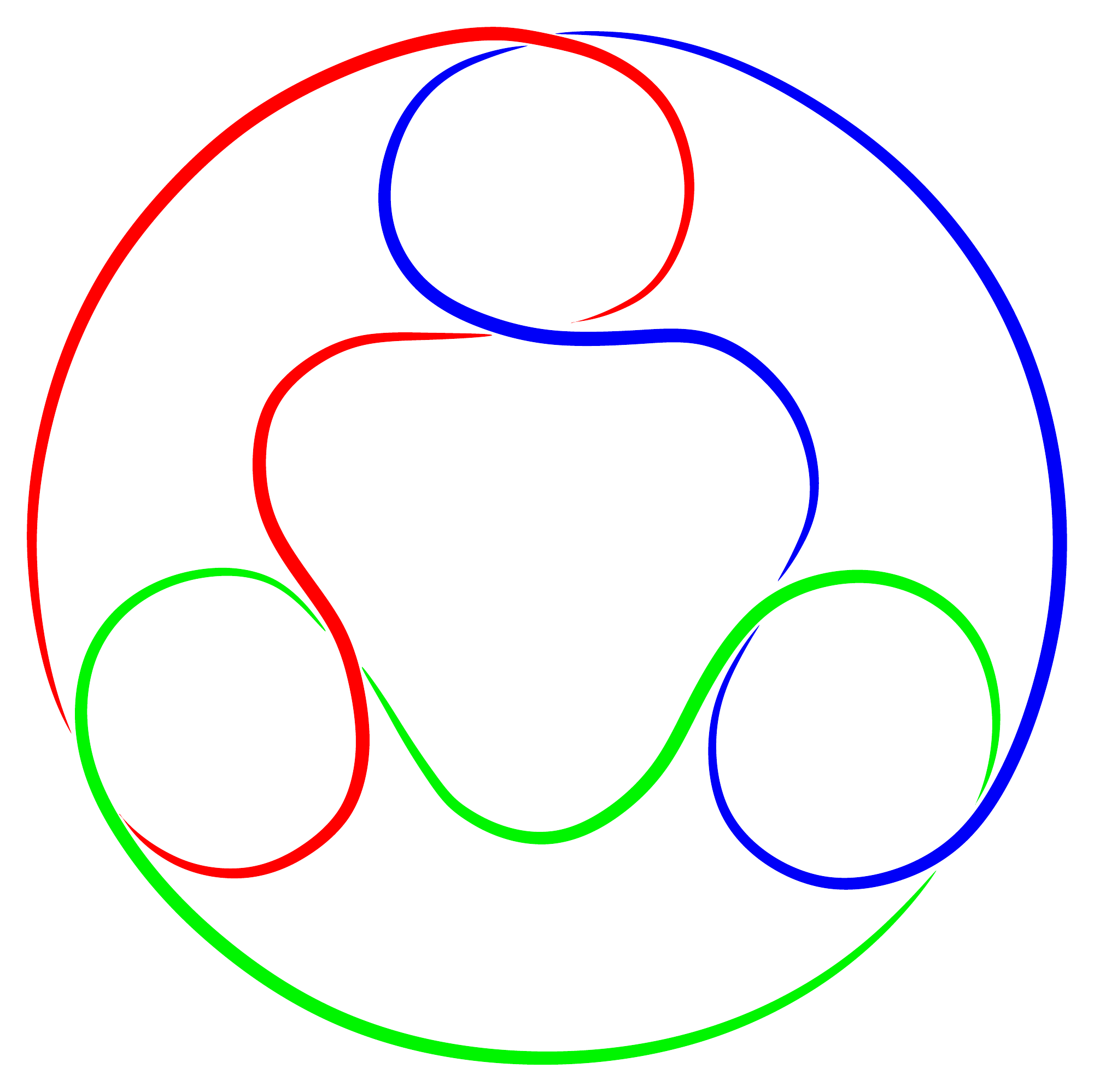}
\caption{The link L6a5 a.k.a.~$6^3_1$ is a hyperbolic link corresponding to a knot in $S^1 \x S^1 \x I$.
This link $f$ has $\pi_1(\L_f/SO_4) \cong \Z\langle \lambda_0, \lambda_1, \lambda_2 \rangle$.}
\label{F:HypLinkVf}
\end{figure}

\section{Spaces of spliced links}
\label{S:Splicing}
We will now complete the proof of Theorem \ref{MainT:FramedLinks}, part (A), by recursively determining the homotopy type of $\tL_F/SO_4$ for any irreducible framed link $F$.
Recall the notions of JSJ trees $G_F$, companionship trees $\mathbb{G}_F$, and the splicing operation $\bowtie$ from Section \ref{S:JSJ}.  
Suppose $F$ has a nontrivial JSJ tree $G_F$, or equivalently, a nontrivial companionship tree $\mathbb{G}_F$.
We arbitrarily choose a distinguished component $F_0$ of $F$.  Its boundary lies in $G_F(v_R)$ for some vertex $v_R$ in $G_F$, now designated as the root.  
Let $L:=\mathbb{G}_F(v_R)$.  Suppose $v_R$ has $n+1$ leaf half-edges (where $n \geq 0$) and $r$ remaining half-edges.  
Then $L=(L_0,\dots,L_{n+r})$ and 
\[
F=(\varnothing, \dots, \varnothing, J_1,\dots, J_r)\bowtie L.
\]
As mentioned in Remark \ref{SplicingRemark}, by writing the expression on the right-hand side above, we imply that it corresponds to a sub-decomposition of the JSJ decomposition of $C_F$.
We visualize this sub-decomposition as a tree shown in Figure \ref{F:TreeF=JbowtieL}, where the half-edges incident to a vertex $v$ labeled by a link $L_v$ correspond to the components of $L_v$, and the one emanating downwards corresponds to the distinguished component of $L_v$.  
The leaf half-edges correspond to components of $F$.
This tree is not $\mathbb{G}_F$ unless all the $J_i$ are Seifert-fibered or hyperbolic, but $\mathbb{G}_F$ can always be obtained from it by inserting each tree $\mathbb{G}_{J_i}$ into the vertex labeled by $J_i$.  
So we will call it a \emph{schematic of} $\mathbb{G}_F$.
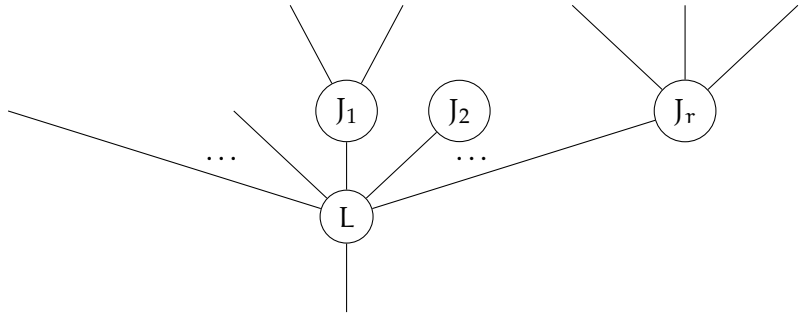
\begin{figure}[h!]
\begin{tikzpicture}
[grow=north, level distance=40pt]
\node[] {}[sibling angle=40]
child {node[draw,circle,fill=white, opacity=1, text opacity=1] {$L$}
child{node[draw,circle,fill=white, opacity=1, text opacity=1] {$J_r$}
child{[]} child{[]} child{[]} }
child { edge from parent[draw=none] node[draw=none] (ellipsis) {$\ldots$} }
child{node[draw,circle,fill=white, opacity=1, text opacity=1] {$J_2$}}
child{node[draw,circle,fill=white, opacity=1, text opacity=1] {$J_1$}
child{[]} child{[]} }
child{[]}
child { edge from parent[draw=none] node[draw=none] (ellipsis) {$\ldots$} }
child{[]}
};
\end{tikzpicture}
\caption{A tree corresponding to the decomposition $F=(\varnothing, \dots, \varnothing, J_1, \dots, J_r) \bowtie L$.   Inserting the tree $\mathbb{G}_{J_i}$ into each vertex labeled by $J_i$ produces the tree $\mathbb{G}_F$.}
\label{F:TreeF=JbowtieL}
\end{figure}

We will describe $\tL_F/SO_4$ in terms of the spaces $\tL_{J_i}/SO_4$.  By repeating this process upward along the tree $\mathbb{G}_F$ from the root half-edge (corresponding to $F_0$), we will obtain a complete description of $\tL_F/SO_4$.  Indeed, consider induction on the maximum distance $d$ from $v_R$ to any other vertex in $\mathbb{G}_F$, counted by the number of edges in a path joining them. 
The basis case of $d=0$ is covered by Corollary \ref{SeifertFramedLinks} and Corollary \ref{DiffHyp}.  The description of $\tL_F/SO_4$ in terms of the $\tL_{J_i}/SO_4$ is the induction step, and it will be given below.  
It varies by cases, according as $L$ is Seifert-fibered or hyperbolic.  The Seifert-fibered case itself has two subcases, one for $L=S_{p,q}$ or $R_{p,q}$ and one for $L=KC_n$.  
The contents of the upcoming subsections, which all generalize Budney's work on knots \cite{BudneyTop}, are as follows:
\begin{itemize}[leftmargin=0.25in]
\item
Section \ref{S:SplicingLemma}: Lemma \ref{SpliceSES}, a key result about splicing used in each of the remaining subsections. 
\item 
Section \ref{S:Cables}: the Seifert-fibered subcase generalizing cabling in Propositions \ref{P:Cable} and \ref{P:SpliceIntoRpq}.
\item
Section \ref{S:ConnectSums}: the Seifert-fibered subcase generalizing connected sum in Proposition \ref{P:ConnectSum}.
\item
Section \ref{S:HypSplices}: the case of hyperbolic splicing in Proposition \ref{P:HypSplice}. 
\end{itemize} 

In most of our examples of links $F$ below, we label each edge or leaf half-edge of $\mathbb{G}_F$ by the two loops of rotations of the corresponding torus.  These are nontrivial loops in $\L_f/SO_4$, except for the meridional rotations of boundary tori, which we  parenthesize.  Note however that in $\pi_1(\L_f/SO_4)$, this set of loops does not always extend to a generating set, and sometimes there are relations among them. 

The setting of arbitrary irreducible links $F$ includes the special cases of those $F$ corresponding to knots in $\T$ or $\V$.  
If $F$ yields a knot in $\T$, then $\mathbb{G}_F$ has only one leaf half-edge in addition to the root half-edge.  Moreover, since the corresponding component is unknotted, one can attach a vertex corresponding to a knot to obtain a knot in $S^3$.  Therefore, every vertex in $\mathbb{G}_F$ is labeled by a KGL, just as in the case of knots in $S^3$, and we can think of $\mathbb{G}_F$ as the result of removing a non-root vertex from the companionship tree of a knot (indeed many different ones).  
If $F$ yields a knot in $\V$, then $\mathbb{G}_F$ has two leaf half-edges in addition to the root half-edge.  The corresponding components are unknotted but linked, so one can attach a vertex labeled by a knot $J$ to only one of these half-edges.  As a result, the links labeling vertices of $\mathbb{G}_F$ need not be KGL's;  for example, a vertex may be labeled by $R_{p,q}$.  If one is solely interested in the homotopy types of $\T_f$ and $\V_f$, these properties of $\mathbb{G}_F$ in these cases would only simplify the proofs of the results in this section by ruling out a few subcases.

\subsection{A general lemma about splicing}
\label{S:SplicingLemma}
We now set up the general framework which we will apply to all of the relevant cases of splicing.

\begin{definition}
\label{EquivRelDef}
Suppose $\vec{J}=(J_1,\dots,J_r)$, where each $J_i$ is a link with a distinguished component $J_{i,0}$.
Let $\iota_0^+(J_i):=J_i$, and let $\iota_0^-(J_i)$ be the result of reversing the orientation on $J_{i,0}$.  
Write $\iota_0^\pm(J_i)$ to mean either possibility.
Define an equivalence relation $\sim$ on $\{J_1, \dots, J_r, \iota_0^-(J_1), \dots, \iota_0^-(J_r)\}$ by $J_i \sim \iota_0^\pm(J_k)$ if  
\begin{itemize}
\item either $J_i$ and $\iota_0^\pm(J_k)$ are isotopic knots
\item or $i=k$ and $J_i$ and $\iota_0^\pm(J_i)$ are isotopic links.
\end{itemize}
The above two possibilities are not mutually exclusive, since knots are links.  The relation $\sim$ gives rise to a (Young) subgroup of $\mathfrak{S}_r^\pm$.  
Define $\mathfrak{S}^\pm(\vec{J})$ to be this subgroup.  
\end{definition}

The following Lemma and its proof are a variant of Budney's work \cite[Lemma 2.2]{BudneyTop} and \cite[p.~20]{BudneyCubes}.

\begin{lemma}
\label{SpliceSES}
Let $F$ be a framed link in $S^3$ with a designated component, so that the JSJ tree for its complement $C_F$ is rooted.  Let $L=(L_0,\dots, L_{n+r})$ be the companion link to the root manifold $C_L$ (so $L$ is either Seifert-fibered or hyperbolic).  Let $C_{J_1}, \dots C_{J_r}$ be the components of $C_F \setminus C_L$, and let $J_1, \dots, J_r$ be the associated companion links.  Then there is an exact sequence of groups
\begin{equation}
\label{Eq:SpliceSES}
\{e\} \to \prod_{i=1}^{r} \pi_0 \Diff(C_{J_i}; \d C_{J_i}) \to \pi_0 \Diff(C_F; \d C_F) \to \pi_0 \Diff(C_L; T_0 \cup \dots \cup T_{n})
\end{equation}
where $T_0, \dots, T_{n}$ are the boundary tori of $C_L$ that are also boundary tori of $C_F$.  

The image $H$ of the rightmost map is represented by those diffeomorphisms $\psi$ which extend to a diffeomorphism $\widetilde{\psi}$ of the pair $(S^3, L)$  whose action on $L_{1}, \dots, L_{n+r}$ is given by an element of $\mathfrak{S}^\pm(\vec{J}) < \mathfrak{S}_r^\pm<\mathfrak{S}_{n+r}^\pm$.

Finally, the image $H$ admits a splitting, i.e., $\pi_0 \Diff(C_F; \d C_F) \cong H \ltimes \prod_{i=1}^{r} \pi_0 \Diff(C_{J_i}; \d C_{J_i})$ where the action on $H$ is via the composition $H \to \mathfrak{S}^\pm(\vec{J}) \incl \mathfrak{S}_r^\pm$.
\end{lemma}

We can restate the condition that $\widetilde{\psi}$ acts by $\mathfrak{S}^\pm(\vec{J})$ by saying that $\widetilde{\psi}$ fixes $L_0,\dots, L_n$ and for any $i,k \in \{1,\dots, r\}$, if $\widetilde{\psi}(L_{n+i}) = \pm L_{n+k}$ then $J_i \sim \iota_0^\pm(J_k)$.  
In other words, the relation induced by the action of $\psi$ is at least as fine as the relation $\sim$.  In general, $H \to \mathfrak{S}^\pm(\vec{J})$ need not be injective or surjective.

\begin{proof}
Let $\Diff(C_F, C_L; \d C_F)$ be the space of diffeomorphisms which preserve $C_L$ setwise and fix $\d C_F$ pointwise.  First, we claim that $\Diff(C_F; \d C_F) \simeq \Diff(C_F, C_L; \d C_F)$.
Let $T_{n+1}, \dots, T_{n+r}$ be the tori separating $C_L$ from $C_{J_1}, \dots, C_{J_r}$.  
There is the following inclusion of fibrations, which are not necessarily surjective on path components:
\[
\xymatrix{
\Diff(C_F; \d C_F \cup \d C_L)  \ar@{=}[d] \ar[r] &  \Diff(C_F, C_L; \d C_F) \ar@{^(->}[d] \ar[r] & \ar@{^(->}[d]  
\Diff(T_{n+1} \sqcup \dots \sqcup T_{n+r}) = 
\mathfrak{S}_r \wr \Diff(S^1 \x S^1) \\
\Diff(C_F; \d C_F \cup \d C_L) \ar[r] & \Diff(C_F; \d C_F) \ar[r] & \Emb(\coprod_{i=1}^{r} T_i, C_F) 
}
\]
The right-hand vertical map is given by post-composing by the inclusion.  
The uniqueness up to isotopy of the tori in the JSJ decomposition together with the fixing of $\d C_F$ guarantee that the image of the lower-right hand horizontal map is contained in the components hit by the composition  from $\Diff(C_F, C_L; \d C_F)$.
Hatcher's theorem on spaces of incompressible surfaces in 3-manifolds guarantees that the right-hand vertical inclusion is a homotopy equivalence on each component in the image.  
Thus the middle inclusion is an equivalence, proving the claim.  

The long exact sequence in homotopy for the fibration 
\[
\prod_{i=1}^{r} \Diff(C_{J_i}; \d C_{J_i}) \to \Diff(C_F, C_L; \d C_F) \to \Diff(C_L; T_0 \sqcup \dots \sqcup T_{n})
\]
then yields at $\pi_0$ the three nontrivial terms in the desired exact sequence.  It only remains to check that $\pi_1 \Diff(C_L; T_0 \cup \dots \cup T_{n})$ is trivial.  This in turn can be seen from the fibration
\[
\Diff(C_L; \d C_L) \to \Diff(C_L; T_0 \sqcup \dots \sqcup T_{n}) \to \Diff(T_{n+1} \sqcup \dots \sqcup T_{\ell}).
\]
The fiber is a $K(\pi,0)$ space, so it suffices to see that the map $\pi_1\Diff(T_{n+1} \sqcup \dots \sqcup T_{\ell}) \to \pi_0 \Diff(C_L; \d C_L)$ is injective.  
If $C_L$ is Seifert-fibered, the injectivity is clear from the descriptions of generators of $\pi_0 \Diff(C_L; \d C_L)$ in the proof of Proposition \ref{DiffSeifertFramed}.  
If $C_L$ is hyperbolic, this injectivity is part of the sequence \eqref{HypLinkSES} in 
the Hatcher--McCullough result, Proposition \ref{HMProp}.  
Hence the exact sequence \eqref{Eq:SpliceSES} is established.

We now show that elements in the image $H$ satisfy the conditions in the lemma statement.  
Let $[\psi] \in H < \pi_0\Diff(C_L; T_0 \cup \dots \cup T_n)$ be the image of $[\phi] \in \pi_0\Diff(C_F, C_L; \d C_F)$.  
Since $\phi$ extends to an orientation-preserving diffeomorphism of the pair $(S^3,L)$ which fixes $L_0, \dots, L_n$, so must $\psi$.  Call this mutual extension $\widetilde{\psi}$.
Then $\widetilde{\psi}$ yields an element $\sigma \in \pi_0\Diff(L_{n+1} \sqcup \dots \sqcup L_{n+r}) \cong \mathfrak{S}_r^\pm$, where the action of $-1$ on $L_{n+i}$ is given by reversing its orientation.
We check that $\sigma \in \mathfrak{S}^\pm(\vec{J})$. 
Indeed, the restriction of $\widetilde{\psi}$ to the $i$-th component $(S^3, L_{n+i}) \overset{\cong}{\to} (S^3, \pm L_{n+k})$ gives rise (via $(S^3, T_{n+i}) \overset{\cong}{\to} (S^3, \pm T_{n+k})$ and $(S^3, C_{J_i}) \overset{\cong}{\to} (S^3, \pm C_{J_k})$) to a diffeomorphism of pairs $(S^3, J_i) \to (S^3, \iota_0^{\pm}(J_k))$.  The latter gives 
an isotopy between $J_i$ and $\iota_0^{\pm}(J_k)$, since $\Diff^+(S^3)$ is connected.  
Moreover, if $J_i$ has more than one component, then at least one component $T$ of $\d C_{J_i}$ lies in $\d C_F$.  
So $\phi$ must preserve $T$ and hence $\widetilde{\psi}$ preserves $L_{n+i}$ (though it may act by $-1$ on $L_{n+i}$).  
In other words, only the $J_i$ that are knots can be permuted. 
Thus the action of $\sigma$ identifies only $\iota_0^\pm(J_i)$ in the same equivalence class under relation $\sim$, so $\sigma \in \mathfrak{S}^\pm(\vec{J})$.

Finally, we check that if $\psi \in \Diff(C_L; T_0 \cup \dots \cup T_{n})$ satisfies the conditions in the lemma statement, then $[\psi]\in H$.
So suppose $\psi$ extends to a diffeomorphism of the pair $(S^3, L)$ which fixes $L_0, \dots, L_n$ and whose action on $L_{n+1}, \dots, L_{n+r}$ is given by $\sigma \in \mathfrak{S}^\pm(\vec{J})$.  
Then $\psi$ gives rise to a diffeomorphism in $\Diff(C_F; \d C_F)$ by choosing for each $i$ a diffeomorphism $C_{J_i} \to \pm C_{J_k}$, where $\sigma(i)=\pm k$.
This establishes both the desired description of $H$ and the splitting as a semi-direct product.  
\end{proof}

As in Sections \ref{S:Seifert} and \ref{S:Hyperbolic}, we will give our results in terms of homotopy types of spaces, rather than fundamental groups.  
So we will apply the classifying space functor $B(-)$ to exact sequences involving the discrete groups $G = \pi_0\Diff(C_F; \d C_F)$ to understand the corresponding $K(G,1)$ spaces $\widetilde{\mathcal{L}}_F/SO_4$.  
For splices into keychain links and hyperbolic links in Sections \ref{S:ConnectSums} and \ref{S:HypSplices}, it will help to understand how $B(-)$ interacts with a semi-direct product.  Let $G$ be a discrete group.  Suppose $G = K \rtimes H$.  Then 
\begin{equation}
\label{Bsemidirect}
BG \simeq EH \x_H (E(K\rtimes H)/K)
 \simeq EH \x_H BK.  
\end{equation}
The bundle $EH \to BH$ is the universal cover of $BH$, and $H$ acts on $EH$ by Deck transformations.  The action of $H$ on $BK$ is induced by the map $H \to \Aut(K)$ from the semi-direct product.  Suppose the latter map factors through a surjection $H \twoheadrightarrow H'' \to \Aut(K)$ with $H' := \ker(H \to H'')$.  Then the quotient $EH \x BK \to EH \x_H BK$ factors through $EH/H' \x BK \simeq BH' \x BK$, and we can simplify the right-hand side of \eqref{Bsemidirect} as 
\[
BG \simeq BH' \x_{H''} BK
\]
where the action of $H''$ on the left-hand factor is by Deck transformations of the covering space $BH' \to BH$, and the action on the right-hand factor is induced by automorphisms of $K$.  We record the result:

\begin{lemma}
\label{SemidirectLemma}
Suppose that $G = K \rtimes H$, and
let $H'$ and $H''$ be the kernel and  image of the map $H \to \Aut(K)$.
Then $BG \simeq BH' \x_{H''} BK$.
\qed
\end{lemma}

For example, if $G$ is the nontrivial semi-direct product $\Z\rtimes \Z$, then 
\[
BG \simeq \R \x_\Z S^1 
\]
where $\Z$ acts by translation on $\R$ and complex conjugation on $S^1$.  Since $\Z \to \Aut(\Z)$ has image isomorphic to $\Z/2$ and kernel $2 \Z \cong \Z$, the above expression simplifies to
\[
BG \simeq S^1 \x_{\Z/2} S^1
\]
where the actions are by rotation by $\pi$ and complex conjugation.  Either space is the Klein bottle.

\subsection{Generalizations of cables}
\label{S:Cables}
We now apply Lemma \ref{SpliceSES} to the Seifert-fibered case.  In this subsection, we first consider the case where $L=(L_0, \dots, L_{n+r})$ is the Seifert link $S_{p,q} = T_{p,q} \cup C_2$, with any ordering of the components (Proposition \ref{P:Cable}).  
This can produce links corresponding to knots in a solid torus.
We then consider the cases where $L=(L_0, \dots, L_{n+r})$ is $R_{p,q}= T_{p,q} \cup C_1 \cup C_2$ or $T_{p,q}$, again with any ordering of the components (Proposition \ref{P:SpliceIntoRpq}).  The case $L=R_{p,q}$ can produce knots in a thickened torus.
The operation $(\varnothing, \dots, \varnothing, J_1, \dots, J_r) \bowtie L$ will not always produce a link in $S^3$ (though it produces a link in some 3-manifold).  For example, if the $J_i$ are all knots, it does only if $r=1$ and $L_{n+r}$ is an unknot.  We assume that $F \in \tL$ is the result of such an operation, without considering the conditions in which this operation produces a link in $S^3$.
We leave keychain links for the next subsection.

\begin{proposition}
\label{P:Cable}
Suppose that $\gcd(p,q)=n+r$ so that $S_{p,q}$ has $n+r+1$ components, and
$F = (\varnothing, \dots, \varnothing, J_1, \dots, J_r) \bowtie S_{p,q}$.
Then  
\[
\boxed{
\tL_F/SO_4 \simeq \Conf(n+r+1,\R^2) \x (S^1)^{2n} \times \prod_{i=1}^r \tL_{J_i}/SO_4
}
\]
where the factor $(S^1)^{2n}$ corresponds to meridional and longitudinal rotations of $L_1, \dots, L_n$.
\end{proposition}

The case $r=0$ is already covered by Proposition \ref{SeifertFramedLinks}, and the only other case relevant to $\T$ or $\V$ is 
when $n=0$ and $r=1$.  That is, $\gcd(p,q)=1$, $S_{p,q}$ is a 2-component $(p,q)$-Seifert link, and $F= J \bowtie S_{p,q}$.  In this case,
\[
\boxed{
\tL_F/SO_4 \simeq S^1 \x \tL_J/SO_4
}
\]
which when $J$ is a knot recovers the result \cite[Theorem 2.3]{BudneyTop}, since $\tL_J/SO_4 \simeq \K_{\underline{j}}$ and $\tL_F/SO_4 \simeq \K_{\underline{f}}$.

\begin{proof}
Write $L=(L_0,\dots, L_{n+r})$ for $S_{p,q}$.
Lemma \ref{SpliceSES} gives us the exact sequence 
\begin{equation}
\label{CableSES}
\{e\} \to \prod_{i=1}^r \pi_0 \Diff(C_{J_i}; \d C_{J_i}) \to \pi_0 \Diff(C_F; \d C_F) \to \pi_0 \Diff(C_L; T_0 \cup \dots \cup T_{n} )
\end{equation}
and the semi-direct product $\pi_0 \Diff(C_F; \d C_F)\cong H \ltimes \prod_{i=1}^r \pi_0 \Diff(C_{J_i}; \d C_{J_i})$ where $H$ is the image of the right-hand map above.

We first determine  
$\pi_0 \Diff(C_L; T_0 \cup \dots \cup T_{n} )$.  
The manifold $C_L$ is Seifert-fibered over the surface $P_{n+r}=D^2 -(D_1 \sqcup \dots \sqcup D_{n+r})$ with one marked point $x$ corresponding to the singular fiber.  
The analysis is similar to that in the proof of Proposition \ref{DiffSeifertFramed}.
By work of Waldhausen \cite[pp.~85-86]{Waldhausen:Annals1968}, $\Diff(C_L)$ is equivalent to the space of diffeomorphisms which respect the Seifert fibering.  
Since an element of $\Diff(C_L; T_0 \cup \dots \cup T_{n})$ fixes the remaining boundary components $T_{n+1}, \dots, T_{n+r}$ setwise, this implies that $\Diff(C_L; T_0 \cup \dots \cup T_{n} )$ fibers over the subspace of $\Diff(P_{n}; \d P_{n} \cup \{x\})$ which setwise fixes the centers $\{y_1, \dots, y_r\}$ of the disks $D_{n+1}, \dots, D_{n+r}$. 
At the level of $\pi_0$, we then have the exact sequence
\begin{equation}
\label{DiffCLrelT0...Tn}
\{e\} \to \Z^n \to \pi_0 \Diff(C_L; T_0 \cup \dots \cup T_{n} ) \to  \Z^n \x \B_{1,\dots, 1, r} \to \{e\}
\end{equation}
 where $\B_{1,\dots, 1,r} < \B_{n+r}$ is the preimage of the Young subgroup $(\mathfrak{S}_1)^n \x \mathfrak{S}_r < \mathfrak{S}_{n+r}$ under the canonical map $\B_{n+r} \to \mathfrak{S}_{n+r}$.  That is, the image in \eqref{DiffCLrelT0...Tn} consists of braids on $n+r$ strands whose induced permutations fix the first $n$ points and which are equipped with framings on the first $n$ strands.
The kernel of \eqref{DiffCLrelT0...Tn} is analogous to that in \eqref{JohannsonSES}: it is identified with $H_1(P_{n})$ and generated by Dehn twists along the fibers $S^1$ of $T_1, \dots, T_n$.  
As in the proof of Proposition \ref{DiffSeifertFramed}, we obtain a splitting of this fibration.
Thus 
$\pi_0\Diff(C_L; T_0 \cup \dots \cup T_{n}) \cong \Z^{2n} \x \B_{1,\dots, 1, r}$.

We now consider the image $H$ of the last map in \eqref{CableSES}.  Because any pair of components of $S_{p,q}$ link nontrivially and $F$ is a link in $S^3$, at most one of the $J_i$ can be a knot.  Then by Definition \ref{EquivRelDef} and Lemma \ref{SpliceSES}, a map $[\psi] \in H < \pi_0\Diff(C_L; T_0 \cup \dots \cup T_{n} )$ must induce the identity permutation of $L_{n+1}, \dots, L_{n+r}$.  It must also preserve their orientations because of the nontrivial linking, so $\psi$ induces the identity in $\mathfrak{S}_r^\pm$.  
Thus $H=\Z^{2n} \x \PB_{n+r} < \Z^{2} \x \B_{1,\dots,1, r}$.  Moreover, the action of $H$ on the kernel in \eqref{CableSES} is by the identity in $\mathfrak{S}_r^\pm$.  
Hence the associated semi-direct product is a direct product.
Applying the classifying space functor $B(-)$ and Proposition \ref{FramedLinksKpi1} completes the proof.
\end{proof}

In the next two examples, $\gcd(p,q)=1$ and $(L_0,L_1)=S_{p,q}$, and we take $L_0=T_{p,q}$.

\begin{example}[Cables of knots]
\label{Ex:CableOfKnot}
Let $F:=J \bowtie S_{p,q}$ be a (framed) $(p,q)$-cable of a (framed) knot $J$.
We consider the classes in $\pi_1(\L_f/SO_4)$ of several canonical loops.  
Consider the following solid tori containing $J$.
Let $U$ be the standard unknotted one (or a slight thickening thereof) 
and let $W$ be a thinner one knotted into the shape of $J$.  
The loops $\mu$ and $\lambda$ of meridional and longitudinal rotations of $U$ each generate $\pi_1(SO_4)$.  
Let $\mu'$ and $\lambda'$ be the loops of meridional and longitudinal rotations of $W$.  
Then $\lambda_0=q \mu' + p \lambda'$.

If $J$ is the figure-eight knot, then $\pi_1(\tL_F/SO_4) \cong \Z\langle \mu_0, \mu', \frac{1}{2}\lambda\rangle$, and $\pi_1(\L_f/SO_4)\cong \Z\langle \mu', \frac{1}{2}\lambda \rangle$, where $\frac{1}{2}\lambda$ can be visualized by rotating the last picture in Figure \ref{F:CableOfKnot} by $180^\circ$ in the plane and then rotating $W$ longitudinally by $180^\circ$.  
If $J$ is a knot $T_{r,s}$, $\pi_1(\L_f/SO_4) \cong \Z\langle \mu'\rangle$ because $\lambda = s \mu + r \lambda$.  

In general, $\mu'$ is called ``rotation of the cabling parameter'' in \cite[Section 5]{BudneyTop}.
Modulo $SO_4$, we can also view $\mu'$ as the Gramain loop $g_J$.
For any nontrivial knot $J$, $\pi_1(\L_f/SO_4) \cong \pi_1(\tL_J/SO_4)$.  
\end{example}

\begin{figure}[h!]
\includegraphics[scale=0.2]{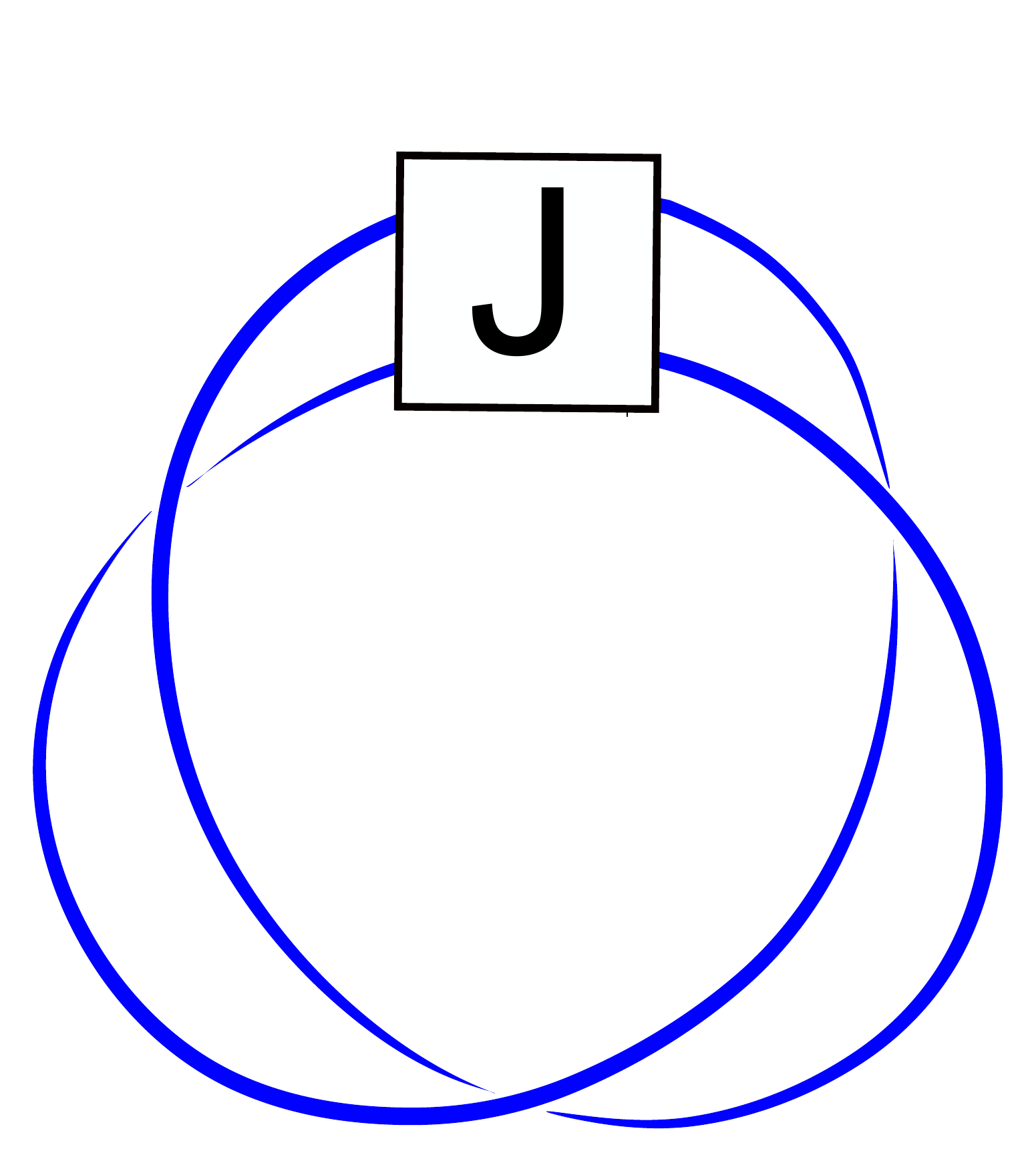} 
\qquad 
\includegraphics[scale=0.2]{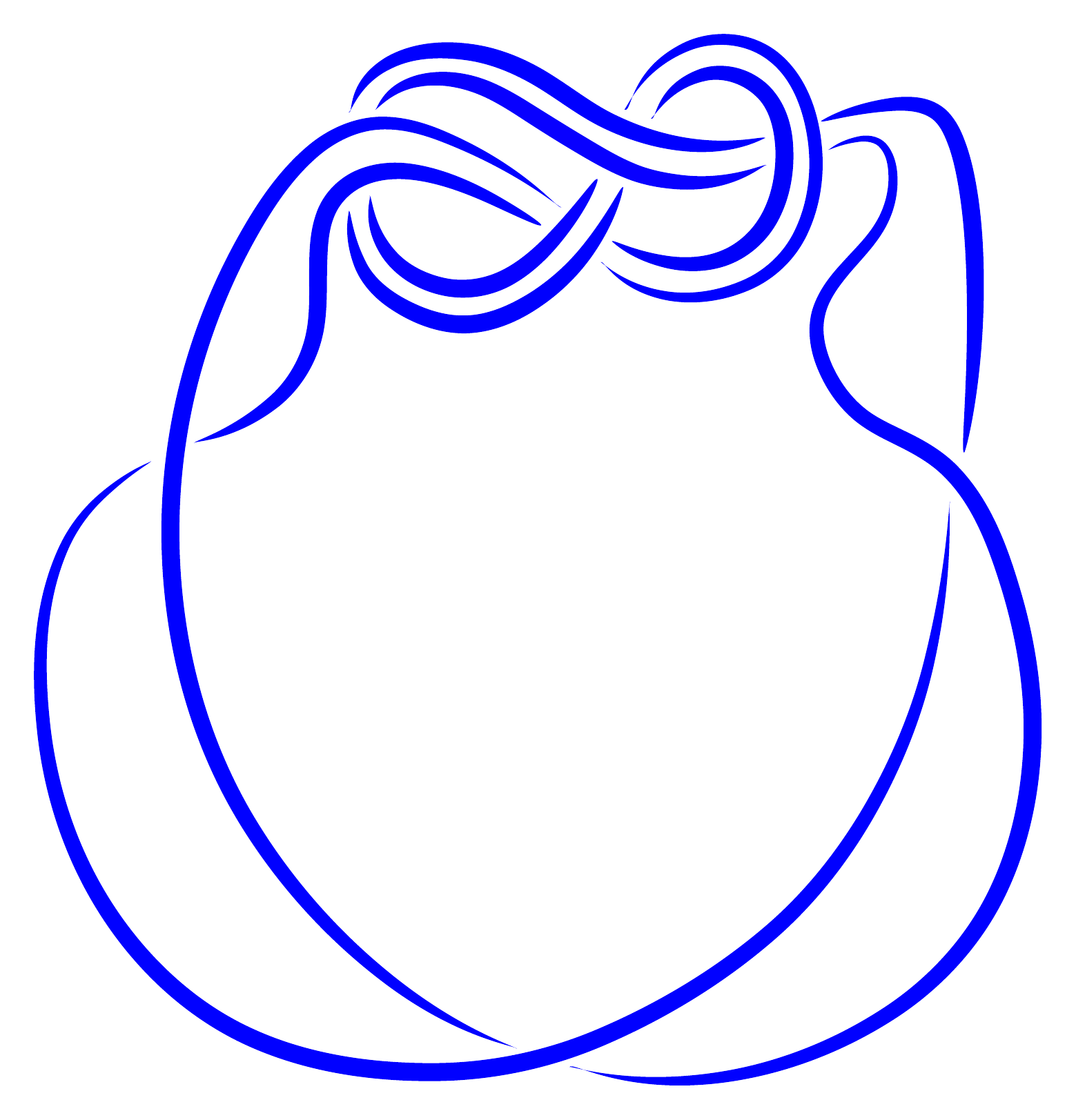}
\qquad
\includegraphics[scale=0.2]{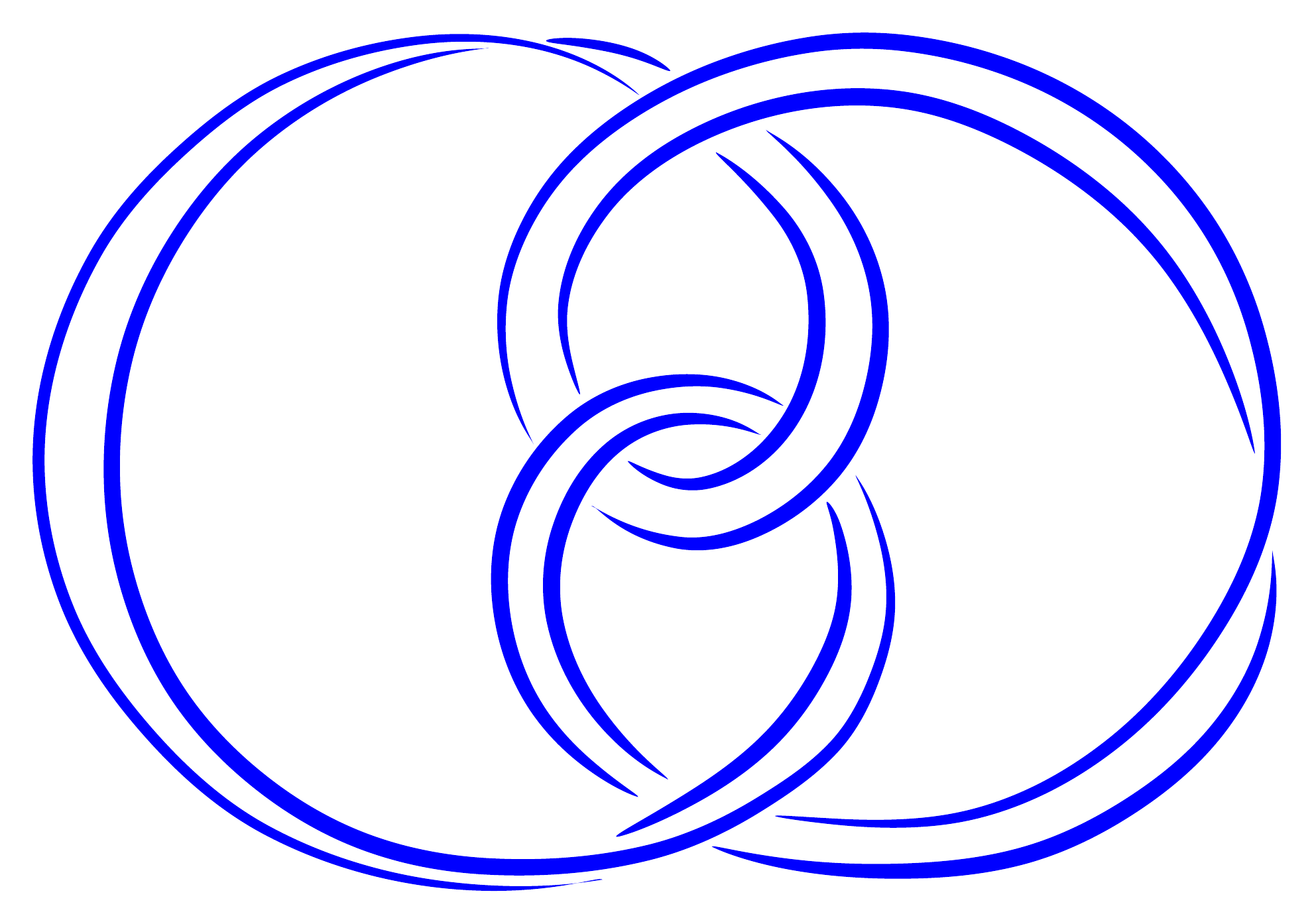}
\qquad 
\begin{forest}
for tree={grow=north, l sep=25pt}
[ [$S_{p,q}$, draw, circle, edge label={node[midway, right] {$\lambda_0$ (and $\mu_0$)}}
[$J$, draw, circle, edge label={node[midway,right] {$\lambda, \mu$}}]]
]
\end{forest}
\caption{
Left: schematic of $J \bowtie S_{2,3}$ for a knot $J$ using the long knot associated to $J$.  Center-left and center-right: the $(2,3)$-cable of the figure-eight knot $J$ (with $0$ framing, which agrees with the blackboard framings of the planar projections shown above, hence no extra twists appear).  
Right: (schematic of) $\mathbb{G}_F$.
Here 
$\pi_1(\L_f/SO_4) \cong \Z \langle \mu, \, \frac{1}{2}\lambda \rangle$.
See Example \ref{Ex:CableOfKnot}.
}
\label{F:CableOfKnot}
\end{figure}

\begin{example}[Cable of Whitehead link]
\label{Ex:CableWh}
Let $F:=\Wh \bowtie S_{p,q}$, the ``$(p,q)$-cable'' of the Whitehead link $\mathrm{Wh}$.  
Then $\pi_1(\tL_F/SO_4) \simeq \Z^5$, and $\pi_1(\L_f/SO_4) \simeq \Z^3$.  
To specify generators, take an embedding of $F$ in which the $(p,q)$-torus knot is in its standard symmetric position on a torus $T$.  
See the first picture in Figure \ref{F:SeifertWhitehead}.
Let $U$ be the solid torus in $S^3$ bounded by $T$ which does not contain the unknotted component.  
The loops of rotations $\mu'$ and $\lambda'$ of $U$ provide two generators.  
Then $\pi_1(\L_f/SO_4) \cong \Z\langle \lambda_1, \mu', \lambda'\rangle$, where $\lambda_1$ reparametrizes the unknotted component.
For $\T_f$, we have $\pi_1(\T_f) \cong \Z\langle \mu_1, \lambda_1,  \mu', \lambda'\rangle$, where $\mu_0$ can be visualized using the second picture in Figure \ref{F:SeifertWhitehead}.  
\end{example}

\begin{figure}[h!]
(a) \includegraphics[scale=0.25]{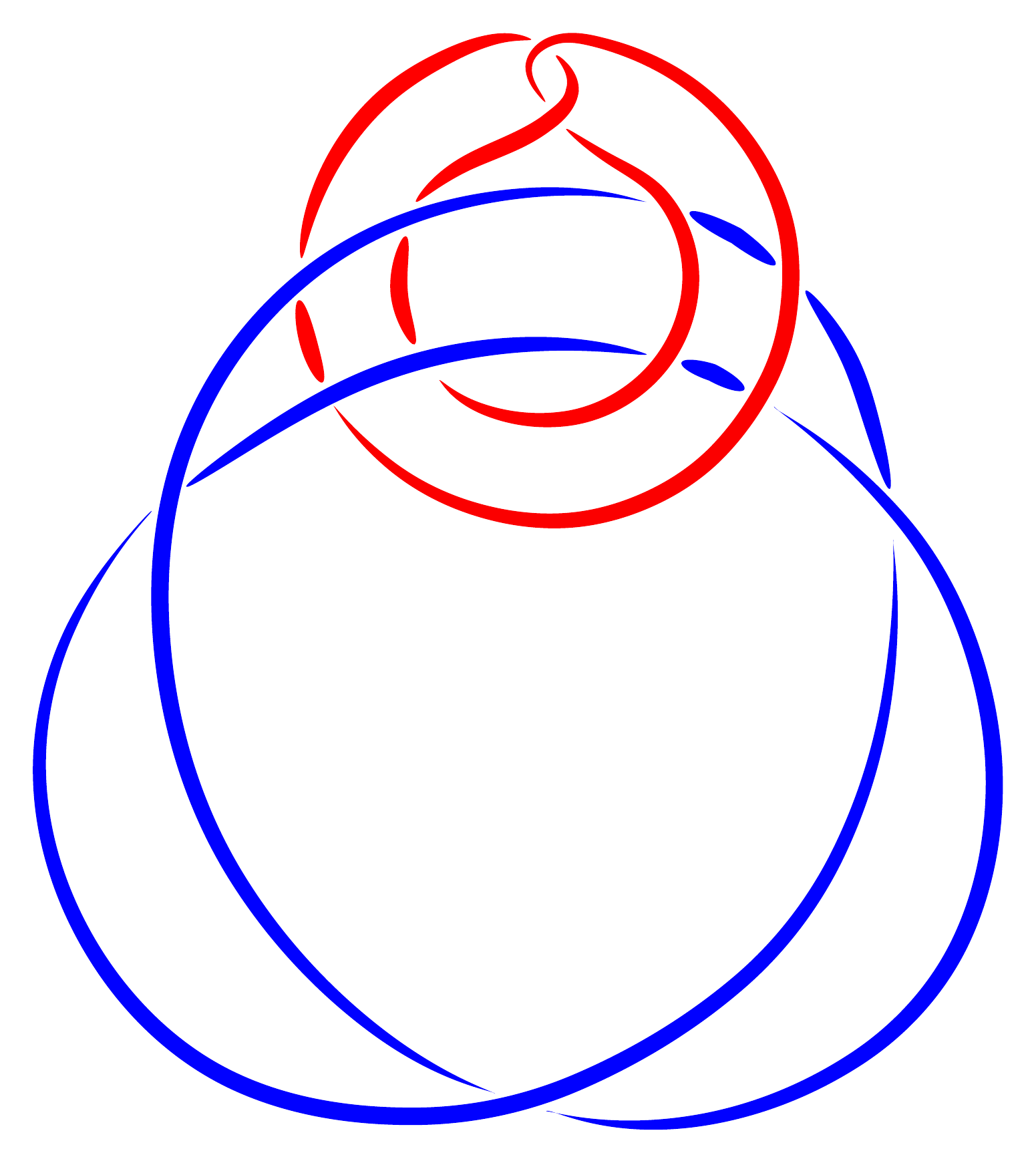} \quad
(b) \raisebox{-1pc}{\includegraphics[scale=0.25]{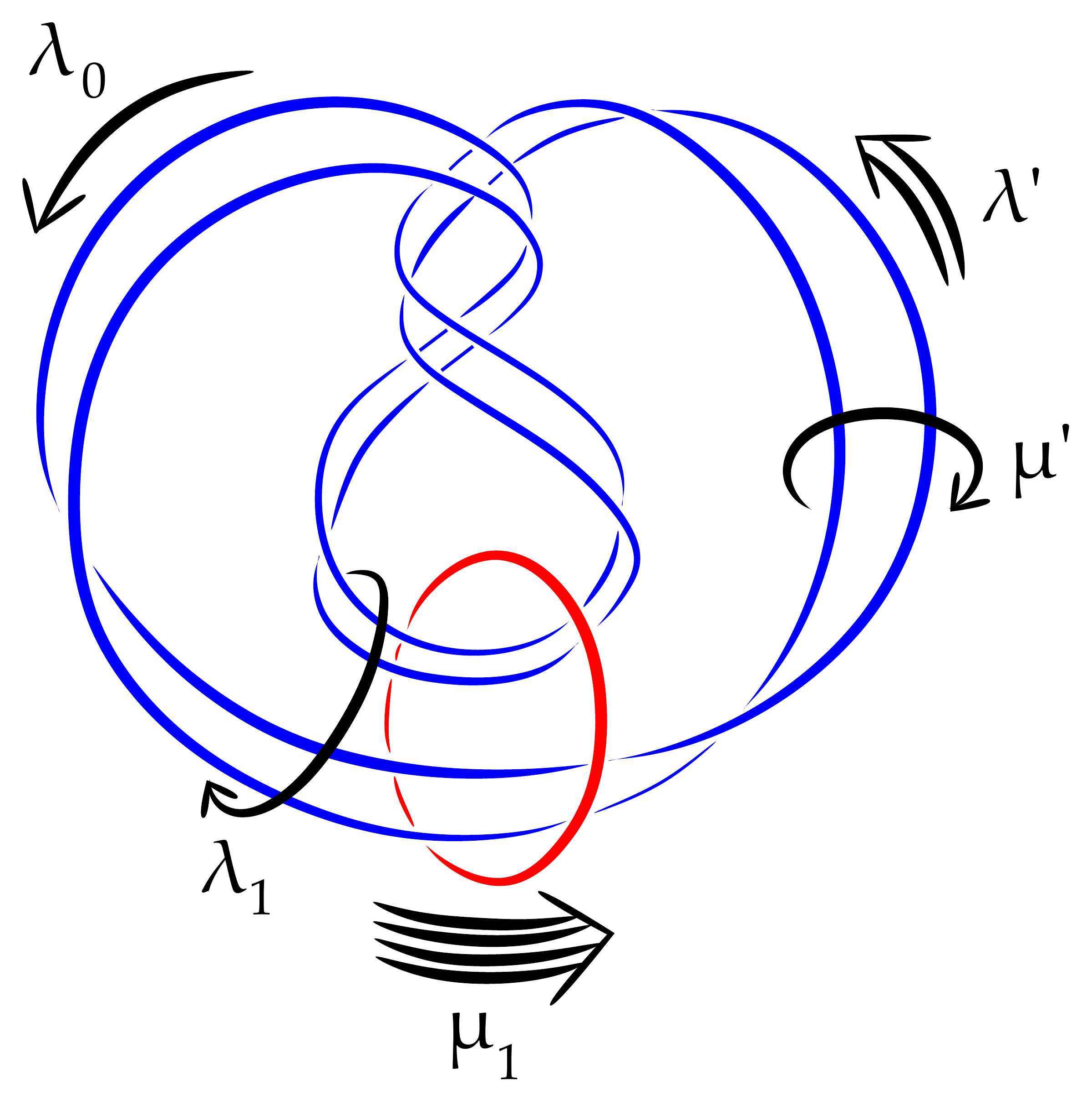}}
\quad 
\begin{forest}
for tree={l sep=25pt}
[ [$\mathrm{Wh}$, circle, draw, edge label={node[midway,right] {$\lambda_1$ (and $\mu_1$)}} 
[$S_{p,q}$, circle, draw, edge label={node[midway,right] {$\lambda', \mu'$}}
[ , edge label={node[midway,right] {$\lambda_0$ (and $\mu_0$)}}]
]]]
\end{forest}
\caption{Two planar projections of the link $F$ in Example \ref{Ex:CableWh} for $(p,q)=(2,3)$, together with $\mathbb{G}_F$.
Here $\pi_1(\L_f/SO_4) \cong \Z\langle \lambda_1,\, \mu',\, \lambda'\rangle$ and 
$\pi_1(\T_f) \cong \Z\langle\mu_1, \, \lambda_1, \, \mu', \, \lambda'\rangle$. }
\label{F:SeifertWhitehead}
\end{figure}

The proof of Proposition \ref{P:Cable} is easily adapted to prove the next result.  For $R_{p,q}$, we simply remove the neighborhood of a singular fiber, which in the base surface of the Seifert fibering corresponds to replacing a marked point by a missing disk.  
For $T_{p,q}$, we instead have an extra singular fiber.
As in Proposition \ref{P:Cable}, the labeling of the components $L_0, \dots, L_{n+r}$ is not needed to determine $\tL_F/SO_4$.

\begin{proposition}
\label{P:SpliceIntoRpq} \ 
\begin{itemize}
\item[(a)]
Suppose that $\gcd(p,q)=n+r-1$ so that $R_{p,q}$ has $n+r+1$ components.  
Let $F = (\varnothing, \dots, \varnothing, J_1, \dots, J_r) \bowtie R_{p,q}$.
Then  
\begin{flalign*}
&& \boxed{
\tL_F/SO_4 \simeq \Conf(n+r,\R^2) \x (S^1)^{2n} \times \prod_{i=1}^r \tL_{J_i}/SO_4.} && 
\end{flalign*}
\item[(b)]
Suppose that $\gcd(p,q)=n+r+1$ so that $T_{p,q}$ has $n+r+1$ components.  
Let $F = (\varnothing, \dots, \varnothing, J_1, \dots, J_r) \bowtie T_{p,q}$.
Then  
\begin{flalign*}
&& \boxed{
\tL_F/SO_4 \simeq \Conf(n+r+2,\R^2) \x (S^1)^{2n} \times \prod_{i=1}^r \tL_{J_i}/SO_4.} && 
\end{flalign*}
\end{itemize}
In either case, the factor of $(S^1)^{2n}$ corresponds to meridional and longitudinal rotations of $L_1, \dots, L_n$.
\qed
\end{proposition}

This operation does not relate to $\T$, since $R_{p,q}$ has at least 3 components and, to obtain a link in $S^3$, a knot can be spliced along only one of them (since any pair of components links nontrivially).  One can however obtain links $F$ corresponding to $f\in \V$ when $\gcd(p,q)=1$, $n=1$, and $r=1$.
Let  $L=(L_0,L_1,L_2)=R_{p,q}$, where we take $L_2$ to be the knotted component.
If $F=(\varnothing, J) \bowtie R_{p,q}$, then 
\[
\boxed{
\tL_F/SO_4 \simeq (S^1)^3 \x \tL_J/SO_4.
}
\]
Such $F$ gives rise to a knot $f\in \V$ if $J$ is a 2-component KGL.

\begin{example}
\label{Ex:WhSpliceRpq}
Suppose $\gcd(p,q)=1$, and let $L=(L_0,L_1,L_2)$ be the link $R_{p,q}$ where $L_2$ is the knotted component.  Let $F=(\varnothing, \Wh) \bowtie L$ where $\Wh$ is the Whitehead link.  
Proposition \ref{P:SpliceIntoRpq} implies that $\pi_1(\tL_F/SO_4) \cong \Z \x \Z^2 \x \Z^4 \cong \Z^7$.
The link $F$ corresponds to a knot $f$ in $S^1 \x S^1 \x I$.  The knot $f$ is a Whitehead double of $T_{p,q}$ embedded in a neighborhood of $T_{p,q}$ in its standard position (see Figure \ref{F:3-4-SeifertLink}), which we then view as lying in a small neighborhood of the torus 
%$T:=U \cap U'$
 in the standard Heegaard decomposition of $S^3$.  The group $\pi_1(\V_f) \cong \Z\langle \mu, \lambda, \lambda_0, \mu' \rangle$ is generated by not only rotations of this torus and reparametrization, but also a meridional rotation of the neighborhood of $T_{p,q}$.
\end{example}

\subsection{Generalizations of connected sums}
\label{S:ConnectSums}

We now consider a generalization to links of a connected sum of knots.  
In the next proposition, we label the components of the keychain link $KC_{n+r}$ in any order.  The resulting link $F$ will differ according as the special component lies among the first $n+1$ or last $r$ components.  
Nonetheless, the homotopy type of $\tL_F/SO_4$ is the same in the two cases, much like the ordering of the components of $S_{p,q}$ and $R_{p,q}$ in Propositions \ref{P:Cable} and \ref{P:SpliceIntoRpq} was inconsequential.  

\begin{proposition}
\label{P:ConnectSum}
Let  $L=(L_0, \dots, L_{n+r})$ be the (framed) $(n+r+1)$-component keychain link $KC_{n+r}$.
Let $F=(\varnothing, \dots, \varnothing, J_1, \dots, J_r)\bowtie KC_{n+r}$.
Then 
\[
\boxed{
\tL_F/SO_4 \simeq   
(S^1)^{2n} \x 
\left( 
\Conf(n + r, \R^2) 
\x_{\mathfrak{S}(\vec{J})} \prod_{i=1}^r \tL_{J_i} / SO_4
\right)
}
\] 
where $\mathfrak{S}(\vec{J})$ is the Young subgroup of $\mathfrak{S}_r$ for the partition given by  $i \sim k$ iff $J_i$ and $J_k$ are isotopic knots, 
and where the factor $(S^1)^{2n}$ corresponds to meridional and longitudinal rotations of the components $L_1, \dots, L_n$.  
\end{proposition}

Note that in the partition $\sim$ above, links $J_i$ with multiple components lie in singleton sets.

\begin{proof}
With $L=KC_{n+r}$, $F$, and the $J_i$ as above, let $T_0, T_{1}, \dots, T_n$ be the tori which are in both $\d C_L$ and $\d C_F$.  
Let $G:=\pi_0 \Diff(C_F; \d C_F)$ and $K:=\prod_{i=1}^r K_i :=\prod_{i=1}^r \pi_0 \Diff(C_{J_i}; \d C_{J_i})$. 
Lemma \ref{SpliceSES} gives a map $\pi_0 \Diff(C_F; \d C_F) \to \pi_0 \Diff(C_L; T_0, T_{1}, \dots, T_{n})$, and if $H$ is the image of this map, then $G\cong H \ltimes K$.

We first describe $H$.  
As in the proof of Proposition \ref{P:Cable}, let $\B_{1,\dots, 1,r} < \B_{n+r}$ be the preimage of $(\mathfrak{S}_1)^n \x \mathfrak{S}_r < \mathfrak{S}_{n+r}$ under the map $\B_{n+r} \to \mathfrak{S}_{n+r}$.  By an argument analogous to the one given for a Seifert link there, $ \pi_0 \Diff(C_L; T_0, T_{1}, \dots, T_{n}) \cong (\Z^2)^{n} \x \B_{1,\dots,1,r}$.
A diffeomorphism in this group acts on $L$ by $\mathfrak{S}_r < \mathfrak{S}_r^\pm$ because it fixes some component and preserves the $+1$ linking number of the special component with the others.
Moreover, it acts by projection to $\B_{1,\dots,1,r}$ followed by the canonical map to $\mathfrak{S}_r$.
Let $\B(\vec{J}) < \B_{1,\dots,1,r}$ be the preimage of $\mathfrak{S}(\vec{J})$ under the map $\B_{1,\dots, 1,r} \to \mathfrak{S}_{r}$.
(Alternatively, it is the preimage of $\mathfrak{S}^\pm(\vec{J})$ under the composition $\B_{1,\dots, 1,r} \to \mathfrak{S}_{r} \incl \mathfrak{S}_{r}^\pm$, with $\mathfrak{S}^\pm(\vec{J})$ is as in Definition \ref{EquivRelDef}.)
Then by Lemma \ref{SpliceSES}, $H = \Z^{2n} \x \B(\vec{J})$.  

We now consider the action of $H$ on $K$ in the semi-direct product.
The factor $\Z^{2n}$ acts trivially on $K$, so $G\cong \Z^{2n} \x (\B(\vec{J}) \ltimes K)$.
The action of $\B(\vec{J})$ on $K=\prod_{i=1}^r K_i$ is given by the composition $\B(\vec{J}) \twoheadrightarrow \mathfrak{S}(\vec{J}) \incl \mathfrak{S}_r$.
The kernel of the surjection $\B(\vec{J}) \twoheadrightarrow \mathfrak{S}(\vec{J})$ is $\PB_{n+r}=\B_{1,\dots, 1, 1, \dots, 1}$.
By Lemma \ref{SemidirectLemma} with $H=\B(\vec{J})$, $H''= \mathfrak{S}(\vec{J})$, and $H'=\PB_{n+r}$, we have
\begin{align*}
\tL_F/SO_4 \simeq BG & \simeq B(\Z^{2n}) \x \left( B(\PB_{n+r})  \x_{\mathfrak{S}(\vec{J})} \prod_{i=1}^r BK_i \right) 
\end{align*}
which yields the claimed boxed formula.
\end{proof}

If $F$ is a link giving rise to a knot in $\T$ or $\V$, then the special component of $KC_{n+r}$ must be one of $L_0, \dots, L_{n+r}$ (without loss of generality $L_0$), for otherwise $F$ would have no unknotted components.    
For arbitrary links one could splice into $KC_{n+r}$ with different orderings of the components to obtain different results.
For example, $F:=(\varnothing, J) \bowtie KC_2$ with $L_0$ as the special component is a link as shown in Figure \ref{F:TrefoilSumEmpty} or \ref{F:Fig8SumEmpty}.  But $F':=(\varnothing, J) \bowtie KC_2$ with $L_2$ as the special component
is a pair of parallel knotted and linked components.
However, $\tL_F/SO_4 \simeq \tL_{F'}/SO_4$, as guaranteed by Proposition \ref{P:ConnectSum}.

We now consider some special cases.  In all the examples below, $L_0$ is the special component of $KC_{n+r}$.
Recall that for the framed 3-component keychain link $KC_2$,  $\pi_1(\tL_{KC_2}/SO_4) \cong \Z^5$.

\begin{example}[Connected sum of knots]
\label{Ex:ConnectSumOfTwo}
If $n=0$ and $r=2$, i.e.~$F:=(J_1,J_2)\bowtie KC_2$, and both $J_i$ are knots, we recover Budney's result for long knots, as ensured by Proposition \ref{FramedKnotsInS3VsLongKnots}.  For distinct $J_i$, 
\[
\pi_1(\tL_F / SO_4) \cong \PB_2 \x \pi_1(\tL_{J_1}/SO_4 \x \tL_{J_2}/SO_4)
\]
while if the $J_i$ are isotopic to the same knot $J$, 
\[
\pi_1(\tL_F / SO_4) \cong \B_2 \ltimes (\pi_1(\tL_J/SO_4))^2.
\]
The action of $\B_2$ on the right-hand factors is given by swapping them.
The description of a generator $p$ of $\PB_2$ most generalizable to $n>2$ is a motion that first pulls $J_1$ through $J_2$ counter-clockwise (thus swapping their positions) and then pulls $J_2$ through $J_1$ counter-clockwise.  This can be visualized using Figure \ref{F:ConnectSumOfTwo} (a).  
Similarly, $\B_2$ is generated by just the first half of that motion.  
Passing to the unframed knot $f$ in either case means taking 
the quotient in $\pi_1$ by $\mu_0$, i.e.~the diagonal in the subgroup $\Z^2 < \pi_1(\tL_{J_1}/SO_4) \x \pi_1(\tL_{J_2}/SO_4)$ spanned by the Gramain loops of the factors.

If the $J_i$ are distinct torus knots, as in Figure \ref{F:ConnectSumOfTwo}, then $\L_f/SO_4 \simeq S^1 \x S^1$ and $\pi_1(\L_f/SO_4) \cong \PB_2 \x \Z$.
If $J_1=J_2 = T_{p,q}$, as in Figure \ref{F:SumTrefoils} (a) but taken as a knot in $S^3$, 
then $\L_f/SO_4 \simeq S^1 \x_{\Z/2} S^1$, the Klein bottle, with $\pi_1(\L_f/SO_4) \cong \B_2 \ltimes \Z$.
For the generator of $\Z$ in either case, we can take the Gramain loop of either $J_1$ or $J_2$.  

Consider a loop $b_1$ that pulls $J_2$ a full circle counter-clockwise by passing through $J_1$, and a loop $b_2$ that moves $J_1$ a full circle counter-clockwise by passing through $J_2$.  
Each $b_i$ is equivalent to the product of the Fox--Hatcher loop of $J_i$ and possibly a loop of reparametrizations $\lambda_0^{\pm 1}$.
If both $J_i$ are torus knots, then $b_1$, $b_2$, $p$, and the loop $\lambda_0$ of reparametrizations are all equivalent,
since the Fox--Hatcher loop of a torus knots is trivial modulo $SO_4$.
See Figure \ref{F:ConnectSumOfTwo} (b).
So any of these four loops represent a generator of $\PB_2$.
(On the other hand, in the presence of a second unknotted component, each $b_i$ is a generator in a braid group; see Example \ref{Ex:MoreGeneralConnectSum}.)
If both $J_i$ are torus knots, then $p=b_i +\lambda_0$, so we can also take $\lambda_0$ to generate the factor $\PB_2$.
\end{example}

\begin{figure}[h!]
(a) \includegraphics[scale=0.23]{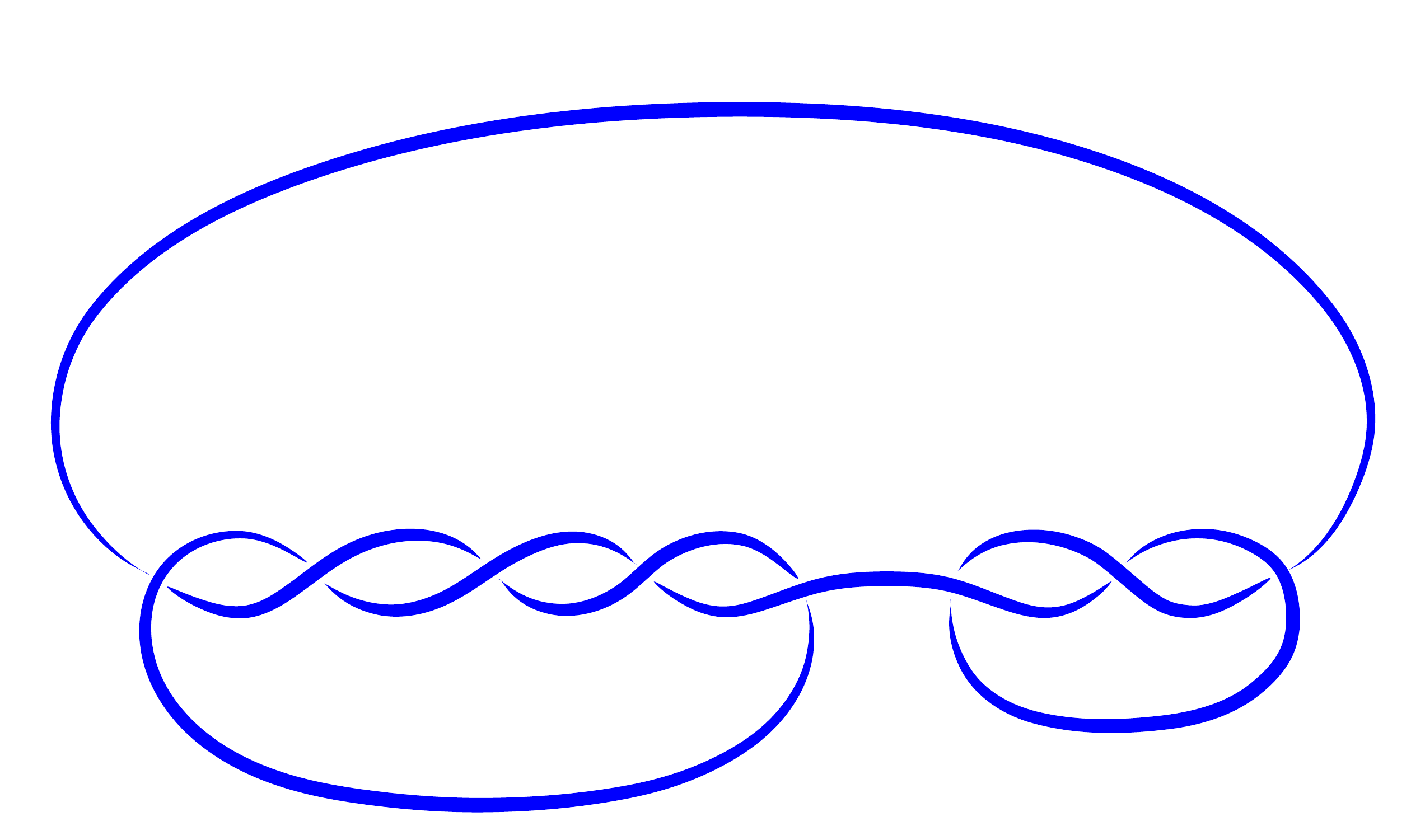}\qquad
(b) \includegraphics[scale=0.23]{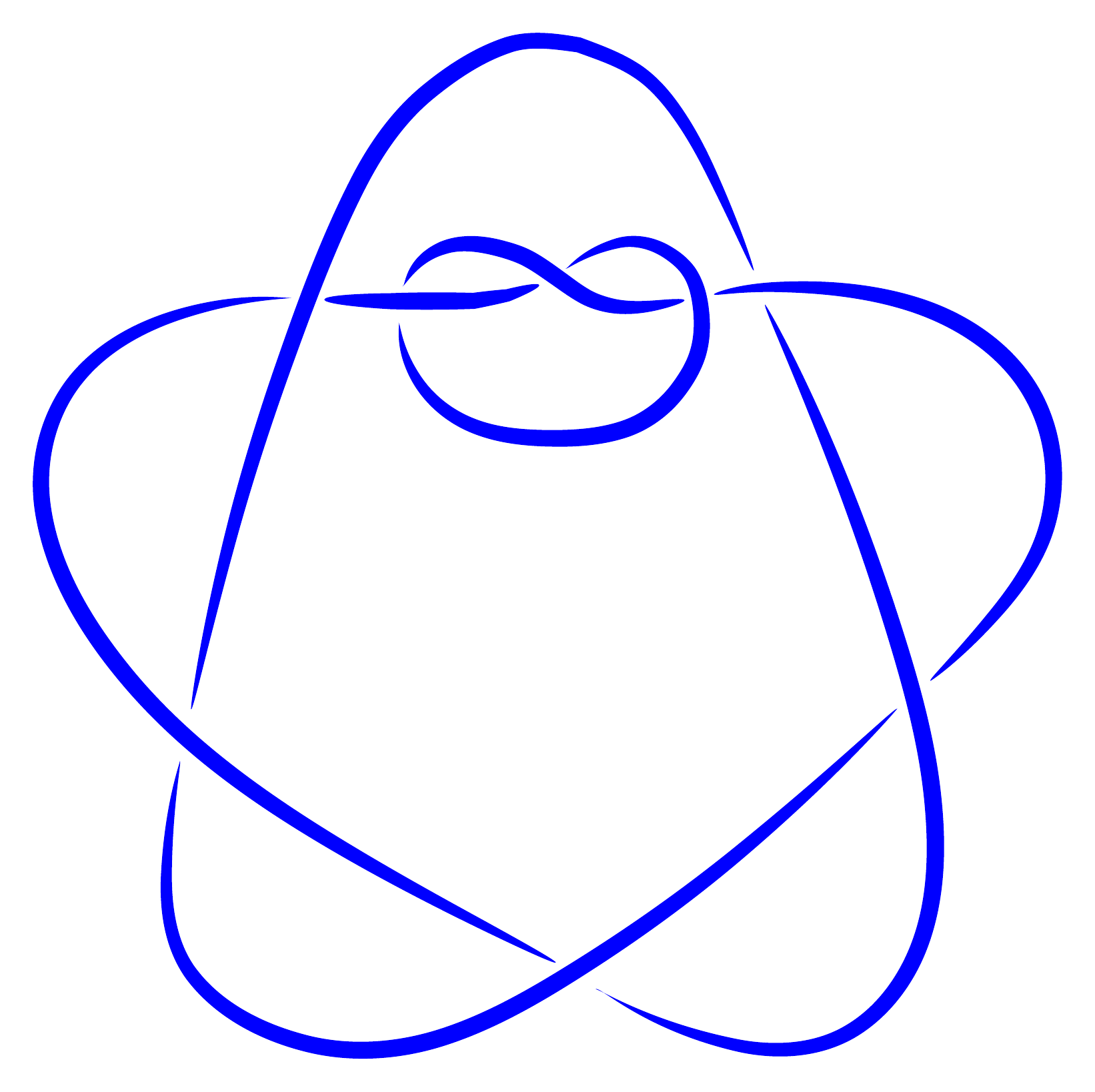}\qquad
\begin{forest}
for tree={grow=north, l sep=25pt, s sep=35pt}
[ [ $KC_2$, circle, draw, edge label={node[midway, right] {$\lambda_0$ (and $\mu_0$)}}
[$J_2$, circle, draw, edge label={node[midway, left] {$g_{J_2}$}}, edge label={node[midway,right] {$g_{J_2}=g_{J_1}^{-1}$}} ]
[ $J_1$, circle, draw, edge label={node[midway,left] {$g_{J_1}$}}]
 ]]
\end{forest}
\caption{Two planar projections of the knot $F$ from Example \ref{Ex:ConnectSumOfTwo}, with $J_1$ as the cinquefoil and $J_2$ as the trefoil, together with $\mathbb{G}_F$.  Here $\pi_1(\L_F/SO_4) \cong \Z \langle \lambda_0, \, g_{J_1},\, g_{J_2}\rangle$, while 
$\pi_1(\L_f/SO_4) \cong \Z \langle \lambda_0,\, g_{J_1}\rangle$.}
\label{F:ConnectSumOfTwo}
\end{figure}

\begin{example}
\label{Ex:ConnectSumJPlusEmpty}  
For $n=1$, $r=1$, i.e.~$F=(\varnothing, J)\bowtie KC_2$we get 
\begin{align*}
\pi_1(\tL_F / SO_4) &\cong \Z\langle \lambda_1,\, \mu_1,\, \lambda_0 \rangle \x \pi_1(\tL_{J}/SO_4)\\
\end{align*}
For $J$ a torus knot, 
\[
\pi_1(\tL_F / SO_4) \cong \Z\langle \lambda_1,\ \mu_1,\ \lambda_0,\ \mu_0\rangle
\]
We could replace either $\mu_0$ or $\lambda_1$ by the Gramain loop $g_J$ of $J$, since modulo $SO_4$, $g_J+\mu_0=\lambda_1$.
See Figure \ref{F:TrefoilSumEmpty}.

For $J$ the figure-eight knot,
\[
\pi_1(\tL_F / SO_4) \cong \Z\langle \lambda_1,\ \mu_1,\ \lambda_0,\ \mu_0,\  \frac{1}{2}(\lambda_0+ \lambda_J) \rangle
\]
where $\lambda_J$ is a loop that pushes the unknotted component a full loop along the knot.  It can also be viewed as the Fox--Hatcher loop of the long figure-eight knot; see the first picture in Figure \ref{F:Fig8SumEmpty}.

Passing to $\pi_1(\T_f)$ means taking the quotient by $\mu_0$.
For a torus knot $J$, this space is homotopically one dimension larger than the space of a torus knot in its most symmetric position in a solid torus.
Passing to $\pi_1(\L_f / SO_4)$ means taking the quotient by $\mu_1$.
\end{example}

\begin{figure}[h!]
\includegraphics[scale=0.25]{empty-sum-trefoil-long.pdf} \quad
\includegraphics[scale=0.25]{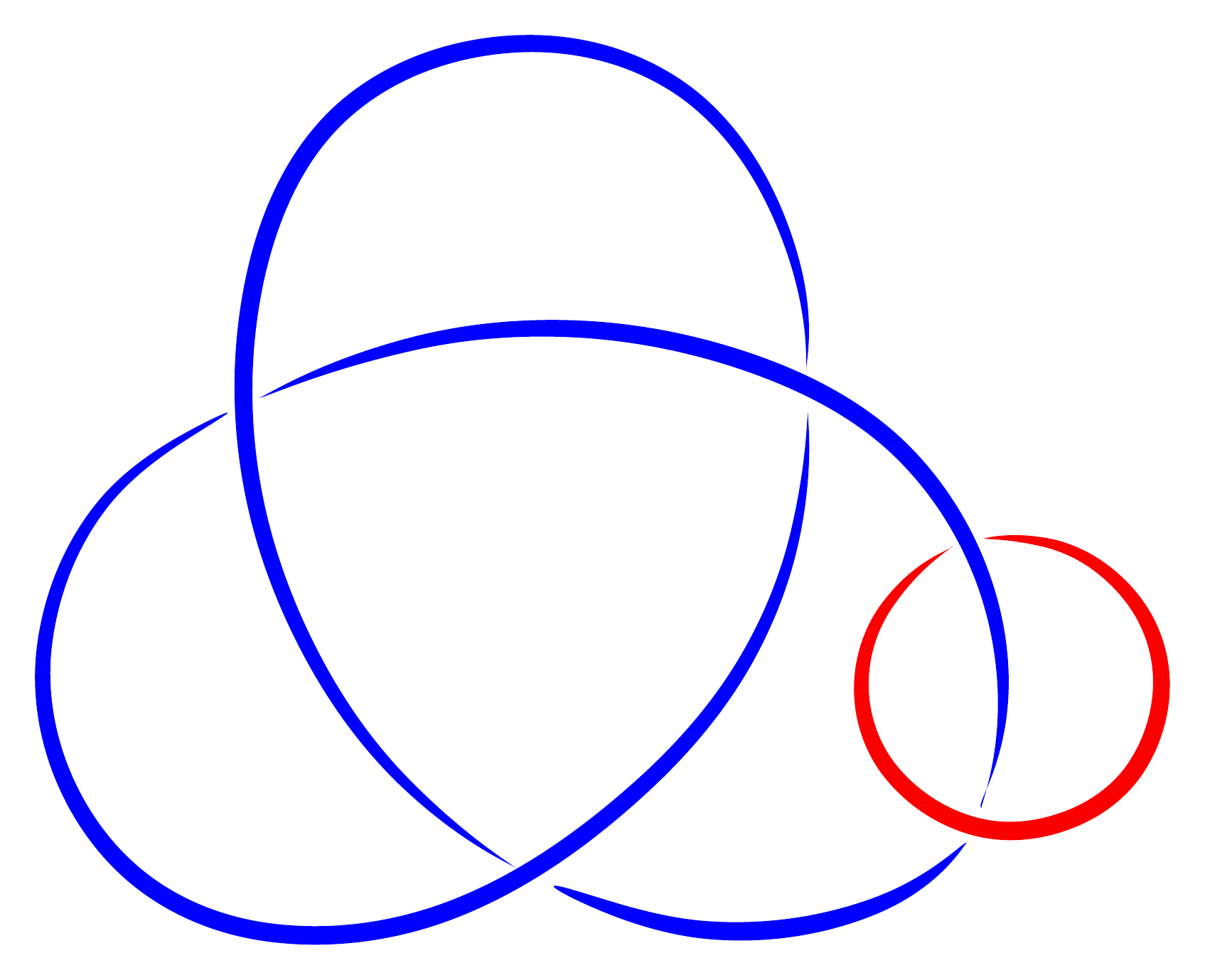} \quad
\begin{forest}
for tree={grow=north, l sep=25pt, s sep=35pt}
[ [ $KC_2$, circle, draw, edge label={node[midway,right] {$\lambda_0$ (and $\mu_0$)}}
[$J$, circle, draw, edge label={node[midway, right] {$g_J=\lambda_1 - \mu_0$}} ]
[ , edge label={node[midway,left] {$\lambda_1$ (and $\mu_1$)}} ]
 ]]
\end{forest}
\caption{Two planar projections of links $F$ from Example \ref{Ex:ConnectSumJPlusEmpty} with $J$ as the trefoil, together with $\mathbb{G}_F$.  Here $\pi_1(\L_f / SO_4) \cong \Z\langle \lambda_0,\ \lambda_1\rangle$.}
\label{F:TrefoilSumEmpty}
\end{figure}

\begin{figure}[h!]
\includegraphics[scale=0.25]{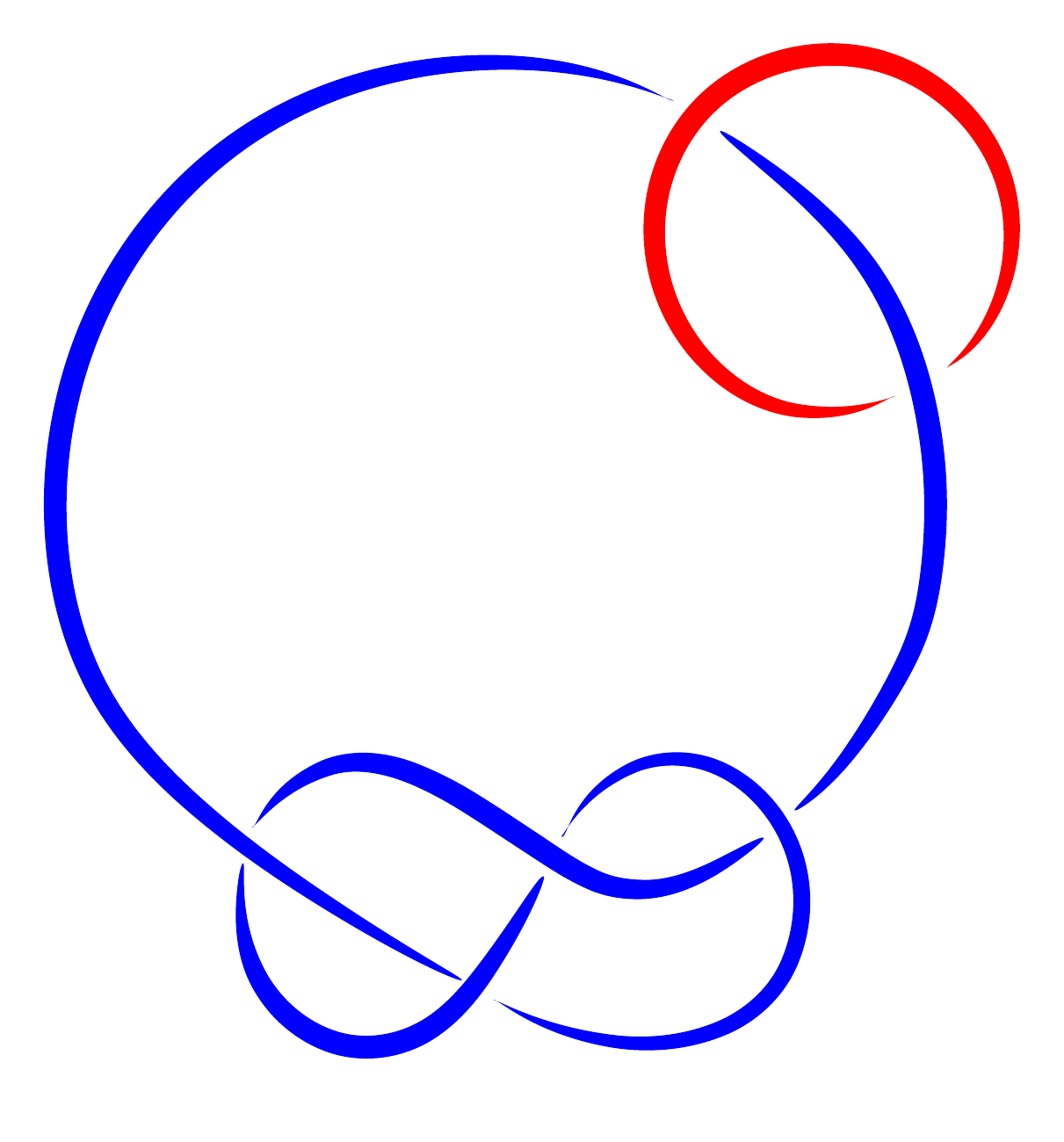} \qquad
\includegraphics[scale=0.25]{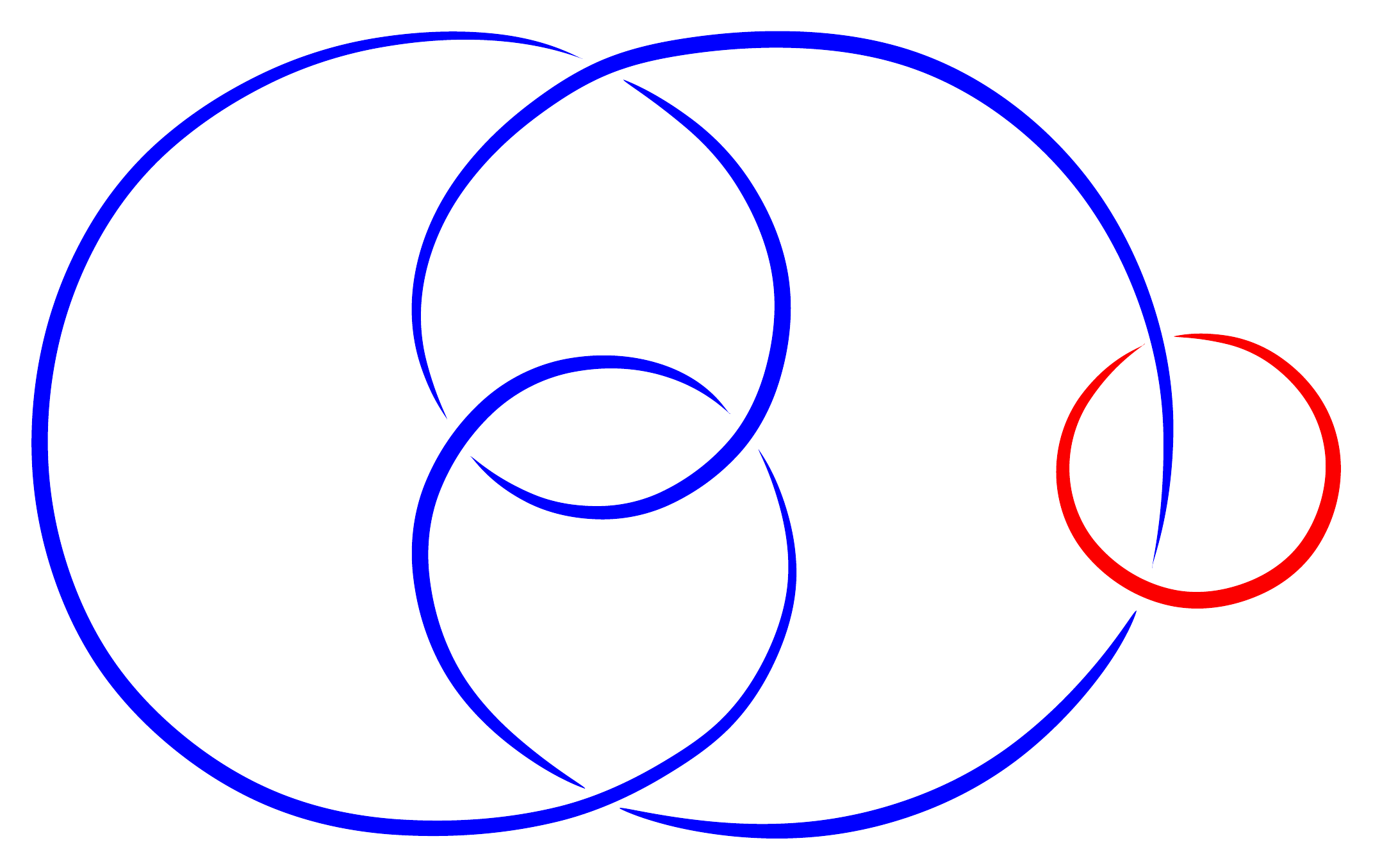} 
\begin{forest}
for tree={grow=north, l sep=25pt, s sep=35pt}
[ [ $KC_2$, circle, draw, edge label={node[midway,right] {$\lambda_0$ (and $\mu_0$)}}
[$J$, circle, draw, edge label={node[midway, right] {$\lambda_J, \ g_J=\lambda_1 - \mu_0$}} ]
[ , edge label={node[midway,left] {$\lambda_1$ (and $\mu_1$)}} ]
 ]]
\end{forest}
\caption{Two planar projections of links $F$ from Example \ref{Ex:ConnectSumJPlusEmpty} with $J$ as the figure-eight knot, together with $\mathbb{G}_F$.  Here $\pi_1(\L_f / SO_4) \cong \Z\langle \lambda_0,\ \lambda_1,\ \frac{1}{2}(\lambda_0+ \lambda_J) \rangle$.}
\label{F:Fig8SumEmpty}
\end{figure}

\begin{example}
\label{Ex:MoreGeneralConnectSum}
More generally, let $F = (\varnothing, J_1, \dots, J_r) \bowtie KC_{r+1} $, the result of splicing $r$
knots into the $(r+2)$-component keychain link $KC_{r+1}$.  
If the $J_i$ are distinct, then 
\begin{align*}
\pi_1(\tL_F/SO_4) &\cong \Z^2 \x \PB_{r+1} \x \prod_{i=1}^r \pi_1(\tL_J/SO_4)
\end{align*}
View $F_0$ as lying in the solid torus $U$.
 For $i<j<r+1$, the standard generators $p_{ij}$ are described similarly as for braids except that ``strand $k$ passes \emph{over} strand $\ell$ to the right/left'' corresponds to ``$J_k$ passes \emph{through} $J_\ell$ counter-clockwise/clockwise''.  So $p_{ij}$ is given by a motion where $J_i$ passes through $J_{i+1}, \dots, J_j$ counter-clockwise, then $J_j$ passes through $J_i$ counter-clockwise, and finally $J_i$ passes clockwise back through $J_{j-1}, \dots, J_{i+1}$, thus returning to its original position.  Each generator $p_{i,r+1}$ is given by $J_i$ passing through all of the other knots, counter-clockwise, until it returns to its original position.  
One may visualize the $(r+1)$-th configuration point at the center of $D^2$, with the remaining points arranged in a circle around it.
The full twist generating $Z(\PB_{r+1})$ corresponds to the product of a loop $\mu_1$ of longitudinal rotations of $U$ and
the loop $\lambda_0$ of reparametrizations of $F_0$.

If the $J_i$ are copies of the same knot $J$, then 
\begin{align*}
 \pi_1(\tL_F/SO_4) &\cong \Z^2 \x \B_{1,r} \x_{\mathfrak{S}_r} \pi_1(\tL_J/SO_4)^r.
\end{align*}
where $\mathfrak{S}_r$ permutes the last $r$ points in a configuration in $\Conf(r+1,\R^2)$ and the $r$ factors of $\tL_J/SO_4$.  
View $\B_{1,r}$ now as the subgroup of $\B_{r+1}$ inducing permutations of $\{0,\dots, r\}$ that fix $0$.  
It is isomorphic to the \emph{annular} (or \emph{circular}) \emph{braid group} $\mathcal{CB}_r$ on strands labeled $\{1,\dots,r\}$; see e.g.~\cite{Kent-Peifer} or \cite{Bellingeri-Bodin}.  Let $\beta_{i,i+1}$ denote the generator of $\mathcal{CB}_r$ that passes the $i$-th strand over the $(i+1)$-th strand, with the subscripts lying in $\{1,\dots,r\}$ and interpreted modulo $r$.  Then $\beta_{i,i+1}$ corresponds to passing $J_i$ through $J_{i+1}$.  A generator $\zeta$ that cyclically permutes the strands (without any crossings after projecting to $S^1 \x I$) corresponds to longitudinal rotation by $2\pi/r$ of a solid torus containing the knotted component followed by reparametrization in the opposite direction by $2\pi/r$.

To compute $\T_f$, suppose for simplicity that $J$ is a torus knot, so $\pi_1(\tL_J/SO_4) \cong \Z$.  
Then 
\[
\pi_1(\tL_F/SO_4) \cong \Z^2 \x \B_{1,r} \ltimes \Z^r = 
\Z \langle \mu_1, \lambda_1 \rangle \x \B_{1,r} \ltimes \Z \langle g_{J,1}, \dots, g_{J,r} \rangle
\]
 where the $g_{J,i}$ are the Gramain loops of the $r$ summands.  
Then $\pi_1(\T_f)$ is the quotient by the group generated by $\mu_0 = \lambda_1 g_{J,1} \dots g_{J_r}$ (and $\B_{1,r}$ acts trivially on the diagonal $\langle g_{J,1} \dots g_{J_r}\rangle$), so
\[
 \pi_1(\T_f) \cong \Z\langle \mu_1 \rangle \x \B_{1,r} \ltimes \Z\langle g_{J,1}, \dots, g_{J,r} \rangle.
 \]
See Figure \ref{F:SumTrefoils}.
For distinct torus knots $J_1, \dots, J_r$, 
\[
 \pi_1(\T_f) \cong \Z^2 \x  \PB_{r+1} \x \Z^{r-1} \cong \Z \langle \mu_1, \, \lambda_1, \,g_{J,1}, \dots, g_{J,r-1} \rangle \x \PB_{r+1}.
 \]
Alternatively, we could discard $\lambda_1$ and instead take all $r$ Gramain loops.
Taking the quotient by $\mu_1$ gives $\pi_1(\L_f/SO_4) \cong\Z \x  \B_{1,r} \ltimes \Z^{r-1}$ when all $J_i=J$ and 
$\pi_1(\L_f/SO_4) \cong\Z^r \x  \PB_{r+1}$ when the $J_i$ are distinct.  So for instance, if $r=2$ and $J_1$ and $J_2$ are distinct torus knots, $\pi_1(\L_f/SO_4) \cong\Z^2 \x  \PB_3$, whereas for $F=J_1 \# J_2$ in Example \ref{Ex:ConnectSumOfTwo}, we had $\pi_1(\L_f/SO_4) \cong \Z \x  \PB_2$.  An extra factor of $\Z$ comes from reparametrizing the second component $f_1$, while the extra braid generators come from loops called $b_1$ and $b_2$ in Example \ref{Ex:ConnectSumOfTwo}, corresponding to
 $p_{13}$ and $p_{23}$ as in the first paragraph of this example.
\end{example}

\begin{figure}
(a)  \includegraphics[scale=0.22]{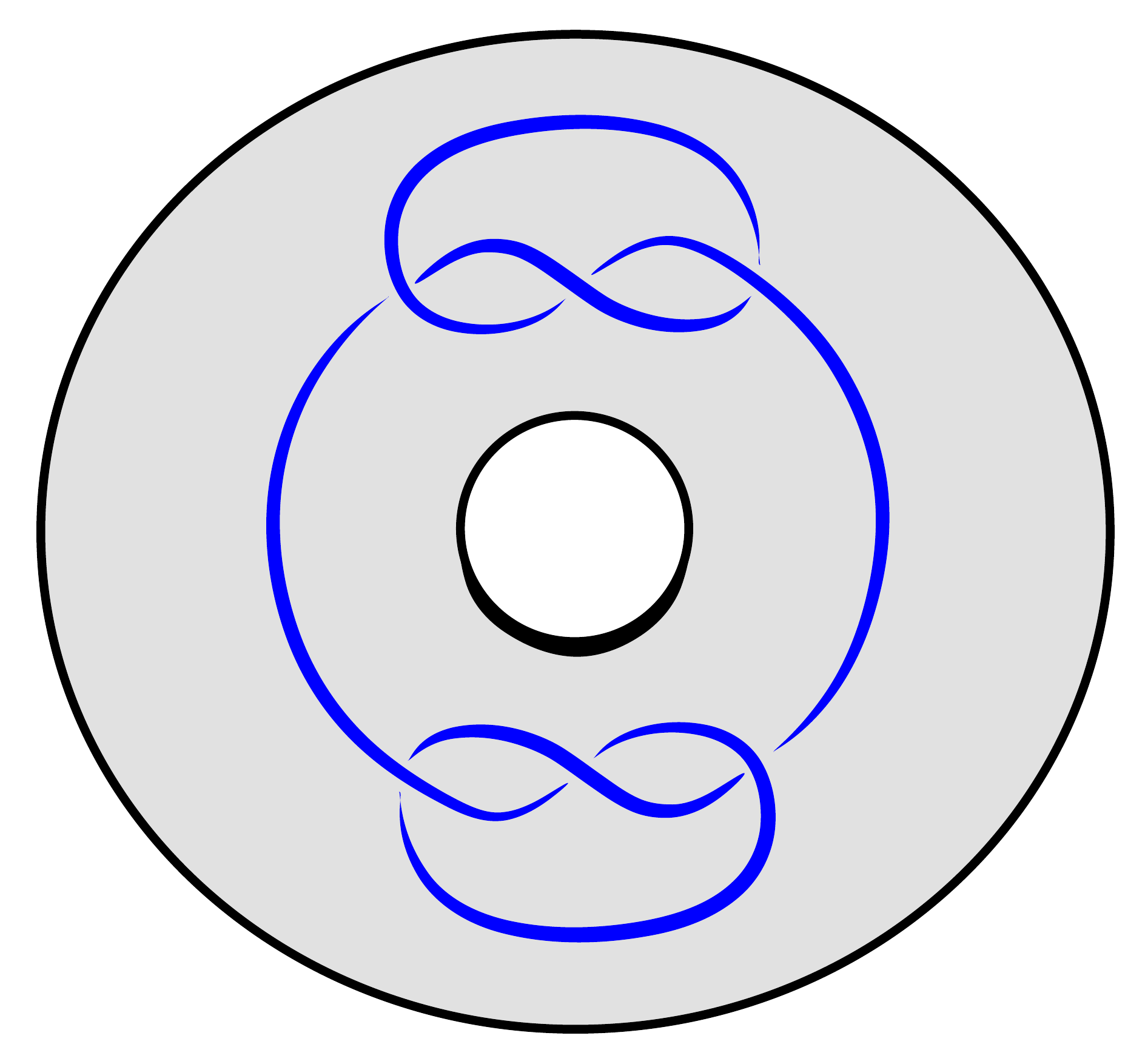}
\qquad \qquad
(b)  \includegraphics[scale=0.22]{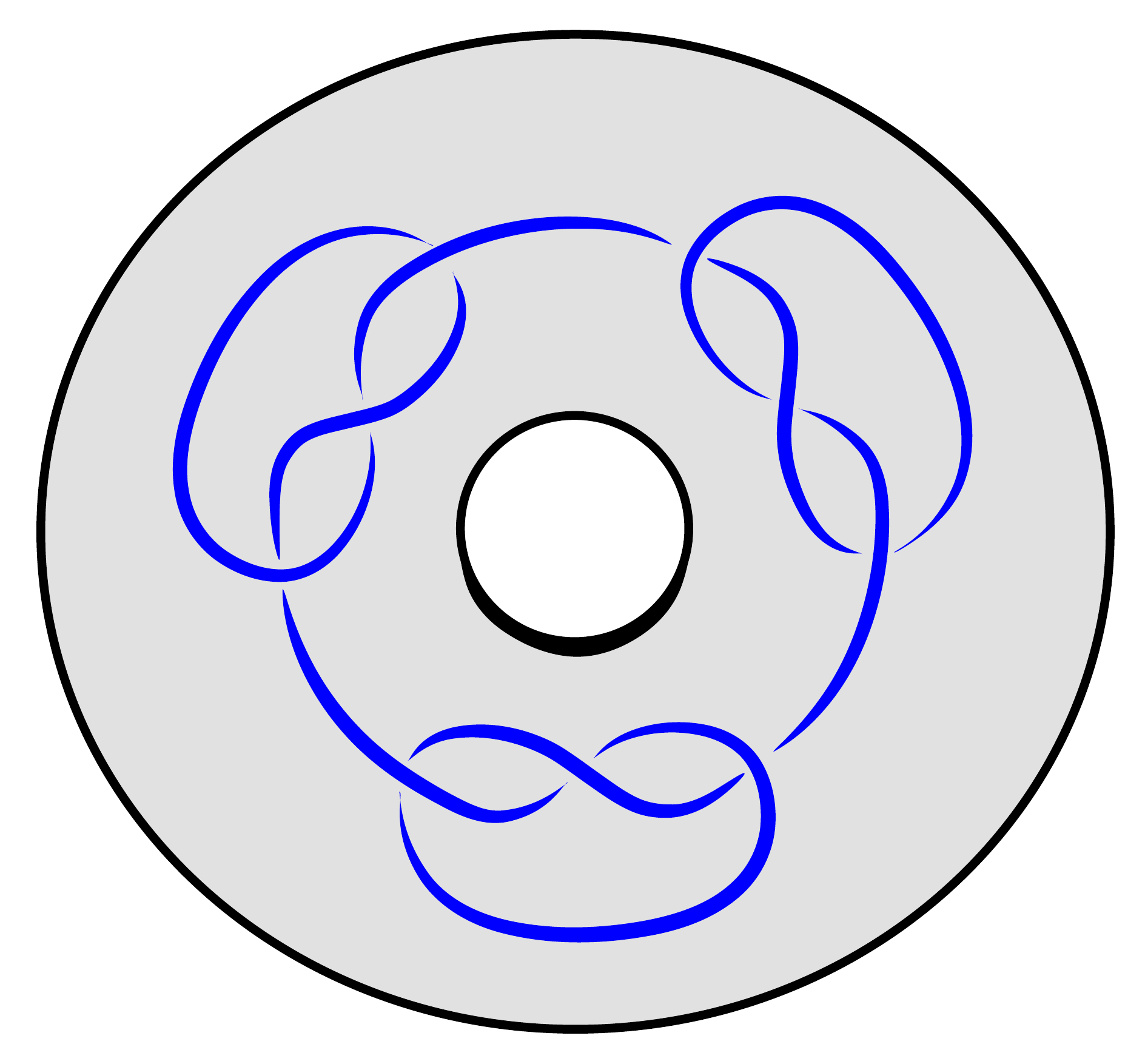}
\caption{
Knots $f$ and $g$ in a solid torus associated to (a) $F=(\varnothing, J,J) \bowtie KC_3$ and (b) $G=(\varnothing, J,J,J) \bowtie KC_4$ where $J$ is the trefoil.
The fundamental groups are given by 
(a) $\pi_1(\T_f) \cong \Z\langle \mu_1 \rangle \x \B_{1,2} \ltimes \Z^2$
and (b) $\pi_1(\T_g) \cong \Z\langle \mu_1 \rangle \x \B_{1,3} \ltimes \Z^3$, where $\B_{1,n}$ is the $n$-strand annular braid group.
See Example \ref{Ex:MoreGeneralConnectSum}
}
\label{F:SumTrefoils}
\end{figure}

%\newpage
\subsection{Splicing links into hyperbolic links}
\label{S:HypSplices}
We now treat the final case of splicing, namely the case of splicing into a hyperbolic link.
Let $L=(L_0,\dots, L_{n+r})$ be a hyperbolic link with boundary tori  $T_i = \d \nu(L_i)$ where $\nu(L_i)$ is a tubular neighborhood of $L_i$.  
The indexing of $L$ is meant to suggest a partitioning of the tori in $\d C_L$, as in Lemma \ref{SpliceSES}, according as we glue or do not glue link complements to them.
Let $k \in \{0,\dots,n+r\}$ (where we are most interested in the case $k=n$).
Let $\Diff_{0, \dots, k}(C_L)$ be the group of diffeomorphisms of $C_L$ whose restrictions to $T_0 \sqcup \dots \sqcup T_k$ are isotopic to the identity.  
By Proposition \ref{HMProp}, the fibration
\[
\Diff(C_L; T_0 \sqcup \dots \sqcup T_n) \to \Diff_{0,\dots,n}(C_L) \to \Diff(T_0 \sqcup \dots \sqcup T_n)
\]
yields the exact sequence 
\begin{equation}
\label{SequenceForHypSplicing}
\{e\} \to \pi_1\Diff(T_0 \sqcup \dots \sqcup T_n) \to  \pi_0 \Diff(C_L; T_0 \sqcup \dots \sqcup T_n) \to \pi_0 \Diff_{0, \dots, n}(C_L)  \to \{e\}.
\end{equation}

By Mostow rigidity, the group $\Diff_{0, \dots, k}(C_L)$ can be identified with a subgroup  of the finite group $\Isom^+(C_L) < \mathrm{Isom}(C_L) \cong \pi_0 \Diff(C_L)$.
Specifically, it consists of those isometries which extend to orientation-preserving diffeomorphisms of $S^3$ preserve $L_0, \dots, L_k$ and their orientations.

\begin{definitions}
For $L=(L_0, \dots, L_{n+r})$ and any $k=0,\dots, n+r$, define $B_{L,0,\dots,k}:=\pi_0 \Diff_{0,\dots,k}(C_L)$.  Define $B_L$ as the subgroup of isometries in $\Isom^+(C_L)$ which extend to diffeomorphisms of $S^3$.
\end{definitions}
 
So the group called $B_L$ in \cite{BudneyTop} is $B_{L,0}$ in our notation.  Note that what the software SnapPy's symmetry group function for a link $L$ returns $\Isom(C_L)$.

\begin{lemma}
Let $\Diff_{0, \dots,n}(L_0 \sqcup \dots \sqcup L_n)$ be the identity component of $\Diff(L_0 \sqcup \dots \sqcup L_n)$.
For any $k=0,\dots, n+r$, the map 
\[B_{L,0,\dots,k} \to \Diff_{0, \dots,k}(L_0 \sqcup \dots \sqcup L_k)\]
 is injective.
\end{lemma}
\begin{proof}
If $k>0$, then the kernel of this map has a finite-order element whose fixed-point set is a $(k+1)$-component link and in particular, not an unknot, contradicting Smith's theorem \cite{Smith:Annals1939}.  The more subtle case of $k=0$ is covered by \cite[Proposition 3.5]{BudneyTop}.  
\end{proof}

Thus $B_{L,0}$ must be cyclic, as a finite subgroup of $\Diff^+(S^1) \cong S^1 \times \Diff(I, \d I)$.
Hence $B_{L,0,\dots,n} < B_{L,0}$ is cyclic. 
Now $B_{L,0,\dots,n}$ maps to $(S^1)^{2n+2} \cong \Isom(T^{n+1}) \cong \Diff_{0,\dots,n}(L_0 \sqcup \dots \sqcup L_n)$.  Let $\widetilde{B}_{L,0,\dots,n}$ be the lift of $B_{L,0,\dots,n}$ to the universal cover $\R^{2n+2}$, i.e.~the pullback below:
\[
\xymatrix{
\widetilde{B}_{L,0,\dots,n} \ar[r] \ar[d] &\R^{2n+2} \ar[d] \\
B_{L,0,\dots,n} \ar@{^(->}[r] & (S^1)^{2n+2}
}
\]
As in the proof of Proposition \ref{HMProp}, one can show that the exact sequence \eqref{SequenceForHypSplicing} is equivalent to the exact sequence
\begin{equation}
\label{SequenceForHypSplicing2}
0 \to \Z^{2n+2} \to \widetilde{B}_{L,0,\dots,n} \to B_{L,0,\dots,n} \to 0.
\end{equation}
By that Proposition \ref{HMProp} (b), $\widetilde{B}_{L,0,\dots,n}$ is free abelian of rank $2n+2$.

As in the proof of Lemma \ref{SpliceSES}, there is a map $B_{L,0,\dots,n} \to \mathfrak{S}_r^\pm$ given by extending $\alpha \in B_{L,0,\dots,n}$ to a diffeomorphism of $S^3$ and considering the induced action on $L_{n+1} \sqcup \dots \sqcup L_{n+r}$.
This map to $\mathfrak{S}_r^\pm$ then gives an action of $B_{L,0,\dots,n}$ on $\{J_1, \dots, J_r, \iota_0^-(J_1), \dots, \iota_0^-(J_r) \}$, where $J_1, \dots, J_r$ are links, each with a distinguished component $J_{i,0}$.
Recall the definition of $\mathfrak{S}^\pm(J)$ from Definition \ref{EquivRelDef}.

\begin{definitions}
\label{AFDefinition}
Suppose $F= (\varnothing, \dots, \varnothing, J_1, \dots, J_r)\bowtie L$.   
Let $A_F$ be the preimage of $\mathfrak{S}^\pm(\vec{J})$ under the map $B_{L,0,\dots,n} \to \mathfrak{S}_r^\pm$.
(So  $A_F$ depends on $L$, the $J_i$, and the association of a component of each $J_i$ to a component of $L$.)  
Then define $\widetilde{A}_F$ as the pullback below, from which it follows that $\widetilde{A}_F$ is a subgroup of $\widetilde{B}_{L,0,\dots,n}$, is free abelian, and is of rank $2n+2$:
\begin{equation}
\label{E:AFTildeDefn}
\xymatrix{
\widetilde{A}_F \ar[r] \ar[d] &\widetilde{B}_{L,0,\dots,n} \ar[d] \\
A_F \ar@{^(->}[r] & B_{L,0,\dots,n}
}
\end{equation}
\end{definitions}

\begin{proposition}
\label{P:HypSplice}
Let  $L=(L_0, \dots, L_{n+r})$ be a (framed) $(n+r+1)$-component hyperbolic link.
Let $F=(\varnothing, \dots, \varnothing, J_1, \dots, J_r)\bowtie L$.
Then 
\[
\boxed{\tL_F/SO_4 \simeq (S^1)^{2n+1} \x \left(S^1 \x_{A_F}  \prod_{i=1}^r \tL_{J_i} / SO_4 \right)}
\]
where the factor $(S^1)^{2n+1}$ corresponds to meridional rotations of $L_0, \dots, L_n$ and longitudinal rotations of some $n$ of those components.
\end{proposition}

\begin{proof}
As in the proof of Proposition \ref{P:ConnectSum}, set $G:=\pi_0 \Diff(C_F; \d C_F)$ and $K:=\prod_{i=1}^r K_i$ where $K_i:= \pi_0 \Diff(C_{J_i}; \d C_{J_i})$.  
Lemma \ref{SpliceSES} then gives us the exact sequence 
\[
\{e\} \to K \to G \to \pi_0 \Diff(C_L;\ T_0 \sqcup \dots \sqcup T_n).
\]
By  \eqref{SequenceForHypSplicing2}, the right-hand group in this hyperbolic case is $\widetilde{B}_{L,0,\dots, n}$.  
By Lemma \ref{SpliceSES} and the definition of $A_F$, the image $H$ in that right-hand group is $\widetilde{A}_F$, and   $G \cong \widetilde{A}_F \ltimes K \cong \Z^{2n+2} \ltimes K$.

We  now claim that there is a rank-$(2n+1)$ free abelian subgroup of $\widetilde{A}_F$ which acts trivially and contains all the meridional rotations.  
Recall the short exact sequence $0 \to \Z^{2n+2} \to \widetilde{A}_F \to A_F \to 0$ induced by \eqref{SequenceForHypSplicing} and\eqref{SequenceForHypSplicing2}, where 
$A_F$ is a finite cyclic group.  
Let  $\alpha$ be the preimage of a generator of $A_F$ (which can be thought of as a fractional linear combination of images of generators of $\Z^{2n+2}$).  
By the characterization of lattice bases, we can find a basis for $\widetilde{A}_F$ by adjoining to $\alpha$ any $(2n+1)$-element subset $S$ of a basis $B$ of $\Z^{2n+2}$ such that $\alpha \notin \mathrm{span}(S)$.  Take $B$ to consist of classes of loops of meridional and longitudinal rotations for each component $L_0, \dots, L_n$.  By the proof of Proposition \ref{DiffS3LHypLink}, we can take all the $n+1$ meridional generators to lie in $S$.  
We can then take some $n$ longitudinal generators to lie in $S$.  
We finally verify the triviality of the action of each element of $S$.
Indeed, the resulting Dehn twist $\tau$ of such an element can be taken to have support in a collar neighborhood $\nu(T_i)$ of $T_i$ and to be the identity on $\d \nu(T_i)$.  Any other diffeomorphism can be taken be supported outside this collar and thus commutes with $\tau$.  Thus $\mathrm{span}(S)$ is the desired rank-$(2n+1)$ subgroup.

So now $G \cong \Z^{2n+1} \x (\Z \ltimes K)$, where the $\Z$ factor in the semi-direct product is generated by $\alpha \in \widetilde{A}_F$ as defined above.  Taking the classifying space $BG$ of $G$ and applying Lemma \ref{SemidirectLemma} with $H=\widetilde{A}_F$, $H''=A_F$, and $H' = \ker(\widetilde{A}_F \to A_F) \cong \Z^{2n+2}$ gives the claimed description of $\tL_F/SO_4$.
\end{proof}

From the decomposition in Proposition \ref{P:HypSplice}, we see that if all the $J_i$ are knots, then $\tL_F/SO_4 \simeq (S^1)^{n+1} \x \L_f/SO_4$, where each factor of $S^1$ corresponds to meridional rotation about a component of $L$.  Thus if $F=(F_0,F_1)$ corresponds to a knot in the solid torus, $\T_f \simeq S^1 \x \L_f/SO_4$, where the $S^1$ corresponds to meridional rotation of the unknot component, and $\tL_F/SO_4 \simeq S^1 \x \T_f$, where the $S^1$ corresponds to meridional rotation of the knotted component.  
(These are special cases of Theorem \ref{T:MeridiansFactor} below.)  Thus in the examples below we often describe only $\L_f/SO_4$.

\begin{example}[Whitehead double of a knot]
\label{Ex:WhiteheadDoubleInS3}
Let $J$ be a (framed) knot (in $S^3$), and let $F$ be the Whitehead double $\Wh(J):=J \bowtie \Wh$ of $J$.
Then $\tL_F/SO_4 \simeq S^1 \x (S^1 \x_{A_F} \tL_J/SO_4)$, recovering an example in \cite[Section 5]{BudneyTop}, since $\tL_F/SO_4\simeq \K_{\underline{f}}$.
Thus $\L_f/SO_4 \simeq S^1 \x_{A_F} \tL_J/SO_4$.  If $J$ is invertible, $A_F=\Z/2$ acts by inversion on $\tL_J/SO_4$ and rotation by $180^\circ$ on $S^1$.  If $J$ is non-invertible, $A_F=\{e\}$.  
If $J$ is non-invertible, the $S^1$ factor corresponds simply to reparametrization.  
If $J$ is invertible, its description is more subtle.  For example, if $J$ is a torus knot, $\L_f/SO_4 \simeq S^1 \x_{\Z/2} S^1$, the Klein bottle, so 
\[
\pi_1(\L_f/SO_4) \cong \Z \ltimes \Z \cong \Z\langle g_J \rangle \ltimes \Z \langle \mu_J^{1/2} \theta^{1/2} \lambda_0^{1/2}\rangle
\]
where $g_J$ is the Gramain loop of $J$,
 $\mu_J^{1/2}$ is a path of rotations by $180^\circ$ along the meridian of the knotted torus, 
 $\lambda_0^{1/2}$ is the path of reparametrizations of $\Wh(J)$ by translating by $\pi$,
$\theta^{1/2}$ is a path of rotations by $180^\circ$ of the picture in Figure \ref{F:WhiteheadDoubleInS3} around the $y$-axis (which is an isotopy from $J$ to its inverse $i(J)$ and which lies entirely in $SO_4$).
Modulo $SO_4$, $g_J$ is equal to a full loop $\mu_J$ of meridional rotations of a torus knotted in the shape of $J$.  
\end{example}

\begin{figure}[h!]
\includegraphics[scale=0.28]{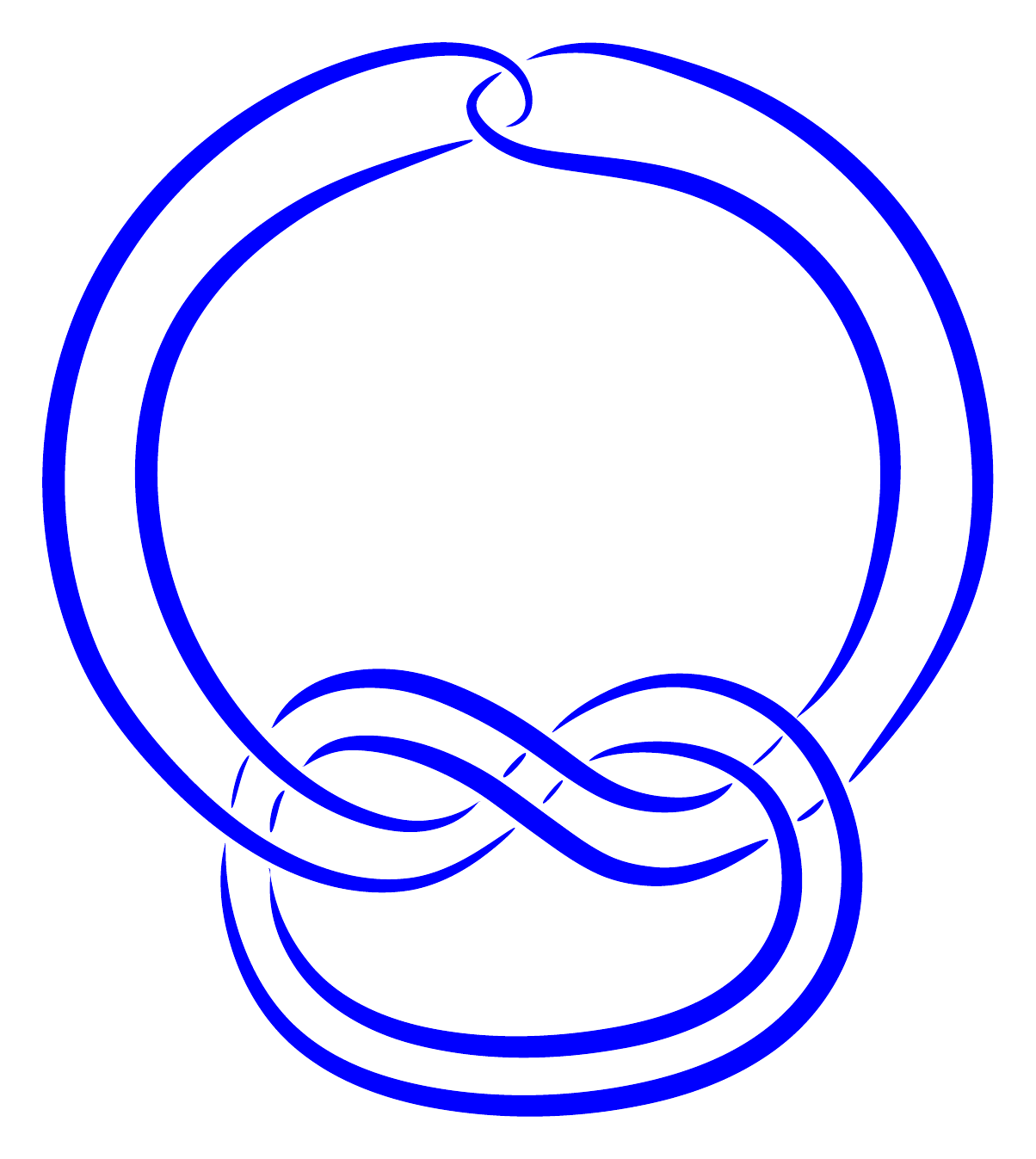}
\qquad \qquad
\begin{forest}
for tree={l sep=25pt}
[ $J$, circle, draw, 
[Wh, circle, draw, edge label={node[midway,right] {$\mu_J= g_J$ (and $\lambda_J = 1 $)}}
[ , edge label={node[midway,right] {$\lambda_0$ (and $\mu_0 = g_{\Wh(J)}$)}}] ]]
\end{forest}
\caption{Left: the Whitehead double $F$ of $J$ in Example \ref{Ex:WhiteheadDoubleInS3}, with $J$ the right-handed trefoil with  framing $+3$ (the blackboard framing of the above projection).  
Right: the tree $\mathbb{G}_F$.  Here $\pi_1(\L_f/SO_4) \cong \langle a, \,\mu_J \, | \,  \mu_J^a = \mu_J^{-1}\rangle$ where $a=\mu_J^{1/2}\lambda_0^{1/2}$, and $\pi_1(\tL_F/SO_4)$ is the product of this group with $\Z\langle \mu_0 \rangle$. }
\label{F:WhiteheadDoubleInS3}
\end{figure}

\begin{example}[Whitehead double of Whitehead link]
\label{Ex:WhSpliceWh}
Let $F = W \bowtie W$.  Then $\tL_F/SO_4 \cong (S^1)^3 \x ( S^1 \x_{\Z/2} (S^1)^2)$.  The group $\Z/2$ acts by the antipodal map on each factor of $S^1$, and those two factors correspond to the meridian and longitude of a torus embedded roughly in an $\infty$ shape in Figure \ref{F:WhSpliceWh}.  We denote the corresponding generators of $\pi_1$ by $b$ and $c$.
So $\L_f/SO_4 \cong S^1 \x ( S^1 \x_{\Z/2} (S^1)^2)$.
Although the projection in Figure \ref{F:WhSpliceWh} is not the most symmetric, we chose it so as to help visualize the blue, non-circular component as a knot in $S^1 \x D^2$.
\end{example}

%\vspace{-2pc}

\begin{figure}[h!]
\includegraphics[scale=0.3]{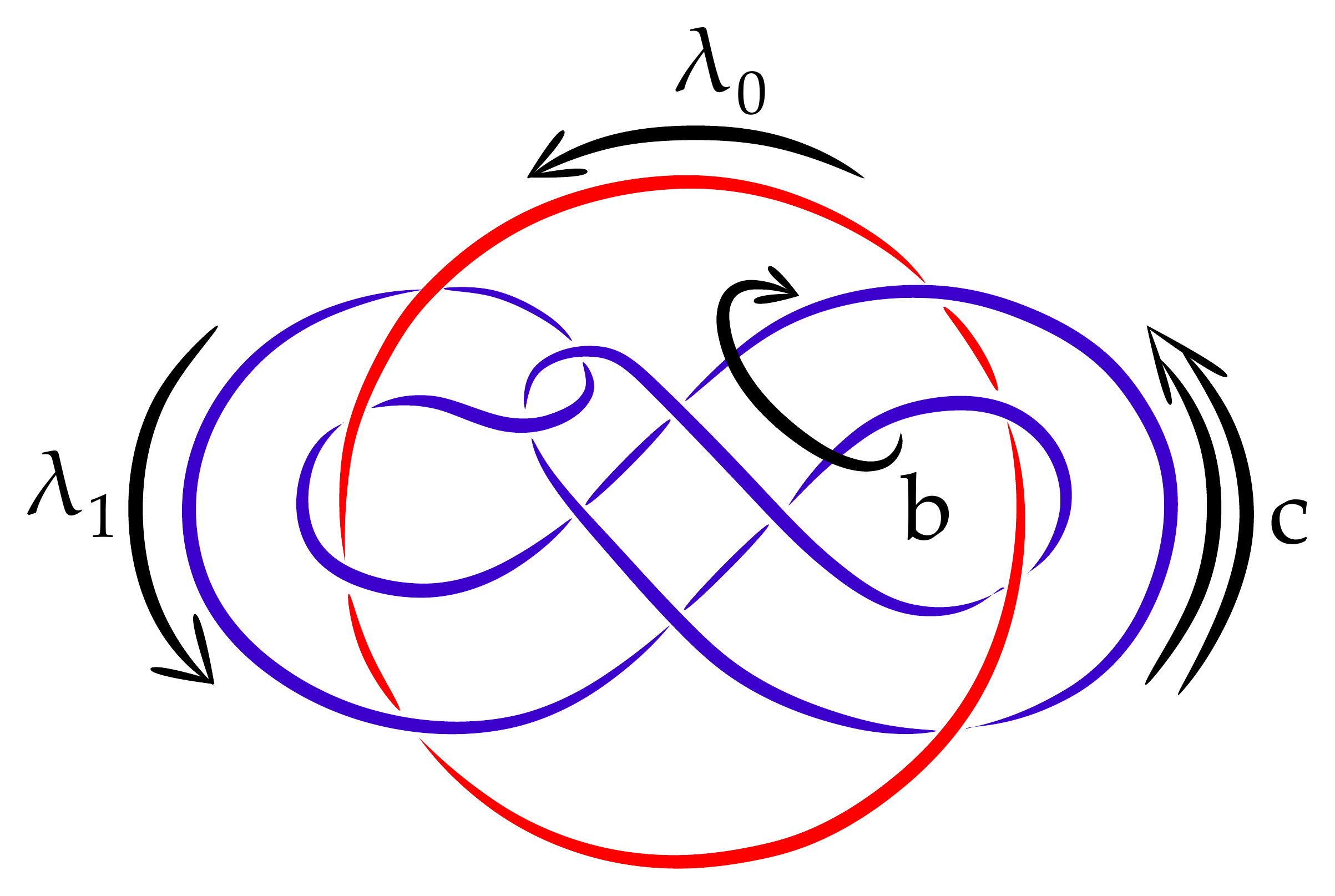}
\qquad \qquad
\begin{forest}
for tree={l sep=25pt}
[ [Wh, circle, draw, edge label={node[midway,right] {$\lambda_1$ (and $\mu_1$)}} 
[Wh, circle, draw, edge label={node[midway,right] {$b, c$}}
[ , edge label={node[midway,right] {$\lambda_0$ (and $\mu_0$)}}]
]]]
\end{forest}
\caption{The link $F = W \bowtie W$ in Example \ref{Ex:WhSpliceWh}, together with $\mathbb{G}_F$.  
Here $\pi_1(\L_f/SO_4) \cong \Z\langle \lambda_0 \rangle \x \langle a,b,c  \ | \ b^a = b^{-1}, \, c^a = c^{-1},\, [b,c]=1  \rangle $ where $a=\lambda_0^{1/2}  \lambda_1^{1/2}$.
}
\label{F:WhSpliceWh}
\end{figure}

\begin{example}[Bing double of a knot]
\label{Ex:BorrPartialSplice}
Let $\mathrm{Borr}$ be the Borromean rings, and let $F=(\varnothing, J)\bowtie \mathrm{Borr}$.
This $F$ is called the Bing double of $J$, is shown in Figure \ref{F:BorrPartialSplice}.  
Here $n=1$, $r=1$, and $B_{L,0,1}$ is trivial: the complement of the Borromean rings has many isometries (including a nontrivial one that preserves two components and the orientation of one them), but none that preserve two components and both orientations.  
Thus $\pi_1(\L_f/SO_4) \cong \Z\langle \lambda_0, \lambda_1 \rangle \x \tL_J/SO_4$.
So for example, if $J$ is a torus knot, $\pi_1(\L_f/SO_4) \cong \Z\langle \lambda_0, \, \lambda_1,\,  g_J \rangle$ and $\pi_1(\T_f) \cong \Z\langle \lambda_0, \, \lambda_1, \, \mu_1,\, g_J \rangle$.
For any knot $J$, both components are unknotted in $S^3$, and there is a symmetry interchanging the two, so either component can be taken to be the knot in the solid torus, and we get isomorphic results for $\pi_1(\T_f)$.
\end{example}

\begin{figure}[h!]
{\includegraphics[scale=0.3]{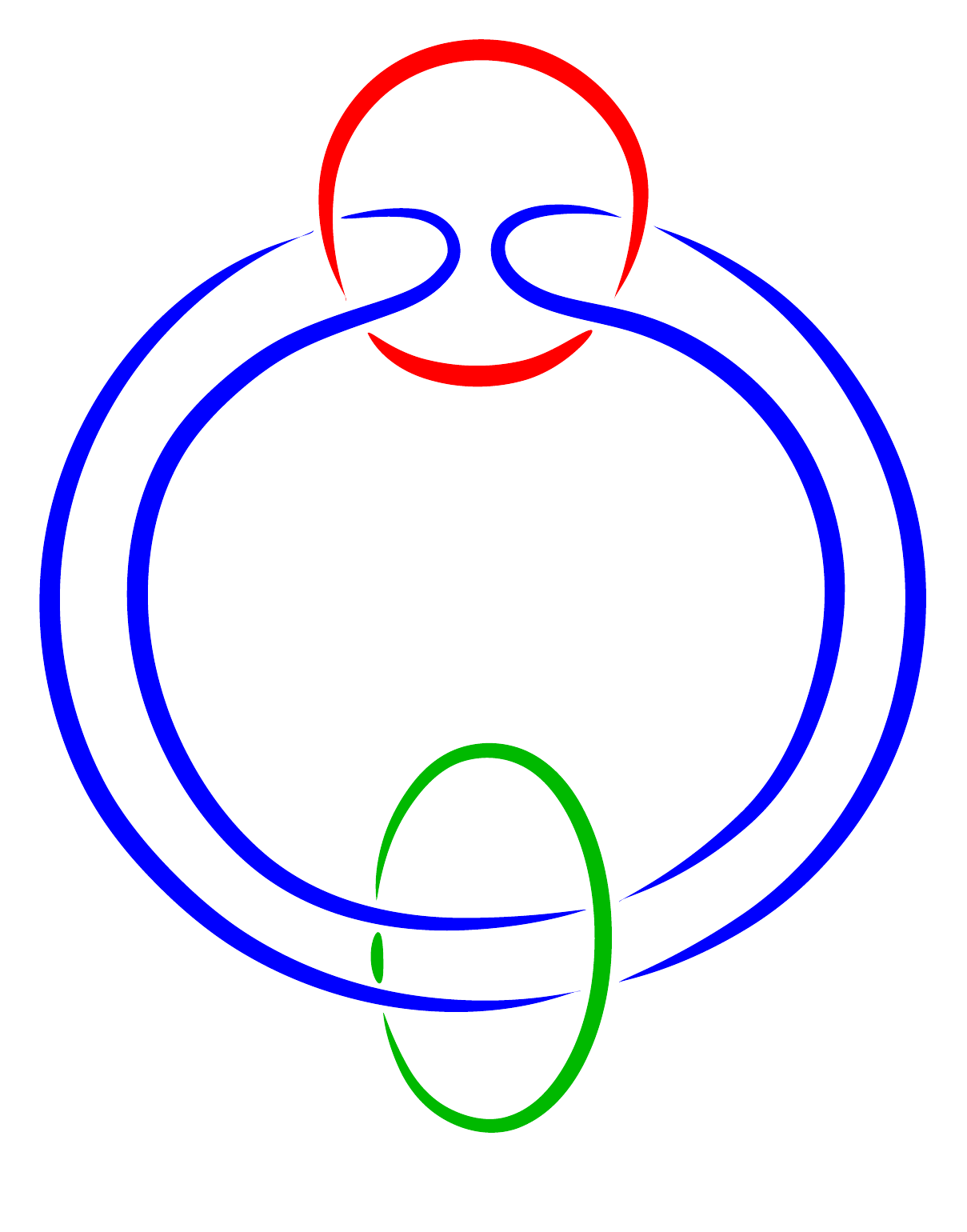}}
\qquad \qquad
\includegraphics[scale=0.3]{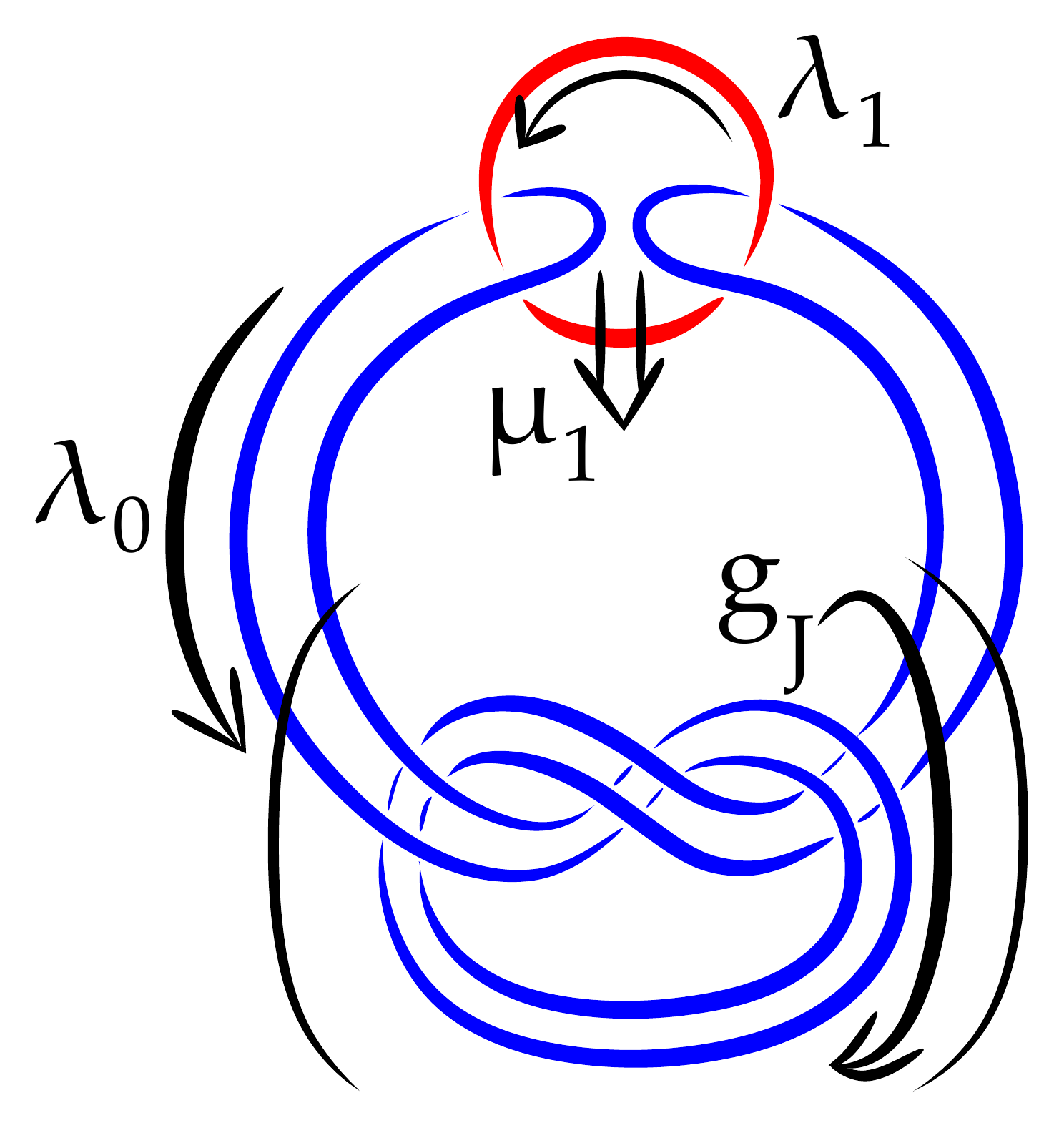}
\qquad \qquad
\begin{forest}
for tree={grow=north, l sep=25pt, s sep=35pt}
[ [ $\mathrm{Borr}$, circle, draw, edge label={node[midway,left] {$\lambda_0$ (and $\mu_0$)}}
[$J$, circle, draw, edge label={node[midway, right] {$g_J$}} ]
[ , edge label={node[midway,left] {$\lambda_1$ (and $\mu_1$)}} ]
 ]]
\end{forest}
\caption{Left: the Borromean rings.  Center: the Bing double $F=(\varnothing, J)\bowtie \mathrm{Borr}$ of the right-handed trefoil $J$ with $+3$ framing, from Example \ref{Ex:BorrPartialSplice}.  Right:  $\mathbb{G}_F$.  Here $\pi_1(\L_f/SO_4) \cong \Z\langle \lambda_0, \lambda_1, g_J \rangle$.}
\label{F:BorrPartialSplice}
\end{figure}

\begin{example}
\label{Ex:Brunnian}
Let $F=(\varnothing, \mathrm{Borr}) \bowtie  ((\varnothing, \mathrm{Borr}) \bowtie ( \dots \bowtie \mathrm{Borr}))$, where the Borromean rings $\mathrm{Borr}$ appear $n-2$ times.  Then $F$ is Milnor's $n$-component Brunnian link 
shown in Figure \ref{F:Borr-Borr-Borr} and \cite[Figure 7]{MilnorLinkGroups}.  As in Example \ref{Ex:BorrPartialSplice}, the symmetry group $B_{L,0,1}$ for $L$ the Borromean rings is trivial.  Thus $\L_f/SO_4 \cong (S^1)^{n}$ where each factor corresponds to a reparametrization of a component.  
\end{example}

\begin{example}[A link $F = (\varnothing, J) \bowtie L$ with nontrivial $A_F$ corresponding to a knot in $S^1 \x D^2$]
\label{Ex:SpliceInto3CompHypLink}
Let $L=(L_0,L_1,L_2)$ be the 3-component hyperbolic KGL shown in Figure \ref{F:SpliceInto3CompHypLink}, where $L_0$ is the red circular component, $L_1$ is the green elliptical component, and $L_2$ is the blue component in the shape of $\infty$.  
Then $B_{L,0,1}\cong \Z/2$, generated by a rotation by $\pi$ in the plane.  (Its full symmetry group is $B_L \cong \Z/2 \oplus \Z/2$ and has rotations by $\pi$ about the $x$- and  $y$-axes.)  The action of $B_{L,0,1}$ reverses the orientation of $L_2$.
Let $F$ be the result of splicing a single knot $J$ along $L_2$.  

If $J$ is invertible, $\L_f/SO_4 \simeq S^1 \x \left(S^1 \x_{\Z/2} \tL_J/SO_4 \right)$.
The left-hand factor of $S^1$ in $\L_f/SO_4$ can be taken to correspond to a full loop of reparametrizations of either component.  The remaining $S^1$ factor corresponds to the rotation by $\pi$ followed by half-loops of reparametrizations of both components $L_0$ and $L_1$.  The action of $\Z/2$ is thus given by rotation by $\pi$ on $S^1$ and by knot inversion on $\tL_J/SO_4$.  If $J$ is a torus knot, $\tL_J/SO_4 \simeq S^1$ and knot inversion corresponds to complex conjugation.
In this case $\L_f/SO_4$ is the product of a circle and a Klein bottle.  

If $J$ is not invertible, $\tL_f/SO_4$ is simply $(S^1)^2 \x  \tL_J/SO_4$.
\end{example}

\begin{figure}[h!]
\includegraphics[scale=0.16]{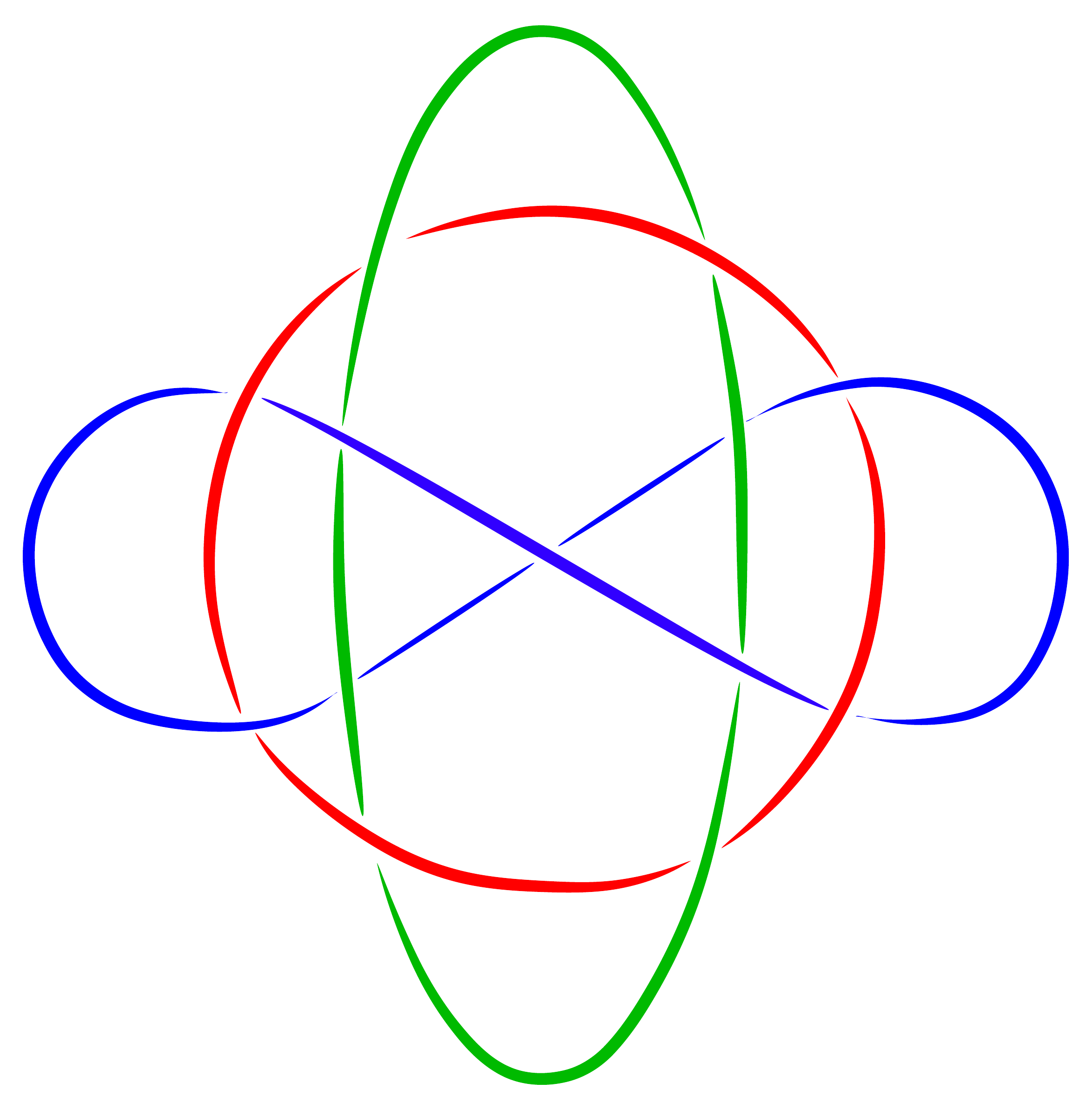}
\qquad \qquad
\begin{tikzpicture}
[grow=north, level distance=40pt]
\node[]{}
child {node[draw,circle] {$L$} 
child{node[draw,circle] {$J$}} 
child{[]}
};
\end{tikzpicture}
\caption{The link $L$ in Example \ref{Ex:SpliceInto3CompHypLink}, together with a schematic of $\mathbb{G}_F$ for $F = (\varnothing, J) \bowtie L$.  }
\label{F:SpliceInto3CompHypLink}
\end{figure}

\begin{example}[A link $F=(\varnothing, J,K,K) \bowtie L$ that gives a knot $S^1 \x D^2$ and has $A_F$ permuting the $K$'s]
\label{Ex:5CompHypLink}
Along similar lines to the previous example, let $L=(L_0,\dots, L_4)$ be either of the 5-component hyperbolic KGL's shown in Figure \ref{F:5CompHypLink}, where $L_0$ is the red circular component, $L_1$ is the green elliptical component, $L_2$ is the blue component in the shape of $\infty$, and $L_3$ and $L_4$ are the remaining components.  
As in the previous example, its full symmetry group is $B_L$, which is the Klein four group.
Let $F =   (\varnothing, J, K, K) \bowtie (L_0,\dots, L_4)$ be the result of splicing a knot $J$ along $L_2$, and a knot $K$ along both $L_3$ and $L_4$.  
If $J$ is invertible, then $A_F = B_{L,0,1} \cong \Z/2$, given by rotation by $\pi$ in the plane.

Then $\L_f/SO_4 \simeq S^1 \x \left(S^1 \x_{\Z/2} \left(\tL_J/SO_4 \x (\tL_K/SO_4)^2\right)\right)$, where the action is given by inversion on $\tL_J/SO_4$ and by swapping the two factors of $\tL_K/SO_4$.  
If $J$ is not invertible, of if we splice two different knots into $L_3$ and $L_4$, then $A_F=\{e\}$, and all the various spaces are just ordinary products.
\end{example}

\begin{figure}[h!]
\includegraphics[scale=0.18]{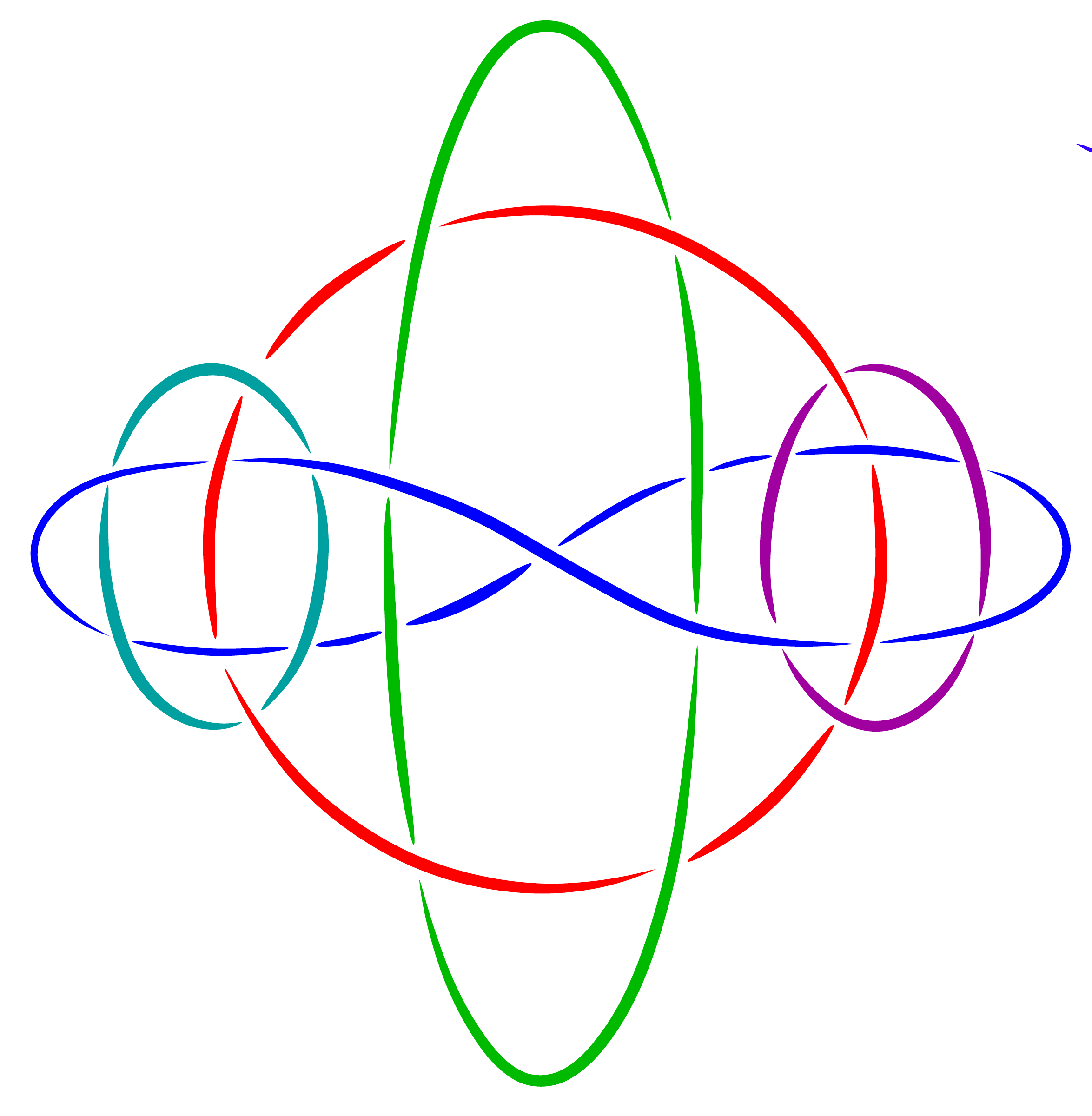}
\qquad 
\includegraphics[scale=0.18]{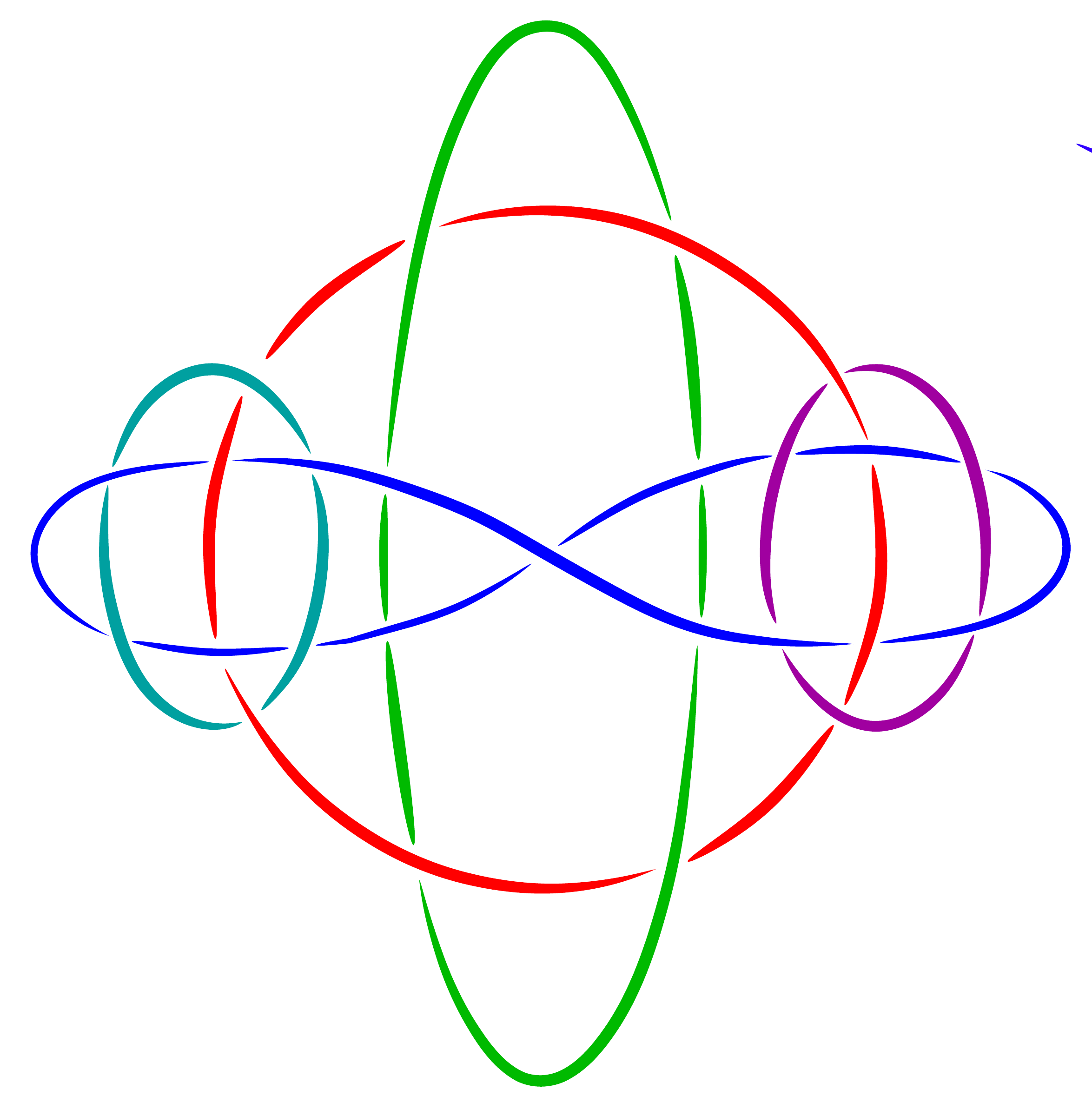}
\qquad 
\begin{tikzpicture}
[grow=north, level distance=40pt]
\node[]{}
child {node[draw,circle] {$L$} 
child{node[draw,circle] {$K$}} 
child{node[draw,circle] {$K$}} 
child{node[draw,circle] {$J$}} 
child{[]}
};
\end{tikzpicture}
\caption{Two possibilities for the link $L$ in Example \ref{Ex:5CompHypLink}, together with a schematic of $\mathbb{G}_F$ for $F = (\varnothing, J, K, K) \bowtie L$.}
\label{F:5CompHypLink}
\end{figure}

\begin{example}[Variant of an example of Stoimenow: $F$ giving a knot in $S^1 \x D^2$ where $A_F$ acts by $\Z/r$]
\label{Ex:Stoimenow}
In the previous two examples, we saw links $L$ with nontrivial representations $B_{L,0,1} \to \mathfrak{S}_r^\pm$ where $B_{L,0,1} \cong \Z/2$.  
A slight variation of an example of Stoimenow \cite[Section 5]{BudneyTop} gives a link $L$ with $B_{L,0,1}$ cyclic of any order and $B_{L,0,1} \to \mathfrak{S}_r^\pm$ faithful (i.e.~injective).  This representation is contained in $\mathfrak{S}_r < \mathfrak{S}_r^\pm$.
 Let $r\geq 3$ be any integer, and let $L=(L_0, \dots, L_{r+1})$ be the link shown in Figure \ref{F:Stoimenow}.  Let $L_0$ be the blue component which links with every other component, and let $L_1$ be the red elliptical component.
The group $B_{L,0,1}$ is $\Z/r$, generated by a rotation of $S^3$ in the $xy$-plane by $2\pi/r$.  
(The full symmetry group $\Isom(C_L)$ is $D_{2r}$, with all isometries extending to the link, and $\Isom^+(C_L) \cong D_r$.)  
The components  $L_2, \dots, L_{r+1}$ are are cyclically permuted by this symmetry.  Let $F = (\varnothing, J, \dots, J) \bowtie L$.  Then $\L_f/SO_4 \simeq S^1 \x \left(S^1 \x_{\Z/r} \left( (\tL_J/SO_4)^r \right) \right)$, where the generator of $\Z/r$ acts on the right-hand factor by the permutation $(1, 2, \dots, k)$ and on the left-hand factor by rotation by $2\pi/r$.  

\end{example}

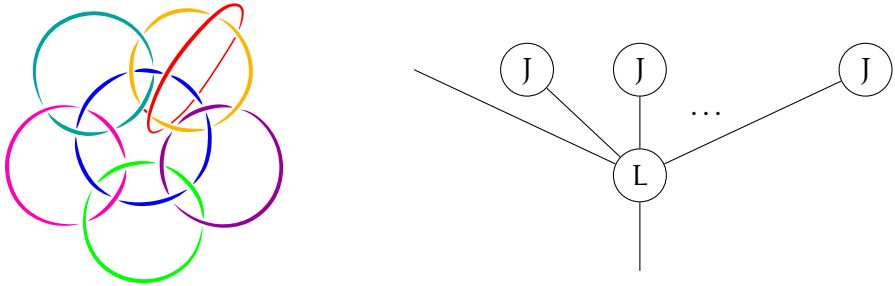
\begin{figure}[h!]
\includegraphics[scale=0.2]{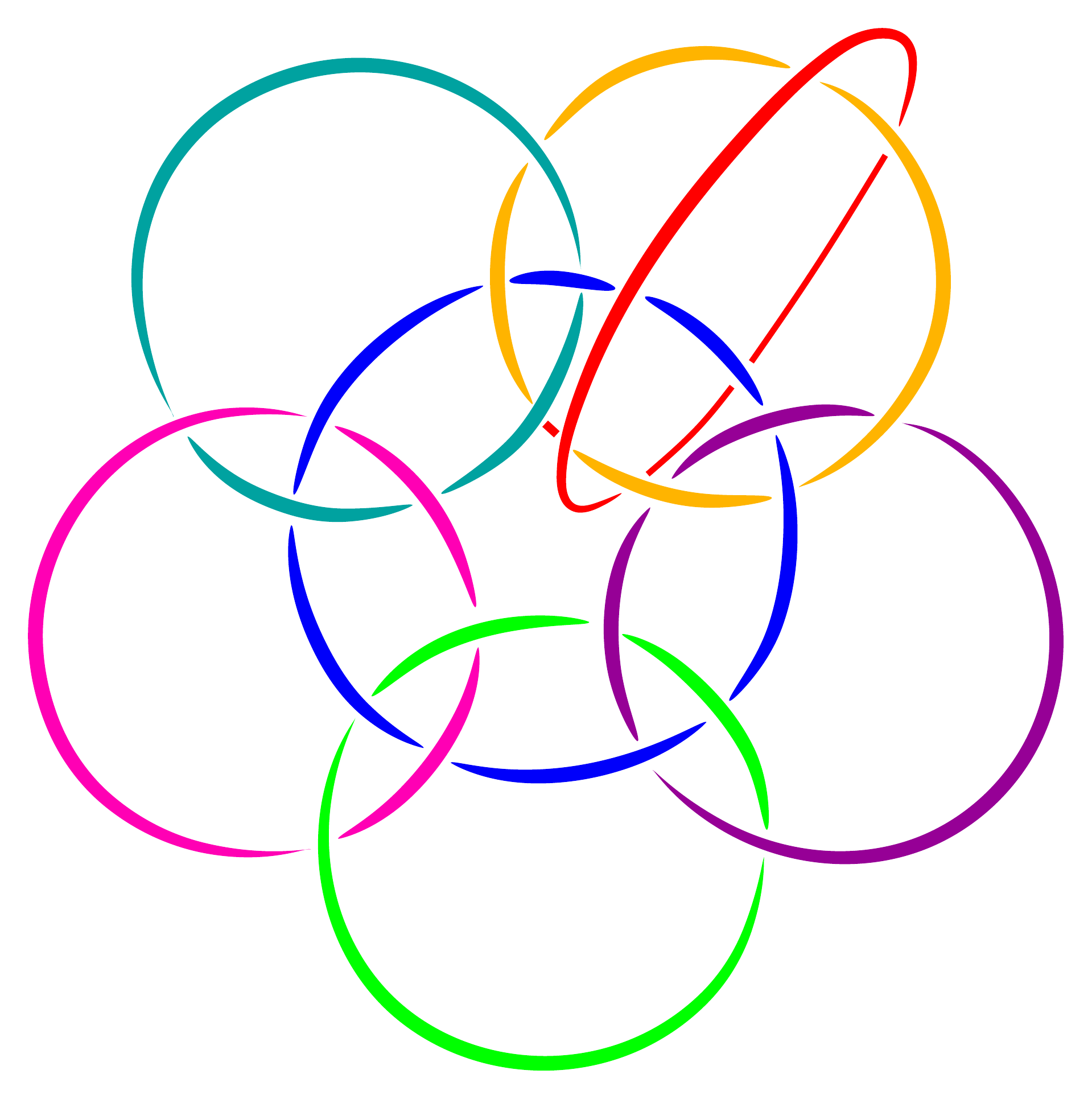} 
\qquad \qquad
\begin{tikzpicture}
[grow=north, level distance=40pt]
\node[]{}
child {node[draw,circle] {$L$}
child{node[draw,circle] {$J$}}
child { edge from parent[draw=none] node[draw=none] (ellipsis) {$\ldots$} }
child{node[draw,circle] {$J$}}
child{node[draw,circle] {$J$}}
child{[]} };
\end{tikzpicture}
\caption{
A hyperbolic KGL $L$ with $B_{L,0,1} \cong \Z/r$ that embeds in $\mathfrak{S}_r$ for $r=5$, together with a schematic of $\mathbb{G}_F$ for $F= (\varnothing, J,\dots, J) \bowtie L$.    The sublink $(L_0, \dots, L_r)$ is an example due to Stoimenow.  See Example \ref{Ex:Stoimenow}.
}
\label{F:Stoimenow}
\end{figure}

\begin{example}[Variant of an example of Sakuma: $F$ giving a knot in $S^1 \x D^2$ where $A_F$ acts by a signed cycle]
\label{Ex:Sakuma}
A slight variation on an example of Sakuma \cite{Sakuma1986}, \cite[Example 5.16]{BudneySplicing} gives a link $L$ with a faithful representation $B_{L,0,1} \to \mathfrak{S}_r^\pm$ that does not factor through $\mathfrak{S}_r$ and with $B_{L,0,1}$ cyclic of any even order.
Let $r$ be odd and let $L=(L_0, \dots, L_{r+1})$ be the link shown in Figure \ref{F:Sakuma}, where $L_0$ is the blue torus knot component and $L_1$ is the red elliptical component.  The group $B_{L,0,1}$
is $\Z/2r$, generated by a rotation of $S^3$ that rotates in the $xy$-plane by $2\pi/r$ and rotates by $\pi$ along $L_1=C_2$.
(The full symmetry group $\Isom(C_L)$ is $\Z/2 \x D_{2r}$, though only the isometries in $\Isom^+(C_L) \cong D_{2r}$ extend to the link.)  
The components  $L_2, \dots, L_{r+1}$ are are cyclically permuted by this symmetry.  Let $F = (\varnothing, J, \dots, J) \bowtie L$.  Then $\L_f/SO_4 \simeq S^1 \x \left(S^1 \x_{\Z/2r} \left( (\tL_J/SO_4)^r \right) \right)$, where the generator of $\Z/2r$ 
acts on the left-hand factor by rotation by $\pi/r$ and on the right-hand factor by the signed permutation given by the cycle $(1, 2, \dots, r)$ and a minus sign for every $i=1,\dots,r$.
\end{example}

\begin{figure}[h!]
\includegraphics[scale=0.2]{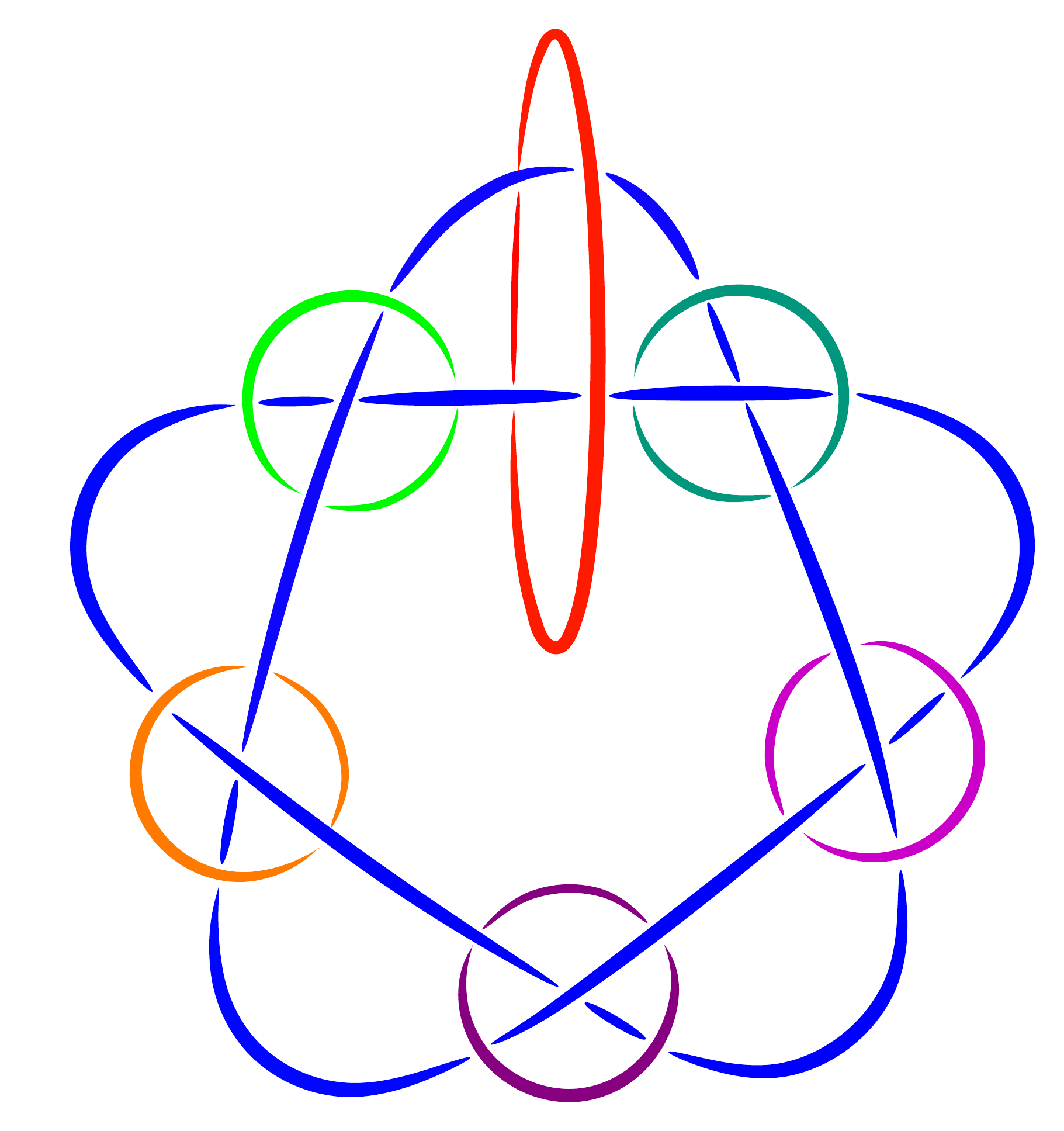} 
\qquad \qquad
\begin{tikzpicture}
[grow=north, level distance=40pt]
\node[]{}
child {node[draw,circle] {$L$}
child{node[draw,circle] {$J$}}
child { edge from parent[draw=none] node[draw=none] (ellipsis) {$\ldots$} }
child{node[draw,circle] {$J$}}
child{node[draw,circle] {$J$}}
child{[]} };
\end{tikzpicture}
\caption{
A hyperbolic KGL $L$ with $B_{L,0,1} \cong \Z/2r$ that embeds as a signed cycle in $\mathfrak{S}_r^\pm$ for $r=5$, together with a schematic of $\mathbb{G}_F$ for $F=(\varnothing, J,\dots, J) \bowtie L$.  The sublink $(L_0, \dots, L_r)$ is an example due to Sakuma.  See Example \ref{Ex:Sakuma}.
}
\label{F:Sakuma}
\end{figure}
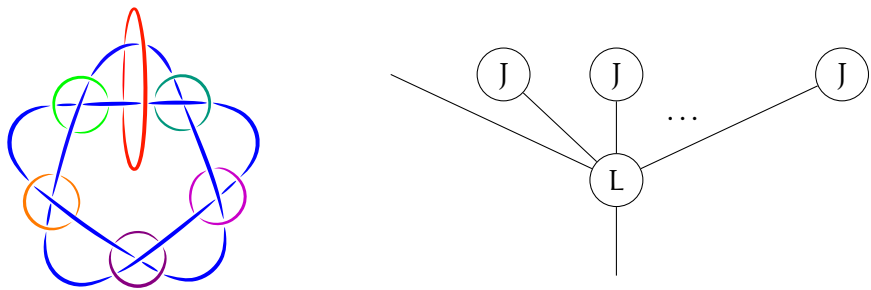

%----------------------------------------------------------------------------------
%BEGIN EXERCISE(S)
%----------------------------------------------------------------------------------

\begin{exercise}
%[A Whitehead double in solid torus]
\label{Ex:WhiteheadDoubleInSolidTorus}
Find the companionship tree $\mathbb{G}_f$, and compute $\pi_1(\L_f/SO_4)$ for the links $f$ shown in 
(a) Figure \ref{F:WhiteheadDoubleInSolidTorus}, 
(b) Figure \ref{F:AnotherWhiteheadDouble}, and
(c) Figure \ref{F:WhSumTrefoil}.
Notice that each of these corresponds to a knot in the solid torus.  \emph{Hints}: the links are all distinct, but some of the fundamental groups may be the same.  The answers to the three parts will include both abelian and nonabelian fundamental groups.

% EXERCISE PART (a)
%Consider the link $f$ shown in Figure \ref{F:WhiteheadDoubleInSolidTorus}, which  becomes a Whitehead double upon removing the unknotted component.  
%DELETED TO LEAVE AS EXERCISE
%Splice a knot $J$ into the 3-component link $L$ that has the Whitehead link as a sublink, as shown in Figure \ref{F:WhiteheadDoubleInSolidTorus}.  The link $L$ has a satellite decomposition into the Whitehead link and the 3-component keychain link $KC_2$, where the splicing is done along the component $K_0$ of $KC_2$ and either component, say $W_1$, of the Whitehead link $\Wh=(W_0,W_1)$.  
%The group $A_F$ is  trivial for any choice of $J$.
%Indeed, the symmetry of $\Wh$ preserving $W_0$ and its orientation reverses the orientation of $W_1$ and hence the orientation of $K_0$.  Because of the remaining unknotted component in $(J, \varnothing)\bowtie KC_2$, this symmetry yields a link inequivalent to $(J, \varnothing) \bowtie KC_2$. 
%
%Thus
%$\L_f/SO_4 \simeq (S^1)^2 \x \Conf(2, \R^2) \x \tL_J/SO_4$.
%So if $J$ is a torus knot, $\pi_1(\L_f/SO_4) \cong \Z \langle \lambda_0, \lambda_1, \lambda', g_J \rangle$, where $\lambda'$ is loop of longitudinal rotations of a solid torus knotted in the shape of $J$.  
%Equivalently (mod $SO_4$), $\lambda'$ pushes the unknotted component along $J$, as seen by putting $J$ in its standard, symmetric position in $S^3$
%The loop $\mu'$ of meridional rotations of this solid torus satisfies $\mu'+ g_J=\lambda_1$.
\end{exercise}

\begin{figure}[h!]
%DELETED TO LEAVE AS EXERCISE
%\includegraphics[scale=0.25]{whitehead-sum-empty.pdf}
%\qquad \qquad
\includegraphics[scale=0.26]{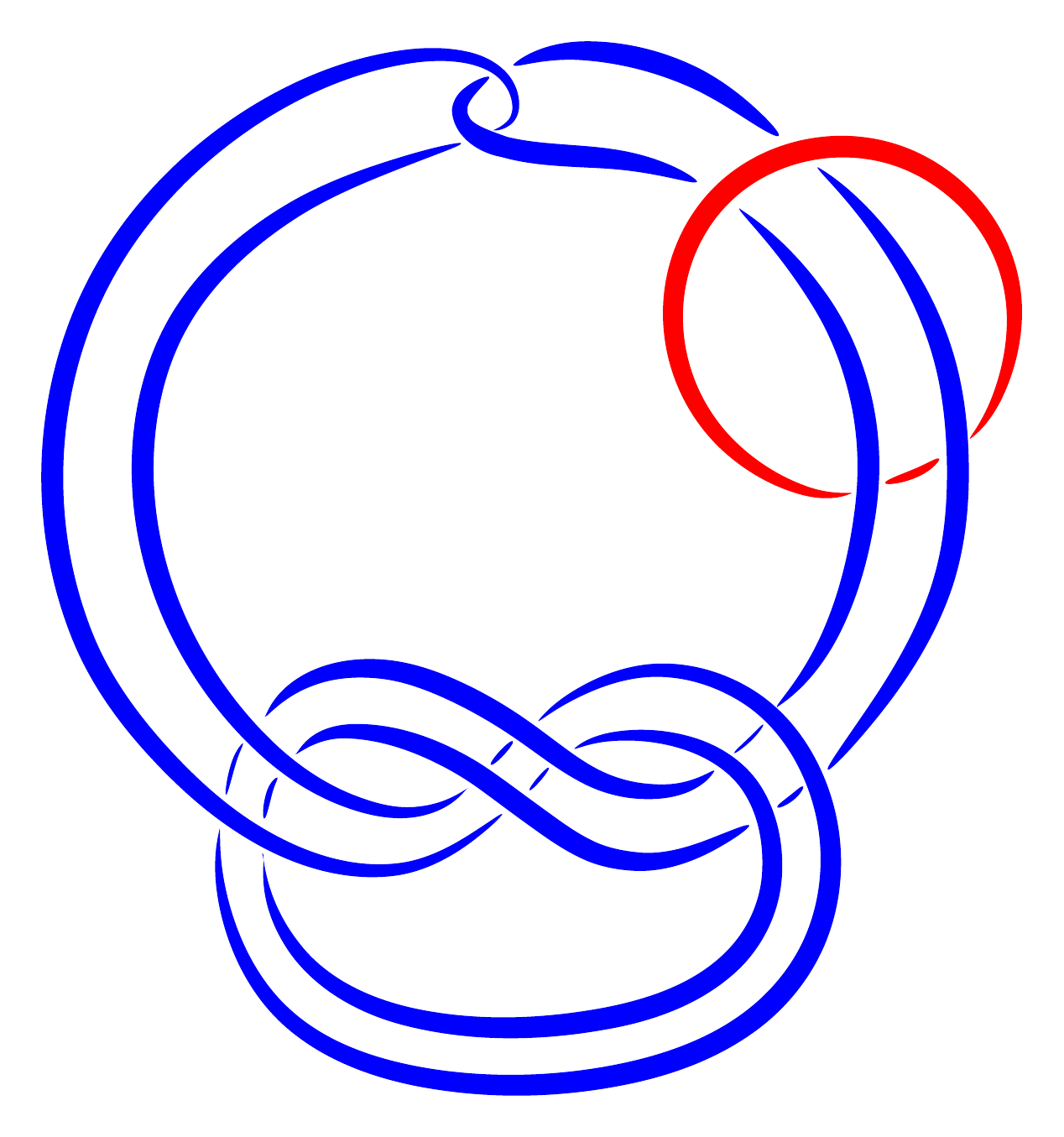}
\qquad \qquad
%DELETED TO LEAVE AS EXERCISE
%\begin{forest}
%for tree={grow=north, l sep=25pt, s sep=35pt}
%[ [Wh, circle, draw, edge label={node[midway, right] {$\lambda_0$ (and $\mu_0$)}}
%[$KC_2$, circle, draw,   edge label={node[midway, right] {$\lambda'$ (and $\mu'=\lambda_1 - g_J$)}}
%[, edge label={node[midway, right] {$\lambda_1$ (and $\mu_1$) }}] [$J$, circle, draw, edge label={node[midway, left] {$g_J$}}  ] ] 
%] ]
%\end{forest}
\caption{The link in Exercise \ref{Ex:WhiteheadDoubleInSolidTorus}, part (a).}
%DELETED TO LEAVE AS EXERCISE
%$L:= KC_2 \bowtie W$ and $F = J \bowtie L$ from Example \ref{Ex:WhiteheadDoubleInSolidTorus}, with $J$ the trefoil, together with the tree $\mathbb{G}_F$.  
%Here $\pi_1(\L_f/SO_4) \cong \Z\langle \lambda_0, \lambda_1, \lambda', g_J\rangle$.}
\label{F:WhiteheadDoubleInSolidTorus}
\end{figure}

\begin{figure}[h!]
\includegraphics[scale=0.26]{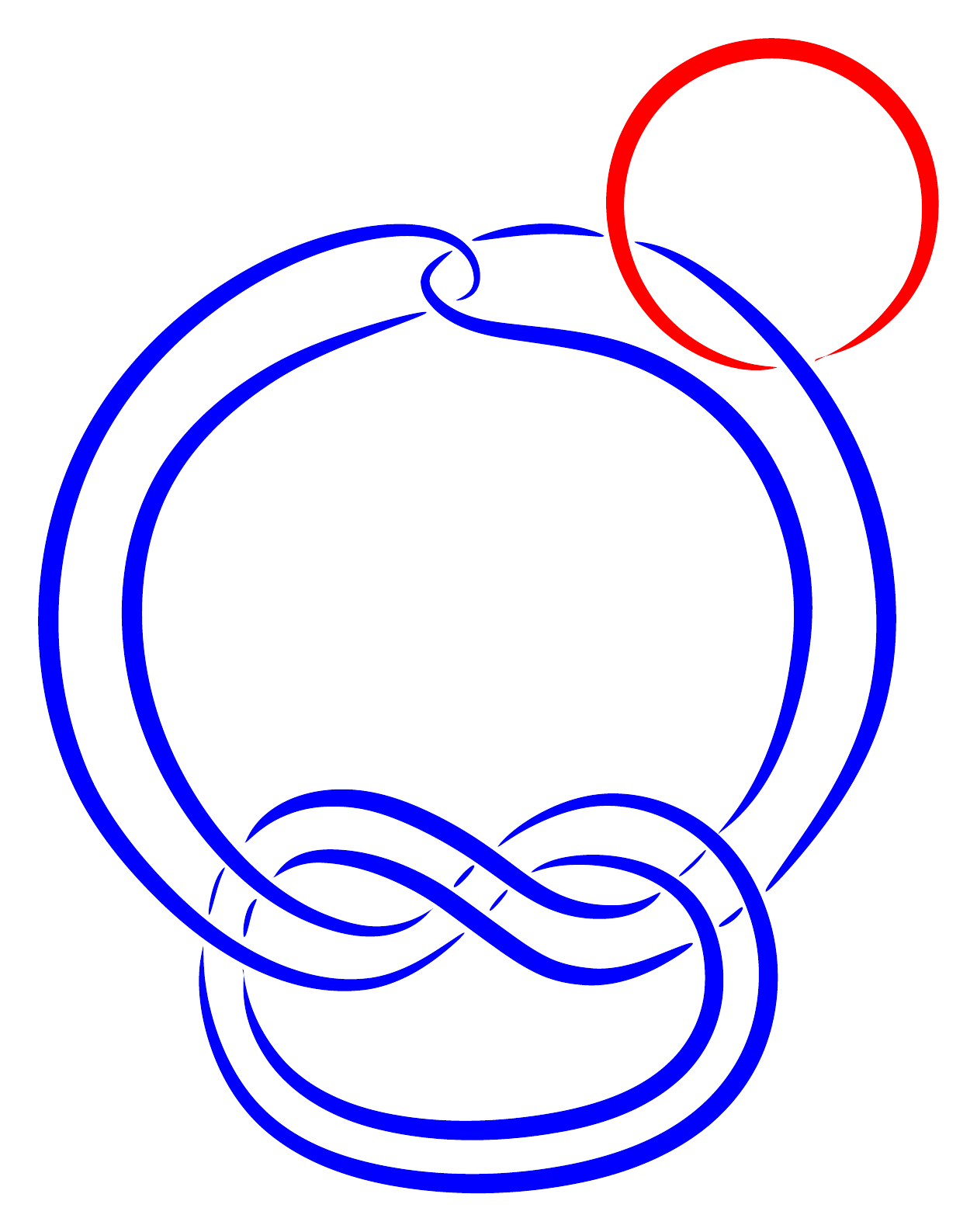}
\qquad \qquad
%DELETED TO LEAVE AS EXERCISE
%\begin{forest}
%for tree={grow=north, l sep=25pt, s sep=45pt}
%[ [$KC_2$, circle, draw, edge label={node[midway, right] {$\lambda_0$ (and $\mu_0$)}}
%[, edge label={node[midway, right] {$\lambda_1$ (and $\mu_1$) }}]
%[Wh, circle, draw,   edge label={node[midway, left] {$\lambda'$ (and $\mu' = g_{\Wh(J)} = \lambda_1 - \mu_0$)}}
% [$J$, circle, draw, edge label={node[midway, left] {$g_J=\mu_J$}}  ] ] 
%] ]
%\end{forest}
\caption{The link in Exercise \ref{Ex:WhiteheadDoubleInSolidTorus}, part (b).}
%DELETED TO LEAVE AS EXERCISE
%$F$ in Example \ref{Ex:AnotherWhiteheadDouble}, with $J$ the trefoil, together with the tree $\mathbb{G}_F$.  Here
%$\pi_1(\L_f/SO_4) \cong \Z \langle \lambda_0, \lambda_1 \rangle \x \langle a,\, g_J\  |\ g_J^a = g_J^{-1}\rangle$ where $a = (\lambda')^{1/2} \mu_J^{1/2}$.}
\label{F:AnotherWhiteheadDouble}
\end{figure}

% EXERCISE PART (c)
%DELETED TO LEAVE AS EXERCISE
%The complement of  $F=(Wh, J) \bowtie KC_2$ has a JSJ decomposition into the same three 3-manifolds as the complements of the links in Examples \ref{Ex:WhiteheadDoubleInSolidTorus} and \ref{Ex:AnotherWhiteheadDouble}.  
%This link also produces a knot in $S^1 \x D^2$, though in $S^3$ the knotted component $f_0$ is $J$, not $\Wh(J)$.  
%Thus the link $F$ is distinct from the ones in the previous examples, though we get the same homotopy type for $\L_f/SO_4$ as in  \ref{Ex:WhiteheadDoubleInSolidTorus}.  If $J$ is a torus knot, $\pi_1(\L_f/SO_4) \cong \Z\langle \lambda_0, \lambda_1, \lambda', g_J\rangle$ where $\lambda'$ can be seen in the second picture in Figure \ref{F:WhSumTrefoil} by pushing the red unknotted component along the knotted component.
%\end{exercise}

\begin{figure}[h!]
\includegraphics[scale=0.26]{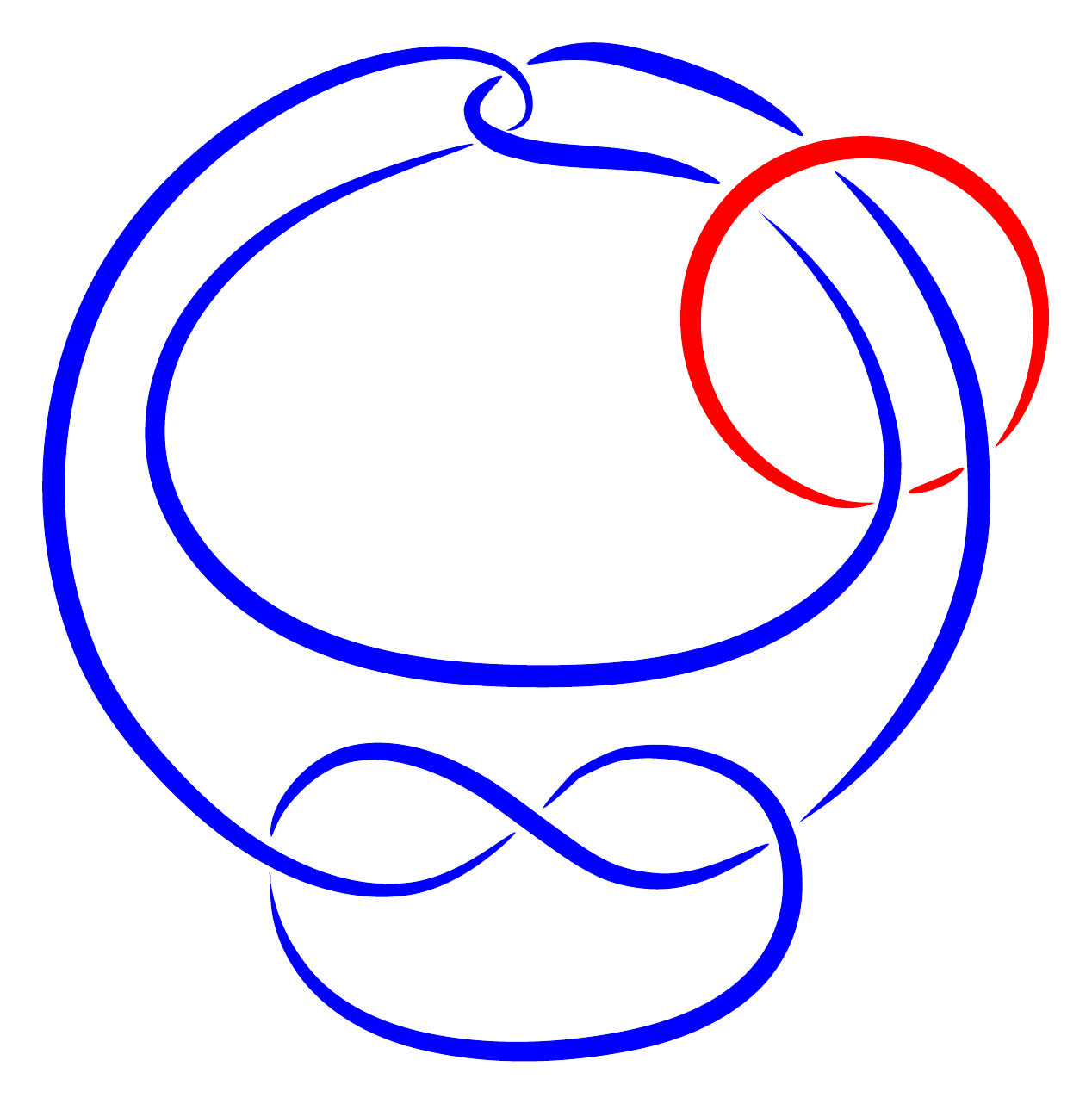}
%DELETED TO LEAVE AS EXERCISE
%\quad 
%\includegraphics[scale=0.25]{whitehead-sum-trefoil-3.pdf}
%\quad
%DELETED TO LEAVE AS EXERCISE
%\begin{forest}
%for tree={grow=north, l sep=25pt, s sep=10pt}
%[ [$KC_2$, circle, draw, edge label={node[midway, right] {$\lambda_0$ (and $\mu_0$)}}
%[$J$, circle, draw,   edge label={node[midway, right]{$g_J$ (and $\lambda_J=1$)}}]
%[Wh, circle, draw, edge label={node[midway, left]  {$\lambda'$ (and $\mu' = \lambda_J -\mu_0$)}} 
%[, edge label={node[midway, left] {$\lambda_1$ (and $\mu_1$) }}  ] ] 
%] ]
%\end{forest}
\caption{The link in Exercise \ref{Ex:WhiteheadDoubleInSolidTorus}, part (c).}
%Two planar projections of the link $F=(Wh, J) \bowtie KC_2$ in Example \ref{Ex:ThirdExample}, together with $\mathbb{G}_F$.  Here $\pi_1(\L_f/SO_4) \cong \Z\langle \lambda_0, \lambda_1, \lambda', g_J\rangle$.}
\label{F:WhSumTrefoil}
\end{figure}

\newpage

\section{Spaces of split links}
\label{S:SplitLinksTorus}

In this section, we consider spaces of framed links that are split (i.e.~reducible), as well knots in a solid torus or a thickened torus which correspond to split links.
See Figure \ref{F:SplitLinkInSolidTorus} for an example of such a knot and the associated link.
In Proposition \ref{P:SplitLinks}, which is part (B) of Theorem \ref{MainT:FramedLinks}, we give a relationship between the space $\tL_F/SO_4$ of a split framed link $F$ and the spaces of the irreducible sublinks of $F$, via a space of embeddings of a punctured 3-sphere in the complement of $F$.  
In Proposition \ref{P:SplitLinksTwoSummands}, we give a more precise description of $\tL_F/SO_4$ in the case that $F$ has just two irreducible summands.
Proposition \ref{P:SplitKnots}, which has a similar proof, describes spaces of knots contained in a 3-ball in an orientable irreducible 3-manifold.
Corollaries \ref{TfSplit} and \ref{VfSplit} gives the homotopy types of $\T_f$ and $\V_f$ of spaces in terms of spaces of knots in a 3-ball; the latter spaces can be described in terms of spaces of long knots.  
This section does not rely on any of the material after Section \ref{S:Asphericity}.

\begin{figure}[h!]
\includegraphics[scale=0.2]{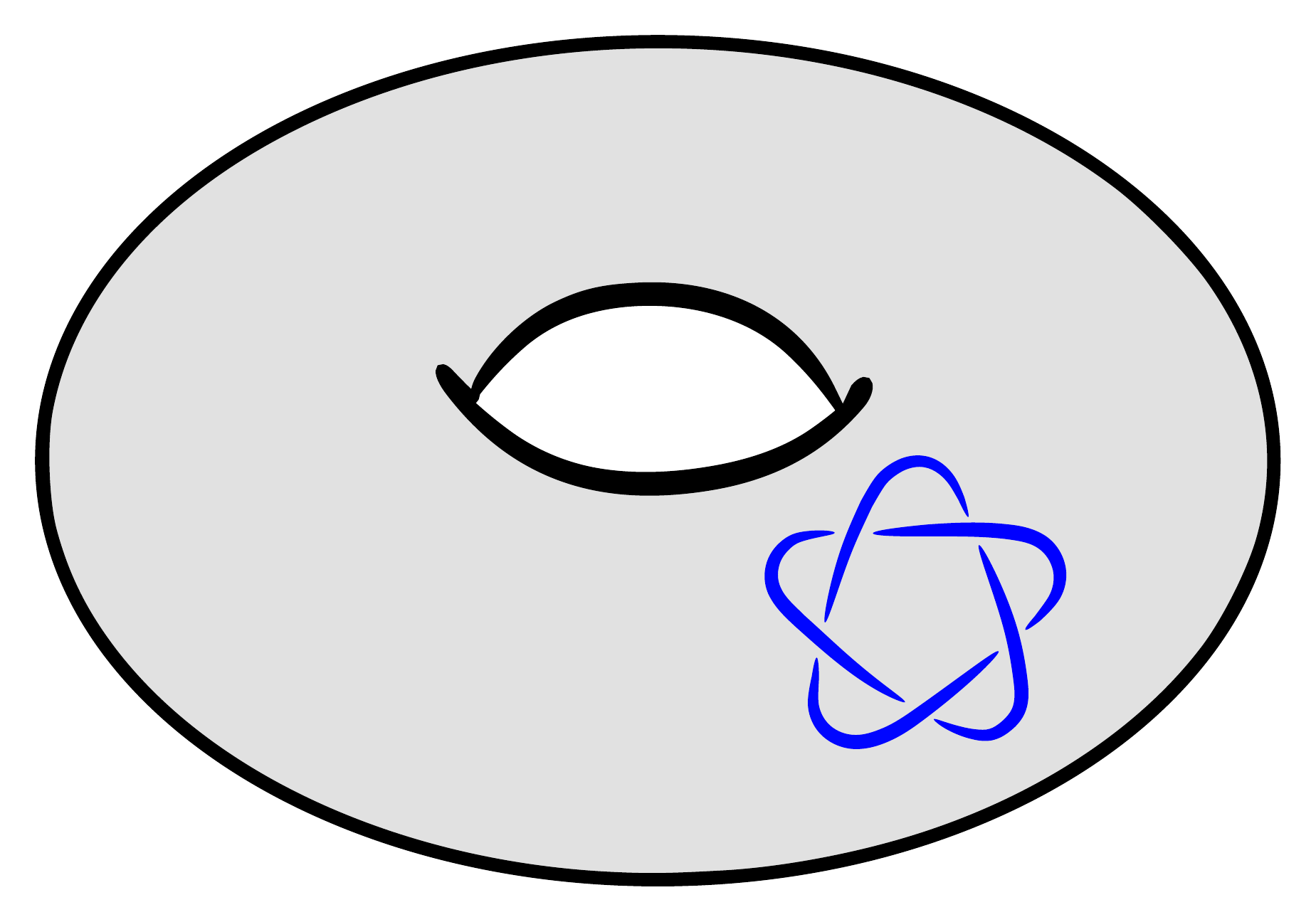}
\qquad \qquad \qquad
\raisebox{1.7pc}{\includegraphics[scale=0.4]{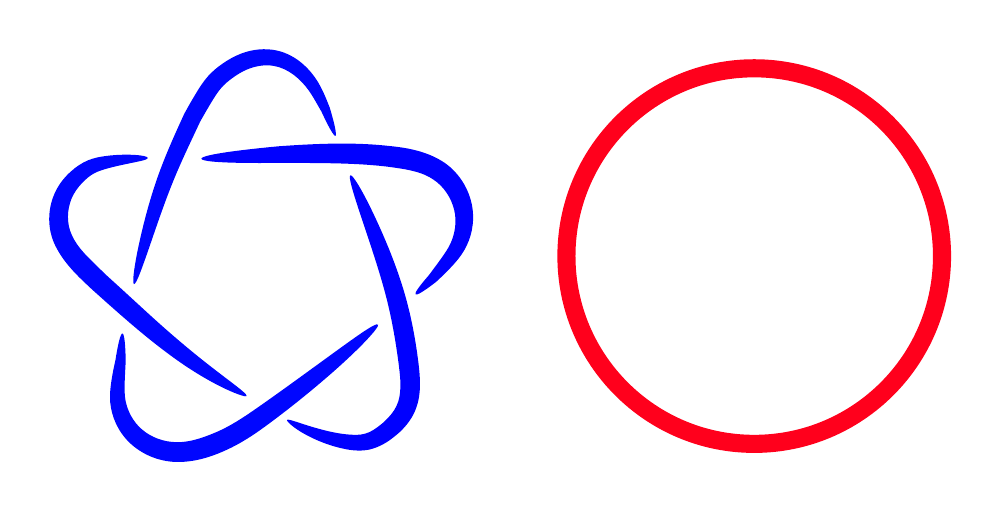}}
\caption{An example of a knot in a solid torus (left), where the associated link is a split link (right).}
\label{F:SplitLinkInSolidTorus}
\end{figure}

Write $F=F_1 \sqcup \dots \sqcup F_k$ if $F$ is a disjoint union of links $F_i$ such that there are disjoint 3-balls $B_1, \dots, B_k$ such that each $B_i$ contains $F_i$.
In this situation, the complement $C_F$ is diffeomorphic to the connected sum of complements $C_{F_1} \# \dots \# C_{F_k}$.  
Let $Q_k:=S^3 - \coprod_{i=1}^k B_i$.
Let $\Emb_0(Q_k, C_F)$ be the subspace of embeddings of $Q_k$ into $C_F$ which extend to diffeomorphisms in $\Diff(C_F; \d C_F)$.  
Below, we write $F \rtimes B$ for a fiber bundle with base $B$ and fiber $F$.  As with semi-direct products of groups, this notation may not completely describe the object (even up to homotopy equivalence).  

\begin{proposition}
\label{P:SplitLinks}
Let $F$ be a split framed link in $S^3$ with a decomposition $F = F_1 \sqcup \dots \sqcup F_k$ into irreducible framed links $F_1, \dots, F_k$.  
Let $C_F$ be the complement of $F$, and let $Q_k \subset C_F$ and $\Emb_0(Q_k, C_F)$ be as above.  
Let $\Emb_{F_i}\left(\coprod S^1 \x  D^2, D^3\right)$ be the component of a framed link in $D^3$ corresponding to $F_i$ in the space of framed links in $D^3$.
Then up to homotopy there is a fibration sequence 
\begin{equation}
\label{Eq:SplitLinksFibn}
\Emb_0 \left(Q_k, C_F\right) \to \prod_{i=1}^k \Emb_{F_i}\left(\coprod S^1 \x  D^2, D^3\right)  \to \tL_F/SO_4
%\begin{split}
%\dots & \to \pi_j\left(\Emb_0 \left(Q_k, C_F\right)\right) \to 
%\pi_j\left(\prod_{i=1}^k \Emb_{F_i}\left(\coprod S^1 \x  D^2, D^3\right) \right)  \to 
%\pi_j\left(\tL_F/SO_4\right) \to \dots \\
%\dots & \to \pi_0\left(\Emb_0 \left(Q_k, C_F\right)\right)
%\end{split}
\end{equation}
and  for $i=1,\dots,k$, there is an equivalence $\Emb_{F_i}\left(\coprod S^1 \x  D^2, D^3\right)\simeq SO_3 \x (C_{F_i} \rtimes (\tL_{F_i}/SO_4))$.
%Then
%\[
%\tL_F/SO_4 \simeq \mathcal{C}(C_F) \rtimes  \prod_{i=1}^n \tL_{F_i}/SO_4.
%\]
The map $\tL_F/SO_4 \to \prod_{i=1}^k\tL_{F_i}/SO_4$ given by restricting to the sublinks $F_1, \dots, F_n$ induces a surjection in $\pi_1$.
\end{proposition}

\begin{proof}
Restricting a diffeomorphism of $C_F$ to $Q_k$ gives a 
%principal bundle  
fibration
\begin{equation}
\label{Eq:SplitPfFibn1}
\prod_{i=1}^k \Diff(C_{F_i} - D^3; \d) \to \Diff(C_F; \d C_F) \to \Emb_0(Q_k, C_F)
\end{equation}
which is surjective on path components.
%%We can apply the classifying space functor $B(-)$ to the inclusion of the fiber above
%We can extend this sequence one step to the right via the classifying map from the base space to $B(-)$ of the fiber, and then one step further by taking $B(-)$ of the inclusion of the fiber.  We thus obtain a fibration sequence 
%
By considering the homotopy fiber of the first map above, we can view it up to homotopy as a fibration
\begin{equation}
\label{Eq:SplitPfFibn2}
 \Omega \Emb_0(Q_k, C_F) \to \prod_{i=1}^k \Diff(C_{F_i} - D^3; \d) \to \Diff(C_F; \d C_F).
 \end{equation}
The classifying space functor $B(-)$ can be applied to H-spaces as well as groups, and for a connected space $X$, $B\Omega X \simeq X$.
By considering \eqref{Eq:SplitPfFibn1} one component in the base space at a time, and
then applying $B(-)$ it to the resulting sequence \eqref{Eq:SplitPfFibn2}, we obtain the fibration
%%Applying the classifying space functor $B(-)$ to this fibration and then shifting it to the left by taking $\Omega(-)$ of the base space yields 
\[
\Emb_0(Q_k, C_F) \to \prod_{i=1}^k B\Diff(C_{F_i} - D^3; \d) \to B\Diff(C_F; \d C_F).
\]
By Proposition \ref{FramedLinksKpi1}, the base space above is equivalent to $\tL_F/SO_4$.  By a similar argument using the contractibility of $\Diff(D^3; \d D^3)$, each factor in the total space is equivalent to the component $\Emb_{F_i}(\coprod S^1 \x D^2, D^3)$ corresponding to $F_i$ in the space of framed links in $D^3$.

We next check that for any framed link $J$, $\Emb_{J}\left(\coprod S^1 \x  D^2, D^3\right)\simeq SO_3 \x (C_{J} \rtimes (\tL_{J}/SO_4))$.  %If the $F_i$ are all knots, a theorem of Budney and Cohen \cite[Proposition 4.4]{Budney-Cohen} \cite[Proposition 2.2]{Budney:Family} from unframed knots to framed links and 
Notice that 
\[
\Emb\left(\coprod S^1 \x  D^2, D^3_+\right) \cong \Emb\left(D^3_- \sqcup \coprod S^1 \x  D^2, S^3\right)/\Diff^+(S^3)
\]
where $D^3_+$ and $D^3_-$ are respectively the upper and lower hemispheres of $S^3$.  Indeed, the map ``$\to$'' is given by taking fixed standard embeddings $e_+$ and $e_-$ of $D^3_+$ and $D^3_-$ into $S^3$.  The inverse ``$\leftarrow$'' is given by first applying an element of $\Diff^+(S^3)$ to get a representative where the embedding $D^3_-$ is $e_-$ and applying $e_+^{-1}$ to get a link in $D^3$.  Since $\Diff^+(S^3)\simeq SO_4$, we get a fibration 
\begin{equation}
\label{Eq:LinkInD3toLinkInS3Fibn}
\Emb(D^3, C_J) \to \Emb_J\left(D^3 \sqcup \coprod S^1 \x  D^2, S^3\right)/SO_4 \to \tL_J/SO_4
\end{equation}
where the fiber is equivalent to $C_J \x SO_3$.  In this fibration, $SO_3$ splits as a direct factor because 
$S^3$ is parallelizable.
%$C_J$ is an orientable 3-manifold and hence parallelizable.

For the last statement, consider the commutative diagram 
\[
\xymatrix{
\prod_{i=1}^k \Emb_{F_i}\left(\coprod S^1 \x  D^2, D^3\right)  \ar[rr] \ar[dr] & & \tL_F/SO_4 \ar[dl] \\
&  \prod_{i=1}^k \tL_{F_i}/SO_4 &
}
\] 
where the map from the top-left space is product of the fibrations \eqref{Eq:LinkInD3toLinkInS3Fibn} for $J=F_1,\dots,F_k$.  The long exact sequence of \eqref{Eq:LinkInD3toLinkInS3Fibn} in homotopy shows that this map is surjective in $\pi_1$, hence so is the map from $\tL_F/SO_4$.
\end{proof}

We study neither $\Emb(Q_k, C_F)$ nor the above fibration sequence in further detail.  The statement of Proposition \ref{P:SplitLinks} resembles a result of Hendriks and McCullough \cite{Hendriks-McCullough}, though its proof is independent of that work, using only restriction fibrations and the Smale conjecture.  Statements of C\'esar da S\'a and Rourke \cite{CesarDeSaRourke} and Hatcher \cite{Hatcher:Reducible3Mfds} could lead to a more effective result than Proposition \ref{P:SplitLinks}, but neither statement has a complete proof.  The (proven) closely related result of Hendriks and Laudenbach \cite{HendriksLaudenbach} applies only to 3-manifolds with a 2-sphere boundary component.
Nonetheless, in the case $k=2$, we can provide a more effective description of $\tL_F/SO_4$.

\begin{proposition}
\label{P:SplitLinksTwoSummands}
Suppose $F=F_1 \sqcup F_2$ where $F_1$ and $F_2$ are irreducible links.  Then 
\[
\tL_F/SO_4 \simeq  SO_3 \x ((C_{F_1} \x C_{F_2}) \rtimes (\tL_{F_1}/SO_4 \x \tL_{F_2}/SO_4)).
\]
\end{proposition}

\begin{proof}
Let $\mathrm{Sph}(C_F)$ be the space of (unparametrized) separating 2-spheres in the complement $C_F$ of $F$; that is, $\mathrm{Sph}(C_F)$ is a component of $\Emb(S^2, C_F)/\Diff(S^2)$.  
The theorem in \cite{Hatcher:1981} together with the validity of the Smale Conjecture implies that it is contractible.  (The contractibility of $\mathrm{Sph}(M)$ for $M$ a connected sum of two irreducible 3-manifolds is equivalent to the fact that $\Diff(M) \simeq Diff(M,S)$ where $S$ is a separating 2-sphere.)  Proofs of this contractibility and the latter equivalent statement can also be found in \cite{Hatcher:Reducible3Mfds} and \cite[Remark 3.10]{Nariman:Reducible3Mfds}.

Now define $ES_F$ by
\[
ES_F:= \left. \left\{ \left(J: \coprod S^1 \x D^2 \incl M, g: S^2 \incl M \right) :  
 \mbox{$J$ is isotopic to $F$ and $g$ separates $J$} \right\} \ \right/ \ 
 %(g \sim -g)
 \Diff(S^2).
\]
%where $g \sim -g$ means that we identify $g$ with the result of pre-composing by the antipodal map on $S^2$.  
The contractibility of the fiber in the fibration 
\[
\mathrm{Sph}(C_F) \to ES_F/\Diff^+(S^3) \to \Emb_F \left. \left(\coprod S^1 \x D^2, S^3 \right)\right/\Diff^+(S^3)
\]
shows that $ES_F/\Diff^+(S^3)$ is equivalent to the base space above, which is equivalent to $\tL_F/SO_4$.  

Next, we claim that there is a homeomorphism 
\begin{equation}
\label{Eq:ES_F}
ES_F/\Diff^+(S^3) \cong \Emb_{F_1}\left(\coprod S^1 \x D^2, D^3\right) \x_{\Diff^+(S^2)} \Emb_{F_2}\left(\coprod S^1 \x D^2, D^3\right).
\end{equation}
On the right-hand side, a left action of $h \in \Diff^+(S^2)$ on the second factor is given by viewing $S^2$ as $\d D^3$ and post-composing by an extension $H\in \Diff(D^3)$ of $h$.  A right action of such $h$ on the first factor is given by post-composition by $H^{-1}$.  
Indeed, the map ``$\leftarrow$'' sends a class $[(J_1, J_2)]$ on the right-hand side to $[(e_+\circ J_1, e_- \circ J_2,e)] \in ES_F/\Diff^+(S^3)$ where $e_{\pm}:D^3 \incl S^3$ are embeddings of the upper and lower hemispheres and $e:S^2 \incl S^3$ is an embedding of the equatorial 2-sphere.  Conversely, an element in $ES_F/\Diff^+(S^3)$ can be represented by $(J,e)$ by applying a diffeomorphism in $\Diff^+(S^3)$.  Then writing $J=J_1 \sqcup J_2$, we define the inverse by sending $[(J_1 \sqcup J_2,e)]$ to $[(J_1,J_2)]$.

Finally, since $\Diff^+(S^2) \simeq SO_3$, the same argument as in the proof of Proposition \ref{P:SplitLinks} shows that the right-hand side of \eqref{Eq:ES_F} is equivalent to 
\[
((C_{F_1} \rtimes \tL_{F_1}/SO_4) \x SO_3) \x_{SO_3}
(SO_3 \x (C_{F_2} \rtimes \tL_{F_1}/SO_4))
\]
which simplifies to the desired result.
\end{proof}

\begin{example}
In particular, if $F= U \sqcup J$ where $U$ is a framed unknot and $J$ is any framed knot or link, then 
\begin{equation}
\label{Eq:FramedLinkCupUnknot}
\tL_F/SO_4 \simeq SO_3 \x S^1 \x (C_{J} \rtimes  \tL_{J}/SO_4).
\end{equation}
As a further special case, if $F=U_1 \sqcup U_2$ with each $U_i$ a framed unknot, we get that the space of 2-component framed unlinks in $S^3$ modulo rotations is
\[
\tL_F/SO_4 \simeq S^1 \x SO(3) \x S^1.
\]
One can visualize this by using an element of $SO_4$ to fix $U_1$; then one factor of $S^1$ corresponds to a loop where $U_2$ passes through the first, the factor of $SO(3)$ corresponds to rotations of a 3-ball containing $U_2$, and the other factor of $S^1$ corresponds to meridional rotations of $U_2$.  
The spaces of unlinks studied by Brendle and Hatcher \cite{Brendle-Hatcher} are related but different because they study unframed unlinks in $\R^3$ rather than framed links $S^3$, and they work modulo reparametrizations rather than modulo rotations. 
%[[Mention K--Kusner?? even though it's in $\R^3$ and modulo reparametrization and not modulo $SO_4$??]]
\end{example}

To treat unframed knots in a solid torus, we could determine the quotient of \eqref{Eq:FramedLinkCupUnknot} by the meridional rotations of a knot $J$.  Instead we will apply Proposition \ref{P:SplitKnots}, which is closely related to Proposition \ref{P:SplitLinksTwoSummands}, though strictly speaking, the two are logically incomparable.

\begin{proposition}
\label{P:SplitKnots}
Suppose $M$ is an orientable, irreducible 3-manifold not homeomorphic to $S^3$.  
Let $f: S^1 \incl M$ be a knot obtained as a composition of embeddings $j:S^1 \incl D^3$ and $D^3 \incl M$.  
%Let $J: S^1 \x D^2 \incl S^3$ be a framed knot in $S^3$ corresponding to the knot $e_+ \circ j: S^1 \incl S^3$, where $e_+: D^3 \incl S^3$ is the inclusion of the upper hemisphere.
Let $\overline{j}: S^1 \incl S^3$ be the composition of $j$ with 
an embedding $e_+:D^3 \incl S^3$ of the upper hemisphere, and let $\underline{j}$ be the long knot whose closure is $\overline{j}$.
Then the component of $f$ in the space of knots in $M$ satisfies 
\begin{equation}
\label{Eq:SplitKnots}
\Emb_{f}\left(S^1, M\right)
\simeq M \x \Emb_j \left(  S^1, D^3\right)
\simeq M \x (SO_3 \x_{SO_2} (C_{\overline{j}} \rtimes \K_{\underline{j}}))
\end{equation}
where $C_J \rtimes \K_{\underline{j}} = \{ (p,g): g \mbox{ is isotopic to } \underline{j} \mbox{ and } p \in C_g \}$, $SO_2$ acts on $SO_3$ as a subgroup fixing an axis, and $SO_2$ acts on  $C_{\overline{j}} \rtimes \K_{\underline{j}}$ via the Gramain loop.
\end{proposition}

\begin{proof}
The exterior $C_f$ of $f$ in $M$ is the connected sum $C_{\overline{j}} \, \# M$ of irreducible 3-manifolds, neither of which is $S^3$.
Let $\mathrm{Sph}(C_f)$ be the space of essential 2-spheres in $C_f$, just like in the proof of Proposition \ref{P:SplitLinksTwoSummands}.  As before, $\mathrm{Sph}(C_f)\simeq \ast$.
%Because $M$ is irreducible and not $S^3$, 
An element of $\mathrm{Sph}(C_f)$ is precisely a 2-sphere which bounds a 3-ball containing $\im(f)$.

Define $ES(M)$ by
\begin{align*}
ES(M):= \{& (f:S^1 \incl M, g: S^2 \incl M) :  
 \mbox{$\im(g)$ bounds a 3-ball containing $\im(f)$} \} \ / \ \Diff(S^2)
\end{align*}
Let $ES_f(M)$ denote the component of $f$ in $ES(M)$.  The fibration 
\[
\mathrm{Sph}(C_f) \to ES_f(M) \to \Emb_f(S^1, M)
\]
shows that $ES_f(M) \simeq \Emb_f(S^1, M)$ for each $f$ and hence $ES(M) \simeq \Emb(S^1, M)$. 

Furthermore, we claim that there is a homeomorphism
\[
ES(M) \cong \Emb(D^3, M) \x_{\Diff^+(D^3)} \Emb(S^1, D^3)
\]
where $\Diff(D^3)$ acts on the left on the first factor by post-composition and on the right on the second factor by pre-composition.  The map from the right-hand side to the left-hand side is given on representatives by $(g,f) \mapsto (g \circ f, g|_{\d D^3})$.  Its inverse is given on representatives by $(f,g) \mapsto (h, \ h^{-1} \circ f)$, where $h$ is an extension of $g$ to an embedding $D^3\incl M$.  The map is independent of the choice of $h$, since if $h$ and $k$ are two such choices, $(k, k^{-1}f) \sim (k(k^{-1}h), (h^{-1}k) k^{-1}f) \sim (h, h^{-1}f)$.

Since $M$ is an orientable 3-manifold, it is parallelizable, so
\[
\Emb(D^3, M)\x_{\Diff^+(D^3)} \Emb(S^1, D^3) \simeq (M \x SO_3) \x_{SO_3} \Emb(S^1, D^3) 
\simeq M \x \Emb(S^1, D^3).
\]
and the first equivalence in \eqref{Eq:SplitKnots} is proven.  
The second equivalence in \eqref{Eq:SplitKnots} follows from a result of Budney on spaces of knots in $D^3$; see \cite[Proposition 4.4]{Budney-Cohen} and \cite[Proposition 2.2]{Budney:Family}.
\end{proof}

\begin{corollary}
\label{TfSplit}
If $f$ is a knot in $S^1 \x D^2$ obtained as a composition $S^1 \overset{j}{\incl} D^3 \incl S^1 \x D^2$, then 
\[
\pushQED{\qed} 
\T_f \simeq S^1 \x \Emb_f(S^1, D^3) \simeq S^1 \x (SO_3 \x_{SO_2} (C_{\overline{j}} \x \K_{\underline{j}})).
\qedhere
\popQED
\]
\end{corollary}

\begin{corollary}
\label{VfSplit}
If $f$ is a knot in $S^1 \x S^1 \x I$ obtained as a composition $S^1 \overset{j}{\incl} D^3 \incl S^1 \x S^1 \x I$, then 
\[
\pushQED{\qed} 
\V_f \simeq S^1 \x S^1 \x \Emb_f(S^1, D^3)
\simeq S^1 \x S^1 \x (SO_3 \x_{SO_2} (C_{\overline{j}} \x \K_{\underline{j}})).
\qedhere
\popQED
\]
\end{corollary}

\begin{example}
\label{Ex:SplitUnknot}
Consider an unknot $f$ in a (small) 3-ball.  As explained by Hatcher \cite{HatcherSmaleConj}, the Smale Conjecture implies that $\Emb_f(S^1, D^3)$ is equivalent to the space of embeddings of Euclidean circles in $\R^3$, which in turn is equivalent to $SO_3$, by taking the tangent vector at the basepoint in $S^1$ and the normal vector to the circle.  
The equivalence to $SO_3$ also follows from the result of Budney mentioned at the end of the proof of Proposition \ref{P:SplitKnots}. 
If we turn $f$ into a knot in $S^1 \x D^2$ by embedding the 3-ball in $S^1 \x D^2$, as in Figure \ref{F:SplitUnknot}, then $\T_f \simeq S^1 \x SO_3$.  Similarly, we can turn $f$ into a knot in $S^1 \x S^1 \x I$, in which case $\V_f \simeq S^1 \x S^1 \x SO_3$.
\end{example}

\begin{figure}[h!]
\includegraphics[scale=0.2]{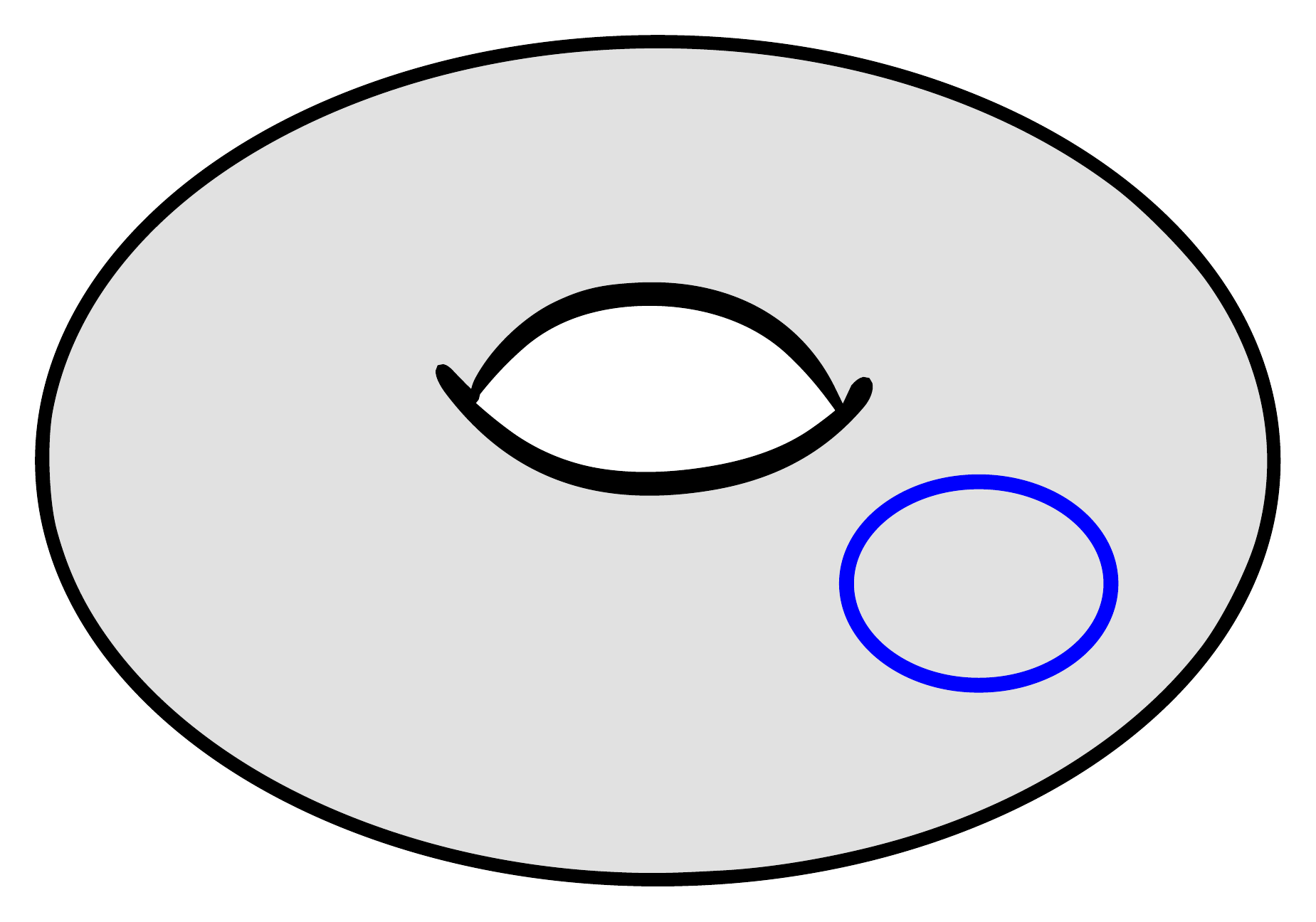}
\caption{An unknot $f$ in the solid torus corresponding to the 2-component unlink, which is a split link.  It satisfies $\T_f \simeq S^1 \x SO_3$.  For the embedding space $\V_f$ of a similar knot in the thickened torus, $\V_f \simeq S^1 \x S^1 \x SO_3$.}
\label{F:SplitUnknot}
\end{figure}

\section{Splittings of subgroups of rotations}
\label{S:Splittings}

In this section, we study certain canonical subgroups of rotations in the fundamental groups of the embedding spaces $\tL_F/SO_4$, $\L_f/SO_4$, $\T_f$, and $\V_f$.  We first deduce a splitting of the subgroup of meridional rotations in $\pi_1(\tL_F/SO_4)$ if $F$ is not a framed unknot (Theorem \ref{T:MeridiansFactor}).  We then consider the subgroup of rotations of the torus in $\pi_1(\T_f)$ (Theorem \ref{T:FactorsInTf}).  The subgroup of rotations in $\pi_1(\V_f)$ (Corollary \ref{FactorsInVf}) is easily understood via Theorem \ref{T:MeridiansFactor}.
The results below depend on analyses of the various cases treated in Sections \ref{S:Seifert}, \ref{S:Hyperbolic}, \ref{S:Splicing}, and \ref{S:SplitLinksTorus}.

\begin{theorem}
\label{T:MeridiansFactor}
Let $m \geq 0$, and let $F=(F_0, F_1,\dots, F_m)$ be an $(m+1)$-component framed link that is not a framed unknot. 
Then the loops of meridional rotations $\mu_0, \dots, \mu_m$ generate a $\Z^{m+1}$ subgroup that splits off of $\pi_1(\tL_F / SO_4)$ as a factor, i.e., $\pi_1(\tL_F / SO_4) \cong \Z^{m+1} \x G$ for some group $G$.
\end{theorem}

\begin{proof}
By Proposition \ref{LfModSO4isKpi1}, the meridional rotations $\mu_0,\dots, \mu_m$ generate a normal subgroup $N \cong \Z^{m+1}$.
Viewing $\pi_1(\tL_F/SO_4)$ as $\pi_0\Diff(C_F; \d C_F)$, we also see that each $\mu_i$ is central, since the corresponding diffeomorphism can be taken to have support in a small neighborhood of the $i$-th boundary component.  Obtaining the splitting requires a more careful analysis. 
We start with irreducible links, treating the Seifert-fibered, hyperbolic, and satellite cases separately.  

If $F$ is Seifert-fibered, then by Corollary \ref{SeifertFramedLinks}, $\pi_1(\tL_F/SO_4) \cong \Z^{2m} \x \PB_{m+r}$ for some $r\in \{0,1,2\}$.  If $F=KC_m$, then $N$ is  a factor $\Z^{m+1} < \Z^{2m}$.  Otherwise, $N$ is  $\Z^{m} \x Z(\PB_{m+r})$ (where the $\Z^{m}$ is a factor of $\Z^{2m}$), and the center of the pure braid group $Z(\PB_{m+r})$ is a direct factor of $\PB_{m+r}$.  

If $F$ is hyperbolic, then $N$ is the kernel in the short exact sequence \eqref{HypLinkSESMeridians}, which 
by Proposition 
%\ref{HMProp} and 
\ref{DiffS3LHypLink} can be written as $0 \to N \to \Z^{2(m+1)} \to \Z^{m+1} \to 0$.  As a short exact sequence of free abelian groups, it splits as a direct product.

If $F$ has a nontrivial companionship tree, we proceed as in the proof of Theorem \ref{MainT:FramedLinks} part (A), outlined at the beginning in Section \ref{S:Splicing}. 
We arbitrarily choose a distinguished component $F_0$ of $F$, thus determining a root vertex $v_R$ in $\mathbb{G}_F$.
We can then write $F=(\varnothing, \dots, \varnothing, J_1,\dots, J_r)\bowtie L$ where $L=(L_0,\dots,L_{n+r})$ is the link $\mathbb{G}_F(v_R)$, and where $n\geq 0$ and $r \geq 1$.
We proceed by induction on the maximum distance $d$ of a vertex from $v_R$, counted by the number of edges.  The basis case $d=0$ is covered by the previous two cases.  For the induction step, there are different cases according to whether $L$ is  $S_{p,q}$, $R_{p,q}$, $KC_{n+r}$ (which has two subcases), or a hyperbolic link.  The induction hypothesis is that the theorem holds for each $J_i$.

Suppose $F= (\varnothing, \dots, \varnothing, J_1,\dots, J_r) \bowtie S_{p,q}$ where $\gcd(p,q)=n+r$.
Let $|J_i|$ denote the number of components of $J_i$, and let $\ell:=\sum_{i=1}^r (|J_i|-1)$.
Then $m+1=n + \ell + 1$.  
By Proposition \ref{P:Cable},  $\pi_1(\tL_F/SO_4) \cong \PB_{n+r+1} \x \Z^{2n} \x \prod_{j=1}^r\pi_1(\tL_{J_i}/SO_4)$, and the subgroup $N \cong \Z^{m+1}$ of meridional rotations is the product of two subgroups, one a subgroup of $\PB_{n+r+1} \x \Z^{2n}$ the other a subgroup of $\prod_{j=1}^r\pi_1(\tL_{J_i}/SO_4)$.  The first factor is the subgroup $Z(\PB_{n+r+1}) \x \Z^n <  \PB_{n+r+1} \x \Z^{2n}$ where $\Z^n= \Z\langle \mu_1, \dots, \mu_n \rangle < \Z \langle \mu_1, \dots, \mu_n, \lambda_1, \dots, \lambda_n\rangle =\Z^{2n}$.
The second factor is the subgroup $\Z^{\ell} < \prod_{j=1}^r\pi_1(\tL_{J_i}/SO_4)$ generated by meridional rotations along all components of the $J_i$ except for the $J_{i,0}$ (along which the splicing is done).
The result is proven in this case because the center of $\PB_{n+2}$ splits as a direct factor, and $\Z^\ell$ is a factor of a factor $\Z^{\ell+r} < \prod_{j=1}^r\pi_1(\tL_{J_i}/SO_4)$ guaranteed by the induction hypothesis.

If $F = J \bowtie R_{p,q}$ with $\gcd(p,q)=n+r-1$, then Proposition \ref{P:SpliceIntoRpq} shows that a similar analysis as for $S_{p,q}$ applies, just with $\PB_{n+r+1}$ replaced by $\PB^{n+r}$ in the argument above.  (With our indexing for $R_{p,q}$, $r=1$ implies $n\geq 1$, even though for $S_{p,q}$ above, $r=1$ and $n=0$ is possible.)

We will treat the case $L=KC_{n+r}$ last, so suppose now that $F = (\varnothing, \dots, \varnothing, J_1, \dots, J_r) \bowtie L$ for some hyperbolic link $L=(L_0,\dots, L_{n+r})$.
Then by Proposition \ref{P:HypSplice}, 
\begin{equation}
\label{E:HypSplicePi1}
\pi_1(\tL_F/SO_4) \cong \Z^{2n+1} \x \left(\Z \ltimes \prod_{i=1}^r \pi_1(\L_{J_i}/SO_4) \right)
\end{equation}
The first $n+1$ meridional rotations $\mu_0, \dots, \mu_{n}$ generate a subgroup $N_0 =\Z^{n+1} < \Z^{2n+1}$ that is a factor of $\pi_1(\tL_F/SO_4)$.  
For $i=1,\dots, r$, let $N_i < \pi_1(\L_{J_i}/SO_4)$ be the subgroup generated by meridional rotations around the components $J_{i,1}, J_{i,2}, \dots$ along which no splicing is done.  Then $N_i \cong \Z^{|J_i|-1}$ and by the induction hypothesis, it is a factor  of $\pi_1(\L_{J_i}/SO_4)$.  
Recall from the proof of Proposition \ref{P:HypSplice} that the copy of $\Z$ in \eqref{E:HypSplicePi1} is a subgroup of the image $H\cong \widetilde{A}_F$ from the group $\pi_0 \Diff(C_F; \d C_F) (\cong \pi_1(\tL_F/SO_4))$ in \eqref{Eq:SpliceSES}.
While it may act nontrivially on each factor $\pi_1(\L_{J_i}/SO_4)$ (e.g., as in the case of $\Wh \bowtie \Wh$ in Example \ref{Ex:WhSpliceWh}), it acts trivially on each subgroup $N_i$, since elements of $N_i$ are rotations of components of $F$.
Since $N=\prod_{i=0}^r N_i$, we conclude that $N$ is a factor.  (In the special case that all the $J_i$ are knots, $N=N_0$ is simply the factor $\Z^{n+1} < \Z^{2n+1}$.)

Next suppose $F= (\varnothing, \dots, \varnothing, J_1, \dots, J_r) \bowtie KC_{n+r}$ where no splicing is done along the special component of $KC_{n+r}$.  
After possibly rechoosing $F_0$ from among the leaf half-edges on $KC_{n+r}$, we may assume $L_0$ is the special component.
If necessary, reorder the $J_i$ so that $J_1, \dots, J_q$ are links with multiple components and $J_{q+1},\dots, J_r$ are knots; thus $0 \leq q \leq r$.   
By Proposition \ref{P:ConnectSum},
\begin{equation}
\label{PiLF-KCSplice}
\pi_1(\tL_F/SO_4) \cong \Z\langle \mu_1, \dots, \mu_n, \lambda_1, \dots, \lambda_n \rangle \x \prod_{i=1}^q \pi_1(\tL_{J_i}/SO_4) \x \left(\B_F \ltimes \prod_{i=q+1}^r \pi_1(\tL_{J_i}/SO_4) \right)
\end{equation}
since $\B_F$ acts trivially on the multiple-component links $J_1, \dots, J_q$,
We first consider the subgroup $N_0 := \langle \mu_0 \rangle$.
By the induction hypothesis, $\prod_{i=q+1}^r \pi_1(\tL_{J_i}/SO_4)$  has a direct factor of $\Z^{r-q}$ generated by meridional rotations of $J_{q+1}, \dots, J_r$.  
Since the action of $\B_F$ factors through $\mathfrak{S}_r < \mathfrak{S}_r^\pm$, it is trivial on the diagonal $\Delta$ in $\Z^{r-q}$, so $\Delta$ splits as a direct factor of $\pi_1(\tL_F/SO_4)$.
Now notice that $N_0$ is the diagonal (and thus a direct factor) in $\Z\langle \lambda_1, \dots, \lambda_n \rangle \x  \Delta \cong \Z^{n+1}$.
Hence $N_0$ is a direct factor of $\pi_1(\tL_F/SO_4)$.  
Let $N_1$ be the subgroup generated by $\mu_1, \dots, \mu_m$.  Then $N_1$ is the product of the factor $\Z\langle \mu_1, \dots, \mu_n\rangle$ with a direct factor $\Z^{m-n} < \prod_{i=1}^q \pi_1(\tL_{J_i}/SO_4)$ by the induction hypothesis.  
The subgroup $N$ of all meridional rotations is then the direct factor $N_0 \x N_1$.  (In the special case where all the $J_i$ are knots, $N=N_0$ is just the diagonal in $\Z^n \x \Z^r < \Z^{2n}  \x \B_F \ltimes \prod_{i=1}^r \pi_1(\tL_{J_i}/SO_4)$.)

Now suppose $F= (\varnothing, \dots, \varnothing, J_0, \dots, J_{r-1}) \bowtie KC_{n+r}$ where the special component of $KC_{n+r}$  is (without loss of generality) $L_{n+1}$.  
In this case the roles of $L_0, \dots, L_n$ are essentially interchangeable like boundary components of the surface $P_n$, but we view $L_0$ as corresponding to the outer boundary in $P_n$.  
Then the isomorphism \eqref{PiLF-KCSplice} still holds, and $N_1$ is as described in the previous case.  However, $N_0$ 
is now the diagonal in $\Z\langle \lambda_1, \dots, \lambda_n \rangle \x \langle \tau \rangle$ where $\tau$ is the full twist in $\B_F$.  Note that $\langle \tau \rangle$ is the center $Z(\B_F)$, since $Z(\PB_{n+r})=\langle \tau \rangle=Z(\B_{n+r})$.
If none of $J_1,\dots, J_{r-1}$ is a knot, then we are done because $\B_F$ lies in $\PB_{n+r}$, which implies $Z(\B_F)$ splits as a direct factor of $\B_F$.
If some of the $J_1,\dots, J_{r-1}$ are knots, then there are no knots 
in the subtree of $\mathbb{G}_F$ attached to the half-edge corresponding to the special component of $KC_{n+r}$, 
by Alexander's theorem on embedded tori in $S^3$.  Then by rechoosing the distinguished component $F_0$ of $F$ to lie in this subtree, we are in one of the previous cases.
This completes the induction step and thus the proof in for an arbitrary irreducible link.

Finally, suppose $F$ is a split link consisting of irreducible sublinks $F_1, \dots, F_k$.  For $i=1,\dots, k$, let $N_i$ be the subgroup of meridional rotations of the components in $F_i$.  Thus $N=N_1 \x \dots \x N_k$.  Consider the composition 
\[
N_i \incl \pi_1(\tL_F/SO_4) \to \pi_1(\tL_{F_i}/SO_4) \twoheadrightarrow N_i
\] 
where the second map is the restriction to the sublink $F_i$ and where the projection onto $N_i$ exists because we have established the Theorem for the irreducible link $F_i$.  Since this composition is the identity, each $N_i$ is a direct factor of $\pi_1(\tL_F/SO_4)$ and thus so is $N$.
\end{proof}

\begin{theorem}
\label{T:FactorsInTf}
Let $f=(f_0,f_1)$ be any 2-component link corresponding to a knot in the solid torus.  Then $\mu_1$, the longitudinal rotation of the solid torus, generates a copy of $\Z$ that splits as a direct factor of $\pi_1(\T_f)$.
If $f$ is irreducible and not the Hopf link, then $\lambda_1$, the meridional rotation of the solid torus, generates another direct factor of $\Z$ in $\pi_1(\T_f)$.  
If $f$ is a split link, then $\lambda_1$ generates a central subgroup isomorphic to $\Z/2$.
\end{theorem}

\begin{proof}
We separately treat the cases of irreducible and split links.

\textbf{Case 1: $f$ is an irreducible link}.
Since $\pi_1(\T_f) \cong \pi_1(\tL_F/SO_4)\,/\,\langle \mu_0\rangle$, where $\mu_0$ is the meridional rotation around the knotted component, the claim about $\mu_1$ is immediate from Theorem \ref{T:MeridiansFactor} in this case.  
So we may suppose $f$ is not the Hopf link.  
We must show that the image of $\langle \mu_1, \lambda_1 \rangle$ in $\pi_1(\T_f)$ is a direct factor isomorphic to $\Z^2$.  
Equivalently, we must show that the subgroup $\langle \mu_1, \lambda_1 \rangle$  of $ \pi_1(\tL_F/SO_4)$ is a factor isomorphic to $\Z^2$ which intersects $\langle \mu_0 \rangle$ trivially.

We consider the companionship tree $\mathbb{G}_F$ of $F$ and proceed by induction on the distance $d$, counted by the number of vertices, from the half-edge corresponding to $F_1$ to the root half-edge corresponding to $F_0$.  
The basis case $d=1$ is when $\mu_1$ and $\mu_0$ both correspond to half-edges on the root vertex $v_R$.  There are three subcases depending on the type of the corresponding link $\mathbb{G}_F(v_R)$:
\begin{itemize}[leftmargin=0.25in]
\item
If $\mathbb{G}_F(v_R)$ is  $S_{p,q}$, then $F=S_{p,q}$, and we saw that we can take $\mu_1$ and $\lambda_1$ to generate $\pi_1(\T_f) \cong \Z^2$.
\item
If $\mathbb{G}_F(v_R)$ is $KC_{1+r}$, then $F = (\varnothing, J_1, \dots, J_r) \bowtie KC_{1+r}$ for some $r \geq 1$ where all the $J_i$ are knots and $L_0$ is the component of $KC_{1+r}$ that links nontrivially with the other components.  
Then $\pi_1(\T_f) \cong (\Z^2 \x (\B_F \ltimes \prod \pi_1(\tL_{J_i}/SO_4)) / \langle \mu_0\rangle$, and $\langle \mu_0\rangle$ is the diagonal in a free abelian factor {properly} containing the factor $\Z^2$.  
Thus the image of this factor in the quotient is again $\Z^2$, and it is generated by $\mu_1$ and $\lambda_1$.  
\item
If $\mathbb{G}_F(v_R)$ is a hyperbolic link $L=(L_0,\dots, L_{1+r})$, then $F = (\varnothing, J_1, \dots, J_r) \bowtie L$ where all the $J_i$ are knots.
Then $\pi_1(\T_f) \cong \Z^{2} \x (\Z \ltimes \prod \pi_1(\tL_{J_i}/SO_4)$, and the factor $\Z^2$ is generated by $\mu_1$ and $\lambda_1$.  
\end{itemize}
This completes the proof of the basis case.
\smallskip

Now suppose we have proven the statement for all $F$ where $d=n-1$, and that $\mathbb{G}_F$ has $d=n\geq 2$ vertices between the two half-edges.  
Logically, we need not consider separate cases, but we observe that $F$ must be of one of the following three forms: 
\begin{itemize}[leftmargin=0.25in]
\item $J_1 \bowtie S_{p,q}$ (with $\gcd(p,q)=1$), 
\item $(J_1,\dots, J_r) \bowtie KC_r$ (where $L_0$ is the component of $KC_r$ linking nontrivially with $L_1, \dots, L_r$), or 
\item  $(J_1,\dots, J_r) \bowtie L$ for a hyperbolic KGL $L=(L_0, \dots, L_r)$.  
\end{itemize}
In each case exactly one $J_i$ is a 2-component link, while the other $J_i$ are knots.  Without loss of generality, $J_1$ is the 2-component link.  (The component $J_{1,1}$ along which no splicing is done must be the unknot.)
By Lemma \ref{SpliceSES}, $\pi_1(\tL_F/SO_4) \cong H \ltimes K$ for a certain group $H$, where $K=\prod_{i=1}^r \pi_1(\tL_{J_i}/SO_4)$.  The action $H \to \mathfrak{S}_r^\pm \to K$ corresponds to a signed permutation $\sigma\in \mathfrak{S}_r^\pm$ with $\sigma(1)=1+$ because $J_1$ has multiple components.
Thus $\pi_1(\tL_{J_1}/SO_4)$ is a direct factor of $\pi_1(\tL_F/SO_4)$.
By the induction hypothesis, $\pi_1(\tL_{J_1}/SO_4)$ has a direct factor $\Z\langle \mu_1, \lambda_1\rangle$, which is then a direct factor of $\pi_1(\tL_F/SO_4)$ which intersects $\langle \mu_0\rangle$ trivially.
This completes the induction step and the proof for irreducible $f$.

\smallskip

\textbf{Case 2: $f$ is a split link}.  By Proposition \ref{P:SplitKnots}, $\T_f \simeq (S^1 \x D^2) \x \Emb_{f_1}(S^1, D^3)$, so $\pi_1(\T_f) \cong \Z \x \pi_1\Emb_{f_1}(S^1, D^3)$.  The factor of $\Z$ is generated by $\mu_1$, so the claim about $\mu_1$ is now proven in all cases.  

It remains to prove the claim about $\lambda_1$.  
By a result of Budney and Cohen \cite[Proposition 4.4]{Budney-Cohen}, 
\[
\Emb_{f_1}(S^1, D^3) \simeq SO_3 \x_{SO_2} (C_{f_1} \rtimes \K_{f_1}).
\]
Here $\K_{f_1}$ is the component in $\K$ of the long knot associated to the closed knot $f_1$, and $C_{f_1} \rtimes \K_{f_1}$ is the total space of a fibration with a section, where the base is $\K_{f_1}$ and the fiber is $C_{f_1}$.  It is defined by $C_{f_1} \rtimes \K_{f_1} := \{(g, p): g \in \K_{f_1}, p \in C_g\}$.  The group $SO_2$ acts on $SO_3$  via a standard inclusion, and it acts on $C_{f_1} \rtimes \K_{f_1}$ diagonally, by rotating around the long axis.

The element $\lambda_1$ is the image of the generator $\theta \in \pi_1(SO_3) \cong \Z/2$ under the projection  
\begin{equation}
\label{EmbS1D3Projection}
SO_3 \x (C_{f_1} \rtimes \K_{f_1}) \to
SO_3 \x_{SO_2} (C_{f_1} \rtimes \K_{f_1}).
\end{equation}
Thus $\langle \lambda_1\rangle \cap \langle \mu_1\rangle = \{e\}$.  
We will now show that $\lambda_1$ is nontrivial.  The fibration \eqref{EmbS1D3Projection} yields the exact sequence
\[
\pi_1(SO_2) \to \pi_1(SO_3) \x (\pi_1(C_{f_1}) \rtimes \pi_1(\K_{f_1})) \to \pi_1(SO_3 \x_{SO_2} (C_{f_1} \rtimes \K_{f_1})) \to \{e\}
\]
and we claim the first map is injective.  The Gramain loop provides a $\Z$ factor in $\pi_1(\K_{f_1})$ (by Proposition \ref{FramedKnotsInS3VsLongKnots} and Theorem \ref{T:MeridiansFactor}, as mentioned in Remark \ref{R:GramainSplitting}).  If $f_1$ is not the unknot, the image of the generator $\tau \in \pi_1(SO_2)$ in $\pi_1(\K_{f_1})$ is a generator of this $\Z$.   If $f_1$ is the unknot, $\tau$ maps to the generator of $\pi_1(C_{f_1} \; \{e\}) \cong \Z$.  This establishes the desired injectivity.  Thus the kernel of 
\begin{equation}
\label{EmbS1D3ProjectionPi1}
\pi_1(SO_3 \x (C_{f_1} \rtimes \K_{f_1})) \twoheadrightarrow
\pi_1(SO_3 \x_{SO_2} (C_{f_1} \rtimes \K_{f_1}))
\end{equation}
 is free abelian of rank one.  Hence the image $\lambda_1$ of the generator $\theta \in \pi_1(SO_3)$ under this map is nontrivial, as claimed.
The subgroup $\langle \lambda_1 \rangle \cong \Z/2$ is central, as the image of a central subgroup under the quotient map \eqref{EmbS1D3ProjectionPi1}.
\end{proof}

\begin{remark}
A shorter argument proves a weaker version of Theorem \ref{T:FactorsInTf} for irreducible $f$.  
Consider the fibration $\Diff(M; \d \nu(f)) \to \Diff(S^1 \x D^2) \to \widetilde{\T}_F$, where $\widetilde{\T}_F \simeq \tL_F/SO_4$ is the space of the framed knot $F$ in the solid torus and $M = S^1 \x D^2 - \nu(f)$ is the exterior of $f$.
Since $f$ is irreducible,  $M$ is Haken.
Hence the fiber is a $K(\pi,0)$ by Hatcher and Ivanov's result (which actually applies to spaces of diffeomorphisms fixing any nonempty union of boundary components).  
Thus $\Z\langle \lambda_1, \mu_1\rangle \cong \pi_1 \Diff(S^1 \x D^2)$ is a subgroup of $\pi_1\widetilde{\T}_f$.  
By viewing $\pi_1\widetilde{\T}_f$ as $\pi_0 \Diff(M; \d M)$ and taking the diffeomorphism corresponding to $\lambda_1$ or $\mu_1$ to be supported near a component of $\d M$, we see that this subgroup is central.  
However, to reach the same conclusion about $\T_f$, one still needs to check that this subgroup intersects $\Z\langle\mu_0\rangle < \pi_1(\tL_F/SO_4)$ trivially.  This argument also does not  give the splitting.  
\end{remark}

\begin{corollary}
\label{TfisS1xLf}
If $f=(f_0,f_1)$ is a link with $f_1$ the unknot, then $\T_f \simeq S^1 \x (\L_f/SO_4)$ where the factor of $S^1$ corresponds to meridional rotation around $f_1$ (or longitudinal rotation of the solid torus).
\qed
\end{corollary}

\begin{remark}
By a similar case-by-case analysis, one can show that the loop $\rho$ given by reparametrizing the knot is always nontrivial in $\pi_1(\T_f)$.  However, it is not always independent of $\mu_1$ and $\lambda_1$, as shown by the case where $f$ corresponds to a $(p,q)$-Seifert link.  In cases where it is independent of $\mu_1$ and $\lambda_1$, it is not always true that there is a generating set containing all three of $\mu_1, \lambda_1$, and $\rho$; see for instance the hyperbolic link in Example \ref{Ex:HypLinkFromFig8}.  Nor does the subgroup $\Z\langle \rho\rangle$ always split as a factor.  Indeed, if $F=(\varnothing, J, J) \bowtie KC_3$ where $J$ is a knot, then $\pi_1(\tL_F/SO_4) \cong \Z^2 \x \B_{1,2} \ltimes \pi_1(\tL_J/SO_4)^2$, and $\rho$ corresponds to a full twist of the disk, which 
generates the center $Z(\B_{1,2})$ of $\B_{1,2}$ (since it generates the centers of $\B_3$ and $\PB_3$).  But this center does not split, since the abelianization of $\B_{1,2}$ is $\Z$, whereas the abelianization of $\B_{1,2}/Z(\B_{1,2})$ is $\Z \x \Z/2$.
\end{remark}

The next result is a corollary to Theorem \ref{T:MeridiansFactor} and the analogue of Theorem \ref{T:FactorsInTf} with $S^1 \x D^2$ replaced by $S^1 \x S^1 \x I$.

\begin{corollary}
\label{FactorsInVf}
Let $f$ be any 2-component knot in the thickened torus.  Then $\mu_1=\lambda_2$ and $\lambda_1=\mu_2$ generate a copy of $\Z^2$ that  splits as a direct factor of $\pi_1(\V_f)$.
\end{corollary}

\begin{proof}
If $f$ corresponds to an irreducible 3-component link, this is immediate from Proposition \ref{VfIsLfModSO4} and Theorem \ref{T:MeridiansFactor} by viewing the two loops as $\mu_1$ and $\mu_2$.

If $f$ corresponds to a split 3-component link, then we may think of $f$ as a knot contained in some 3-ball in $S^1 \x S^1 \x I$.  
By Proposition \ref{P:SplitKnots}, $\V_f \simeq (S^1 \x S^1 \x I) \x \Emb_f (S^1, D^3)$,  so $\pi_1(\V_f) \cong \Z^2 \x \pi_1\Emb_f (S^1, D^3)$, where $\mu_1$ and $\lambda_1$ generate the factor of $\Z^2$.
\end{proof}

\section{Tables summarizing examples}
\label{S:Tables}

Tables  \ref{T:ExamplesSolidTorus} and \ref{T:ExamplesThickenedTorus} summarize our examples of 2-component links corresponding to knots in a solid torus $U$ and 3-component links corresponding to knots in a thickened torus $\nu(T)=\nu(\d U)$ respectively.
We list the homotopy types of the embedding spaces and their fundamental groups, with generators.  For simplicity, generic knots $J$ or $K$ are taken to be torus knots.  
The loop $\mu_1$ is a meridional rotation around the unknot (or longitudinal rotation $\lambda$ of $U$), while $\lambda_1$ is a longitudinal rotation along (i.e.~reparametrization of) the unknot (or meridional rotation $\mu$ of $U$).  
Similarly, $\lambda_0$ is reparametrization of the knotted component.  
Rotations along an incompressible, non-boundary torus are in some cases denoted $\mu', \lambda'$.  
The Gramain loop of  $J$ is denoted $g_J$.
We abbreviate $C_n(\R^2):=\Conf(n,\R^2)$, and $\beta_{ij}$ (respectively $p_{ij}$) stand for generators of the braid (respectively pure braid) group.  The group $\B_{1,n-1} < \B_{n}$ consists of braids whose permutations fix the first strand and can be viewed as the annular $(n-1)$-strand braid group $\mathcal{CB}_{n-1}$.  
The links are grouped as split links, Seifert-fibered links, hyperbolic links, splices into Seifert-fibered KGL's, and splices involving hyperbolic KGL's.

%\newpage
\newgeometry{margin=1cm} 
% modify this if you need even more space
\begin{landscape}

%\noindent
%\textbf{Table}.  
%
%\bigskip

\begin{table}
\caption{Summary of examples of links corresponding to knots in a solid torus}

\centering
%\begin{center}
\begin{tabular}{|c| c| c| c| c|}
\hline 
\rule{0pt}{1.0\normalbaselineskip}
Example $\#$ & 2-comp.~link $f$  & 
$\T_f$ & 
$\pi_1(\T_f)$ &  generators of $\pi_1(\T_f)$ \\[0.5ex] 
\hline 
\hline
\rule{0pt}{1.0\normalbaselineskip}
Example \ref{Ex:SplitUnknot} & unlink & $S^1 \x SO_3$ & $\Z \x \Z/2$ & $\mu_1, \lambda_0$ \\[0.5ex]
\hline 
\hline
\rule{0pt}{1.0\normalbaselineskip}
- & Hopf link & $S^1$ & $\Z$ & $\mu_1 = \lambda_0$ \\[0.5ex]
\hline 
\rule{0pt}{1.0\normalbaselineskip}
Figs.~\ref{F:3-4-SeifertLink} and \ref{F:2-5-SeifertLink} & Seifert link $S_{p,q}$, $(p,q)=1$ & $(S^1)^2$ & $\Z^2$ & $\mu_1,\ \lambda_1$ \\[0.5ex]
\hline 
\hline
\rule{0pt}{1.0\normalbaselineskip}
Example \ref{Ex:WhiteheadLink} & Whitehead link $\Wh$ & $(S^1)^3$ &  $\Z^3$ & $\mu_1,\ \lambda_1,\ \lambda_0$ \\[0.5ex]
\hline 
\rule{0pt}{1.0\normalbaselineskip}
Example \ref{Ex:HypLinkFromFig8} & hyp.~link L8n1=$8^2_{16}$ (from $4_1$) &  $(S^1)^3$ & $\Z^3$ & $\mu_1,\ \lambda_1,\  \frac{1}{2}(\lambda_0 +\lambda_1)$ \\[0.5ex]
\hline 
\rule{0pt}{1.0\normalbaselineskip}
Example \ref{Ex:HypLinkFromTrefoil} & hyp.~link from $T_{2,k}$, $k$ odd  & $(S^1)^3$  & $\Z^3$ & $\mu_1,\ \lambda_1,\ 
 \frac{1}{k}(\lambda_0 + \lambda_1)$ \\[0.5ex]
\hline 
\rule{0pt}{1.0\normalbaselineskip}
Example \ref{Ex:HypLinkFrom8-18} & hyp.~link from $8_{18}$  & $(S^1)^3$  & $\Z^3$ & 
$\mu_1,\ \lambda_1,\ \frac{1}{4}(\lambda_0 + \lambda_1)$ \\[0.5ex]
\hline 
\hline
\rule{0pt}{1.0\normalbaselineskip}
Example \ref{Ex:CableWh} & $\Wh \bowtie S_{p,q}$  & $(S^1)^4$  & $\Z^4$ & $\mu_1,\ \lambda_1,\ \mu', \ \lambda'$ \\[0.5ex]
\hline 
\rule{0pt}{1.0\normalbaselineskip}
Example \ref{Ex:ConnectSumJPlusEmpty} & $(\varnothing, J) \bowtie KC_2$ & $(S^1)^3$  & $\Z^3$ & $\mu_1,\ \lambda_1,\ \lambda_0$ \\[0.5ex]
\hline 
\rule{0pt}{1.0\normalbaselineskip}
Example \ref{Ex:MoreGeneralConnectSum} & $(\varnothing, J, J) \bowtie KC_3$ & $((S^1) \x (C_3(\R^2) \x_{\mathfrak{S}_2} (S^1)^2)$ & 
$\Z \x (\B_{1,2} \ltimes \Z^2)$ & $\mu_1,\, \beta_{23},\, p_{12},\, p_{13}$,  \\[0.5ex]
Fig.~\ref{F:SumTrefoils} (a) & & $\mathfrak{S}_2 \cong \mathrm{Aut}\{2,3\}$ & & $g_{J,1}, \ g_{J,2}$ \\[0.5ex]
%
%\hline 
%\rule{0pt}{1.0\normalbaselineskip}
% & $(\varnothing, J, J,J) \bowtie KC_4$ & $((S^1)^2 \x (C_4(\R^2) \x_{\mathfrak{S}_3} (S^1)^3))/S^1$ & $(\Z^2 \x (\mathcal{B}_3 \ltimes \Z^3))/\langle\ ((0,1),(1,1,1))\ \rangle \cong$ & $\mu_1,\, \beta_{23},\, \beta_{34},$ \\[0.5ex]
%& &  & $\Z \x (\mathcal{B}_3 \ltimes \Z^3) \cong \Z^2 \x (\mathcal{B}_3 \ltimes \Z^2)$
%& $\, g_{J,1},\, g_{J,2}$, and $g_{J,3}$ or $\lambda_1$ \\[0.5ex]
\hline
\hline
\rule{0pt}{1.0\normalbaselineskip}
Example \ref{Ex:WhSpliceWh} & $\Wh \bowtie \Wh$ & $(S^1)^2 \x (S^1 \x_{\Z/2} (S^1)^2)$ & 
$\Z^2 \x \langle a,b,c \ |\ [b,c]=1, b^a=b^{-1}, c^a=c^{-1} \rangle$ & 
$\mu_1, \ \lambda_0, \ a=\lambda_0^{1/2}  \lambda_1^{1/2}, \ b, \ c$ \\[0.5ex]
\hline
\rule{0pt}{1.0\normalbaselineskip}
Example \ref{Ex:BorrPartialSplice} & $(\varnothing, J) \bowtie \mathrm{Borr}$  & $(S^1)^4$  & $\Z^4$ & $\mu_1,\ \lambda_1,\ \lambda_0$, \ $g_J$ \\[0.5ex]
\hline
\rule{0pt}{1.0\normalbaselineskip}
Example \ref{Ex:SpliceInto3CompHypLink} & $(\varnothing, J) \bowtie (L_0,L_1,L_2)$ & $(S^1)^2 \x (S^1 \x_{\Z/2} S^1)$  &  $\Z^2 \x \langle a,b \ | \ b^{a} = b^{-1}\rangle$ & $\mu_1,\ \ \lambda_1,$ \\
 & $L$ shown in Fig.~\ref{F:SpliceInto3CompHypLink} & & & 
 $a=\lambda_0^{1/2} \lambda_1^{1/2},\  b=g_J$ \\[0.5ex]
\hline
\rule{0pt}{1.0\normalbaselineskip}
Example \ref{Ex:5CompHypLink} & $(\varnothing, J, K,K) \bowtie (L_0,\dots,L_4)$ & 
$(S^1)^2 \x \left(S^1 \x_{\Z/2} \left(S^1 \x (S^1)^2\right)\right)$ &  
$\Z^2 \x \langle a,b,c,d \ | \ [b,c]=[c,d]=[b,d]=1,$  & $\mu_1,\ \ \lambda_1, \ a=\lambda_1^{1/2} \lambda_0^{1/2},$ \\
 & $L$ shown in Fig.~\ref{F:5CompHypLink} & & $b^a=b^{-1}, c^a=d, d^a=c \rangle$ & 
 $b=g_J,\ c=g_{K,1},\ d=g_{K,2}$  \\[0.5ex]
\hline
\rule{0pt}{1.0\normalbaselineskip}
Example \ref{Ex:Stoimenow}, & $(\varnothing, J, \dots, J) \bowtie (L_0, \dots, L_{r+1})$ & 
$(S^1)^2 \x \left(S^1 \x_{\Z/r} (S^1)^r \right)$ &
$\Z^2 \x \langle a, b_1,\dots, b_r \ | \ b_i^a = b_{i+1}\rangle$ & $\mu_1,\ \ \lambda_1, \ a=\lambda_0^{1/r},$ \\
Fig.~\ref{F:Stoimenow} & $(L_0, \dots, L_r) = L_{\mathrm{Stoimenow}}$ & & indices taken mod $r$ &
$b_1=g_{J,1},\dots, b_r=g_{J,r}$ \\[0.5ex]
\hline
\rule{0pt}{1.0\normalbaselineskip}
Example \ref{Ex:Sakuma}, & $(\varnothing, J, \dots, J)  \bowtie (L_0, \dots, L_{r+1})$ & 
$(S^1)^2 \x \left(S^1 \x_{\Z/2r} (S^1)^r \right)$ & 
$\Z^2 \x \langle a, b_1,\dots, b_r \ | \ b_i^a = b_{i+1}^{-1}\rangle$ & $\mu_1,\ \ \lambda_1, \ a=\lambda_0^{1/r}\lambda_1^{1/2},$ \\
Fig.~\ref{F:Sakuma} & $(L_0, \dots,L_r) = L_{\mathrm{Sakuma}}$& & indices taken mod $r$ & 
$b_1= g_{J,1},\dots, b_k=g_{J,r}$ \\[0.5ex]
\hline
\rule{0pt}{1.0\normalbaselineskip}
%(exercises) & & & & \\[0.5ex]
%\hline
%\rule{0pt}{1.0\normalbaselineskip}
Exercise \ref{Ex:WhiteheadDoubleInSolidTorus} (a) 
& %$((\varnothing, J) \bowtie KC_2) \bowtie \Wh$ 
& %$(S^1)^5$ 
& %$\Z^5$ 
& %$\mu_1,\ \lambda_1,\ \lambda_0, \  \lambda', \ g_J$ 
\\[0.5ex]
\hline
\rule{0pt}{1.0\normalbaselineskip}
%Ex.~\ref{Ex:AnotherWhiteheadDouble} 
Exercise \ref{Ex:WhiteheadDoubleInSolidTorus} (b) 
& %$(\varnothing, J \bowtie \Wh) \bowtie KC_2$ 
& % $(S^1)^3 \x (S^1 \x_{\Z/2} S^1) \simeq$ 
& %$\Z^3 \x \langle a,b \ | \ b^a = b^{-1}\rangle$ 
& %$\mu_1,\ \ \lambda_1 = g_{\Wh(J)} = \mu', \ \lambda_0$ 
 \\[0.5ex]
% & &  %$(S^1) \x \tL_{\Wh(J)}/SO_4$ 
% & & %$a=(\lambda')^{1/2} \mu_J^{1/2},\ \ g_J=\mu_J$  
% \\[0.5ex]
\hline
%NO ROOM FOR MORE
\rule{0pt}{1.0\normalbaselineskip}
%Ex.~\ref{Ex:ThirdExample} 
Exercise \ref{Ex:WhiteheadDoubleInSolidTorus} (c) 
& %$(\Wh, J) \bowtie KC_2$ 
& %$(S^1)^5$ 
& %$\Z^5$ 
& %$\mu_1,\ \lambda_1,\ \lambda_0, \  \lambda', \ g_J$ 
\\[0.5ex]
\hline
%\rule{0pt}{1.0\normalbaselineskip}
\end{tabular}
%\end{center}

\label{T:ExamplesSolidTorus}
\end{table}

\end{landscape}
\restoregeometry

\begin{table}[h!]
\caption{Summary of examples of links corresponding to knots in a thickened torus}

\centering
%\begin{center}
\begin{tabular}{|c| c| c| c| c|}
\hline 
\rule{0pt}{1.0\normalbaselineskip}
Example $\#$ & 3-comp.~link $f$  & $\V_f$ &  $\pi_1(\V_f)$ &  generators of $\pi_1(\V_f)$ \\[0.5ex] 
\hline 
\hline
\rule{0pt}{1.0\normalbaselineskip}
Ex.~\ref{Ex:SplitUnknot} & Hopf link $\sqcup$ unknot & $(S^1)^2 \x SO_3$ & $\Z^2 \x \Z/2$ & $\mu, \lambda, \lambda_0$ \\[0.5ex]
\hline
\rule{0pt}{1.0\normalbaselineskip}
Fig.~\ref{F:3-4-SeifertLink} (c) ($(p,q)=(3,4)$),  & $R_{p,q}$ & $(S^1)^2$ & $\Z^2 $ & $\mu, \lambda$ \\[0.5ex]
\rule{0pt}{1.0\normalbaselineskip}
Ex.~\ref{Ex:T_3,3} ($(p,q)=(1,1)$)  &  &  &  &  \\[0.5ex]
\hline
\rule{0pt}{1.0\normalbaselineskip}
Ex.~\ref{Ex:KC2}  & $KC_2$ & $(S^1)^2$ & $\Z^2 $ & $\mu, \lambda$ \\[0.5ex]
\hline
\hline
\rule{0pt}{1.0\normalbaselineskip}
Ex.~\ref{Ex:L6a5} & hyp.~link L6a5$=6^3_1$ & $(S^1)^3$ & $\Z^3 $ & $\mu, \lambda, \lambda_0$ \\[0.5ex]
\hline
\hline
\rule{0pt}{1.0\normalbaselineskip}
Ex.~\ref{Ex:WhSpliceRpq} & $(\varnothing, \Wh)\bowtie R_{p,q}$ & $(S^1)^4$ & $\Z^4 $ & $\mu, \lambda, \lambda_0, \mu'$ \\[0.5ex]
\hline
%\rule{0pt}{1.0\normalbaselineskip}
%
%\hline
\end{tabular}
%\end{center}

\label{T:ExamplesThickenedTorus}
\end{table}

%\vspace{-2pc}

     \bibliographystyle{alpha}    
    \bibliography{refs}

\end{document}